\documentclass[11pt]{article}

\usepackage{amsfonts,amssymb, latexsym, amsmath, amsthm,mathrsfs, verbatim,geometry}
\usepackage{graphics}
\usepackage{slashed}
\usepackage{url,color}
\usepackage{enumerate}
\usepackage{esint}
\usepackage{pifont}
\usepackage{upgreek}
\usepackage{times}
\usepackage{calligra}
\usepackage{graphicx}
\usepackage{caption}
\usepackage{tikz}
\usetikzlibrary{calc}
\usetikzlibrary{calc}
\usetikzlibrary{arrows.meta}
\usetikzlibrary{hobby}
\usepackage{tkz-euclide}
\usetkzobj{all}

\usepackage{hyperref}
\hypersetup{
    colorlinks,
    citecolor=red,
    filecolor=green,
    linkcolor=blue,
    urlcolor=black
}

\newcommand\J{\mathscr J}

\def\D{\mathcal{D}}
\def\be{\begin{equation}}
\def\ee{\end{equation}}
\def\bea{\begin{eqnarray}}
\def\eea{\end{eqnarray}}
\def\beas{\begin{eqnarray*}}
\def\eeas{\end{eqnarray*}}

\def\g{\partial}

\def\l{\lambda}
\def\pa{\partial }

\def\r{ r\partial_r}
\def\rr{(r\partial_r)}
\def\et{(\frac{r^n}{\tau})}

\def\l{\lambda}
\def\pet{p_{\l,-\frac2n}\et}
\def\lv{\left\vert}
\def\rv{\right\vert}
\def\bcr{\begin{color}{red}}
\def\ec{\end{color}}
\def\phia{\phi_{\text{app}}}
\def\bcv{\begin{color}{violet}}
\def\bcg{\begin{color}{green}}
\def\chid{\chi_{\text{dust}}}
\def\ve{\varepsilon}

\def\DL{\Lambda}

\newtheorem{theorem}{Theorem}[section]
\newtheorem{definition}[theorem]{Definition}
\newtheorem{proposition}[theorem]{Proposition}

\newtheorem{corollary}[theorem]{Corollary}
\newtheorem{lemma}[theorem]{Lemma}
\newtheorem{remark}[theorem]{Remark}
 
\renewcommand{\theequation}{\arabic{section}.\arabic{equation}}

\title{Continued Gravitational Collapse for Newtonian Stars}

\author{Yan Guo\thanks{Division of Applied Mathematics, Brown University, Providence, RI 02912, USA, Email: Yan\_Guo@brown.edu.}, \ Mahir Had\v zi\'c\thanks{Department of Mathematics, King's College London, Strand, London WC2S 2LR, UK. Email: mahir.hadzic@kcl.ac.uk.}, \ and Juhi Jang\thanks{Department of Mathematics, University of Southern California, Los Angeles, CA 90089, USA, and Korea Institute for Advanced Study, Seoul, Korea.  Email: juhijang@usc.edu.}}

\date{}

\begin{document}

\maketitle

\abstract{
The classical model of an isolated selfrgavitating  gaseous
star is given by the Euler-Poisson system with a polytropic pressure law 
$P(\rho)=\rho^\gamma$, $\gamma>1$. For any $1<\gamma<\frac43$,
we construct an infinite-dimensional family of collapsing solutions to the Euler-Poisson system whose density
is in general space inhomogeneous and
undergoes gravitational blowup along a prescribed space-time surface, with continuous mass
absorption at the origin. The leading order singular behavior is described by an explicit collapsing solution
of the pressureless Euler-Poisson system.
}

\tableofcontents

\section{Introduction}

\setcounter{equation}{0}

The basic model of a Newtonian star is given by the 3-dimensional compressible Euler-Poisson system~\cite{ZeNo,BiTr,Ch},
\begin{subequations}
\label{E:EULERPOISSON}
\begin{alignat}{2}
\g_t\rho + \text{div}\, (\rho \mathbf{u})& = 0, && \,\label{E:CONTINUITYEP}\\
\rho\left(\pa_t  \mathbf{u}+ ( \mathbf{u}\cdot\nabla) \mathbf{u}\right) +\nabla P(\rho) +\rho \nabla\Phi&= 0, && \, \label{E:VELOCITYEP} \\
\Delta \Phi  = 4\pi \,\rho, \ \lim_{|x|\to\infty}\Phi(t,x) & = 0.&& \label{E:POISSONEULER}
\end{alignat}
\end{subequations}
Here $\rho,{\bf u}, P(\rho), \Phi$ denote the gas density, the gas velocity vector, the gas pressure, and the gravitational potential respectively. 
To close the system we impose the so-called polytropic equation of state: 
\begin{align}\label{E:EOS}
P(\rho) = \rho^\gamma, \ \ \gamma>1.
\end{align} 
The power $\gamma$ is called the adiabatic exponent.

Here star is modelled as a compactly supported compressible gas surrounded by vacuum, which interacts with a self-induced gravitational field. 
To describe the motion of the boundary of the star we must consider the corresponding free-boundary formulation of~\eqref{E:EULERPOISSON}. In this case,
a further unknown in the problem is the support of $\rho(t,\cdot)$ denoted by $\Omega(t)$. We prescribe the natural boundary conditions
\begin{subequations}
\label{E:EULERPOISSONBDRY}
\begin{alignat}{2}
\rho&=0,&& \ \text{ on } \ \partial \Omega(t),\label{E:VACUUMEP} \\
\mathcal{V}(\partial\Omega(t))&= {\bf u}\cdot {\bf n}  && \ \text{ on } \ \partial\Omega(t),\label{E:VELOCITYBDRYEP}
\end{alignat}
\end{subequations} 
and the initial conditions
\be
(\rho(0,\cdot),  {\bf u}(0,\cdot))=(\rho_0, {\bf u}_0)\,,  \ \Omega(0)=\Omega.\label{E:INITIALEP}
\ee
Here $\mathcal{V}(\partial\Omega(t))$ is the normal velocity of the moving boundary $\pa\Omega(t)$ and condition~\eqref{E:VELOCITYBDRYEP} 
simply states that the movement of the boundary in normal direction is determined by the normal component of the velocity vector field.
We refer to the system~\eqref{E:EULERPOISSON}--\eqref{E:EULERPOISSONBDRY} as the EP$_\gamma$-system. 
We point the reader to the classical text~\cite{Ch} where the existence of static solutions of EP$_\gamma$ is studied under the natural boundary condition~\eqref{E:VACUUMEP}.

We next impose the {\em physical vacuum condition} on the initial data:
\begin{align}\label{E:PHYSICALVACUUM}
-\infty<\nabla\left(\frac{dP}{d\rho}(\rho)\right)\cdot {\bf n}\big\vert_{\pa\Omega} <0.
\end{align}
Condition~\eqref{E:PHYSICALVACUUM} implies that the normal derivative of the squared speed of sound $c_s^2(\rho)=\frac{dP}{d\rho}(\rho)$ is discontinuous at the vacuum boundary. This condition is famously satisfied by the 
well-known class of steady states of the EP$_\gamma$-system known as the Lane-Emden stars. At the same time, condition~\eqref{E:PHYSICALVACUUM} is the key 
assumption that guarantees the well-posedness of the Euler-Poisson system with vacuum regions.

For any $\bar\ve>0$ consider the mass preserving rescaling applied to the EP$_\gamma$-system:
\begin{align}
\rho = \bar\ve^{-3} \tilde\rho(s,y), \ \  
u  = \bar\ve^{-1/2}\tilde u(s,y), \ \ 
\Phi  = \bar\ve^{-1} \tilde \Phi(s,y), \label{E:SCALING}
\end{align}
where 
\[
s=\bar\ve^{-3/2}t, \ \ y = \bar\ve^{-1}x.
\]

It is easy to see that the above rescaling is mass-critical, i.e.
$
M[\rho] = M[\tilde\rho].
$
A simple calculation reveals that if $(\rho, {\bf u}, \Phi)$ solve the EP$_\gamma$-system, then the rescaled quantities $(\tilde\rho, \tilde{\bf u}, \tilde\Phi)$ solve 
\begin{subequations}
\label{E:EULERPOISSON2}
\begin{alignat}{2}
\g_s\tilde\rho + \text{div}\, (\tilde\rho \tilde u)& = 0, && \label{E:CONTINUITYEP2}\\
\tilde\rho\left(\pa_s  \tilde u+ ( \tilde u\cdot\nabla) \tilde u\right) +\ve\nabla (\tilde\rho^\gamma) 
+\tilde\rho \nabla\tilde\Phi&= 0, &&\label{E:VELOCITYEP2} \\
\Delta \tilde\Phi  = 4\pi \,\tilde\rho, \ \lim_{|x|\to\infty}\tilde\Phi(t,x) & = 0,&& \label{E:POISSONEULER2}
\end{alignat}
\end{subequations}
where
\[
\ve: = \bar\ve^{4-3\gamma}.
\]

Observe that for $\bar\ve \ll1$ the factor $\ve$ in front of the pressure in~\eqref{E:VELOCITYEP2} is small precisely in the supercritical range $1<\gamma<\frac43$. The system obtained by dropping the $\ve$-term in~\eqref{E:VELOCITYEP2} is known as the pressureless- or dust-Euler system.
This gives a vague heuristics that, if one for a moment thinks of $\ve$ as a sufficiently small length scale of density concentration, the effects of the pressure term may become negligible and the leading order singular behavior will be driven by the pressure-less dynamics. On the other hand, at this stage, this scaling heuristics is at best doubtful, as the pressure term enters the equation at the top order from the point of view of the derivative count. 

Parameter $\ve$ serves the purpose of a ``small" parameter in our analysis. 
Defining $\tilde\Omega(s) = \bar\ve^{-1}\Omega(t)= \ve^{-\frac{1}{4-3\gamma}} \Omega(t)$, a homothetic image of $\Omega(t)$, 
boundary conditions~\eqref{E:EULERPOISSONBDRY} take the form
\begin{subequations}
\label{E:EULERPOISSONEPS}
\begin{alignat}{2}
\tilde\rho&=0,&& \ \text{ on } \ \partial \tilde\Omega(s),\label{E:VACUUMEPS} \\
\mathcal{V}(\partial \tilde\Omega(s))&= \tilde{\bf u}\cdot \tilde{\bf n}  && \ \text{ on } \ \partial \tilde\Omega(s),\label{E:VELOCITYBDRYEPS}
\end{alignat}
\end{subequations} 
and the initial conditions read
\be
(\tilde\rho(0,\cdot),  \tilde{\bf u}(0,\cdot))=(\tilde\rho_0,  \tilde{\bf u}_0)\,,  \ \tilde\Omega(0)=\tilde\Omega.\label{E:INITIALEPS}
\ee

\subsection{Lagrangian coordinates}

To address the problem of collapse we  express~\eqref{E:EULERPOISSONEPS} in the Lagrangian coordinates.
Firstly, if we wish to follow the collapse
process in its entirety until all of the stellar mass is absorbed, it is clear that the Eulerian description becomes inadequate at and after the first collapse time. In order to describe particle trajectories after the first collapse time, it is advantageous to work in a coordinate system that avoids this issue.
Secondly, the free boundary is automatically fixed in Lagrangian description and thus more amenable to rigorous analysis.

For the remainder of the paper we make the assumption of radial symmetry and 
assume that the reference domain $\tilde\Omega$ is the unit ball $\{y\in\mathbb R^3\,\big| \ |y|\le1\}$. 
Let the flow map $\eta:\tilde\Omega\to\tilde\Omega(s)$ be a solution of: 
\begin{align}
\pa_s\eta(s,y) &= \tilde {\bf u}(s,\eta(s,y)), \label{E:FLOWMAP} \\
\eta(0,y) &= \eta_0(y). \label{E:ETAINITIAL}
\end{align}
Here the boundary $\pa\tilde\Omega$ is mapped to the moving boundary $\pa\tilde\Omega(s)$.
The choice of $\eta_0$ corresponds to the initial particle labelling and represents a gauge freedom in the problem. Equation~\eqref{E:FLOWMAP} automatically incorporates the dynamic boundary condition~\eqref{E:VELOCITYBDRYEPS} when we pull-back the problem on $\tilde\Omega$

Since the flow is spherically symmetric, $\eta$ is parallel to the vectorfield $y$. We introduce the ansatz
\be\label{E:ETARADIAL}
\eta(s,y) = \chi(s,r) y, \ \ r=|y|, \ \ r\in [0,1],
\ee
and denote $\chi(0,r)$ by $\chi_0(r)$.
The Jacobian determinant of $D\eta$ expressed in terms of $\chi$ takes the form
\be
\mathscr J[\chi] := \chi^2(\chi+r\pa_r \chi). \label{E:JFORMULA}
\ee

Since $\pa_s \J = \J (\text{div}\tilde{\bf u})\circ \eta$, as
a consequence of the continuity equation the Eulerian density $\tilde\rho$ evaluated along the particle world-lines satisfies 
\be\label{E:CONTLAGR}
\frac{d}{ds} \big( \tilde \rho(s,\chi(s,r)y) \J[\chi](s,r)\big) =0.
\ee

Let
\be\label{E:ALPHADEF}
\alpha := \frac1{\gamma-1}.
\ee
The fluid enthalpy is a function $r\mapsto w(r)$ defined through the relationship
\begin{align}
w(r)^\alpha & = \tilde\rho_0(\chi_0(r) r) \J[\chi_0](r), \label{wrho_0}
\end{align}
and it is a fundamental object in our work. Instead of specifying $\tilde\rho_0$ and $\chi_0$, 
throughout the paper we fix the choice of the fluid enthalpy $w$ satisfying properties 
(w1)--(w3) below.

\begin{enumerate}
\item[(w1)]
We assume that $w:\tilde\Omega\to \mathbb R_+$ is a non-negative radial function such that $[0,1)\ni r\mapsto w(r)^\alpha$ is $C^\infty$, $w>0$ on $[0,1)$ and $w(1)=0$. 
\end{enumerate}

Assuming further that $\chi_0(r)$ is uniformly bounded from below and $C^2$, from $\tilde\rho(\chi_0(1))=0$ and the physical vacuum condition~\eqref{E:PHYSICALVACUUM},  we conclude $\nabla w\cdot \tilde {\bf n}<0$ at the boundary $\pa\tilde\Omega$ of the reference domain. 
\begin{enumerate}
\item[(w2)]
This leads us to the second basic assumption on $w$ 
\[
\nabla w\cdot\tilde {\bf n}\Big|_{\pa\tilde\Omega} = w'(1)<0.
\]

\item[(w3)]
Finally we denote the {\em mean density} of the gas by
\be\label{E:GDEF}
G(r): = \frac{1}{r^3}\int_0^r4\pi w^{\alpha} s^2\,ds,
\ee
and let 
\begin{align}\label{E:LITTLEGDEF}
g(r) :=3\sqrt{\frac{G(r)}{2}}, \ \ r\in[0,1].
\end{align}
Clearly $g>0$. Observe that $G(0) = \frac{4\pi}{3}w(0)^\alpha$. We shall require that $g:[0,1]\to\mathbb R$ is a smooth function such that 
there exist positive constants $c_1,c_2>0$ and $n\in\mathbb Z_{>0}$ so that 
\begin{align}
c_1 r^n \le -r\pa_r \left(\log g(r)\right) & \le c_2r^n, \ \ r\in[0,1].
\label{E:FM}
\end{align}
\end{enumerate}

The purpose of the following lemma is to show that there exist choices of the enthalpy $w$
consistent with the above assumptions. 


\begin{lemma}\label{L:ENTHALPYOK}
For any $n\in\mathbb N$ there exists a choice of the enthalpy $w$ satisfying properties {\em (w1)}--{\em (w3)}. In particular, the resulting map $g$ defined by~\eqref{E:LITTLEGDEF} satisfies~\eqref{E:FM}. 
\end{lemma}


\begin{proof}
Let $w(r) =a (1-r^n)_+$ 
We observe that for any $r\in(0,1]$
$G(r)=\frac {a^\alpha}{r^3} \int_0^r 4\pi (1-s^n)^\alpha s^2\,ds = \frac{4\pi a^\alpha}{3} - \frac{a^\alpha c_{n,\alpha}}{n}r^{n} + o_{r\to0}(r^n)$, with $1\lesssim c_{n,\alpha}\lesssim1$.
Note that 
\[
\r (\log g(r)) = \frac12 \r (\log G(r))  
\]
which implies~\eqref{E:FM}.
\end{proof}


\begin{remark}
It is evident from the proof that one can easily modify the enthalpy $w$ in the regions away from $r=0$ so that~\eqref{E:FM} is still satisfied. In fact, the family of enthalpies $w$ which satisfy the assumptions {\em (w1)--(w3)} is infinite-dimensional.
\end{remark}


As a simple, but important corollary of (w3), specifically~\eqref{E:FM}, we have

\begin{corollary}
Let $g$ be given by~\eqref{E:LITTLEGDEF}. Then the following properties hold:
\begin{enumerate}
\item[(i)] the map $r\mapsto g(r)$ is monotonically decreasing on $[0,1]$;
\item[(ii)] in the vicinity of the origin 
the following Taylor expansion for $g$ holds:
\be\label{E:GTAYLOR}
g(r)= g(0)-\frac cn r^n +o_{r\to0}(r^n)
\ee
for some constant $c>0$;
\item[(iii)]
for any $k\in\mathbb N$ there exists a positive constant $c_k$ such that 
\be\label{E:GTAYLOR2}
\left\vert \rr^k g(r)\right\vert \le c_k r^n.
\ee
\end{enumerate}
\end{corollary}

As shown in~\cite{Jang2014}, the momentum equation~\eqref{E:VELOCITYEP2} expressed in the Lagrangian variables $(s,r)$ reduces to a nonlinear second order degenerate hyperbolic equation for $\chi$:
\be\label{E:BASICPDE}
 \chi_{ss} +\frac{G(r)}{\chi^2}+ \ve P[\chi] = 0,
\ee
where $r\mapsto G(r)$ is given above in~\eqref{E:GDEF}
and the nonlinear pressure operator $P$ is given by 
\be\label{E:FCHI}
P[\chi] : = \frac{\chi^2}{ w^\alpha r^2}(r\pa_r)\left(w^{1+\alpha}\mathscr J[\chi]^{-\gamma}\right).
\ee
We may explicitly relate the Eulerian density, the fluid enthalpy and the Jacobian determinant; 
as long as $\J[\chi]>0$ by~\eqref{E:CONTLAGR} and~\eqref{wrho_0} we have the fundamental formula
\be\label{E:DENSITYFORMULA}
\tilde \rho(s,\chi(s,r)y) = w^\alpha(r) \J[\chi]^{-1}.
\ee

\begin{remark}
Without being precise about the definition of the gravitational collapse for the moment, our goal is to prove that there exists a choice of initial conditions
\be\label{E:INITIALCONDITIONS}
\chi(0) = \chi_0, \ \ \pa_s\chi(0) = \chi_1
\ee
with a particular choice of the enthalpy $w$ so that $\J[\chi]$ becomes zero in finite time. 
We shall then show that there indeed exists a density $\tilde\rho_0$ satisfying the physical vacuum condition 
\[
\nabla(\tilde\rho_0^{\gamma-1})\cdot \tilde{\bf n} <0 \ \ \text{ on } \ \ \partial\tilde\Omega(0),
\]
as both the profile $w^\alpha$ and the initial labelling of the particles $\chi_0$ are necessary to recover the Eulerian density $\tilde\rho_0$, see~\eqref{wrho_0}. 
\end{remark}

\begin{remark}
In the special  case when $\chi_0=1$, $w^\alpha$ and $\tilde\rho_0$ coincide. We refrain from imposing the initial condition $\chi_0=1$, but we shall prove a posteriori that the initial conditions 
that we use for the construction of the collapsing stars indeed satisfy $\chi_0=1+O(\ve)$ in a suitable norm.
\end{remark}


Finally, from~\eqref{E:GDEF} we have $r\pa_r G + 3G = 4\pi w^\alpha$ and therefore 
\be\label{E:IMPORTANTAPRIORI0}
\r \log g+ \frac32 = \frac{9\pi w^\alpha}{g^2}. 
\ee
Since $\pa_r g \le 0$, $w^{\alpha}|_{r=1}=0$ and $\frac{9\pi w^{\alpha}}{g^{2}}>0$ for $r\in[0,1)$ it follows that 
\be\label{E:IMPORTANTAPRIORI}
\left\vert r\pa_r (\log g)\right\vert < \frac32, \ \ r\in[0,1), \ \  r\pa_r (\log g)\big|_{r=1}=\frac32.
\ee
Bounds~\eqref{E:IMPORTANTAPRIORI0}--\eqref{E:IMPORTANTAPRIORI} are crucial in proving sharp coercivity properties of our high-order energies later in the article.

\subsection{Pressureless collapse}
\label{SS:DUST}

The first step in our analysis is to describe the solutions of~\eqref{E:BASICPDE} when $\ve=0$.
We are led to the ordinary differential equation (ODE)
\be\label{E:BASICPDEDUST}
 \chi_{ss} +\frac{G(r)}{\chi^2}= 0,
\ee
with initial conditions
\be\label{E:DUSTINITIAL}
\chi(0,r) =\chi_0>0, \ \ \chi_s(0,r)=\chi_1.
\ee

We now give a detailed description of the dust collapse from both the Lagrangian and Eulerian perspective, as this will serve as the leading order description of the collapsing stars for the EP$_\gamma$ system.

Notice that for any fixed $r\in[0,1]$ the coefficient $G(r)$ merely serves as a parameter in the above ODE.
The total energy 
\begin{align}\label{E:ENERGYDUST}
E(s) = \frac12\chi_s^2 - \frac{G(r)}{\chi}
\end{align}
is clearly a conserved quantity. We are interested in the collapsing solutions, i.e. solutions of~\eqref{E:BASICPDEDUST}--\eqref{E:DUSTINITIAL} 
such that there exists a $0<T<\infty$
so that $\lim_{s\to T^-}\chi(s,r)=0$ for some $r\in[0,1]$. We consider the inward moving initial velocities with $\chi_1<0$. From the conservation of~\eqref{E:ENERGYDUST}
we obtain the formula
\begin{align}\label{E:PDEDUST2}
\chi_s = - \sqrt{\chi_1^2 + 2G\left(\frac1{\chi}-\frac1{\chi_0}\right)}.
\end{align}
Integrating~\eqref{E:PDEDUST2} one sees that for every $r$ there exists a $0<t^\ast(r)<\infty$ such that $\chi(t^\ast(r),r)=0$. 
A simple calculation reveals that for any $r\in[0,1]$ 
we have the universal blow-up exponent $2/3$ 
\begin{align}\label{E:UNIVERSALSCALING}
\chi(s,r) \sim c(r) (t^\ast(r) - s)^{\frac23}, \ \ s\to t^\ast(r). 
\end{align}

We may further define
the first blow-up time  
\[
t^\ast:=\min_{r\in[0,1]}t^\ast(r).
\]
Observe that the Eulerian description of the solution seizes to make sense at and after time $s\ge t^\ast$. On the
other hand, for different values of $r$ the Lagrangian solution may make sense even after $t^\ast$. In particular, when $t^\ast(r)$ is a non-constant 
function, we can speak of a ``fragmented" or ``continued" collapse, wherein particles with a different Lagrangian label $r$ collapse at different times.
This is the hallmark behavior of {\em inhomogeneous} collapse.

For simplicity, we shall consider a special subclass of solutions of~\eqref{E:BASICPDEDUST}--\eqref{E:DUSTINITIAL} with zero energy. Up to multiplication by a constant
such profiles have the form
\begin{align} \label{E:FUNDPROFILE0}
\chid(s,r) = (1-g(r)s)^{\frac23},
\end{align}
where $g$ is given by~\eqref{E:LITTLEGDEF}.
It follows that $\chid$ becomes zero along the space-time curve 
\begin{align}\label{E:GAMMADEF}
\Gamma : = \{(s,r)\, | \, 1-g(r)s = 0\}.
\end{align}
The solution is only well-defined 
in the region 
\[
\Xi:=\{(s,r)\, \big|\, 1-g(r)s>0\}.
\]

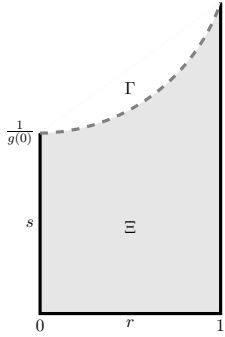
\begin{figure}

\begin{center}
\begin{tikzpicture}

\begin{scope}[scale=0.6, transform shape]
\coordinate [label=below:$0$] (A) at (-3,7){};
      \coordinate [label=below:$1$] (B) at (1,7) {};
      \coordinate [label=below:$r$] (C) at (-1,7) {};
      \coordinate [label=left:$s$] (E) at (-3,9) {};

\coordinate [label=left:$\frac1{g(0)}$] (D) at (-3,11){};
\coordinate  (L) at (0.2,12.5){};
\coordinate  (M) at (1,13.9){};
\tkzCircumCenter(D,L,M)\tkzGetPoint{O}
\draw [fill=gray!20, very thick] (D) -- (A) -- (B) -- (M);
\tkzDrawArc[fill=white, dashed, very thick](O,D)(M)
\coordinate [label=below: $\Xi$] (B) at (-1,9.2) {};
\node at (-1,12) {$\Gamma$};

\end{scope}
\end{tikzpicture}
 \caption{Dust collapse in Lagrangian coordinates}
    \label{compression}

\end{center}
\end{figure}

After a simple calculation we have
\[
\J[\chid](s,r) = (1-g(r)s)^2\left(1-\frac23\frac{srg'(r)}{1-g(r)s}\right), \ \ (s,r)\in\Xi.
\]
In particular, $\chid$ and $\J[\chid]$ vanish along $\Gamma$ and therefore, since the Eulerian density  satisfies
\begin{align}\label{E:DUSTDENSITY}
\tilde\rho_{\text{dust}}(s,\chid(s,r)y) = w^\alpha(r) \J[\chid](s,r)^{-1}, \ \ r=|y|, \ sg(r)<1,
\end{align}
the value of $\tilde\rho_{\text{dust}}(s,0)$ diverges to infinity at the {\em first blow-up time} $t^\ast:=\frac1{g(0)}$. In the region $\chid>0$, the Eulerian density $Y\mapsto \tilde\rho_{\text{dust}}(s,Y)$ is always well-defined away from the origin $Y=0$.
Moreover for any $r\in[0,1]$
\[
\lim_{s\to\frac1{g(r)}} \tilde\rho_{\text{dust}}(s,\chid(s,r)y)=\infty.
\]
Since $r\mapsto g(r)$ is monotonically decreasing, 
particles that start out closer to the boundary of the star take longer to vanish into the singularity.

\medskip
\noindent
{\em Remaining mass.}
For any time $s\in(\frac1{g(0)},\frac1{g(1)})$ the remaining star mass is given by
\be\label{E:REMAININGMASS}
M(s) =4\pi \int_{g^{-1}\circ(\frac1s)}^1 w^\alpha(z) z^2\,dz
= \int_{(0,  \chi_{\text{dust}}(s,1))} 4\pi \tilde\rho_{\text{dust}} (s, Z) Z^2 dZ, 
\ee
where we have changed variables: $z \to Z=\chi(s,z) z$ and used $w^\alpha (z) =\tilde\rho(s, \chi(s,z) z) \J[\chid] $ and $4\pi  \J[\chid]  z^2 dz = 4\pi Z^2 dZ$. 
Since for any $\frac1{g(0)}< s<\frac1{g(1)}$ $\J[\chid](s,r)>0$ for all  $r\in(g^{-1}\circ (\frac1s),1]$, this change of variables is justified.

Finally, the support of the collapsing dust star shrinks to zero as $s\to\frac{1}{g(1)}$. This is clear, as the free boundary in the Eulerian description is at distance $\chid(s,1)=(1-g(1)s)^{\frac23}$ from the origin.
As $s\to\frac1{g(1)}$ the star concentrates with its mass completely absorbed at the origin:
\[
\lim_{s\to \frac1{g(1)}}\chid(s,1)=0 \ \text{ and } \ \lim_{s\to \frac1{g(1)}} M(s)=0.
\]

Therefore the time $s=\frac1{g(1)}$ has a natural interpretation as the {\em end-point of star collapse} for the dust example considered here.


\subsection{Main theorem and related works}

Stellar collapse is one of the most important phenomena of both Newtonian and relativistic astrophysics. Even though extensively studied in the physics literature, very little is rigorously known about the compactly supported solutions to EP$_\gamma$-system that lead to the gravitational collapse. 

\begin{enumerate} 

\item[(1)] When $P(\rho)=0$ and therefore the star content is the pressureless dust, there exists an infinite-dimensional family of collapsing dust solutions, as described in Section~\ref{SS:DUST}.

\item[(2)] If $\gamma=\frac43$ in~\eqref{E:EOS}, due to the special symmetries of the problem, ``homologous" self-similar collapsing solutions exist and were discovered by Goldreich and Weber~\cite{GoWe} in 1980. Further rigorous mathematical works about such solutions are given in~\cite{Makino92,FuLin,DengXiangYang}. Here all the gas  contracts to a point at the same time and the dynamics is described by a reduction to a finite-dimensional system of ODEs.

\item[(3)] When $\gamma>\frac43$ it is shown in~\cite{DLYY} that the collapse by density concentration cannot occur.
\end{enumerate}

We refer to the values $1<\gamma<\frac43$, $\gamma=\frac43$, and $\gamma>\frac43$ of the adiabatic exponent as the {\em mass supercritical}, {\em mass-critical}, and {\em mass subcritical} cases respectively. This terminology is motivated by the invariant scaling analysis of the EP$_\gamma$-system, see e.g.~\cite{HaJa2016-1}.

It has  been an outstanding open problem to prove or disprove the existence of collapsing  solutions in the supercritical range $1<\gamma<\frac43$.


\begin{theorem}[Main theorem] \label{T:MAINMAIN}
For any $\gamma\in(1,\frac43)$ there exist  classical solutions $\chi(s,r)$ of~\eqref{E:BASICPDE} defined in $\Xi=\{(s,r) \,\big| \, 1-g(r)s>0\}$.
The solution behaves qualitatively like the collapsing dust solution $\chid$ and in particular
\begin{align}
 1\lesssim \lv \frac{\chi}{\chid} \rv \lesssim 1, \ \ 1\lesssim\lv \frac{\J[\chi]}{\J[\chid]}\rv &\lesssim  1,  \ \ (s,r)\in
 \Xi. \label{E:CLAIM1}
 \end{align}
 Further, for any $r\in[0,1]$
 \begin{align}
 \lim_{s\to\frac1{g(r)}}\frac{\chi}{\chid}= \lim_{s\to\frac1{g(r)}}\frac{\J[\chi]}{\J[\chid]} &= 1. \label{E:CLAIM2}
\end{align}
Finally, the following three properties hold
\begin{enumerate}
\item[(1)]{\em (Density blows up)}
For any $r\in[0,1]$
\be\label{E:CP1}
 \lim_{s\to\frac1{g(r)}} \tilde\rho(s,\chi(s,r)r) = \lim_{s\to\frac1{g(r)}} w(r)^\alpha \J[\chi]^{-1} = \infty.
\ee

\item[(2)]{\em (Support shrinks to a point)}
\be\label{E:CP2}
\lim_{s\to \frac1{g(1)}} \chi(s,1) = 0.
\ee

\item[(3)]{\em (Mass is continuously absorbed into the singularity)}
\be\label{E:CP3}
\lim_{s\to \frac1{g(1)}} M(s) = \lim_{s\to \frac1{g(1)}} 4\pi \int_{g^{-1}\circ\frac1s}^1 w(z)^\alpha z^2\,dz = 0.
\ee

\end{enumerate}
\end{theorem}


\begin{remark}
One distinctive feature of our proof is that the singularity occurs along the {\em prescribed} space-like surface 
$\Gamma$~\eqref{E:GAMMADEF} 
which coincides with the blow-up surface of the underlying dust solution $\chid$.
\end{remark}

Theorem~\ref{T:MAINMAIN} identifies an infinite-dimensional family of monotonically decreasing initial densities that lead to the gravitational collapse. This is a global characterization of the dynamics, as the region $\Xi$ corresponds to the maximal forward development of the data at $s=0$.

The best known class of global solutions to the EP$_\gamma$ system are the famous static Lane-Emden stars~\cite{Ch,BiTr,ZeNo}. In the range $\frac65<\gamma<2$ one finds compactly supported radially symmetric
time-independent solutions of finite mass, whose stability still remains an outstanding open problem. In the subcritical range $\gamma>\frac43$ the question of nonlinear stability is open despite the promising conditional nonlinear stability result proven by Rein~\cite{Rein} (see~\cite{LuSm} for rotating stars). If the solution exists globally in time when $\gamma>\frac43$ and the energy is strictly positive, then the support of the star must grow at least linearly in $t$, as shown in~\cite{MaPe1990}. A similar conditional result holds when $\gamma=\frac43$~\cite{DLYY}. In the supercritical range $\frac65\le\gamma<\frac43$ it has been shown by Jang~\cite{J0,Jang2014} that the Lane-Emden stars are dynamically nonlinearly unstable. Besides the stationary states and the homologous collapsing stars in the mass-critical case $\gamma=\frac43$, the only other global solutions of EP$_\gamma$ were constructed by Had\v zi\'c and Jang~\cite{HaJa2016-1,HaJa2017}. 

Since the works of Sideris~\cite{Si1985,Si1991} it has been well-known that solutions of the compressible Euler equation (without gravity) develop singularities even with small and smooth initial perturbations 
of the steady state $(\rho,{\bf u})=(1,{\bf 0})$. This type of blow up is generally attributed to the loss of regularity in the fluid unknowns which typically results in a shock. Under the assumption of irrotationality, Christodoulou~\cite{Chr2007} gave a very precise information on the dynamic process of shock formation for the relativistic Euler equation. In the context of nonrelativistic fluids, a related result was given by Christodoulou and Miao in~\cite{ChMi}, while a wider range of quasilinear wave equations is treated extensively by Speck~\cite{Sp}, Holzegel, Luk, Speck, and Wong~\cite{HoLuSpWo}. Most recently shock formation results have been obtained even in the presence of vorticity by Luk and Speck~\cite{LuSp}, for an overview we refer the reader to~\cite{Sp2}. A very different type of singular behavior which results in a wild nonuniqueness for the weak solutions of compressible Euler flows was obtained by Chiodaroli, DeLellis, and Kreml~\cite{ChDeKr}, inspired by the methods of convex integration, see~\cite{DeSz} for an overview. 

The above mentioned mechanisms of singularity formation are different from the singularity exhibited in Theorem~\ref{T:MAINMAIN}, where the density and the velocity remain smooth in the vicinity of the origin and no shocks are formed before the gravitational collapse occurs.

In the absence of gravity, a finite dimensional class of special {\em affine} expanding solutions to the vacuum free boundary compressible Euler flows was constructed by Sideris~\cite{Sideris,Sideris2014}. Their support takes on the shape of an expanding ellipsoid. Related finite-dimensional reductions of compressible flows with the affine ansatz on the Lagrangian flow map go back to the works of Ovsiannikov~\cite{Ov1956} and Dyson~\cite{Dyson1968}, with different variants of the equation of state. Nonlinear stability of the Sideris motions was shown by Had\v zi\'c and Jang~\cite{HaJa2016-2} for the range of adiabatic exponents $1<\gamma\le\frac53$ and it was later extended to the range $\gamma>\frac53$ by Shkoller and Sideris~\cite{ShSi2017}. 

In the setting of compressible non-isentropic gaseous stars (where the equation of state~\eqref{E:EOS} is replaced by the requirement $p=P(\rho,T)$, $T$ being the internal temperature) it is possible to impose an affine ansatz (separation of variables) for the Lagrangian flow map and thus reduce the infinite-dimensional PDE dynamics to a finite-dimensional system of ODEs.  The resulting solutions have space-homogeneous gas densities and the system is therefore closed - the star takes on the shape of a moving ellipsoid. For an overview we refer to~\cite{Bo, BoKiMa}.
A number of  finite-dimensional reductions in the absence of vacuum regions relying on self-similarity and scaling arguments can be found in the physics literature e.g.~\cite{La, Pe, Shu, Yahil1983, To, BoFeFiMu, BlBoCh, BrWi, RB}.

Without the free boundary, in the context of finite-time break up of $C^1$-solutions for the gravitational Euler-Poisson system with a fixed background we refer to~\cite{ChTa} and references therein. 
There are various models in the literature where the stabilizing effects of the pressure are contrasted to the attractive effects of a nonlocal interaction; we refer the reader to~\cite{CaWrZa,CaCaHo,CaCaHo2} for a review and many references for different choices of repulsive/attractive potentials.

The analogues of the collapsing dust solutions in the general relativistic context were discovered in 1934 by Tolman~\cite{To1934}. In their seminal work from 1939, Oppenheimer and Snyder~\cite{OpSn1939} studied in detail the causal structure of a subclass of asymptotically flat Tolman solutions with space-homogeneous density distributions, thus providing basic intuition for the concept of gravitational collapse. Nevertheless, in 1984 Christodoulou~\cite{Ch1984} showed that the causal structure of solutions described in~\cite{OpSn1939} is in a certain sense non-generic in the wider family of Tolman collapsing solutions, proving thereby that for densities given as small inhomogeneous perturbations of the Oppenheimer-Snyder density, one generically obtains naked singularities. This in particular highlights the importance of the rigorous study of the gravitational collapse of gaseous stars with more realistic equations of state, i.e. with nontrivial pressure. In the absence of any matter, existence of singular solutions containing black holes has been known since 1915. This is the 1-parameter family of Schwarzschild solutions, which is embedded in the larger family of Kerr solutions. The nonlinear stability of the Kerr solution has been an important  open problem in the field. Substantial progress has been made over the recent years by Dafermos, Rodnianski, Holzegel, Shlapentokh-Rothman, Taylor, see~\cite{DaHoRo,DaRoSh,Ta} and references therein.

\subsection{Foliation by the level sets of $\chid$}
\label{SS:FOLIATION}

We would like to build a solution of~\eqref{E:BASICPDE} ``around" the fundamental collapsing profile~\eqref{E:FUNDPROFILE0}. To that end it is natural to consider the change of variables
\begin{align}
\tau = 1 - g(r)s,  
\end{align}
and introduce the unknown
\[
\phi(\tau,r):=\chi(s,r).
\]
Note that $0\le\tau\le 1$ and $\tau=0$ corresponds to the space-time curve $\Gamma$, while $\tau= 1$ represents the initial time. It is clear that the change of variables $(s,r)\mapsto (\tau,r)$ is nonsingular since $g(r)>0$ on $[0,1]$.
\begin{figure}
\begin{center}
\begin{tikzpicture}
\begin{scope}[scale=0.7, transform shape]
\coordinate [label=below:$0$] (A) at (-2,0);
      \coordinate [label=below:$1$] (B) at (2,0);

\tkzDefPoint(-2,1){P}\tkzDefPoint(-0.5,1.3){Q}\tkzDefPoint(2,2.4){S}

\tkzDefPoint(-2,2.1){A2}\tkzDefPoint(-0.6,2.6){M2}\tkzDefPoint(2,4.4){B2}

\tkzDefPoint(-2,4){D}\tkzDefPoint(1.2,5.5){L}\tkzDefPoint(2,6.9){M}

\draw [fill=gray!20, very thick] (D) -- (A) -- (B) -- (M);

\tkzCircumCenter(P,Q,S)\tkzGetPoint{T}
\tkzDrawArc[thick](T,P)(S);

\node at (0,1.5) {$1-g(r)s=\text{const.}$};

\tkzCircumCenter(A2,M2,B2)\tkzGetPoint{T2}
\tkzDrawArc[thick](T2,A2)(B2);

\node at (0,2.8) {$1-g(r)s=\text{const.}$};

\tkzCircumCenter(D,L,M)\tkzGetPoint{O}
\tkzDrawArc[fill=white, dashed, very thick](O,D)(M)

\node at (0,5) {$1-g(r)s=0$};

\coordinate [label=below:$0$] (X) at (6,0){};
      \coordinate [label=below:$1$] (Y) at (10,0) {};
        \coordinate [label=right:] (Z) at (10,4) {};
      \coordinate [label=left:] (U) at (6,4) {};
\draw [fill=gray!20, very  thick] (U) -- (X) -- (Y) -- (Z);
\draw [dashed, very thick] (Z) -- (U);

\coordinate [label=left:] (X1) at (6,1){};
      \coordinate [label=right:] (X2) at (6,2.1) {};
\coordinate [label=left:] (Y1) at (10,1){};
      \coordinate [label=right:] (Y2) at (10,2.1) {};
\draw [] (X1) to (Y1);
\draw [] (X2) to (Y2);

\node at (8,1.3) {$\tau=\text{const.}$};
\node at (8,2.4) {$\tau =\text{const.}$};
\node at (8,4.3) {$\tau=0$};

\coordinate [label=left:] (L1) at (2.5,2.5){};
\coordinate [label=left:] (R) at (5.5,2.5){};
\draw [<->, thick] (L1) to [out=40,in=140] (R);

\node at (4,3.5) {$(s,r)\mapsto (\tau,r)$};
\end{scope}
\end{tikzpicture}
\caption{Foliation by the level sets of $\chid$}
\end{center}

\end{figure}
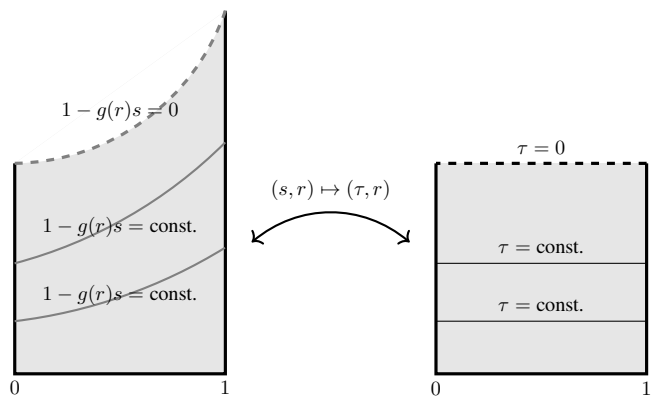

The operator $r\pa_r$ expressed in the new variables is denoted by $\DL$ and it reads
\begin{align}
\DL : = - \frac{ r  g'( r )(1-\tau)}{g( r )} \pa_\tau + \r = (\tau-1)\r (\log g) \pa_\tau + \r. \label{E:DDEF}
\end{align}
We also use the abbreviation
\be\label{E:BETADEF}
M_g ( \tau, r) := (\tau-1)\r (\log g),
\ee
so that 
\be\label{E:DLAMBDADEF}
\DL = M_g  \pa_\tau + \r.
\ee
From~\eqref{E:BASICPDE} we immediately see that the unknown $\phi$ solves 
\be\label{E:BASICPDEPHI}
\phi_{\tau\tau} + \frac{2}{9\phi^2} + \ve P[\phi]=0,
\ee
where
\begin{align}
P[\phi] & : =  
\frac{\phi^2}{ g^2( r )w^\alpha r^2}\DL
 \left(w^{1+\alpha}\left[\phi^2\left(\phi+\DL\phi\right)\right]^{-\gamma}\right) \label{E:PRESSURETERM}
\end{align}
is the pressure term in new variables $(\tau,r)$.
In $( \tau, r)$-coordinates the dust collapse solution~\eqref{E:FUNDPROFILE0} is denoted by $\phi_0$, it solves
\be\label{E:DUST}
\pa_{\tau\tau}\phi_0 + \frac29\phi_0^{-2} = 0, 
\ee
and is given explicitly by 
\begin{align} \label{E:FUNDPROFILE}
\phi_0(\tau,r) = \tau^{\frac23}.
\end{align}
After a simple calculation we obtain 
\begin{align}\label{E:DUSTJACOBIAN}
\J[\phi_0](\tau,r) = \tau^2 \left(1+\frac23 \frac{M_g}{\tau}\right). 
\end{align}
In particular $\J[\phi_0](\tau,r)>0$ for all $(\tau,r)\in(0,1]\times[0,1]$ and
\[
\lim_{\tau\to0^+}\J[\phi_0]=0, \ \ \J[\phi_0]\Big|_{\tau=1} =1.
\]
The connection between the above formulas and mass conservation for the dust solution is detailed in Section~\ref{SS:DUST}.
From the formula~\eqref{E:DUSTJACOBIAN},~\eqref{E:BETADEF}, and~\eqref{E:FM} we conclude that 
for $0<\tau\ll1$ 
\be\label{E:DUSTJACOBIAN2}
\J[\phi_0](\tau,r)\approx \tau^2 (1+\frac{r^n}{\tau}),
\ee
wherefrom the scale $r^n/\tau$ emerges naturally and will play an important role in our work.

We will prove Theorem~\ref{T:MAINMAIN} in the $(\tau,r)$-coordinate system, using~\eqref{E:BASICPDEPHI} as a starting point. This is natural, as the collapse surface in the new coordinates takes on a simpler description
$\Gamma=\{\tau=0\}$.

\subsection{Methodology and outline of the proofs}\label{SS:METHODS}

The continuity equation in Lagrangian coordinates reduces to~\eqref{E:DENSITYFORMULA}, which implies
that the blow-up points of the density
coincide with
the zero set of the Jacobian determinant
\[
\J[\phi]:=\phi^{2}(\phi+\DL\phi), \ \ \DL = M_g\pa_\tau + r\pa_r.
\]
Therefore, the key goal of this work is to identify a class of initial data
that in a suitable sense mimic the bahavior of the dust solution and we do that by showing
\[
1\lesssim\frac{\J[\phi]}{\J[\phi_{0}]}\lesssim1.
\] 

A natural idea is to consider the dynamic splitting
\be\label{E:ANSATZINTRO0}
\phi=\phi_{0}+\varepsilon\phi_{0}R
\ee
where the relative remainder $R$ is expected to be small in an appropriate sense.  
A straightforward calculation gives a partial differential equation satisfied by $R$, 
which at the leading order  takes the schematic form,
\begin{equation}
\bar g^{00}\partial_{\tau\tau
}R+ \bar g^{01}\partial_{r}\partial_{\tau
}R
+\frac4{3\tau}\pa_\tau R-\frac{2}{3\phi_{0}^{3}}R-\varepsilon\gamma c[\phi_{0}]\frac{1}{w^{\alpha}%
}\partial_{r}(\frac{w^{1+\alpha}}{r^{2}}\partial_{r}[r^{2}%
R])=\bar F,\label{linear0}%
\end{equation}
where one can show that 
\[
\bar g^{00}\approx1, \ \bar g^{01}\approx \frac \ve \tau, \ c[\phi_0] \approx \frac{\tau^{\frac53 -\gamma}}{\tau + r^n}.
\]
The simplest way of interpreting the relative ``strength" of each of the terms in~\eqref{linear0}
is to compute the associated energies by 
taking the inner product with $\pa_\tau R$. 
A key term emerges
\begin{align*}
-(\frac{2}{3\phi_{0}^{3}}R,\pa_\tau R)_{L^2} & =-\frac13\pa_\tau\int \frac{R^2}{\phi_0^3} + \frac13\int \pa_\tau
\left(\frac{1}{\phi_{0}^{3}}\right)R^{2} \\
& = -\frac13\pa_\tau\int \frac{R^2}{\tau^2} - \frac23 \int \frac{R^2}{\tau^{3}}.
\end{align*}
The severe $\tau^{-3}$-singularity has a bad sign and
it cannot be controlled by the positive definite part of the natural energy.
This issue is not a mere technicality, it is deeply connected with the focusing nature of
the dust collapse and it is unlikely that the ansatz~\eqref{E:ANSATZINTRO0} can be successful.

We note that we have already implicitly used the assumption $\gamma<\frac43$ via
the scaling transformation~\eqref{E:SCALING}, which resulted in the occurrence of the small parameter $\ve$
in~\eqref{E:BASICPDEPHI}.
We want to further use $\gamma<\frac43$, but with a more refined dynamic splitting ansatz.
Namely, our main idea is to seek a more \textit{special} solution $\phi$
of the form
\be\label{E:ANSATZINTRO}
\phi=\phi_{\text{app}}+\frac{\tau^{m}}{r}H
\ee
where $\phia$ will be chosen as a more accurate {\em approximate solution} of the Euler-Poisson system~\eqref{E:BASICPDEPHI} in hope of mitigating the issue explained above. The exponent $m>0$ is a sufficiently large positive number, so that $H$ is a weighted remainder, small relative to $\tau^m = \phi_0^{\frac32 m}\ll \phi_0$ for small values of $\tau$.

\medskip 
\textit{Step 1. Hierarchy and the construction of the approximate solution {\em $\phia$} (Section~\ref{S:HIERARCHY})}.

We shall find the approximate profile $\phia$ as a finite order expansion into the powers of $\ve$ around
the background dust profile $\phi_0$, i.e.
\be\label{E:EXPANSIONINTRO}
\phia = \phi_0 + \varepsilon \phi_1 + \varepsilon^2\phi_2 + \dots + \varepsilon^M \phi_M, \ \ M\gg1.
\ee
With the solution ansatz~\eqref{E:EXPANSIONINTRO} we can formally Taylor expand the pressure term
$\ve P[\phi_0 + \ve \phi_1 + \dots]$ into the powers of $\ve$, thus giving us a hierarchy
of ODEs satisfied by the $\phi_j$:
\begin{align}\label{E:ODEINTRO}
\partial_{\tau\tau}\phi_{j+1} - \frac{4}{9\tau^2}\phi_{j+1} = f_{j+1}[\phi_0,\phi_1,\dots,\phi_j], \ \ j=0,1,\dots, M.
\end{align} 
Functions $f_{j+1}$, $j=0,1,\dots,M$ are explicit and generally depend nonlinearly on $\phi_k$, $0\le k\le j$, 
 and their spatial derivatives (up to the second order). 
 
The system of ODEs~\eqref{E:ODEINTRO} can be solved iteratively as the right-hand side $f_{j+1}$ is always
known as a function of the first $j$ iterates. To show that finite sums of the form~\eqref{E:EXPANSIONINTRO} 
are good approximate solutions of~\eqref{E:BASICPDEPHI}, we must prove that the iterates $\phi_j$, $j\ge1$, are effectively ``small" with respect to $\phi_0$.
The mechanism by which this is indeed true is one of the key ingredients of the paper, in both the conceptual and the technical sense. In particular we shall have to choose special solutions of~\eqref{E:ODEINTRO}, as they are in general not unique (the two general solutions of the homogeneous problem are $\tau^{4/3}$ and $\tau^{-1/3}$), which will allow us to see the above mentioned {\em gain}.

We now proceed to explain these ideas in more detail.
To provide a quantitative statement, we assume that the enthalpy profile 
$w$ satisfies
\be\label{E:RHOTAYLORINTRO}
w^\alpha(r) = 1 - c r^n + o_{r\to0}(r^n)
\ee
in a neighbourhood of the center of symmetry $r=0$. The exponent $n\in\mathbb N$ is our effective measure of {\em flatness} of the star close to the center. 
For a given $\gamma\in(1,\frac43)$ we consider densities~\eqref{E:RHOTAYLORINTRO} with $n$ so large that 
\be\label{E:DELTAINTRO}
\delta : = 2\left(\frac43-\gamma-\frac1n\right) >0.
\ee

With this assumption in place we prove that the iterates $\{\phi_j\}_{j\in\mathbb N}$ ``gain" smallness and this conclusion is summarized in the following theorem:


\begin{theorem}\label{T:MAINBOUNDPHI}
Let $M,K\in\mathbb Z_{>0}$ be given. There exists a sequence $\{\phi_{j}\}_{j\in\{0,\dots,M\}}$  of
solutions to~\eqref{E:ODEINTRO} with $\phi_0(\tau,r)=\tau^{\frac23}$, constants $C_{jkm}$ depending on $K$ and $M$, and a $\l>\frac2n$ such that 
for $j\in \{1,\dots,M\}$ and $\ell,m \in \{0,1,\dots,K\}$ we have
\be\label{E:GAININTRO}
\lv \pa_\tau^m(r\pa_r)^\ell \phi_j \rv \le C_{jkm} \tau^{\frac23+j\delta-m} \frac{\et^{\l-\frac2n}}{(1+\et)^\l}.
\ee 
\end{theorem}


Therefore
the iterates $\phi_{j}$ exhibit a crucial gain of 
$\tau^{j\delta}$ with respect to the dust profile $\phi_0=\tau^{\frac23}$!
This is one manifestation of the supercriticality (i.e. $1<\gamma<\frac43$) of the problem
and it can be viewed as the gain of smallness in the singular regime $0<\tau\ll1$.

To motivate~\eqref{E:GAININTRO}, we explain informally how the gain happens for $\phi_1$.
To find $\phi_1$ we solve the ODE
\be\label{E:PHI1MOTIVATION}
\partial_{\tau\tau}\phi_{1} - \frac{4}{9\tau^2}\phi_{1} = -P[\phi_0] = -\frac{\phi_0^2}{ g^2( r )w^\alpha r^2}\DL
 \left(w^{1+\alpha} \J[\phi_0]^{-\gamma}\right)
\ee
For $0\le r\ll1$ we have $w\approx g(r)\approx 1$. Approximating 
$\DL \approx r^n\pa_\tau + r\pa_r$, and by~\eqref{E:DUSTJACOBIAN2} 
$\J[\phi_0]\approx \tau (\tau +r^n)$, $r\ll1$, we obtain
\[
P[\phi_0] \approx 
\tau^{\frac23} \tau^{\frac23-2\gamma} \frac{\et^{1-\frac2n}}{(1+\frac{r^n}{\tau})^{\gamma}}, \ \ r\ll1.
\]
We expect $\phi_1$ to ``gain" 2 powers of $\tau$ with respect to the right-hand side of~\eqref{E:PHI1MOTIVATION}, and thus 
\[
\phi_1 \approx \tau^{\frac23} \tau^{\frac83-2\gamma-\frac2n} \frac{\et^{1-\frac2n}}{(1+\frac{r^n}{\tau})^{\gamma}}
= \tau^{\frac23+\delta}  \frac{\et^{1-\frac2n}}{(1+\frac{r^n}{\tau})^{\gamma}},
\]
with $\delta$ defined in~\eqref{E:DELTAINTRO}. Of course $\delta$ can be positive if and only if $\gamma<\frac43$ and the exponent $n$ from~\eqref{E:RHOTAYLORINTRO} is sufficiently large!

Most important consequence  of Theorem~\ref{T:MAINBOUNDPHI} is that it leads to a source term $\mathscr S(\phia)$ generated by $\phia$ (see Lemma~\ref{L:SOURCETERMBOUNDS}) which satisfies a natural {\em improved} bound 
\[
\lv \mathscr S(\phia) \rv \lesssim \varepsilon^{2M+1}\tau^{-\frac{4}{3}+(M+1)\delta-1}.
\]
Therefore, if $M \gg1,$ it is reasonable to expect that the remainder ansatz  
$\tau^{m}\frac{H}{r}$ (with $m\ge M\delta$)
is consistent  with our strategy of treating $H$ as an error term, in the regime $\tau\ll1$.

Another crucial input in~\eqref{E:GAININTRO} is the factor
$\frac{\et^{\l-\frac2n}}{(1+\et)^\l}\le 1$. The gain is visible only in the asymptotic 
regime $r^n/\tau\ll1$, which suggests that the scale 
\[
r^n/\tau
\]
plays a critical role in our problem. Indeed,
this gain is important in the closure of the energy estimate for $H$ --
it is used to absorb 
various negative powers of $r$ which inevitably appear 
in our high order energy scheme intimately tied to the assumed spherical symmetry.

The proof of Theorem~\ref{T:MAINBOUNDPHI} is complex and delicate. It is based on
the introduction of \textit{special} solution operators $S_{1}$ and $S_{2}$~\eqref{E:S1}--\eqref{E:S2} 
for the ODE~\eqref{E:ODEINTRO}.
In addition to a careful and precise
tracking of the powers of $\tau$ and $\frac{r^{n}}{\tau}$, to
see the gain of $\tau^{j\delta}$ one has to
use different solution operators $S_{1}$ and $S_{2}$ for $j\leq
\lfloor\frac{2}{3\delta}\rfloor$ and for $j>[\frac{2}{3\delta}]$, respectively.
The precise estimate~\eqref{E:GAININTRO} and the emergence of $\frac{r^n}{\tau}$ as a critical
quantity is intimately tied to the algebraic structure of $f_j$, $j=1,\dots,M$, which in turn possesses a rich
geometric information related to the Taylor expansion of the negative powers of the Jacobian determinant
$\J[\phi]$.

\medskip

{\em Step 2. Equation for the remainder $H$ (Section~\ref{S:MAINRESULT})}.

Thanks to the crucial gain of $\tau^{j\delta}$ and in the presence of
$\tau^{m}$ factor, now $H$ satisfies the following quasilinear
wave-like equation:
\begin{align}
g^{00}\partial_{\tau\tau}H 
+ 2g^{01}\partial_{r}\partial_{\tau}H +\frac{2m}{\tau}\partial_{\tau}H &
+ \left[  \frac{m(m-1)}{\tau^{2}}-\frac{4}{9\phia^{3}}\right] H \notag \\
&-\varepsilon\gamma c[\phi]\frac{1}{w^{\alpha}}\partial_{r}(\frac{w^{1+\alpha}}{r^{2}}\partial_{r}[r^{2}%
H])=F, \label{linear}
\end{align}
where at the leading order
\[
g^{00}=g^{00}[\phi]\approx1, \  g^{01}=g^{01}[\phi] \approx \frac\ve{\tau}, \ c[\phi]\approx c[\phia] \approx \frac{\tau^{\frac53 -\gamma}}{\tau + r^n}.
\]
The precise formulas for the right-hand side $F$, $g^{00}$, $g^{01}$, and $c[\phi]$ are given in~\eqref{E:H2}--\eqref{E:GZEROONEDEF},~\eqref{E:CDEF} respectively. 
In comparison to (\ref{linear0}), the remarkable new feature of (\ref{linear})
is the presence of the coefficient $\frac{m(m-1)}{\tau^{2}}$ so that
\[
\frac{m(m-1)}{\tau^{2}}>\frac{4}{9\phia^{3}}\backsim\frac{1}{\tau
^{2}},
\]
for $m$ sufficiently large.
This leads to a {\em coercive positive definite} control of the solution at the 
singular surface $\{\tau=0\}$.

\medskip 

{\em Step 3. The physical vacuum and weighted energy spaces (Section~\ref{S:ENERGIES})}

Much of the difficulty in producing energy estimates for~\eqref{linear} comes
from an antagonism between two different singularities present in the equation.

\begin{itemize}
\item at $\tau=0$ the coefficient $c[\phia]$ and various others formally blow up to infinity. This is the singularity associated with the collapse at the singular surface $\tau=0$ and already explained above;
\item  at $r=1$ we have $w=0$ and therefore the elliptic part of the quasilinear operator on the left-hand side of~\eqref{linear} does not scale like the Laplacian as $r\to1$. This is a well-known degeneracy associated with the presence of the vacuum boundary. 
\end{itemize}

The assumption of physical vacuum can be recast as the requirement that the enthalpy $w\ge0$ behaves like a distance function when $r\sim1$, i.e. 
\be\label{E:PHYSICALVACUUM2}
\frac1C (1-r) \le w(r) \le C(1-r), \ \ r\in[0,1].
\ee
Requirement~\eqref{E:PHYSICALVACUUM2} is important in 
establishing the well-posedness of~\eqref{linear}. The local well-posedness
theory for the physical vacuum problem was first developed in the Euler case~\cite{JaMa2015, CoSh2012},
while the well-posedness statements for the gravitational Euler-Poisson system can be found in~\cite{Jang2014, LXZ, GuLe, HaJa2016-1, HaJa2017}.
Nevertheless, the well-posedness theory cannot be directly applied to our setting, as~\eqref{linear} differs from the above mentioned 
works in two important aspects: the problem has explicit singularities at $\tau=0$ and the space time domain $(\tau,r)\in(0,1]\times[0,1]$ is strictly larger
than the domain $(s,r)\in[0,\frac1{g(0)})\times[0,1]$ which only covers the star dynamics up to the first stipulated 
collapse time $t^\ast=\frac 1{g(0)}$, see Section~\ref{SS:FOLIATION}.

\medskip 

{\em Step 4. Energy estimates and the conclusion (Sections~\ref{S:ENERGYESTIMATES} and~\ref{S:COMPACTNESS}).}

Since $\phia\sim \phi_{0}=\tau^{2/3},$ $\tau$-derivatives of $\phia$ create severe
singularities in $\tau$ as $\tau\to0$, which leads to difficulties in our energy estimates. We
must in particular abandon the use vector field $\partial_{\tau}$ to form the natural high-order energy and
instead rely on purely spatial derivatives. 
Due to very precise and
delicate features of the approximate solution $\phia$ near the center $r=0$ (as described in Theorem~\ref{T:MAINBOUNDPHI}),
we are forced to use polar coordinates throughout $[0,1]$, which results in the introduction of many 
novel analytic
tools to control the singularity at $r=0.$ 

To motivate the definition of high-order energy spaces we isolate the leading order spatial derivatives contribution 
from the left-hand side of~\eqref{linear}:
\begin{align}
L_\alpha H := -\frac{1}{w^\alpha} \pa_r \left[w^{1+\alpha} D_r H \right],
\end{align}
where 
\[
D_r : = \frac 1{r^2}\pa_r\left(r^2\cdot\right)
\]
is the radial expression for the three-dimensional divergence operator. 
This particular form of $L_\alpha$ suggests that we have to carefully 
apply high-order derivatives to~\eqref{linear} in order to avoid singularities
at $r=0$. We therefore introduce a class of operators defined as concatenations
of $\pa_r$ and $D_r$:
\begin{equation}
\label{E:FUNDOP}
\mathcal D_j : = 
\begin{cases}
( \pa_r D_r)^\frac{j}{2} & \text{ if $j$ is even}\\
D_r( \pa_r D_r)^\frac{j-1}{2} & \text{ if $j$ is odd}
\end{cases}
\end{equation}
and set $\D_0 =1$. The operators $\D_j$ are then commuted with the equation~\eqref{linear}.
For some $N$ sufficiently large,  the idea is to form the energy spaces by evaluating the inner product of the commuted equation 
with $\D_j H_\tau$, $j=1,\dots, N$. 
However, following the ideas developed in~\cite{JaMa2015, HaJa2016-1, HaJa2016-2}, we
need to perform our energy estimates in a cascade of weighted Sobolev-like spaces. For any 
given $j\in\{1,\dots,N\}$ the correct choice is the inner product associated with the weights $w^{\alpha+j}$.


\begin{definition}[Weighted spaces]\label{Def:InnerProduct}
For any $i\in\mathbb Z_{\ge0}$ 
we define weighted spaces $L^2_{\alpha+i}$ as a completion of the space $C_c^\infty(0,1)$ with respect to the norm $\|\cdot\|_{\alpha+i}$ 
generated by the inner product 
\be\label{E:WEIGHTEDINNERPRODUCT}
(\chi_1,\chi_2)_{\alpha+i} : = \int_0^1 \chi_1\chi_2 w^{\alpha +i} r^2 dr 
\ee
and denote the associated norm by $\|\cdot\|_{\alpha+i}.$ 
\end{definition}

\begin{definition}[Weighted space-time norm]\label{D:NORMDEF}
For any $0< \kappa\le1$, $N\in\mathbb Z_{>0}$, $\kappa\le \tau\le1$ we define the weighted space-time norm
\begin{align}
&S_\kappa^N(H, H_\tau)(\tau) =S_\kappa^N(\tau):= \notag \\
&
 \sum_{j=0}^N \sup_{\kappa\le \tau'\le\tau} \left\{(\tau')^{\gamma - \frac53} \|\mathcal D_j H_\tau\|_{\alpha+j}^2 
+(\tau')^{\gamma - \frac{11}3} \|\mathcal D_j H\|_{\alpha+j}^2
+ \ve (\tau')^{-\gamma-1}\|q_{-\frac{\gamma+1}{2}}\left(\frac{r^n}{\tau'}\right)  \mathcal D_{j+1} H\|_{\alpha+j+1}^2\right\} \notag \\
& +  \sum_{j=0}^N \int_\kappa^\tau \left\{ (\tau')^{\gamma-\frac83}\|\mathcal D_j H_\tau\|_{\alpha+j}^2  
+(\tau')^{\gamma - \frac{14}3} \| \mathcal D_j H\|_{\alpha+j}^2
+ \ve (\tau')^{-\gamma-2}\| q_{-\frac{\gamma+2}{2}}\left(\frac{r^n}{\tau'}\right)  \mathcal D_{j+1} H\|_{2\alpha+j+1}^2 \right\} \,d\tau' \notag
\end{align}
where $q_\nu(x) = (1+x)^\nu$, $\nu\in\mathbb R$. 
\end{definition}

We see that the powers of the $w$-weights increase with the number of derivatives.
Such spaces are carefully designed to control the motion of the free boundary at $r=1$
and the key technical tool in our estimates is the Hardy inequality. This
is natural since $w\sim 1-r$ near $r=1$. Similarly, the presence of $\tau$-weights
allows us to precisely capture the degeneration of our 
wave operator at the singular space-time curve $\{\tau=0\}$.

The positive function 
$x\mapsto q_\nu(x)$ 
serves as a weight for the top order spatial derivative contributions in the above definition, with powers $\nu=-\frac{\gamma+1}{2}$ and $\nu=-\frac{\gamma+2}{2}$ 
respectively. Such weights appear in the dust Jacobian $\J[\phi_0]$ and by means of expanding the true solution around $\phi_0$, functions $q_\nu$ appear naturally in our energies. The presence of $q_\nu$
highlights again the importance of the characteristic scale $r^n/\tau$ in our problem. We prove the key theorem:


\begin{theorem}[The $\kappa$-problem]\label{T:MAINEXISTENCETHEOREM} Let $\gamma\in(1,\frac43)$ and $m\ge \frac52$ be given. Set
 $N=N(\gamma)=\lfloor \frac1{\gamma-1} \rfloor +6$. For a sufficiently large $n=n(\gamma)\in\mathbb Z_{>0}$,
there exist  $\sigma_\ast,\ve_\ast>0$, $M=M(m,\gamma,n)\gg1$ and $C_0>0$, such that for any $0<\sigma<\sigma_\ast$ and any $0<\ve<\ve_\ast$
the following is true: for any $\kappa\in(0,1)$ and any initial data $(H_0^\kappa,H_1^\kappa]$ satisfying 
\[
S_\kappa^N(H_0^\kappa,H_1^\kappa)(\tau=\kappa) \le\sigma^2,
\]
there exists a unique solution solution $\tau\mapsto H^\kappa(\tau,\cdot)$ to~\eqref{linear} on $[\kappa,1]$ satisfying 
\[
S_\kappa^N(H^\kappa, H^\kappa_\tau)(\tau) \le C_0\left(\sigma^2+\ve^{2M+1}\right), \quad  \tau \in[\kappa,1].
\]
\end{theorem}


Theorem~\ref{T:MAINEXISTENCETHEOREM} gives uniform-in-$\kappa$ bounds 
for the sequence $H^\kappa$ with initial data specified at time $\tau=\kappa$.
One may for example choose trivial data at $\tau=\kappa$, i.e. set $\sigma=0$ in the above theorem to generate a family of solutions $\{H_\kappa\}_{\kappa\in(0,1]}$.
As $\kappa\to0$ we conclude the existence of
a solution $H$ on $(0,1]$. By~\eqref{E:ANSATZINTRO}, 
this gives a solution $\phi = \phia + \tau^m\frac{H}{r}$ 
of the original problem~\eqref{E:BASICPDEPHI}, thus allowing us to prove Theorem~\ref{T:MAINMAIN} (after going back to the $(s,r)$ coordinate system).
The proof of Theorem~\ref{T:MAINEXISTENCETHEOREM} is given in Section~\ref{SS:PROOFOFKAPPATHEOREM}, while the proof of Theorem~\ref{T:MAINMAIN} is given in Section~\ref{S:COMPACTNESS}.

\begin{remark}
Note that the small parameter $\ve$ used for the construction of the approximate solution $\phia$ enters explicitly in~\eqref{linear}.
\end{remark}

\begin{remark}
As part of the proof of Theorem~\ref{T:MAINEXISTENCETHEOREM}, we also obtain a lower bound on the parameter $M$ - the expansion order of the approximate solution $\phia = \sum_{j=0}^M\ve^j\phi_j.$ A precise formula is given in~\eqref{E:MFORMULA}.
\end{remark}

Many of our energy estimates depend crucially on both the gain of a $\tau^\delta$-power and a power of $\frac{r^{n}}{\tau}$ in
Theorem~\ref{T:MAINBOUNDPHI}. The former allows us to obtain a crucial gain of integrability-in-$\tau$ close to the singular surface 
$\tau=0$, while the latter is needed to absorb negative powers of $r$ arising from the application of the operators $\D_j$ on $\phia$.
This delicate interplay works out, but requires a certain ``numerological" constraint, namely the coefficient $n$ has to be large enough
relative to the total number of derivatives $N$ used in our energy scheme.

Despite the delicate tools and analysis, one term stands
out and seemingly causes a major obstruction to our method.
After commuting the equation with high-order operators $\D_j$ and
evaluating the $(\cdot,\cdot)_{\alpha+j}$-inner product with $\D_j H_\tau$,
an error term $\mathscr M[H]$ defined in~\eqref{E:MH} emerges. A 
simple counting argument suggests that the number of powers of $w$ in $\mathscr M[H]$
is insufficient to close the estimates near the vacuum boundary,
but we carefully exploit a remarkable algebraic structure within the term
and obtain the necessary cancellation, see Lemma~\ref{L:CANCEL}.

The last claim of Theorem~\ref{T:MAINMAIN} shows that the infinitesimal volume of the shrinking domain of our collapsing solution behaves like the infinitesimal volume of the collapsing dust profile. More importantly, 
using~\eqref{E:DENSITYFORMULA}, one can conclude that the qualitative behavior on approach to the singular surface $\tau=0$ 
of the Eulerian density $\tilde\rho$ is the same as that of the dust density, see Section~\ref{S:COMPACTNESS}.

\subsubsection*{Plan of the paper}

Section~\ref{S:HIERARCHY} is devoted to the derivation of the hierarchy of ODEs~\eqref{E:ODEINTRO} and the proof of Theorem~\ref{T:MAINBOUNDPHI}. In Section~\ref{S:MAINRESULT} we derive the equation for the remainder term $H$.  In Section~\ref{S:ENERGIES} we introduce the high-order differentiated version of the $H$-equation derived in Section~\ref{S:MAINRESULT}. We also define high-order energies that arise naturally from integration-by-parts and show (Section~\ref{SS:APRIORI}) that they are  equivalent to norm $S_N^\kappa$ from Definition~\ref{D:NORMDEF}. The remainder of the section is devoted to various a priori estimates and preparatory bounds. In Section~\ref{S:ENERGYESTIMATES} we prove the key energy estimates, culminating in the proof of Theorem~\ref{T:MAINEXISTENCETHEOREM} in Section~\ref{SS:PROOFOFKAPPATHEOREM}. Finally, Theorem~\ref{T:MAINMAIN} is shown in Section~\ref{S:COMPACTNESS}. In Appendices~\ref{A:proofs}--\ref{A:HSembedding} many important properties and analytic tools used in our estimates are shown. We present details of the product and chain rule within vector field classes $\mathcal P$ and $\bar{\mathcal P}$ (Appendix~\ref{A:proofs}), commutator identities (Appendix~\ref{A:COMM}), and the Hardy-Sobolev embeddings (Appendix~\ref{A:HSembedding}). Finally, for 
the sake of completeness, we sketch the local well-posedness argument in Appendix~\ref{A:LOCAL}.


\subsection{Notation}\label{SS:NOTATION}

\begin{itemize}
\item 
By $\mathbb Z_{\ge0}$, $\mathbb Z_{>0}$ we denote the sets of non-negative  and strictly positive integers respectively.

\item
$C^0([a,b], [c,d])$ denotes the space of continuous functions $(\tau,r)\mapsto f(\tau,r)$ on the set $[a,b]\times [c,d]$.

\item
We use $\| \cdot\|_{L^2}$ to denote $\| \cdot\|_{L^2([0,1];r^2dr)}$ and $\|\cdot\|_\infty$ to denote $\| \cdot\|_{L^\infty([0,1])}$.  

\item
Writing $A\lesssim B$ means that there exists a universal constant $C>0$ such that $A\le CB$. $A\gtrsim B$ simply means $B\lesssim A$. If we write $A\approx B$ we mean $A\lesssim B$ and $A\gtrsim B$.

\item 
For a given $a>0$ we denote the closed three-dimensional ball of radius $a$ centered at $0$ by $B_a(0)$.
\end{itemize}

\section{The hierarchy}\label{S:HIERARCHY}

Formally we would like to 
build a solution of~\eqref{E:BASICPDEPHI}
as a sum of the approximate profile $\phia$ (given as a finite series expansion in the powers of $\ve$) and the remainder term $\theta$ which we
hope to show to be suitably small. In other words, we are looking to write
\be\label{E:BASICPHI}
\phi = \phia+ \theta = \sum_{j=0}^M\ve^j \phi_j + \theta.
\ee

Plugging~\eqref{E:BASICPHI} into~\eqref{E:BASICPDEPHI}, we will now derive a formal
hierarchy of ODEs satisfied by the functions $\phi_j$, $j\in\{1,\dots, M\}$. 
We define the {\em source term $S(\phia)$}:
\begin{align}
S(\phia):=- \pa_\tau^2 \phia - \frac{2}{9\phia^2} - \ve P[\phia] \label{E:SOURCETERM}
\end{align}

We first recall the formula of Faa Di Bruno (see e.g.~\cite{Jo}) which will be repeatedly used in this section. 
Given two functions $f,g$ with formal power series expansions,
\begin{align}\label{E:series}
f(x) = \sum_{n=0}^\infty \frac{f_n}{n!} x^n,  \ \ g(x) = \sum_{n=1}^\infty \frac{g_n}{n!} x^n,
\end{align}
we can compute the formal Taylor series expansion of the composition $h= f \circ g$ 
via 
\[
h(x) =  \sum_{n=0}^\infty \frac{h_n}{n!} x^n, 
\]
where $f_n$, $g_n$ and $h_n$ are constants with respect to $x$. Faa Di Bruno's formula gives 
\be\label{E:FDB}
h_n = \sum_{k=1}^n\sum_{\pi(n,k)} \frac{n!}{\lambda_1!\dots\lambda_{n}!} f_k \left(\frac{g_1}{1!}\right)^{\lambda_1}\dots \left(\frac{g_{n}}{n!}\right)^{\lambda_{n}}, \ \ 
h_0 = f_0,
\ee
where 
\begin{align}\label{E:PELLKDEF}
\pi(n,k) = \{(\lambda_1,\dots, \lambda_n) : \lambda_i\in\mathbb Z_{\ge0}, \, \sum_{i=1}^n\lambda_i = k, \,
\sum_{i=1}^n i \lambda_i = n\}.
\end{align}

An element of $\pi(n,k)$ encodes the partitions of the first $n$ numbers into $\lambda_i$ classes of cardinality $i$ for $i\in\{1,\dots,k\}$. 
Observe that by necessity $\lambda_j=0$ for any $n-k+2\le j\le n$.

\begin{lemma}[Detailed structure of the source term $S(\phia)$] \label{L:SOURCEDETAIL}
The source term $S(\phia)$ given by~\eqref{E:SOURCETERM} satisfies
\begin{align}\label{E:SOURCETERMEXACT}
S(\phia)& = - \sum_{j=0}^M\ve^j\left(\pa_{\tau\tau} \phi_j - \frac49 \phi_j\tau^{-2} - f_j\right) - \ve^{M+1} \left( R_P^\ve+ \phi_0^{-2}R^\ve_{M,2}\right) 
\end{align}
where $R_P^\ve = R_P^\ve[\phi_0,\phi_1,\dots,\phi_M]$ and $R^\ve_{M,2}=R^\ve_{M,2}[\phi_0,\phi_1,\dots,\phi_M]$
are explicitly given by~\eqref{E:RPEPSILONDEF},~\eqref{E:RMEPSILON} below,
$f_0 := 0$, and 
\[
f_j : = - \phi_0^{-2}\frac{\tilde O_{j}}{j!} - \sum_{m+i=j-1,\atop 0\le m,i\le M}\sum_{k=0}^i
\frac{\phi_k\phi_{i-k}}{w^\alpha  r^2} \DL\left(w^{1+\alpha}\mathscr J[\phi_0]^{-\gamma}\frac{h_m}{m!}\right), \ \ j=1,\dots,M-1,
\]
with $\tilde O_j $ and $h_m$ given explicitly below by~\eqref{E:HJTILDE} and~\eqref{E:HFORMULAJ}.
\end{lemma}


\begin{proof}
For any $m\in\mathbb N$, $\nu\in\mathbb R$, there exists a smooth function $R^\ve_{m,\nu}:\mathbb R^m\to \mathbb R$ such that 
\be\label{E:EPSILONEXPANSION}
(1+ \ve x_1 + \ve^2 x_2 + \dots \ve^m x_m)^{-\nu} = 1 + \sum_{j=1}^m \ve^j \frac{F_j}{j!} + \ve^{m+1}R^\ve_{m,\nu}(x_1,\dots,x_m),
\ee
where by the formula of Faa Di Bruno
\[
F_j = \sum_{k=1}^j\sum_{\pi(j,k)}  (-\nu)_k \frac{j!}{\lambda_1!\dots\lambda_{j}!} x_1^{\l_1}\dots x_j^{\l_j}, \ \ j=1,\dots,m,
\]
and $R^\ve_{m,\nu}$ is smooth in a neighbourhood of ${\bf 0}$, 
\be\label{E:REPSILONMNUPROPERTIES}
R^\ve_{m,\nu}({\bf 0})=0, \ \ 
\pa_{x^1}R^\ve_{m,\nu}({\bf 0})=0,
\ee
and for any $\ve\in(0,1)$
\[
\|R^\ve_{m,\nu}\|_{C^\ell} \le C_\ell, 
\]
for some constant $C_\ell>0$ which grows as $\ell$ gets larger. 

Recalling~\eqref{E:BASICPHI}, 
\begin{align}
\phia^{-2} & = \phi_0^{-2} \left(1+\sum_{j=1}^M \ve^j \frac{\phi_j}{\phi_0}\right)^{-2} \notag \\
& = \phi_0^{-2} \left(1 + \sum_{j=1}^M \ve^j \frac{O_j }{j!} + \ve^{M+1}R^\ve_{M,2}(\frac{\phi_1}{\phi_0},\dots,\frac{\phi_M}{\phi_0})\right) \notag \\
& = \phi_0^{-2} + \sum_{j=1}^M \ve^j \phi_0^{-2}\frac{O_j }{j!}+  \ve^{M+1}\phi_0^{-2}R^\ve_{M,2}(\frac{\phi_1}{\phi_0},\dots,\frac{\phi_M}{\phi_0}), \label{E:RMEPSILON}
\end{align}
where
\begin{align}
O_j   &= \sum_{\pi(j,k)} \frac{j!}{\lambda_1!\dots\lambda_{j}!} (-2)_k
\left(\frac{\phi_1}{\phi_0}\right)^{\lambda_1}\dots \left(\frac{\phi_{j}}{\phi_0}\right)^{\lambda_{j}} \notag \\
& = \sum_{\pi(j,k)} \frac{j!}{\lambda_1!\dots\lambda_{j}!} (-2)_k \phi_0^{-k}\left(\phi_1\right)^{\lambda_1}\dots \left(\phi_{n}\right)^{\lambda_{j}} \notag \\
& =  \Big[ \sum_{\pi(j,k) \atop k\ge 2}  + \sum_{\pi(j,k) \atop k=1}    \Big] \frac{j!}{\lambda_1!\dots\lambda_{j}!} (-2)_k \phi_0^{-k}\left(\phi_1\right)^{\lambda_1}\dots \left(\phi_{j}\right)^{\lambda_{j}}   \notag\\
& = \sum_{\pi(j,k) \atop k\ge 2} \frac{j!}{\lambda_1!\dots\lambda_{j}!} (-2)_k \phi_0^{-k}\left(\phi_1\right)^{\lambda_1}\dots \left(\phi_{n}\right)^{\lambda_{j}}
-2 \frac{\phi_j}{\phi_0} \notag \\
& = : \tilde O_j  - 2\frac{\phi_j}{\phi_0}. \label{E:HJTILDE}
\end{align}

Note that from \eqref{E:PELLKDEF}, for $k=1,$ $\lambda_{j}=1,$ and $\lambda_{i}=0$ for
$i<j,$ so that the summation for $k=1$ is given by $-2\frac{\phi_{j}}{\phi_{0}}.$ From \eqref{E:PELLKDEF} again, 
for $k\geq2,$ $\lambda_j=0$, and therefore in the definition of $\tilde O_j $, the expression depends only on $\phi_1,\dots\phi_{j-1}$, justifying the notation
$
\tilde O_j  = \tilde O_j[\phi_0,\dots,\phi_{j-1}].
$ 
Note that $\tilde O_1=0$.

Our next goal is to expand the function $\J[\phia]^{-\gamma}$ in the powers of $\ve$. To that end we first observe that 
\begin{align}
\J[\phia]  & = \J[\sum_{k=0}^{M}\ve^k \phi_k] \notag \\
& = \mathscr J[\phi_0] \sum_{k=0}^{M-1} \ve^k \bar{\mathscr J}_k + \ve^{M} R_{\J}[\phi_0,\phi_1,\dots,\phi_M], \label{E:JBARDEFPHI}
\end{align}
where
\be\label{E:JAFORMULAPHI}
\bar{\mathscr J}_k   = \frac{\J_k}{\J[\phi_0]} : = \sum_{m+i+j=k \atop 0\le m,i,j\le M-1} \frac{\phi_m \phi_i \left(\phi_j+\DL\phi_j\right)}{\J[\phi_0]}, \ \ k\in \{0,1,\dots,M-1\}
\ee
where we note that $\bar{\mathscr J}_0=1$ since $\mathscr J_0=\mathscr J[\phi_0]$. The remainder $R_{\J}$ is given by the formula
\be\label{E:RJDEF}
R_{\J}[\phi_0,\phi_1,\dots,\phi_M]:= \sum_{m+i+j\ge M \atop 0\le m,i,j\le M} \ve^{m+i+j-M}\frac{\phi_m \phi_i \left(\phi_j+\DL\phi_j\right)}{\J[\phi_0]}.
\ee

We have
\begin{align}
\left(\J[\phia]\right)^{-\gamma} & = \left(\mathscr J[\phi_0]\right)^{-\gamma} \left(1+\sum_{k=1}^{M-1} \ve^k\bar\J_k + \ve^{M} R_{\J}[\phi_0,\phi_1,\dots,\phi_M]\right)^{-\gamma} \notag \\ 
& = \left(\mathscr J[\phi_0]\right)^{-\gamma} \left( \sum_{j=0}^{M-1} \ve^j \frac{h_j}{j!} + \ve^{M}\frac{h_{M}}{M!} + \ve^{M+1}R^\ve_{M,\gamma}(\bar\J_1,\dots,\bar \J_{M-1},R_{\J})\right), 
\end{align}
where we use~\eqref{E:EPSILONEXPANSION}. Here $h_0=1$ and the formula of Faa Di Bruno gives
\begin{align}
h_j &= \sum_{k=1}^j \sum_{\pi(j,k)} \frac{j!}{\lambda_1!\dots\lambda_{j}!} (-\gamma)_k \left(\bar\J_1\right)^{\lambda_1}\dots \left(\bar\J_j\right)^{\lambda_{j}} \notag \\
& = \sum_{k=1}^j \sum_{\pi(j,k)} \frac{j!}{\lambda_1!\dots\lambda_{j}!} (-\gamma)_k \J[\phi_0]^{-k}\left(\J_1\right)^{\lambda_1}\dots \left(\J_j\right)^{\lambda_{j}} \label{E:HFORMULAJ}, 
\ \ j=1,\dots M-1,
\end{align}
and
\[
h_{M} =  \sum_{k=1}^{M}\sum_{\pi(M,k)} \frac{M!}{\lambda_1!\dots\lambda_{M}!} (-\gamma)_k \J[\phi_0]^{-k}\left(\J_1\right)^{\lambda_1}\dots \left(\J_{M-1}\right)^{\lambda_{M-1}}R^\ve_{M,\gamma}(\bar\J_1,\dots,\bar \J_{M-1},R_{\J})^{\l_{M}}.
\]
Similarly
\[
\phia^2  = \sum_{j=0}^{M-1} \ve^j \sum_{k=0}^j \phi_k\phi_{j-k} + \ve^{M}R^\ve[\phi_1,\dots,\phi_M],
\]
where
\[
R^\ve[\phi_1,\dots,\phi_M] := \sum_{k+m\ge M \atop 0\le k,m\le M} \ve^{k+m-M}\phi_k\phi_m
\]
we finally have 
\begin{align}
& P[\phia] \notag \\
&  = \frac{\sum_{j=0}^{M-1} \ve^j \sum_{k=0}^j \phi_k\phi_{j-k} + \ve^{M}R^\ve}{g^2(r) w^\alpha  r^2} \DL\left(w^{1+\alpha}\mathscr J[\phi_0]^{-\gamma}
\left(\sum_{j=0}^{M-1} \ve^j \frac{h_j}{j!} + \ve^{M}\frac{h_{M}}{M!} + \ve^{M+1}R^\ve_{M,\gamma}\right) \right) \notag \\
& = \sum_{j=0}^{M-1} \ve^j \left\{ \sum_{m+i=j,\atop 0\le m,i\le M-1}\sum_{k=0}^i
\frac{\phi_k\phi_{i-k}}{w^\alpha  r^2} \DL\left(w^{1+\alpha}\mathscr J[\phi_0]^{-\gamma}\frac{h_m}{m!}\right) \right\} 
+\ve^{M} R_P^\ve[\phi_0,\phi_1,\dots,\phi_M], 
\label{E:PEXPANSION}
\end{align}
where 
\begin{align}
R_P^\ve  = & 
 \sum_{m+i\ge M,\atop 0\le m,i\le M}\ve^{m+i-M}\sum_{k=0}^j
\frac{\phi_k\phi_{i-k}}{w^\alpha  r^2} \DL\left(w^{1+\alpha}\mathscr J[\phi_0]^{-\gamma}\frac{h_m}{m!}\right) \notag \\
 & + \frac{R^\ve}{g^2(r) w^\alpha  r^2} \DL\left(w^{1+\alpha}\J[\phia]^{-\gamma} \right) \notag \\
& +\frac{\sum_{j=0}^{M-1} \ve^j \sum_{k=0}^j \phi_k\phi_{i-k}}{g^2(r) w^\alpha  r^2} \DL\left(w^{1+\alpha}\mathscr J[\phi_0]^{-\gamma}\left(\frac{h_{M}}{M!} + \ve R^\ve_{M,\gamma} \right)\right) \label{E:RPEPSILONDEF}
\end{align}

From the definition of the source term~\eqref{E:SOURCETERM} in therefore follows
\begin{align}
S(\phia) & = - \sum_{j=0}^M \ve^j \pa_{\tau\tau} \phi_j + \frac49 \sum_{j=0}^M\ve^j\phi_j\tau^{-2} - \sum_{j=0}^{M-1} \ve^{j+1} \phi_0^{-2}\frac{\tilde O_{j+1}}{(j+1)!} \notag \\
& \ \ \ \ - \sum_{j=0}^{M-1} \ve^{j+1} \left\{ \sum_{m+i=j,\atop 0\le m,i\le M}\sum_{k=0}^i
\frac{\phi_k\phi_{i-k}}{w^\alpha  r^2} \DL\left(w^{1+\alpha}\mathscr J[\phi_0]^{-\gamma}\frac{h_m}{m!}\right) \right\}  \notag \\
& \ \ \ \ - \ve^{M+1} \left( R_P^\ve[\phi_0,\phi_1,\dots,\phi_M] + \phi_0^{-2}R^\ve_{M,2}[\frac{\phi_1}{\phi_0},\dots,\frac{\phi_M}{\phi_0}]\right) \notag \\
& = - \sum_{j=0}^M\ve^j\left(\pa_{\tau\tau} \phi_j - \frac49 \phi_j\tau^{-2} - f_j\right) - \ve^{M+1} \left( R_P^\ve+ \phi_0^{-2}R^\ve_{M,2}\right),
\end{align}
with $f_0 := 0$ and 
\be\label{E:FNDEFPHI}
f_j : = - \phi_0^{-2}\frac{\tilde O_{j}}{j!} - \sum_{m+i=j-1,\atop 0\le m,i\le M}\sum_{k=0}^i
\frac{\phi_k\phi_{i-k}}{w^\alpha  r^2} \DL\left(w^{1+\alpha}\mathscr J[\phi_0]^{-\gamma}\frac{h_m}{m!}\right), \ \ j=1,\dots,M-1,
\ee
as claimed.
\end{proof}

Motivated by the previous lemma, we  define the {\em hierarchy} of ODEs 
\begin{align}\label{E:HPHI}
\pa_{\tau\tau}\phi_{j+1} - \frac{4\phi_{j+1}}{9\tau^2} = f_{j+1}, \ \ j\in\{0,1,\dots M-1\},
\end{align}
where $\phi_{0}^{3}=\tau^{2}$, $f_j$ is given by~\eqref{E:FNDEFPHI} and 
$\tilde O_j $ and $h_j$ given by~\eqref{E:HJTILDE} and~\eqref{E:HFORMULAJ} respectively.

With $\{\phi_j\}_{j=1,\dots,M}$ satisfying~\eqref{E:HPHI}, Lemma~\ref{L:SOURCEDETAIL}
in particular implies that 
\begin{align}\label{E:SOURCETERM1}
S(\phia) = - \ve^{M+1} \left( R_P^\ve+ \phi_0^{-2}R^\ve_{M,2}\right), 
\end{align}
with $R_P^\ve = R_P^\ve[\phi_0,\phi_1,\dots,\phi_M]$ and $R^\ve_{M,2}=R^\ve_{M,2}[\phi_0,\phi_1,\dots,\phi_M]$
 given by~\eqref{E:RPEPSILONDEF},~\eqref{E:RMEPSILON}. Therefore, by solving the hierarchy up to order $M$ we force the source term 
 to be of order $\ve^{M+1}$.

\subsection{Solution operators and definition of $\phi_j$, $j\in\mathbb Z_{>0}$}
For any $\gamma\in(1,\frac43)$  we define
\be\label{E:NDEF}
N= N(\gamma) := 
\left\lfloor \frac1{\gamma-1}\right\rfloor + 6 
=\lfloor\alpha\rfloor+6. 
\ee
The number $N$ will later correspond to the total number of derivatives used in our energy estimates. 


\begin{definition}[The ``gain" $\delta$ and $\delta^\ast$]
Let $\gamma\in(1,\frac43)$ be given and let $\bar\gamma = \frac43-\gamma$. For any natural number $n>\frac{N+2}{2\bar\gamma}$ 
we define 
\begin{align}
\delta = \delta(n) & : = 2 \left(\frac43-\gamma-\frac1n\right) \label{E:DELTADEF} \\
\delta^\ast =\delta^\ast(n)& : = \delta(n) - \frac{N}{n} = \frac83-2\gamma-\frac{N+2}n \label{E:DELTASTARDEF}
\end{align}
\end{definition}


\begin{lemma}\label{L:NUMERICAL}
Let $\gamma\in(1,\frac43)$ be given and fix an arbitrary natural number $a\in \mathbb Z_{>0}$. Then there exists an $n^*=n^*(\gamma, a)$ such that 
\be\label{E:NUMERICAL}
\left\lfloor \frac2{3\delta(n)}\right\rfloor \delta(n) + \frac2n <\frac23 <\left(\left\lfloor \frac2{3\delta(n)}\right\rfloor+1\right)\delta(n) - \frac{a}{n}, \ \ n\ge n^*.
\ee
In fact
\[
\left\lfloor \frac2{3\delta(n)}\right\rfloor = \left\lfloor \frac1{3\bar\gamma}\right\rfloor, \ \ n\ge n^*.
\]
\end{lemma}


\begin{proof}
For the simplicity of notation let $j:=\left\lfloor \frac2{3\delta(n)}\right\rfloor$. Then it is easy to check that~\eqref{E:NUMERICAL} is equivalent to
\begin{align}
j+  \frac 1{n\bar\gamma-1}<\frac1{3\bar\gamma} + \frac1{3\bar\gamma(n\bar\gamma-1)} < j+1 - \frac a{2(n\bar\gamma-1)}.
\end{align}
Since $1<\frac1{3\bar\gamma}$ it is
clear that the above inequality will be true if  $n$ is chosen sufficiently large.
\end{proof}


\begin{remark}
Lemma~\ref{L:NUMERICAL} implies in particular $\frac2{3\delta} \notin \mathbb Z_{>0}$ since by~\eqref{E:NUMERICAL} $\left\lfloor \frac2{3\delta}\right\rfloor  <\frac{2}{3\delta}$.
\end{remark}


\begin{definition}[Regularity parameter $\l$]\label{D:ALAMBDADEF}
Let $\gamma\in(1,\frac43)$ be given. Choose an 
$n>n^*(\gamma,2N(\gamma))$ (where $n^*(\gamma,a)$ is given by Lemma~\ref{L:NUMERICAL}) sufficiently large so that 
\[
\l: = \frac{2N}{n}<1.
\] 
\end{definition}


\begin{remark}
A simple consequence of Lemma~\ref{L:NUMERICAL} and Definition~\ref{D:ALAMBDADEF}
is the bound
\begin{align}\label{E:DELTALOWBOUND}
\delta > \frac{2N+2}{n}, \ \ \text{ i.e. } \ \delta^\ast >\frac{N+2}{n}.
\end{align}
\end{remark}

Motivated by~\eqref{E:HPHI}, consider for a moment a general inhomogeneous ODE of the form
\[
\pa_{\tau\tau}\phi - \frac{4}{9\tau^2}\phi = f.
\]
A simple calculation shows that the previous ODE is formally equivalent to
\[
\tau^{-\frac43}\pa_\tau \left(\tau^{\frac83}\pa_\tau \left(\tau^{-\frac43} \phi\right)\right) = f.
\]
This motivates the following definition of the {\em solution} operators:
\begin{align}\label{E:S1}
S_1[f,g,h]( \tau, r) = f( \tau, r) \int_\tau^1 g(\tau', r ) \int_{\tau'}^0 h(\tau'', r )\,d\tau''d\tau', \ \ \tau\in[0,1],
\end{align}

\begin{align}\label{E:S2}
S_2[f,g,h]( \tau, r) = f( \tau, r) \int_0^\tau g(\tau', r ) \int_{0}^{\tau'} h(\tau'', r )\,d\tau''d\tau', \ \ \tau\in[0,1].
\end{align}

By direct inspection, one can check that for a given $f$ functions $S_i[\tau^{\frac43},\tau^{-\frac83},\tau^{\frac43} f]$, $i=1,2$ are
solutions of $\pa_{\tau\tau}\phi - \frac{4}{9\tau^2}\phi = f$.
We define
\begin{align} \label{E:CHINFORMULA}
\phi_j : = 
\begin{cases}
 S_1[\tau^{\frac43},\tau^{-\frac83},\tau^{\frac43} f_j]& \text{ if } \ j\le 
 \left\lfloor \frac1{3\bar\gamma}\right\rfloor 
 ,\\  
S_2[\tau^{\frac43},\tau^{-\frac83},\tau^{\frac43} f_j]& \text{ if } \ j>
\left\lfloor \frac1{3\bar\gamma}\right\rfloor. 
\end{cases}
\end{align}

The above definition of the solution is designed to enforce the gain of $\tau^\delta$ with respect to the previous iterate for all $j\in\{1,\dots,M\}$. Since $M\gg\lfloor \frac1{3\bar\gamma}\rfloor$, the above choice of the formula at the index values $j>\lfloor \frac1{3\bar\gamma}\rfloor$ is crucial, see Proposition~\ref{P:MAINBOUNDPHI} and Lemma~\ref{L:lem2.12}.

\subsection{Bounds on $\phi_j$ and proof of Theorem~\ref{T:MAINBOUNDPHI} }
The main goal of this section is the proof of Theorem~\ref{T:MAINBOUNDPHI}. To that end, we need a number of preparatory steps. We first introduce the notation
\begin{align}
q_\nu(x)&: = (1+x)^\nu, \ \ \nu\in\mathbb R, \ \ x\ge0, \label{E:QDEF} \\
p_{\mu,\nu}(x) & : = \frac{x^{\mu+\nu}}{(1+x)^\mu}, \ \ \mu,\nu\in\mathbb R, \ \  x\ge0.  \label{E:PFUNCTIONDEF}
\end{align}

For the remainder of the section, constants $M,K\in\mathbb Z_{>0}$ are arbitrarily large fixed constants. 


\begin{lemma}[Basis of the induction]\label{L:BASIS}
Let $\phi_1$ be given by \eqref{E:CHINFORMULA}. 
Then
\begin{align}
\left\vert\pa_\tau^m \rr^{\ell}\phi_{1}\right\vert & \lesssim  \tau^{\frac23 +\delta-m}\pet, \ \ \ell,m \in \{0,1,\dots,K\}.
\label{E:INDUCTIONBASIS}
\end{align}
\end{lemma}


The main result of this section is the quantitative estimate on the space-time derivatives of the iterates $\phi_j$. 


\begin{proposition}[Inductive step]\label{P:MAINBOUNDPHI}
Let $\phi_j$ be given by \eqref{E:CHINFORMULA}. Let $1\le I< M$ be given and assume that for any $j\in\{1,\dots, I\}$ 
and any $\ell,m \in \{0,1,\dots,K\}$ we have
\begin{align}
\left\vert \pa_\tau^m \rr^{\ell} \phi_j \right\vert  & \lesssim  \tau^{\frac23 + j\delta-m} p_{\l,-\frac2n}\et,
\ \ 
\text{(Inductive Assumptions).} 
\label{E:ASS1} 
\end{align}
Then for any $\ell,m \in \{0,1,\dots,K\}$ the following bound holds
\begin{align}
\left\vert\pa_\tau^m \rr^{\ell}\phi_{I+1}\right\vert & \lesssim  \tau^{\frac23 + (I+1)\delta-m}\pet,\label{E:ASS1RECOVERED}
\end{align}
where $\l = \frac{2N}{n}$ is given in Definition~\ref{D:ALAMBDADEF}. 
\end{proposition}

\begin{remark}
The constants in the above statement depend on $K,M\in \mathbb Z_{>0}$ and they generally grow as $K$ and $M$ get larger.
\end{remark}

Proofs of Lemma~\ref{L:BASIS}, Proposition~\ref{P:MAINBOUNDPHI}, and finally Theorem~\ref{T:MAINBOUNDPHI} are contained in Section~\ref{SSS:PROOFS}. Before that we need a number of 
auxiliary bounds.

\subsubsection{Auxiliary lemmas}

Since 
\[
[(\tau-1)\r(\log g)\pa_\tau, \r] = - \rr^2(\log g) (\tau-1)\pa_\tau,
\]
it is easy to see that for any $\ell\in\mathbb N$ there exist some universal constants $k_{abc_1\dots c_a}\ge0$, $a,b,c_j=1,\dots, \ell$, $j=1,\dots a$, such that 
\be\label{E:DELLLEIBNIZ}
\DL^\ell = \left((\tau-1)\r(\log g) \pa_\tau + \r\right)^\ell = \sum_{a+b+c_1+\dots c_a= \ell \atop 1\le a,b,c_j\le \ell} 
k_{abc_1\dots c_a} \prod_{j=1}^a \rr^{c_j}(\log g)\left((\tau-1)\pa_\tau\right)^a \rr^b.
\ee

\begin{lemma}[Auxiliary estimates]
Let $\ell,m \in \{0,1,\dots,K\}$ be given nonnegative integers. 
Under the inductive assumptions~\eqref{E:ASS1} the following estimates hold:
\begin{align}
\left\vert \pa_\tau^m\rr^{\ell} \J[\phi_0] \right\vert &\lesssim 
\begin{cases}
\tau^{2} r^{-nm}q_{m+1}\et, & \ \ell=0 \\
\tau^{2} r^{-nm}q_{m}\et \frac{ r^n}\tau, & \ \ell>0;
\end{cases} 
\label{E:PARTIALTAUMRRELLJBOUND} \\
\left\vert  \pa_\tau^m\rr^{\ell}\left(\phi_0^{-k}\right)\right\vert  & \lesssim \tau^{-\frac{2k}3-m}, \ \ k\in\mathbb Z_{\ge0};\label{E:PARTIALTAULCHI0} \\
\left\vert  \pa_\tau^m\left(\J[\phi_0]^{-k}\right)\right\vert 
&\lesssim  \tau^{-2k-m }q_{-k}\et,  \ \ k\in\mathbb Z_{\ge0}; \label{E:PARTIALTAUJPOWERMINUSK}\\
\left\vert  \pa_\tau^m\rr^{\ell}\left(\J[\phi_0]^{-k}\right)\right\vert 
&\lesssim  \tau^{-2k -m}q_{-k-1}\et \frac{ r^n}{\tau}, \ \ k\in\mathbb Z_{\ge0}, \ \ell>0; 
\label{E:BARLAMBDAPOWERMINUSK}\\
\left\vert \pa_\tau^m\rr^{\ell}\mathscr J_k\right\vert &\lesssim 
 \tau^{2+k\delta-m} q_{1}\et \pet, \ \ k\in \{1,\dots,I\}, \label{E:PARTIALBARLAMBDAJBOUNDK} \\
\left\vert \pa_\tau^m\rr^{\ell}\left(\left(\J_k\right)^a\right) \right\vert &\lesssim 
 \tau^{(2+k\delta)a-m} q_{a}\et \pet, \ \ k\in \{1,\dots,I\}, \ a\ge0,\label{E:PARTIALBARLAMBDAJBOUNDKALT} \\
\left\vert \pa_\tau^m\rr^{\ell}\left(\left(\phi_k\right)^a\right) \right\vert &\lesssim  \, \tau^{(\frac23+k\delta)a-m} \pet^a, \ \ k\in \{1,\dots,I\}, \ a\ge0. \label{E:PARTIALBARLAMBDACHIBOUNDKM} 
\end{align}
\end{lemma}

\begin{proof}
\noindent
{\bf Proof of~\eqref{E:PARTIALTAUMRRELLJBOUND}.}
By the Leibniz rule for any $k_1,k_2\in\mathbb N$ and any smooth function $\varphi$ we have
\begin{align}
& \left\vert \pa_\tau^{k_1}\rr^{k_2} \DL \varphi \right\vert  \notag \\
& \lesssim \left\vert  \pa_\tau^{k_1}\rr^{k_2+1} \varphi \right\vert + 
\sum_{m=0}^{k_2} \left\vert \rr^{m+1}\left(\log g\right) \right\vert \left( \left\vert \rr^{k_2-m}\pa_\tau^{k_1}\varphi \right\vert+ |(\tau-1)
\rr^{k_2-m}\pa_\tau^{k_1+1}\varphi |  \right)
\notag \\
& \lesssim  \left\vert  \pa_\tau^{k_1}\rr^{k_2+1} \varphi \right\vert + 
 r^n \sum_{m=0}^{k_2} \left( \left\vert \pa_\tau^{k_1}\rr^{k_2-m}\varphi\right\vert  + \left\vert \pa_\tau^{k_1+1}\rr^{k_2-m}\varphi\right\vert \right), \label{E:PHIUSEFUL}
\end{align}
where we have used~\eqref{E:GTAYLOR} and $\tau\le1$ in the last estimate.
Letting $\varphi=\phi_0 = \tau^{\frac23}$ above we obtain
\begin{align}\label{E:DPHIZEROBOUND}
\left\vert \pa_\tau^{k_1}\rr^{k_2} \DL \phi_0 \right\vert \lesssim  r^n\tau^{-\frac13-k_1}.
\end{align}
Now for any $\ell,m \in \{0,1,\dots,K\}$, we have
\begin{align}
&\left\vert \pa_\tau^m\rr^\ell\left(\mathscr J[\phi_0]\right)\right\vert \notag\\
&\lesssim  \sum_{\alpha_1+\alpha_2+\alpha_3 = m \atop
\beta_1+\beta_2+\beta_3 = \ell} \left\vert \pa_\tau^{\alpha_1}\rr^{\beta_1}\phi_0\pa_\tau^{\alpha_2}\rr^{\beta_2}\phi_0\left(\pa_\tau^{\alpha_3}\rr^{\beta_3}\phi_0
+ \pa_\tau^{\alpha_3}\rr^{\beta_3}\DL\phi_0\right)\right\vert 
\end{align}
where we recall $\J[\phi_0]=\phi_0^2(\phi_0+\DL\phi_0)$. Therefore, if $\ell=0$, we have 
\begin{align}
\left\vert \pa_\tau^m \left(\mathscr J[\phi_0]\right)\right\vert& \lesssim \tau^{\frac23-\alpha_1} \tau^{\frac23-\alpha_2} \left( \tau^{\frac23-\alpha_3} + r^n \tau^{-\frac13-\alpha_3}\right) = \tau^{2-m} q_1\et \notag 
\end{align}
and if $\ell>0$, since $(r\partial_{r})^{\beta}\phi_{0}=0$ for $\beta\neq0,$ we have 
\[
\left\vert \pa_\tau^m\rr^\ell\left(\mathscr J[\phi_0]\right)\right\vert \lesssim\tau^{\frac{2}{3}-\alpha_{1}}\tau^{\frac{2}{3}-\alpha_{2}}r^{n}\tau^{-\frac
{1}{3}-\alpha_{3}}=\tau^{2-m}(\frac{r^{n}}{\tau})
\]
which lead to \eqref{E:PARTIALTAUMRRELLJBOUND}.

\noindent
{\bf Proof of~\eqref{E:PARTIALTAULCHI0}}
The bound is obvious. 

\noindent
{\bf Proof of~\eqref{E:PARTIALTAUJPOWERMINUSK}}
We use the formula of Faa Di Bruno.
We may write $\J[\phi_0]^{-k}(t,r) = f(h(r))$ where $f(x) = x^{-k}$ and $h(r) = \J[\phi_0]$. 
Derivatives of $x\mapsto f(x)$ are easily computed:
\[
f^{(j)}(\sigma) =  
\left(-k\right)_{j} \sigma^{-k-j}, \ \ k\in\mathbb N.
\]
Formula of Faa Di Bruno then gives
\begin{align}
\pa_\tau^m\left(\J[\phi_0]^{-k}\right) &= \sum_{j=1}^m \left(-k\right)_{j} \J[\phi_0]^{-k-j}
 \sum_{\pi(m,j)} \,  m !  \, \prod_{i=1}^m \,  \frac{ 
 \left(\pa_\tau^i\J[\phi_0]\right)^{\lambda_i}}{\lambda_i! (i!)^{\lambda_i}},  \label{E:FDB1}
\end{align}
where we refer to~\eqref{E:PELLKDEF} for the definition of $\pi(m,j)$.  
Since 
\be
\J[\phi_0] = \tau^2\left(1+\frac23\frac{(\tau-1)}\tau \r(\log g)\right)
\ee
and 
\be
 r^n\gtrsim - \r(\log g) \gtrsim  r^n, \ \  r \in[0,1],
\ee
it follows 
\be
\tau^2q_1\et\gtrsim \J[\phi_0] \gtrsim \tau^2q_1\et, \ \ \tau\in(0,1].
\ee

Therefore, since $\sum\lambda_{i}=j$ and $\sum i\lambda_{i}=m,$ we have
\begin{align}
\left\vert \pa_\tau^m\left(\J[\phi_0]^{-k}\right)  \right\vert
& \lesssim \sum_{j=1}^m \tau^{-2k-2j} q_{-(k+j)}\et \sum_{\pi(m,j)}\prod_{i=1}^m \tau^{(2-i)\l_i} q_{\l_i}\et \notag \\
& \lesssim \tau^{-2k-m} q_{-(k+j)}\et q_{j}\et \notag \\
& \lesssim \tau^{-2k-m} q_{-k}\et.
\end{align}

\noindent
{\bf Proof of~\eqref{E:BARLAMBDAPOWERMINUSK}}
Using the formula of Faa Di Bruno like above, replacing formally $\pa_\tau^m$ by $\rr^\ell$ we obtain
\begin{align}
\rr^\ell\left(\J[\phi_0]^{-k}\right) &= \sum_{j=1}^\ell \left(-k\right)_{j} \J[\phi_0]^{-k-j}
 \sum_{\pi(\ell,j)} \,  \ell !  \, \prod_{i=1}^\ell \,  \frac{ \left(\rr^i\J[\phi_0]\right)^{\lambda_i}}{\lambda_i! (i!)^{\lambda_i}},  \label{E:FDB1}
\end{align}
where we refer to~\eqref{E:PELLKDEF} for the definition of $\pi(\ell,k)$.  Therefore, for $\ell\in\mathbb Z_{\ge1}$, by using the Leibniz rule
\begin{align}
&\pa_\tau^m\rr^\ell\left(\J[\phi_0]^{-k}\right) \notag \\
&= \sum_{m'=0}^m{m \choose m'}\sum_{j=1}^\ell \left(-k\right)_{j} \pa_\tau^{m-m'}\left(\J[\phi_0]^{-k-j}\right) 
 \sum_{\pi(\ell,j)} \,  \ell !  \, \pa_\tau^{m'}\left(\prod_{i=1}^\ell \,  \frac{ \left(\rr^i\J[\phi_0]\right)^{\lambda_i}}{\lambda_i! (i!)^{\lambda_i}}\right). \label{E:FDB1PRIME}
\end{align}
Notice that for any $d,\ell\in\mathbb Z_{\ge0}, i\in\mathbb Z_{\ge1}$,
\begin{align}
\left\vert \pa_\tau^d\left(\left(\rr^i\J[\phi_0]\right)^{\ell}\right)\right\vert 
&\lesssim \sum_{d_1+\dots + d_{\ell}=d} \prod_{j=1}^{\ell} \left\vert \left(\pa_\tau^{d_j}\rr^i\J[\phi_0]\right)\right\vert \notag \\
& \lesssim \sum_{d_1+\dots + d_{\ell}=d} \prod_{j=1}^{\ell} \left(\tau^{2-d_j} \frac{ r^n}{\tau}\right) \notag\\
& \lesssim \tau^{2\ell -d} \et^\ell \label{E:PARTIALTAUDRRI}
\end{align}
where we have made use of~\eqref{E:PARTIALTAUMRRELLJBOUND}.
From this bound, the product rule, and~\eqref{E:FDB1PRIME}, we conclude 
\begin{align}
&\left\vert \pa_\tau^m\rr^\ell\left(\J[\phi_0]^{-k}\right) \right\vert \notag \\ 
& \lesssim  \sum_{m'=0}^m\sum_{j=1}^\ell \tau^{-2(k+j)-m+m'}q_{-(k+j)}\et   \sum_{\pi(\ell,j)} \sum_{m_1+\dots +m_\ell = m'} \prod_{i=1}^\ell \tau^{2\l_i -m_i}
\et^{\l_i} \notag\\
& \lesssim \tau^{-2k-m}\sum_{m'=0}^m\sum_{j=1}^\ell \sum_{\pi(\ell,j)} \sum_{m_1+\dots +m_\ell = m'} \tau^{-2j} \tau^{2\sum_{i=1}^\ell\l_i} q_{-k-j}\et \et^{\sum_{i=1}^\ell\l_i} \notag\\
& \lesssim  \tau^{-2k-m}q_{-k-1}\et \frac{ r^n}{\tau}\label{E:FDB2}
\end{align}
where we have used~\eqref{E:PARTIALTAUMRRELLJBOUND},~\eqref{E:PARTIALTAUDRRI},~\eqref{E:QDEF}, the identity
$\sum_{i=1}^\ell \l_i = j$ which follows from the definition of the index set $\pi(\ell,j)$, and the trivial estimate $x\le q_1(x)$.

\noindent
{\bf Proof of~\eqref{E:PARTIALBARLAMBDAJBOUNDK}}
By letting $\varphi=\phi_j$, $j\in\{1,\dots, I\}$, in~\eqref{E:PHIUSEFUL} we obtain
\begin{align}
\left\vert \pa_\tau^a\rr^{b} \DL\phi_j  \right\vert& \lesssim  \left\vert  \pa_\tau^{a}\rr^{b+1} \phi_j \right\vert 
+  r^n \sum_{m=0}^{b} \left( \left\vert\pa_\tau^{a}\rr^{b-m}\phi_j\right\vert+\left\vert\pa_\tau^{a+1}\rr^{b-m}\phi_j\right\vert\right) \notag \\
& \lesssim \tau^{\frac23 +j\delta-a}p_{\l,-\frac2n}\et  +  r^n \tau^{\frac23 +j\delta-(a+1)}p_{\l,-\frac2n}\et \notag \\
& \lesssim \tau^{\frac23 +j\delta-a} q_1\et \pet, \label{E:PHIJUSEFUL}
\end{align}
where we have used the inductive assumption~\eqref{E:ASS1}. 
If $j=0$, 
from~\eqref{E:DPHIZEROBOUND} we have 
\be\label{E:PHIZEROUSEFUL}
\left\vert \pa_\tau^a\rr^{b} \DL\phi_0  \right\vert \lesssim \tau^{\frac{2}{3}-a}(\frac{r^{n}}{\tau})\leq\tau
^{\frac{2}{3}-a}q_{1} \et .
\ee
Recalling $\J_{k}$ from \eqref{E:JAFORMULAPHI}, applying the Leibniz rule and using~\eqref{E:ASS1} and~\eqref{E:PHIJUSEFUL}--\eqref{E:PHIZEROUSEFUL} we obtain
\begin{align}
\left\vert \pa_\tau^m\rr^{\ell} {\mathscr J}_k \right\vert &\lesssim \sum_{d+n+j=k \atop d,n,j\ge0}\sum_{\alpha_1+\alpha_2+\alpha_3 = m \atop
\beta_1+\beta_2+\beta_3 = \ell} \left\vert
\pa_\tau^{\alpha_1}\rr^{\beta_1}\phi_d\pa_\tau^{\alpha_2}\rr^{\beta_2}\phi_n\left(\pa_\tau^{\alpha_3}\rr^{\beta_3}\phi_j+\pa_\tau^{\alpha_3}\rr^{\beta_3}\DL\phi_j\right) \right\vert \notag \\
& \lesssim  \sum_{d+n+j=k \atop d,n,j\ge0}\sum_{\alpha_1+\alpha_2+\alpha_3 = m \atop
\beta_1+\beta_2+\beta_3 = \ell} \tau^{\frac23+d\delta-\alpha_1} \tau^{\frac23 +n\delta-\alpha_2} \notag \\
& \ \ \ \ \left( \tau^{\frac23 +j\delta-\alpha_3} +\tau^{\frac23 +j\delta-\alpha_3} q_{1}\et \right) \pet \notag \\
& \lesssim  \sum_{d+n+j=k \atop d,n,j\ge0}\sum_{\alpha_1+\alpha_2+\alpha_3 = m \atop
\beta_1+\beta_2+\beta_3 = \ell} 
\tau^{3\times \frac23+(d+n+j)\delta- (\alpha_1+\alpha_2+\alpha_3)}q_{1}\et\pet \notag \\
&\lesssim  \tau^{2+k\delta-m}q_{1}\et \pet.
\label{E:PARTIALRLEIBNIZ1}
\end{align}

\noindent
{\bf Proof of~\eqref{E:PARTIALBARLAMBDAJBOUNDKALT}.}
We use the Faa Di Bruno formula again. 
Analogously to~\eqref{E:FDB1PRIME} we obtain
\begin{align*}
&\pa_\tau^m\rr^\ell\left(\J_k^a\right) = \sum_{m'=0}^m{m \choose m'}\sum_{j=1}^\ell \left(a\right)_{j} \pa_\tau^{m-m'}\left(\J_k^{a-j}\right) 
 \sum_{\pi(\ell,j)} \,  \ell !  \, \pa_\tau^{m'}\left(\prod_{i=1}^\ell \,  \frac{ \left(\rr^i\J_k\right)^{\lambda_i}}{\lambda_i! (i!)^{\lambda_i}}\right), 
 \end{align*}
 where $(a)_j=a(a-1)\dots(a-j+1)$.
Notice that for any $d,\ell\in\mathbb Z_{\ge0}, i\in\mathbb Z_{\ge1}$,
\begin{align*}
\left\vert \pa_\tau^d\left(\left(\rr^i\J_k\right)^{\ell}\right)\right\vert 
&\lesssim \sum_{d_1+\dots + d_{\ell}=d} \prod_{j=1}^{\ell} \left\vert \left(\pa_\tau^{d_j}\rr^i\J_k\right)\right\vert \notag \\
& \lesssim \sum_{d_1+\dots + d_{\ell}=d} \prod_{j=1}^{\ell} \left(\tau^{2+k\delta-d_j} q_1\et \pet\right) \notag\\
& \lesssim \tau^{(2+k\delta)\ell -d}q_\ell \et \pet^{\ell}
\end{align*}
where we have made use of~\eqref{E:PARTIALBARLAMBDAJBOUNDK}. Using this bound just like in~\eqref{E:FDB2}, we obtain~\eqref{E:PARTIALBARLAMBDAJBOUNDKALT}.

\noindent
{\bf Proof of~\eqref{E:PARTIALBARLAMBDACHIBOUNDKM}.}
Using the formula of Faa Di Bruno, for any $k\in\{1,\dots, I\}$ we have
\begin{align*}
\left\vert\pa_\tau^m\left(\phi_k^a\right)\right\vert &\lesssim \sum_{j=1}^m|\phi_k|^{a-j}\sum_{\pi(m,j)}\prod_{i=1}^m\left\vert \pa_\tau^i \phi_k\right\vert^{\l_i} \\ 
& \lesssim \sum_{j=1}^m \tau^{(\frac23+k\delta)(a-j)} \pet^{a-j} \sum_{\pi(m,j)}\prod_{i=1}^m \left(\tau^{\frac23+k\delta-i} \pet\right)^{\l_i} \\
& \lesssim \tau^{(\frac23+k\delta)a-m} \pet^a,
\end{align*}
where we have used the inductive assumption~\eqref{E:ASS1}, identities $\sum_{i=1}^m\l_i = j$, $\sum_{i=1}^m(i\l_i)=m$ from \eqref{E:PELLKDEF}, and the additive property of $p_{\mu,\nu}$.

By analogy to~\eqref{E:FDB2} we have
\begin{align*}
&\left\vert \pa_\tau^m\rr^\ell \left(\phi_k\right)^a \right\vert 
\lesssim \sum_{m'=0}^m\sum_{j=1}^\ell   \left\vert \pa_\tau^{m-m'}(\phi_k)^{a-j}\right\vert  \sum_{\pi(\ell,j)}\sum_{m_1+\dots +m_\ell = m'} \prod_{i=1}^\ell \,  
\left\vert \pa_\tau^{m_i}\left(\left(\rr^i\phi_k\right)^{\l_i}\right) \right\vert \\
& \lesssim \sum_{m'=0}^m\sum_{j=1}^\ell \tau^{(\frac23+k\delta)(a-j)-m+m'}\pet^{a-j} 
\sum_{\pi(\ell,j)}\sum_{m_1+\dots +m_\ell = m'} \tau^{(\frac23+k\delta)\l_i-m_i} \pet^{\l_i}\\
& \lesssim \tau^{(\frac23+k\delta)a-m} \pet^a, \ \ k\in\{1,\dots, I\},
\end{align*}
where we have used the inductive assumption~\eqref{E:ASS1}, identities $\sum_{i=1}^m\l_i = j$, $\sum_{i=1}^m(i\l_i)=m$ and the additive property of $q_\nu$.
\end{proof}

\begin{lemma} Recall $h_j$ and $\tilde O_j $ from \eqref{E:HFORMULAJ} and \eqref{E:HJTILDE}. 
Under the inductive assumptions~\eqref{E:ASS1} the following estimates hold:
\begin{align}
\left\vert \pa_\tau^m \rr^{\ell} h_j\right\vert & \lesssim   \tau^{j\delta-m} p_{1+\l,-\frac{2}n}\et, \ \ j\in\{1,\dots, I\},
\label{E:LAMBDABARSMALLHBOUND}\\
\left\vert\pa_\tau^m \rr^{\ell}\tilde O_{j+1}\right\vert & \lesssim \tau^{j\delta-m} p_{2\l,-\frac4n}\et, \ \ j\in\{1,\dots, I\},\label{E:LAMBDABARHBOUND}\\
\left\vert \pa_\tau^m \rr^{\ell}\left(w^{-\alpha}\DL\left(w^{1+\alpha}\mathscr J[\phi_0]^{-\gamma}h_j\right)\right) \right\vert 
&
\lesssim
\begin{cases}  
\tau^{-2\gamma +j\delta-m} q_{-\gamma+1}\et p_{2+\l,-\frac2n}\et & \ j\in\{1,\dots, I\} \\
\tau^{-2\gamma -m} q_{-\gamma+1}\et p_{1,0}\et& \ j=0.
\end{cases}
\label{E:JSMALLHBOUND}
\end{align}
\end{lemma}

\begin{proof}
\noindent
{\bf Proof of~\eqref{E:LAMBDABARSMALLHBOUND}.} Recall  \eqref{E:HFORMULAJ}. 
For any $j\in\{1,\dots, I\}$ by the Leibniz rule
\begin{align*}
&\left\vert \pa_\tau^m \rr^{\ell} h_j\right\vert  \\
&\lesssim \sum_{k=1}^j\sum_{\pi(j,k)}\sum_{\alpha_0+\alpha_1+\dots +\alpha_j = m \atop
\beta_0+\beta_1+\dots+\beta_j = \ell} \left\vert \pa_\tau^{\alpha_0} \rr^{\beta_0}
\left(\J[\phi_0]^{-k}\right)\right\vert \left\vert \pa_\tau^{\alpha_1} \rr^{\beta_1}\left(\left(\J_1\right)^{\lambda_1}\right)\right\vert \dots 
\left\vert \pa_\tau^{\alpha_j} \rr^{\beta_j}\left(\left(\J_j\right)^{\lambda_{j}}\right) \right\vert \\
& \lesssim 
 \sum_{k=1}^j\sum_{\pi(j,k)}\sum_{\alpha_0+\alpha_1+\dots +\alpha_j = m \atop
\beta_0+\beta_1+\dots+\beta_j= \ell} \tau^{-2k+ \sum_{i=1}^j(2+i\delta)\l_i -(\alpha_0+\alpha_1+\dots\alpha_j)} q_{-k}\et p_{1,0}\et   \\
& \ \ \ \ \quad q_{\l_1}\et \pet^{\l_1}\dots 
q_{\l_j}\et \pet^{\l_j} \\
& \lesssim \tau^{j\delta-m} \sum_{k=1}^j p_{1+k\l,-\frac{2k}n}\et \lesssim \tau^{j\delta-m} p_{1+\l,-\frac{2}n}\et,
\end{align*}
where we have used~\eqref{E:BARLAMBDAPOWERMINUSK},~\eqref{E:PARTIALBARLAMBDAJBOUNDK}, the additive property of $p_{\mu,\nu}$ and the exponent of $\tau$ is simplified from \eqref{E:PELLKDEF}:%
\[
-2k+\sum_{i=1}^{j}(2+i\delta)\lambda_{i}-(\alpha_{0}+...\alpha_{j}%
)=-2k+2j+2j\delta-m\leq j\delta-m.
\]

\noindent
{\bf Proof of~\eqref{E:LAMBDABARHBOUND}.} Recall \eqref{E:HJTILDE}. 
By the Leibniz rule
\begin{align*}
&\left\vert \pa_\tau^m \rr^{\ell}  \tilde O_{j+1}\right\vert   \\
&\lesssim\sum_{\alpha_0+\alpha_1+\dots +\alpha_{j+1} = m \atop
\beta_0+\beta_1+\dots+\beta_j= \ell}\sum_{k=2}^{j+1}\sum_{\pi(j+1,k)} 
\left\vert\pa_\tau^{\alpha_0} \rr^{\beta_0}\left(\phi_0^{-k}\right)\pa_\tau^{\alpha_1} \rr^{\beta_1}\left(\left(\phi_1\right)^{\lambda_1}\right)\dots 
\pa_\tau^{\alpha_{j+1}} \rr^{\beta_{j+1}}\left(\left(\phi_{j+1}\right)^{\lambda_{j+1}}\right)\right\vert 
 \\
& \lesssim \sum_{\alpha_0+\alpha_1+\dots +\alpha_{j+1} = m \atop
\beta_0+\beta_1+\dots+\beta_{j+1} = \ell} \sum_{k=2}^{j+1} \sum_{\pi(j+1,k)} \tau^{-\frac23 k +\sum_{i=1}^{j+1}(\frac23+i\delta)\l_i-\sum_{i=0}^{j+1}\alpha_i}
 \pet^{\sum_{i=1}^{j+1}\l_i} \\
& \lesssim \tau^{(j+1)\delta-m} p_{2\l,-\frac4n}\et, \ \ j\in\{1,\dots,I\},
\end{align*}
where we have used~\eqref{E:PARTIALTAULCHI0},~\eqref{E:PARTIALBARLAMBDACHIBOUNDKM}, additive property of $p_{\mu,\nu}$,  
the bound $p_{\lambda, -\frac2n}\le1$ and the bound $\sum_{i=1}^j\l_i = k\ge2$. 
Note that for any $k\ge2$ and $(\l_1,\dots,\l_{j+1})\in \pi(j,k)$, we have
$\l_{j+1}=0$. 

\noindent
{\bf Proof of~\eqref{E:JSMALLHBOUND}.}
Assume first $j\in\{1,\dots, I\}$. 
Note that 
\[
w^{-\alpha}\DL\left(w^{1+\alpha}\mathscr J[\phi_0]^{-\gamma}h_j\right) = (1+\alpha) \r w \mathscr J[\phi_0]^{-\gamma}h_j + w \DL\left(\mathscr J[\phi_0]^{-\gamma}h_j\right).
\]
Using the bound $|\rr^{a} w| \lesssim r^n$ for  $a\ge1$,  by the previous identity and~\eqref{E:PHIUSEFUL} 
\begin{align*}
&\left\vert\pa_\tau^m \rr^{\ell} \left(w^{-\alpha}\DL\left(w^{1+\alpha}\mathscr J[\phi_0]^{-\gamma}h_j\right)\right)\right\vert \\
& \lesssim \pa_\tau^m \rr^{\ell+1}\left(\mathscr J[\phi_0]^{-\gamma}h_j\right) \\
& \ \ \ \ + r^n \sum_{d=0}^\ell\left( \left\vert \pa_\tau^{m+1} \rr^{d}\left(\mathscr J[\phi_0]^{-\gamma}h_j\right)\right\vert 
+ \left\vert \pa_\tau^{m} \rr^{d}\left(\mathscr J[\phi_0]^{-\gamma}h_j\right)\right\vert \right) \\
&\lesssim\sum_{\alpha_1+\alpha_2 = m \atop
\beta_1+\beta_2 = \ell+1} \left\vert \pa_\tau^{\alpha_1} \rr^{\beta_1}\left(\J[\phi_0]^{-\gamma}\right)
\pa_\tau^{\alpha_2} \rr^{\beta_2}h_j\right\vert  \\
& \ \ \ \ +  r^n \sum_{d=0}^\ell \sum_{\alpha_1+\alpha_2 = m+1 \atop
\beta_1+\beta_2 = d} \left\vert\pa_\tau^{\alpha_1} \rr^{\beta_1}\left(\J[\phi_0]^{-\gamma}\right)
\pa_\tau^{\alpha_2} \rr^{\beta_2}h_j\right\vert \\
& \ \ \ \ +  r^n \sum_{d=0}^\ell \sum_{\alpha_1+\alpha_2 = m \atop
\beta_1+\beta_2 = d} \left\vert \pa_\tau^{\alpha_1} \rr^{\beta_1}\left(\J[\phi_0]^{-\gamma}\right)
\pa_\tau^{\alpha_2} \rr^{\beta_2}h_j\right\vert \\
& 
\lesssim  \sum_{\alpha_1+\alpha_2=m}\tau^{-2\gamma +j\delta-\alpha_1-\alpha_2}q_{-\gamma}\et p_{1,0}\et p_{1+\l,-\frac2n}\et \\
& \quad +
\sum_{\alpha_1+\alpha_2=m+1} r^n\tau^{-2\gamma +j\delta-\alpha_1-\alpha_2}q_{-\gamma}\et p_{1,0}\et p_{1+\l,-\frac2n}\et \\
& \lesssim \tau^{-2\gamma +j\delta-m} q_{-\gamma+1}\et p_{2+\l,-\frac2n}\et,
\end{align*}
where we have used~\eqref{E:LAMBDABARSMALLHBOUND},~\eqref{E:BARLAMBDAPOWERMINUSK}, and the additive property of $q_\nu, p_{\mu,\nu}$.
If on the other hand $j=0$, then the above proof and $h_0=1$ give
\begin{align*}
\left\vert\pa_\tau^m \rr^{\ell}\left(w^{-\alpha} \DL\left(w^{1+\alpha}\mathscr J[\phi_0]^{-\gamma}h_0\right)\right)\right\vert 
&=\left\vert\pa_\tau^m \rr^{\ell} \left(w^{-\alpha} \DL\left(w^{1+\alpha}\mathscr J[\phi_0]^{-\gamma}\right)\right)\right\vert  \\
& \lesssim  \tau^{-2\gamma-m} q_{-\gamma+1}\et p_{1,0}\et.
\end{align*}
\end{proof}

\begin{lemma}\label{L:FNPLUSONEBOUND}
Under the inductive assumptions~\eqref{E:ASS1}, for any $\ell,m \in \{0,1,\dots,K\}$
the following estimate holds:
\begin{align}
\left\vert \pa_\tau^m \rr^{\ell} f_{j} \right\vert & \lesssim \tau^{-\frac43 + j\delta-m} \pet, \ \ j\in\{1,\dots,I+1\}  \label{E:LAMBDABARFBOUND}
\end{align}
\end{lemma}

\begin{proof}
Using the Leibniz rule and the formula~\eqref{E:FNDEFPHI}, we have
\begin{align*}
&\left\vert\pa_\tau^m \rr^{\ell} f_{j} \right\vert \\
& = \left\vert\pa_\tau^m \rr^{\ell}\left(-\frac29\phi_0^{-2} \frac{\tilde O_{j}}{j!}+  \sum_{d+i=j-1,\atop d,j\ge0}\sum_{k=0}^i
\frac{\phi_k\phi_{i-k}}{w^\alpha  r^2} \DL\left(w^{1+\alpha}\mathscr J[\phi_0]^{-\gamma}\frac{h_d}{d!}\right),\right)\right\vert \\
& \lesssim\sum_{\alpha_1+\alpha_2 = m \atop
\beta_1+\beta_2= \ell}  \left\vert  \pa_\tau^{\alpha_1} \rr^{\beta_1}\left(\phi_0^{-2}\right) 
\pa_\tau^{\alpha_2} \rr^{\beta_2}\tilde O_{j}\right\vert  \\
& \ \ \ \ +   \sum_{d+i=j-1,\atop d,i\ge0}\sum_{k=0}^i \sum_{\alpha_1+\dots +\alpha_4 = m \atop \beta_1+\dots+\beta_4= \ell}  \\
& \ \ \ \ \left\vert\pa_\tau^{\alpha_1} \rr^{\beta_1}\phi_k \pa_\tau^{\alpha_2} \rr^{\beta_2}\phi_{i-k} 
\pa_\tau^{\alpha_3} \rr^{\beta_3}( r^{-2})\pa_\tau^{\alpha_4} \rr^{\beta_4}\left(w^{-\alpha}\DL\left(w^{1+\alpha}\mathscr J[\phi_0]^{-\gamma}\frac{h_d}{d!}\right)\right)\right\vert.
\end{align*}
The worst case is $j-1=0$ where we have $d=i=k=0$, as in \eqref{E:JSMALLHBOUND}, since $p_{2+\lambda,-2/n}\le p_{1,0}$ by our choices of $\lambda$ in Definition \ref{D:ALAMBDADEF}. We now choose $p_{1,0}$ in \eqref{E:JSMALLHBOUND} to obtain: 
\begin{align*}
& \lesssim \sum_{\alpha_1+\alpha_2 = m} \tau^{-\frac43 +j\delta -(\alpha_1+\alpha_2)}p_{2\l,-\frac4n}\et\\
&\ \ \ \  +p_{1,0}\et \sum_{d+i=j-1,\atop d,i\ge0}\sum_{k=0}^i \sum_{\alpha_1+\dots +\alpha_4= m \atop \beta_1+\dots+\beta_4= \ell}
 \tau^{\frac23+k\delta-\alpha_1} \tau^{\frac23+(i-k)\delta-\alpha_2} \\
& \quad \ \ \  r^{-2}\tau^{-2\gamma + d\delta-\alpha_4 }  q_{-\gamma+1}\et \\
& \lesssim  \tau^{-\frac43 + j\delta -m} p_{2\l,-\frac4n}\et  + p_{1,0}\et\sum_{d+i=j-1,\atop d,i\ge0}
\tau^{\frac43 -2\gamma+ (d+i)\delta-m}  r^{-2}q_{-\gamma+1}\et  \\
& \lesssim  \tau^{-\frac43 + j\delta-m}  p_{2\l,-\frac4n}\et+ \tau^{\frac43 -2\gamma+ (j-1)\delta-m }  r^{-2} q_{-\gamma+1}\et p_{1,0}\et \\
& =  \tau^{-\frac43 + j\delta-m} \left(p_{2\l,-\frac4n}\et+ q_{-\gamma}\et \et^{1-\frac2n}\right) \\
& \lesssim  \tau^{-\frac43 + j\delta-m}p_{\l,-\frac2n}\et
\end{align*}
where we have used \eqref{E:ASS1}, \eqref{E:LAMBDABARHBOUND}, \eqref{E:JSMALLHBOUND}, and from the definition of $\delta$ \eqref{E:DELTADEF}, the exponent
\begin{align*}
\frac{4}{3}-2\gamma+(j-1)\delta-m  & =-\frac{4}{3}+j\delta-m+\frac{8}%
{3}-2\gamma-\delta\\
& =-\frac{4}{3}+j\delta-m+\frac{2}{n}
\end{align*}
and the estimates $p_{2\l,-\frac4n}\le p_{\l,-\frac2n}$ and $q_{-\gamma}(x) x^{1-\frac2n} = \frac{x^{1-\frac2n}}{(1+x)^\gamma} \le \frac{x^{\l-\frac2n}}{{(1+x)^\l}} = p_{\l,-\frac2n}(x)$.
(We remind the reader of definitions~\eqref{E:PFUNCTIONDEF} and~\eqref{E:QDEF} of $x\mapsto p_{\mu,\nu}(x)$ and $x\mapsto q_\nu(x)$ respectively.)
\end{proof}

\begin{remark}\label{R:BASIS}
When $j=1$ we have $\tilde O_1=0$, and thus from~\eqref{E:FNDEFPHI}
\[
f_1 = - \frac{\phi_0^2}{w^\alpha  r^2} \DL\left(w^{1+\alpha}\mathscr J[\phi_0]^{-\gamma}\right)= - P[\phi_0].
\]
In particular $f_1$ depends only on $\phi_0$ and the inductive assumption~\eqref{E:ASS1} is not used in the proof of~\eqref{E:LAMBDABARFBOUND}.
\end{remark}

\begin{lemma}\label{L:lem2.12}
\begin{enumerate}
\item
Let $0<\lambda<1$ be given and let $\beta$ satisfy 
\[
\beta -\l + \frac2n >-1.
\]
Then the following bound holds
\begin{align} \label{E:BETABOUND}
\left\vert \int_0^{\tau} (\tau')^\beta p_{\l,-\frac2n}(\frac{ r^n}{\tau'}) \,d\tau' \right\vert \lesssim \tau^{\beta+1} p_{\l,-\frac2n}\et,
\end{align}
where $x\mapsto p_{\mu,\nu}(x)$ is defined in~\eqref{E:PFUNCTIONDEF}.
\item
Let $b$ satisfy 
\[
b + \frac2n <-1.
\]
Then the following bound holds
\begin{align} \label{E:BBOUND}
\left\vert \int_\tau^{1} (\tau')^b p_{\l,-\frac2n}\left(\frac{ r^n}{\tau'}\right) \,d\tau' \right\vert \lesssim \tau^{b+1} p_{\l,-\frac2n}\et.
\end{align}
\end{enumerate}
\end{lemma}

\begin{proof}
{\bf Proof of part (i).}
Applying the change of variables $x=\tau'/ r^n$ we have
\begin{align}\label{E:INTIDENTITY}
\int_0^{\tau} (\tau')^\beta p_{\l,-\frac2n}\left(\frac{ r^n}{\tau'}\right) \,d\tau' =  r^{n(\beta+1)} \int_0^{\frac{\tau}{ r^n}} \frac{x^{\beta+\frac2n}}{(1+x)^\l}\,dx.
\end{align}
Case 1:  $ r^n\ge\tau$. We have
\begin{align*}
 r^{n(\beta+1)} \int_0^{\frac{\tau}{ r^n}} \frac{x^{\beta+\frac2n}}{(1+x)^\l}\,dx 
& \le  r^{n(\beta+1)} \int_0^{\frac{\tau}{ r^n}}x^{\beta+\frac2n}\,dx \\
& \lesssim \tau^{\beta+1+\frac2n}  r^{-2} = \tau^{\beta+1} \left(\frac{ r^n}{\tau}\right)^{-\frac2n} \lesssim  \tau^{\beta+1}p_{\l,-\frac2n}\et, 
\end{align*}
where the very last inequality follows from $1\lesssim p_{\l,0}\et$ which in turn relies on $ r^n\ge\tau$. 

\noindent
Case 2:  $ r^n\le\tau$. We have  from $\beta-\lambda+\frac{2}{n}>-1$ 
\begin{align*}
 r^{n(\beta+1)} \int_0^{\frac{\tau}{ r^n}} \frac{x^{\beta+\frac2n}}{(1+x)^\l}\,dx 
& \le  r^{n(\beta+1)} \int_0^{\frac{\tau}{ r^n}}x^{\beta-\l+\frac2n}\,dx \\
&  \lesssim \tau^{\beta+1} \et^{\l-\frac2n} \lesssim  
\tau^{\beta+1} p_{\l,-\frac2n}\et, 
\end{align*}
where the very last inequality follows from $1\lesssim \frac1{(1+\frac{ r^n}{\tau})^\l}$ which in turn relies on $ r^n\le\tau$. 

\noindent
{\bf Proof of part (ii).}
By the same change of variables as in~\eqref{E:INTIDENTITY} we have
\begin{align}\label{E:INTIDENTITY2}
\int_\tau^{1} (\tau')^b p_{\l,-\frac2n}\left(\frac{ r^n}{\tau'}\right) \,d\tau' =  r^{n(b+1)} \int_{\frac{\tau}{ r^n}}^{\frac1{ r^n}} 
\frac{x^{b+\frac2n}}{(1+x)^\l}\,dx .
\end{align}
We distinguish two cases again. 

\noindent
Case 1:  $ r^n\ge\tau$. We have from $b+\frac{2}{n}<-1,$
\begin{align*}
 r^{n(b+1)} \int_{\frac{\tau}{ r^n}}^{\frac{1}{ r^n}} 
\frac{x^{b+\frac2n}}{(1+x)^\l}\,dx &\le  r^{n(b+1)}  \int_{\frac{\tau}{ r^n}}^{\infty} x^{b+\frac2n}\,dx \\
& \lesssim  r^{n(b+1)} \et^{-b-1-\frac2n}  \\
& = \tau^{b+1} \et^{-\frac2n}\lesssim \tau^{b+1} p_{\l,-\frac2n}\et,
\end{align*}
where the last inequality follows from $1\lesssim p_{\l,0}\et$ which in turn relies on $ r^n\ge\tau$, just like in Case 1 in part (i). 

\noindent
Case 2:  $ r^n\le\tau$. 
We have 
\begin{align*}
 r^{n(b+1)} \int_{\frac{\tau}{ r^n}}^{\infty} 
\frac{x^{b+\frac2n}}{(1+x)^\l}\,dx \le 
 r^{n(b+1)}  \int_{\frac{\tau}{ r^n}}^{\infty}  x^{b-\l+\frac2n}\,dx
\lesssim \tau^{b+1} \et^{\l-\frac2n}\le \tau^{b+1} p_{\l,-\frac2n}\et,
\end{align*}
where the very last inequality follows from $1\le \frac1{(1+\frac{ r^n}{\tau})^\l}$ which in turn relies on $ r^n\le\tau$. 
The two previous estimates together with~\eqref{E:INTIDENTITY2} give~\eqref{E:BBOUND}.
\end{proof}

\subsubsection{Proofs of Proposition~\ref{P:MAINBOUNDPHI}, Lemma~\ref{L:BASIS}, and Theorem~\ref{T:MAINBOUNDPHI}} \label{SSS:PROOFS}

{\em Proof of Proposition~\ref{P:MAINBOUNDPHI}.}
We first assume that $m=0$. 
Let $k\in\mathbb N$ and $|f|\lesssim \tau^{4/3}$, $|g|\lesssim \tau^{-8/3}$, $|h(\tau, r)|\lesssim \tau^{k\delta}\pet$. 
If $k\le 
\left\lfloor \frac1{3\bar\gamma}\right\rfloor 
$ we then have
\begin{align}
\left\vert S_1[f,g,h]( \tau, r) \right\vert
&\lesssim  \tau^{\frac43} \int_\tau^1 (\tau')^{-\frac83} \int_0^{\tau'} \left(\tau''\right)^{k\delta}p_{\l,-\frac2n}\left(\frac{ r^n}{\tau''}\right)\,d\tau''d\tau' \notag \\
& \lesssim    \tau^{\frac43} \int_\tau^1 (\tau')^{-\frac53+k\delta} p_{\l,-\frac2n}\left(\frac{ r^n}{\tau'}\right) \,d\tau'  \lesssim \tau^{\frac23+k\delta}\pet,\label{E:S1BOUND}
\end{align}
since $k\delta-\lambda+\frac{2}{n}>-1$, where we have first used~\eqref{E:BETABOUND} and then~\eqref{E:BBOUND}. 
Note that we have used the assumption $k\le 
\left\lfloor \frac1{3\bar\gamma}\right\rfloor 
$ and Lemma~\ref{L:NUMERICAL} to ensure that
$
-\frac53+k\delta + \frac2n <-1 
$
and therefore~\eqref{E:BBOUND} is applicable in the last line of~\eqref{E:S1BOUND}.
If $k>
\left\lfloor \frac1{3\bar\gamma}\right\rfloor 
$ we then have
\begin{align}
\left\vert S_2[f,g,h](t, r ) \right\vert &\lesssim  \tau^{\frac43} \int_0^\tau (\tau')^{-\frac83} \int_0^{\tau'} \left(\tau''\right)^{k\delta} p_{\l,-\frac2n}\left(\frac{ r^n}{\tau''}\right)\,d\tau''d\tau'  \notag \\
& \lesssim    \tau^{\frac43}\int_{0}^{\tau} (\tau')^{-\frac53+k\delta}p_{\l,-\frac2n}\left(\frac{ r^n}{\tau'}\right)\,d\tau' 
 \lesssim \tau^{\frac23+k\delta}p_{\l,-\frac2n}\et,\label{E:S2BOUND}
\end{align}
where we have used~\eqref{E:BETABOUND} twice. We note that for any $k>
 \left\lfloor \frac1{3\bar\gamma}\right\rfloor 
 $ we have by Lemma~\ref{L:NUMERICAL}
$-\frac53+k\delta -\l +\frac{2}{n}>-1$ where we set $a=2N-2$ and we recall Definition~\ref{D:ALAMBDADEF} of $\lambda$. Therefore~\eqref{E:BETABOUND} is applicable in the second line.

By \eqref{E:HPHI} and \eqref{E:CHINFORMULA}, and the facts $\phi_{0}%
^{2}=\tau^{4/3},\phi_{0}^{-4}=\tau^{-8/3},$
\begin{align*}
\left\vert \rr^\ell \phi_{I+1}\right\vert \lesssim \sum_{\ell_1+\ell_2+\ell_3+\ell_4=\ell}
\left\vert S_i[\rr^{\ell_1} \left(\phi_0^2\right), \rr^{\ell_2}(\phi_0^{-4}),\rr^{\ell_3}(\phi_0^2)\rr^{\ell_4}f_{I+1}] \right\vert,
\end{align*}
where $i=1$ or $i=2$ according to~\eqref{E:CHINFORMULA}.
Since $\left\vert \rr^{\ell_1} \left(\phi_0^2\right)\right\vert\lesssim \tau^{\frac43}$, $\left\vert \rr^{\ell_2} \left(\phi_0^{-4}\right)\right\vert\lesssim \tau^{-\frac83}$
by~\eqref{E:BARLAMBDAPOWERMINUSK} and $\left\vert \rr^{\ell_4}f_{I+1}\right\vert\lesssim\tau^{-\frac43+(I+1)\delta}\pet$ by~\eqref{E:LAMBDABARFBOUND}, we may apply~\eqref{E:S1BOUND}--\eqref{E:S2BOUND} to conclude 
\begin{align}
\left\vert \rr^\ell \phi_{I+1}\right\vert & \lesssim  \tau^{\frac23+(I+1)\delta} \pet.
\end{align}
When $m=1$ we observe by taking $\tau$ derivative of \eqref{E:S1} and \eqref{E:S2}
and by Lemma ~\ref{L:FNPLUSONEBOUND}, \ref{L:lem2.12} 
\begin{align*}
\left\vert \pa_\tau \phi_{I+1}\right\vert 
& \le \frac43 \left\vert \tau^{-1}\phi_{I+1}\right\vert + \tau^{-\frac43} \left\vert \int_0^{\tau'}(\tau'')^{(I+1)\delta}p_{\l,-\frac2n}\left(\frac{ r^n}{\tau''}\right) \,d\tau''\right\vert \\
& \lesssim \tau^{\frac23+(I+1)\delta -1} \pet. 
\end{align*}
Similarly, using the Leibniz rule like above, 
\begin{align*}
\left\vert \pa_\tau \rr^\ell \phi_{I+1}\right\vert 
& \lesssim \tau^{\frac23+(I+1)\delta -1} \pet.
\end{align*}

For $m\ge 2$ we simply use the equation
\be\label{E:PHINPLUSONE}
\pa_{\tau\tau} \phi_{I+1} - \frac49 \phi_{I+1}\tau^{-2} = f_{I+1}
\ee
Applying $\pa_\tau^{m-2}\rr^\ell$ to~\eqref{E:PHINPLUSONE} we obtain
\begin{align*}
\left\vert \pa_\tau^m\rr^{\ell} \phi_{I+1} \right\vert &\lesssim \sum_{m'=0}^{m-2} \left\vert \pa_\tau^{m'}(\tau^{-2})\right\vert  \left\vert \pa_\tau^{m-2-m'}\phi_{I+1} \right\vert
+ \left\vert \pa_\tau^{m-2}\rr^{\ell} f_{I+1}\right\vert  \\
& \lesssim  \sum_{m'=0}^{m-2} \tau^{-2-m'}\tau^{\frac23+(I+1)\delta -m+2+m'}\pet \\
& \ \ \ \  +\tau^{-\frac43+(I+1)\delta-(m-2)}\pet\\
& \lesssim \tau^{\frac23+(I+1)\delta -m}   \pet,
\end{align*}
where
we have used the inductive assumption (that has been verified for all $m'<m$) and Lemma~\ref{L:FNPLUSONEBOUND}.
This completes the proof of Proposition~\ref{P:MAINBOUNDPHI}.

{\em Proof of Lemma~\ref{L:BASIS}.}
It remains to show the basis of induction, i.e. Lemma~\ref{L:BASIS}. By Lemma~\ref{L:FNPLUSONEBOUND} for any $\ell,m \in \{0,1,\dots,K\}$ we have
the bound
\begin{align}\label{E:F1BOUND}
\left\vert \pa_\tau^m \rr^{\ell} f_{1} \right\vert & \lesssim \tau^{-\frac43 + \delta-m} \pet.
\end{align}
By Remark~\ref{R:BASIS}, bound~\eqref{E:F1BOUND} does not rely on the inductive assumptions~\eqref{E:ASS1}. Using an argument identical to the proof of Proposition~\ref{P:MAINBOUNDPHI}, we conclude Lemma~\ref{L:BASIS}. 

{\em Proof of Theorem~\ref{T:MAINBOUNDPHI}.}
The proof follows by induction on the index $j\in\{1,\dots,M\}$. The claim is shown for $j=1$ in  Lemma~\ref{L:BASIS}, while the inductive step follows from Proposition~\ref{P:MAINBOUNDPHI}.


\section{Remainder equations and the main results}\label{S:MAINRESULT}


We look for a solution of~\eqref{E:BASICPDEPHI} in the form
\be\label{E:THETAQDEF}
\phi( \tau, r) = \sum_{k=0}^M \ve^k\phi_k( \tau, r) + \theta( \tau, r) =: \phia+\theta
\ee
where $M$ is to be specified later. 


\subsection{Derivation of the remainder equations}


\begin{lemma}[PDE satisfied by $\theta$]\label{L:THETAEQUATION}
Let $\phi$, $\phia$, and $\theta$ be related by~\eqref{E:THETAQDEF}. Then the equation satisfied by $\theta$ reads
\be\label{E:theta1}
\begin{split}
&\left(1-\ve \gamma w c \frac{M_g ^2}{ r^2}\right)\pa_\tau^2\theta   - 2\ve \gamma w c \frac{M_g }{ r^2} \pa_\tau \pa_r  ( r  \theta) - \ve\gamma  c   \frac{1}{ r w^{\alpha}}\pa_r \left(w^{1+\alpha} \frac{1}{ r^2}\pa_r[ r ^3 \theta] \right) + \ve \mathfrak K_3[\theta]   \\
&-\frac{4\theta}{9\phia^3} + 2\ve \frac{P[\phia]}{\phia} \theta + \ve \mathfrak K_1[\theta] 
+\frac{2}{9}\left(  \frac{1}{\phi^2}-\frac{1}{\phia^2} +  \frac{2\theta}{\phia^3} \right)+\ve \frac{P[\phia]\theta^2}{\phia^2}+\ve \mathfrak K_2[\theta] =S(\phia)
\end{split}
\ee
where the source term $S(\phia)$ and the expressions $\mathfrak K_j[\theta]$, $j=1,2,3$ are given by \eqref{E:SOURCETERM}, \eqref{E:MATHFRAKK1}, \eqref{E:MATHFRAKK2}, \eqref{E:MATHFRAKK3} below respectively and $c$ is given by~\eqref{E:CDEF}.
\end{lemma}


\begin{proof}
We recall the formulas~\eqref{E:BETADEF},~\eqref{E:DLAMBDADEF},~\eqref{E:PRESSURETERM} of $M_g$, the operator $\DL$, and the nonlinear pressure term $P[\phi]$ respectively. Finally, recall the fundamental formula
\[
\J[\phi]=\phi^2 (\phi + \DL\phi)
\]

Let 
\be\label{E:KM}
K_{m}[\theta]:=  \J[\phi]^{m} - \J[\phia]^{m}.
\ee 
Then
\[
\begin{split}
K_{1}[\theta]&= (2\phia\theta +\theta^2) (\phia + \DL\phia) + \phi^2(\theta + \DL\theta)\\
&=\phi^2 \DL\theta + (3\phia^2 + 2\phia \DL\phia )\theta +(3\phia + \DL\phia  )\theta^2 +\theta^3
\end{split}
\]

We want to find alternative expression for 
\begin{align}
P[\phi] - P[\phia]&=  \frac{\phi^2}{ g^2( r )w^\alpha  r^2}\DL
 \left(w^{1+\alpha} \left( \J[\phi]^{-\gamma} - \J[\phia]^{-\gamma}\right)\right) \notag \\
&\quad +  \frac{\phi^2 - \phia^2 }{ g^2( r )w^\alpha  r^2}\DL
 \left(w^{1+\alpha}\J[\phia]^{-\gamma}\right) \label{E:PMINUSQ}
\end{align}

Note that 
\begin{equation}
\begin{split}\label{exp0}
\frac{1}{w^\alpha} \DL\left(w^{1+\alpha} \left( \J[\phi]^{-\gamma} - \J[\phia]^{-\gamma}\right)\right)& =
\frac{1}{w^\alpha} \DL\left(w^{1+\alpha} K_{-\gamma}[\theta]\right) \\
&= w M_g  \pa_\tau K_{-\gamma}[\theta]  + w r\pa_r K_{-\gamma}[\theta]  +(1+\alpha)  r w'    K_{-\gamma}[\theta]
\end{split}
\end{equation}

Since 
\[
\begin{split}
\pa_\tau K_{-\gamma}[\theta]&=-\gamma \J[\phi]^{-\gamma-1} \pa_\tau\left(  \J[\phi] -\J[\phia]\right)
-\gamma (\J[\phi]^{-\gamma-1} -\J[\phia]^{-\gamma-1} ) \pa_\tau \J[\phia] \\
&=-\gamma \J[\phi]^{-\gamma-1} \pa_\tau K_1[\theta] - \gamma K_{-\gamma-1}[\theta] \pa_\tau \J[\phia] 
\end{split}
\]
and 
\begin{align}
\pa_\tau K_1[\theta] &= \phi^2 (M_g  \pa_\tau^2 \theta +\pa_\tau \pa_r  ( r  \theta) )  \\
&\quad+ \left[\phi^2\pa_\tau M_g  + 2\phi\pa_\tau\phi M_g  +2\phi^2+ 2\phi \DL\phia 
\right] \pa_\tau \theta \\
&\quad+ 2\phi\pa_\tau\phi \r\theta+\left[ \pa_\tau (3\phia^2 + 2\phia \DL\phia ) +\pa_\tau(3\phia + \DL\phia  )\theta\right]\theta \\
&=:  \phi^2 (M_g  \pa_\tau^2 \theta + \pa_\tau \pa_r  ( r  \theta) ) + \mathcal K_1
\end{align}
where we have used the identity $3\phia^2  + 6\phia \theta+3\theta^2 = 3\phi^2$. We may rewrite 
\begin{align}
wM_g  \pa_\tau K_{-\gamma}[\theta] = & - \gamma w \J[\phi]^{-\gamma-1} \phi^2 
(M_g ^2 \pa_\tau^2 \theta +M_g  \pa_\tau \pa_r  ( r  \theta) ) - \gamma w M_g  \J[\phi]^{-\gamma-1} \mathcal K_1 \notag \\
& - \gamma wM_g   K_{-\gamma-1}[\theta] \pa_\tau \J[\phia]  \label{exp1}
\end{align}

Similarly, 
\[
\begin{split}
\pa_r K_{-\gamma}[\theta]
&=-\gamma \J[\phi]^{-\gamma-1} \pa_r K_1[\theta] - \gamma K_{-\gamma-1}[\theta] \pa_r \J[\phia] 
\end{split}
\]
and 
\begin{align*}
\pa_r K_1[\theta] &= \phi^2 (M_g  \pa_r \pa_\tau \theta +  r  \pa_ r^2 \theta ) \\
&\quad + \left[\phi^2 + 2\phi\pa_r\phi  r +3\phi^2 + 2\phia \DL\phia  + 2\DL\phia \theta 
\right] \pa_r \theta \\
&\quad+(\phi^2 \pa_rM_g  + 2\phi\pa_r\phi M_g )\pa_\tau\theta+\left[ \pa_r (3\phia^2 + 2\phia \DL\phia ) +\pa_r(3\phia + \DL\phia  )\theta\right]\theta
\end{align*}
We may write 
\begin{equation}\label{exp2}
\begin{split}
wr\pa_r K_{-\gamma}[\theta] &= - \gamma w  \J[\phi]^{-\gamma-1} \phi^2  M_g  \pa_\tau \pa_r  ( r  \theta) - \gamma w r  \J[\phi]^{-\gamma-1} \phi^2 \left(  r  \pa_ r^2\theta + 4\pa_r\theta  \right) \\
&\quad  - \gamma w  r  \J[\phi]^{-\gamma-1} \mathcal K_2 - \gamma w r K_{-\gamma-1}[\theta] \pa_r \J[\phia] 
\end{split}
\end{equation}
where 
\begin{equation}
\begin{split}
\mathcal K_2 &:= \left[ 2\phi\pa_r\phi  r  + 2\phia \DL\phia   + 2\DL\phia \theta \right] \pa_r \theta \\
&\quad+(\phi^2 \pa_rM_g  + 2\phi\pa_r\phi M_g  - r^{-1}\phi^2 M_g  )\pa_\tau\theta+\left[ \pa_r (3\phia^2 + 2\phia \DL\phia ) +\pa_r(3\phia + \DL\phia  )\theta\right]\theta
\end{split}
\end{equation}

Plugging \eqref{exp1} and \eqref{exp2} into \eqref{exp0}, 
we deduce that 
\begin{equation}
\begin{split}\label{exp3}
&\frac{1}{w^\alpha} \DL\left(w^{1+\alpha} \left( \J[\phi]^{-\gamma} - \J[\phia]^{-\gamma}\right)\right)\\
& = - \gamma w \J[\phi]^{-\gamma-1} \phi^2 
(M_g ^2 \pa_\tau^2 \theta +2M_g  \pa_\tau \pa_r  ( r  \theta) )  - \gamma  r  \J[\phi]^{-\gamma-1} \phi^2w^{-\alpha} \pa_r \left(w^{1+\alpha} \frac{1}{ r^2}\pa_r[ r ^3 \theta] \right)  \\
& - \gamma w M_g  \J[\phi]^{-\gamma-1} \mathcal K_1
 - \gamma w  r  \J[\phi]^{-\gamma-1} \mathcal K_2 - \gamma w  K_{-\gamma-1}[\theta] \DL \J[\phia] \\
 & +(1+\alpha) r w' \left(  K_{-\gamma}[\theta] + \gamma \J[\phi]^{-\gamma-1} \phi^2[\r\theta + 3\theta]\right)
\end{split}
\end{equation}
Note that 
\[
\begin{split}
&- \gamma w M_g  \J[\phi]^{-\gamma-1} \mathcal K_1 - \gamma w  r  \J[\phi]^{-\gamma-1} \mathcal K_2 \\
&= - \gamma w \J[\phi]^{-\gamma-1} \big[  (\DL(\phi^2M_g ) + \phi^2M_g  + 2\phi \DL\phia M_g ) \pa_\tau \theta +(\DL(\phi^2) + 2\phi \DL\phia)  \r\theta + \DL(3\phia^2 + 2\phia \DL\phia ) \theta  \big] \\
&\quad  - \gamma w \J[\phi]^{-\gamma-1}   \DL(3\phia + \DL\phia  )\theta^2  
\end{split}
\]

By writing $$K_{-\gamma}[\theta] = - \gamma \J[\phia]^{-\gamma-1} K_1[\theta] + \left( K_{-\gamma}[\theta] + \gamma \J[\phia]^{-\gamma-1} K_1[\theta]\right) $$ and $$  \gamma \J[\phi]^{-\gamma-1} \phi^2[\r\theta + 3\theta]= \gamma
\J[\phia]^{-\gamma-1}\phi^{2}[r\partial_{r}\theta+3\theta]+\gamma K_{-\gamma-1}%
\phi^{2}[r\partial_{r}\theta+3\theta],$$
 the last line of \eqref{exp3} can be rewritten as 
\[
\begin{split}
& K_{-\gamma}[\theta] + \gamma \J[\phi]^{-\gamma-1} \phi^2[\r\theta + 3\theta]\\
 &= -\gamma \J[\phia]^{-\gamma-1} \big[ \phi^2M_g \pa_\tau\theta + 2\phia \DL\phia  \theta \big] -\gamma \J[\phia]^{-\gamma-1}\big[ (\DL\phia -3\phia  )\theta^2 -2\theta^3 \big]\\
 &\quad+K_{-\gamma}[\theta] + \gamma \J[\phia]^{-\gamma-1} K_1[\theta] + \gamma K_{-\gamma-1}[\theta]\phi^2[\r\theta + 3\theta]
 \end{split}
 \] 
 Observe that 
\[
K_{-\gamma}[\theta]  + \gamma \J[\phia]^{-\gamma-1} K_1[\theta]= \gamma(\gamma+1)\J[\phia]^{-\gamma-2} \left(\int_0^1(1-s)(1+s \frac{K_1[\theta]}{\J[\phia]}  )^{-\gamma-2} ds\right)  (K_1[\theta])^2
\]
which asserts that the expression is a nonlinear term. 
Therefore by splitting $$K_{-\gamma-1}[\theta]=-(\gamma+1)\J[\phia]^{-\gamma-2}K_{1}[\theta]+\left(K_{-\gamma-1}[\theta]+(\gamma+1)\J[Q]^{-\gamma-2}%
K_{1}[\theta]\right),$$ we obtain 
 \be
 \begin{split}
& \frac{\phi^2}{ g^2(r) w^\alpha  r^2}\DL
 \left(w^{1+\alpha} \left( \J[\phi]^{-\gamma} - \J[\phia]^{-\gamma}\right)\right)\\
& = - \gamma w  \frac{\phi^4}{g^2\J[\phi]^{\gamma+1}  r^2}
(M_g ^2 \pa_\tau^2 \theta +2M_g \pa_\tau \pa_r  ( r \theta) )  - \gamma  \frac{\phi^4}{ g^2\J[\phi]^{\gamma+1}  r w^\alpha}  \pa_r \left(w^{1+\alpha} \frac{1}{ r^2}\pa_r[ r^3 \theta] \right) \\
&\quad + \mathfrak K_1[\theta]  + \mathfrak K_2[\theta] +  \mathfrak K_3[\theta] 
 \end{split}
 \ee
where 
\begin{align}
\mathfrak K_1[\theta]:=& -\gamma w\frac{\phi^2}{g^2\J[\phi]^{\gamma+1}  r^2} \big[ (\DL(\phi^2M_g ) + \phi^2M_g  + 2\phi \DL\phia M_g ) \pa_\tau \theta +4\phi \DL\phia \r\theta + \DL(3\phia^2 + 2\phia \DL\phia ) \theta  \big] \notag \\
&+ \gamma(\gamma+1) w\frac{\phi^2}{g^2\J[\phia]^{\gamma+2}  r^2} \big[ \phi^2 \DL\theta + (3\phia^2 + 2\phia \DL\phia )\theta \big] \DL\J[\phia] \notag \\
&-\gamma (1+\alpha) r w' \frac{\phi^2}{g^2\J[\phia]^{\gamma+1}  r^2} \big[ \phi^2M_g \pa_\tau\theta + 2\phia \DL\phia  \theta \big]  \label{E:MATHFRAKK1}
\end{align}
\begin{align}
\mathfrak K_2[\theta]:=&  -2\gamma w\frac{\phi^3}{g^2\J[\phi]^{\gamma+1}  r^2}  (\r\theta)^2-2\gamma w\frac{\phi^3}{g^2\J[\phi]^{\gamma+1}  r^2} M_g  \pa_\tau \theta (\r\theta) \notag  \\
& -\gamma w\frac{\phi^2}{g^2\J[\phi]^{\gamma+1}  r^2} \DL(3\phia + \DL\phia  )\theta^2 \notag \\
&+ \gamma(\gamma+1) w\frac{\phi^2}{g^2\J[\phia]^{\gamma+2}  r^2} \big[  (3\phia + \DL\phia  )\theta^2 +\theta^3 \big] \DL\J[\phia]\notag\\
&-\gamma (1+\alpha) r w' \frac{\phi^2}{g^2\J[\phia]^{\gamma+1}  r^2} \big[ (\DL\phia -3\phia  )\theta^2 -2\theta^3 \big] \notag \\
&-\gamma w\frac{\phi^2}{g^2 r^2} (K_{-\gamma-1}[\theta] + (\gamma+1) \J[\phia]^{-\gamma-2} K_1[\theta] )  \DL\J[\phia] \label{E:MATHFRAKK2}
\end{align} 
and
\be
\mathfrak K_3[\theta]:= (1+\alpha) r w' \frac{\phi^2}{g^2  r^2} \left( K_{-\gamma}[\theta] + \gamma \J[\phia]^{-\gamma-1} K_1[\theta] +\gamma  K_{-\gamma-1}[\theta]\phi^2[\r\theta + 3\theta]  \right)  \label{E:MATHFRAKK3}
\ee
Note that $\mathfrak K_1[\theta]$ contains both linear and nonlinear terms in terms of $\theta$ and we view them as linear terms with nonlinear coefficients.  $\mathfrak K_2[\theta]$ and   $\mathfrak K_3[\theta]$ consist of quadratic and higher terms. We have distinguished them because $\mathfrak K_3[\theta]$ needs to be estimated together with the main linear elliptic operator in higher order estimates due to the presence of nonlinear factor $c$. 

The $\phi$ equation \eqref{E:BASICPDEPHI} can be written as 
\be\label{E:theta1}
\begin{split}
&\left(1-\ve \gamma w c \frac{M_g ^2}{ r^2}\right)\pa_\tau^2\theta   - 2\ve \gamma w c \frac{M_g }{ r^2} \pa_\tau \pa_r  ( r \theta) - \ve\gamma  c   \frac{1}{ r w^{\alpha}}\pa_r \left(w^{1+\alpha} \frac{1}{ r^2}\pa_r[ r^3 \theta] \right) + \ve \mathfrak K_3[\theta]   \\
&-\frac{4\theta}{9\phia^3} + 2\ve \frac{P[\phia]}{\phia} \theta + \ve \mathfrak K_1[\theta] 
+\frac{2}{9}\left(  \frac{1}{\phi^2}-\frac{1}{\phia^2} +  \frac{2\theta}{\phia^3} \right)+\ve \frac{P[\phia]\theta^2}{\phia^2}+\ve \mathfrak K_2[\theta] =S(\phia)
\end{split}
\ee
where 
the source term $S(\phia)$ is given by~\eqref{E:SOURCETERM} and
\begin{align}
c:=c[\phi]= \frac{\phi^4}{g^2\J[\phi]^{\gamma+1}} \label{E:CDEF}
\end{align}

\end{proof}


\begin{lemma}[The H equation]
Let 
\be\label{E:HDEF}
H : = \tau^{-m} r \theta.
\ee 
Then $H$ solves 
\be\label{E:H1}
\begin{split}
&\left(1-\ve \gamma w c[\phi] \frac{M_g ^2}{ r^2}\right)\pa_\tau^2 H - 2\ve \gamma w c[\phi] \frac{M_g }{r} \pa_r\pa_\tau H+\frac{2m}{\tau}  \pa_\tau H +\left[ \frac{m(m-1)}{\tau^2}  -\frac{4}{9\phia^3} \right]H  \\
& \qquad- \ve\gamma  c[\phi]   \frac{1}{ w^{\alpha}}\pa_r \left(w^{1+\alpha} \frac{1}{ r^2}\pa_r[ r^2 H] \right) +
\ve \mathscr N_0[H]
+\ve \mathscr L_{\text{low}} H 
 =\mathscr S(\phia) + \mathscr N [H]
\end{split}
\ee
where 
\begin{align}
\mathscr N_0[H] &:= \frac{r}{\tau^m}  \mathfrak K_3[\frac{\tau^m H}{r}] \label{E:NZERODEF}\\
\mathscr L_{\text{low}} H&: =  - \gamma w c[\phi] \frac{M_g ^2}{ r^2} \left[ \frac{2m}{\tau}  \pa_\tau H  + \frac{m(m-1)}{\tau^2} H \right]
 - 2m \gamma w c[\phi] \frac{M_g }{ r\tau} \pa_r H + 2 \frac{P[\phia]}{\phia} H  + \frac{r}{\tau^m}  \mathfrak K_1[\frac{\tau^m H}{r}] \\
\mathscr S(\phia)&: = \frac{r}{\tau^m} S(\phia) \\
\mathscr N [H]&:=- \frac{r}{\tau^m}\mathfrak N[\frac{\tau^m H}{r}] , \quad \mathfrak N[\theta]:=\ve \mathfrak K_2[\theta]+ \frac{2}{9}\left(  \frac{1}{\phi^2}-\frac{1}{\phia^2} +  \frac{2\theta}{\phia^3} \right)+\ve \frac{P[\phia]\theta^2}{\phia^2}, \label{E:NONLINEARTERM}
\end{align}
where the source term $S(\phia)$ and the expressions $\mathfrak K_j[\theta]$, $j=1,2,3$ are given by \eqref{E:SOURCETERM},
 \eqref{E:MATHFRAKK1}, \eqref{E:MATHFRAKK2}, \eqref{E:MATHFRAKK3}. 
\end{lemma}


\begin{proof}
The proof follows by a direct verification after plugging in $\theta = \tau^m  r^{-1} H$ in~\eqref{E:theta1}. 
\end{proof}


We rewrite~\eqref{E:H1} in the form
\begin{align}
&g^{00}\pa_\tau^2 H + 2g^{01}\pa_r\pa_\tau H+\frac{2m}{\tau}  \pa_\tau H +d(\tau,r)^2\frac{H}{\tau^2} 
- \ve\gamma  c[\phi]   \frac{1}{ w^{\alpha}}\pa_r \left(w^{1+\alpha} \frac{1}{ r^2}\pa_r[ r^2 H] \right)  + \ve\mathscr N_0[H]
\notag \\
& \qquad  =\mathscr S(\phia) -\ve \mathscr L_{\text{low}} H + \mathscr N [H], \label{E:H2}
\end{align}
where
\begin{align}
g^{00}& :=1-\ve \gamma w c[\phi] \frac{M_g ^2}{ r^2}, \label{E:GZEROZERODEF}\\
g^{01}& : = -\ve \gamma w c[\phi] \frac{M_g }{r}, \label{E:GZEROONEDEF}\\
d^2( \tau, r) & : =  m(m-1)  -\frac{4\tau^2}{9\phia^3}.
\end{align}

The leading order operator
\be
\Box:= g^{00}\pa_\tau^2  + 2g^{01}\pa_r\pa_\tau  - \ve\gamma  c[\phi]   \frac{1}{ w^{\alpha}}\pa_r \left(w^{1+\alpha} \frac{1}{ r^2}\pa_r[ r^2 \cdot] \right)
\ee
will be shown to be hyperbolic due to the bound $1\lesssim g^{00} \lesssim1$ shown later in 
Lemma~\ref{L:IMPORTANT}. We shall see that the former estimate is crucially tied to the supercriticality ($\gamma<\frac43$) and the flatness assumption on the enthalpy $w$ near $r=0$ (i.e. $n$ sufficiently large in~\eqref{E:FM}, i.e. Lemma~\ref{L:ENTHALPYOK}). 
Moreover, $\Box$ is also manifestly quasilinear as $c[\phi]$ depends on the space-time derivatives of $H$. 
The twofold singular nature of $\Box$ coming from the gravitational singularity at $\tau=0$ and the vacuum singularity at $r=1$
is discussed at length in Section~\ref{SS:METHODS}.

The basic equation for our energy estimates is obtained by dividing~\eqref{E:H2} by $g^{00}$:
\begin{align}
&\pa_\tau^2 H + 2\frac{g^{01}}{g^{00}}\pa_r\pa_\tau H+\frac{2m}{g^{00}}  \frac{\pa_\tau H}{\tau} +\frac{d^2}{g^{00}}\frac{H}{\tau^2} 
- \ve\gamma  \frac{c[\phi]}{g^{00}}   \frac{1}{ w^{\alpha}}\pa_r \left(w^{1+\alpha} \frac{1}{ r^2}\pa_r[ r^2 H] \right)  + \ve\frac{\mathscr N_0[H]}{g^{00}}
\notag \\
& \qquad  =\frac1{g^{00}}\left(\mathscr S(\phia) -\ve \mathscr L_{\text{low}} H + \mathscr N [H]\right).
\label{E:H3}
\end{align}


We denote the first summation without $\sup$ in Definition~\ref{D:NORMDEF} of $S_\kappa^N$ by $E^N$ and the second summation without the time integral by $D^N$, i.e. for any $\tau\in(0,1]$ we let
\begin{align}
E^N(\tau) &:= \sum_{j=0}^N  \left\{\tau^{\gamma - \frac53} \|\mathcal D_j H_\tau\|_{\alpha+j}^2 
+\tau^{\gamma - \frac{11}3} \|\mathcal D_j H\|_{\alpha+j}^2
+ \ve \tau^{-\gamma-1}\|q_{-\frac{\gamma+1}{2}}\et  \mathcal D_{j+1} H\|_{\alpha+j+1}^2\right\} \label{E:ENDEF}\\
D^N(\tau) & :=  \sum_{j=0}^N \left\{ \tau^{\gamma-\frac83}\|\mathcal D_j H_\tau\|_{\alpha+j}^2  
+\tau^{\gamma - \frac{14}3} \| \mathcal D_j H\|_{\alpha+j}^2
+ \ve \tau^{-\gamma-2}\| q_{-\frac{\gamma+2}{2}}\et  \mathcal D_{j+1} H\|_{2\alpha+j+1}^2 \right\}.
\label{E:DNDEF}
\end{align}
Then the space-time norm can be written as 
$$S_\kappa^N(\tau)= \sup_{\kappa\le \tau'\le\tau} E^N(\tau') + \int_\kappa^\tau D^N(\tau') d\tau'.$$


\section{High-order energies and preparatory bounds} \label{S:ENERGIES}


\subsection{High-order equations and energies}


In order to derive high-order equations, we first introduce the elliptic operators
\begin{align}
L_k f &:= -\frac{1}{w^k} \pa_r \left[w^{1+k} D_r f \right], \label{E:LBETADEF}\\
L_k^* h & := - \frac{1}{w^k} D_r \left[w^{1+k} \pa_r h\right]. \label{E:LBETASTARDEF}
\end{align}
Then for any $f,h$ we have
\begin{align}
( f, L_k h)_{k} = (D_r f, D_r h)_{1+k} \ \  \text{ and } \ \  (f, L_k^*h)_k= (\pa_r f, \pa_rh)_{1+k} 
\end{align}
where we recall the inner product $(\cdot,\cdot)_k$ given in \eqref{E:WEIGHTEDINNERPRODUCT}. 

We recall here the definition of the fundamental high-order differential operators $\D_j$ given in~\eqref{E:FUNDOP}.
We then define 
\begin{equation}
\mathcal L_{j+\alpha} \mathcal D_j : = 
\begin{cases}
L_{j+\alpha} \mathcal D_j & \text{ if $j$ is even}\\
L^\ast_{j+\alpha}\mathcal D_j & \text{ if $j$ is odd}
\end{cases}. 
\end{equation}

Important role is played by the operator
$\bar{\mathcal D}_i$ defined as 
\be
\bar{\mathcal D}_i =
\begin{cases}
 \mathcal D_0 & \text{ for } \ \  i=0 \\
 \mathcal D_{i-1}  \pa_r & \text{ for } \  \ i \geq 1 
\end{cases}
\ee

Let $1\le i\le N$. After applying $\mathcal D_i$ to \eqref{E:H3}  we use  Lemmas~\ref{L:COMM1}--\ref{L:COMM2} to derive the equation for $\mathcal D_i H$:  
\begin{align}
&\pa_\tau^2 \mathcal D_i H + 2\frac{g^{01}}{g^{00}}\pa_r \mathcal D_i \pa_\tau  H+\frac{2m}{g^{00}}  \frac{ \mathcal D_i \partial_\tau H}{\tau} +\frac{d^2}{g^{00}}\frac{\mathcal D_i H}{\tau^2} 
+ \ve\gamma  \frac{c[\phi]}{g^{00}}  \mathcal L_{i+\alpha} \mathcal D_i H \notag \\
&  = \mathcal D_i \left( \frac1{g^{00}}\left(\mathscr S(\phia) -\ve \mathscr L_{\text{low}} H + \mathscr N [H]\right)\right) + \mathcal C_i[H] + \bar\D_{i-1} \mathscr M[H].\label{E:DiHP}
\end{align}
Here $\mathcal C_i$ contains all the commutators
\begin{align}
\mathcal C_i[H] : =&  - 2\left[\D_i,  \frac{g^{01}}{g^{00}} \pa_r\right]\pa_\tau H - 2m\left[\D_i, \frac{1}{g^{00}} \right] \frac{\pa_\tau H }{\tau} - \left[\D_i, \frac{d^2}{g^{00}}\right]
\frac{H}{\tau^2}  \notag \\
 &- \ve\gamma  \frac{c[\phi]}{g^{00}}  \sum_{j=0}^{i-1} \zeta_{ij} \mathcal D_{i-j} H -  \ve\gamma \left[\bar\D_{i-1},   \frac{c[\phi]}{g^{00}}  \right] D_r L_\alpha H,   
 \label{E:COMMTOTALDEF}
 \end{align}  
where the functions $\zeta_{ij}$ are given by~\eqref{pij} and the commutators $[\cdot,\cdot]$ are defined in~\eqref{E:COMMDEF}. 
Furthermore, 
\be\label{E:MH}
\mathscr M[H] : = -\ve\gamma \pa_r ( \frac{c[\phi]}{g^{00}}  ) L_\alpha H -  \ve D_r ( \frac{\mathscr N_0[H]}{g^{00}}).
\ee
Note that we have written for $i\geq 1$, 
\[
\begin{split}
&\D_i (  \ve\gamma  \frac{c[\phi]}{g^{00}} L_\alpha H + \ve  \frac{\mathscr N_0[H]}{g^{00}} ) \\
&=  \ve\gamma  \frac{c[\phi]}{g^{00}}  \mathcal L_{i+\alpha} \mathcal D_i H +  \ve\gamma  \frac{c[\phi]}{g^{00}}  \sum_{j=0}^{i-1} \zeta_{ij} \mathcal D_{i-j} H +  \ve\gamma \left[\bar\D_{i-1},   \frac{c[\phi]}{g^{00}}  \right] D_r L_\alpha H + \bar\D_{i-1} \mathscr M[H] 
\end{split}
\]

\begin{definition}[Weighted high-order energies] \label{D:ENERGYDEF}
For any $0< \kappa\le1$ and $N\in\mathbb N$ we define the high-order energies
\be\label{E:TOTALENERGY}
\mathscr E^N(\tau) = \sum_{j=0}^N  \mathscr E_j(\tau'), \ \ \ \ \mathscr D^N(\tau) = \sum_{j=0}^N \mathscr D_j(\tau),
\ee
where for any $0\leq  j \leq N$ we have 
\begin{equation}
\mathscr E_j(\tau)=  \frac12 \int_0^1 \left\{ \tau^{\gamma-\frac53} \left\vert \mathcal D_jH_\tau\right\vert^2  
 +    \frac{d^2}{g^{00}}\tau^{\gamma-\frac{11}3}\left\vert \mathcal D_jH\right\vert^2+ \ve\gamma \tau^{\gamma-\frac53} \frac{c[\phi]}{g^{00}} w \left\vert \mathcal D_{j+1}H\right\vert^2\right\} \, w^{\alpha+j}  r^{2}\,d r
\end{equation}
and 
\begin{equation}
\begin{split}
\mathscr D_j(\tau)= &  \int_0^1
\left( \left[ \frac{2m}{g^{00}} + \frac12(\frac53-\gamma) \right]\tau^{\gamma-\frac83} 
- \tau^{\gamma-\frac53}\frac{\pa_r\left(\frac{g^{01}}{g^{00}} w^{\alpha +j} r^2 \right)}{w^{\alpha+j} r^2} \right)\left\vert \mathcal D_jH_\tau\right\vert^2 \, w^{\alpha+j} r^{2}\,d r \notag \\
 &   -\frac12\ve\gamma \int_0^1 \left(\tau^{\gamma-\frac53} c[\phi_0]\right)_\tau \frac{c[\phi]}{c[\phi_0]g^{00}}  \left\vert \mathcal D_{j+1}H\right\vert^2 \,w^{1+\alpha+j} r^{2}\,d r \\
&   - \frac12 \int_0^1 \left(\frac{d^2}{g^{00}}\tau^{\gamma-\frac{11}3}\right)_\tau \left\vert \mathcal D_jH\right\vert ^2 \, w^{\alpha+j} r^{2}\,d r. 
\end{split}
\end{equation}
\end{definition}

\begin{remark}
It will be shown in Section~\ref{SS:APRIORI}, 
Lemma~\ref{L:IMPORTANT}, 
that every summand appearing in the definition of $\mathscr D_j$ above is positive in our bootstrap regime.
\end{remark}


\begin{proposition}\label{P:HIGHORDER}
Assume that $H$ is a sufficiently smooth solution to~\eqref{E:H2}.
The the following energy identity holds
\begin{align}\label{E:ENERGYIDENTITY}
\pa_\tau \mathscr E^N(\tau) + \mathscr D^N(\tau) = \sum_{i=0}^N\mathcal R_i,
\end{align}
where for any $i\in\{1,\dots, N\}$, the error terms $\mathcal R_i$ are explicitly given by 
\begin{align}
\mathcal R_i  = & \tau^{\gamma-\frac53}\left( \mathcal D_i \left(\frac{\mathscr S(\phia)}{g^{00}} - \frac\ve{g^{00}}\mathscr L_{\text{low}} H + \frac{\mathscr N[H]}{g^{00}}\right) \ , \ \mathcal D_i H_\tau\right)_{\alpha+i} \notag \\
& + \tau^{\gamma-\frac53}\left(\mathcal C_{i}[H] +\bar{\D}_{i-1}\mathscr M[H] \ , \ \mathcal D_i H_\tau\right)_{\alpha+i} \notag \\
&  \frac12 \ve\gamma \tau^{\gamma-\frac53}  
\int_0^1  c[\phi_0]\left(\frac{c[\phi]}{c[\phi_0]g^{00}}\right)_\tau w^{1+\alpha} \left\vert \mathcal D_{j+1}H\right\vert^2 \,w^j r^{2}\,d r, \label{E:Ri}
\end{align}
where $\mathcal C_i[H]$ is given by~\eqref{E:COMMTOTALDEF} and $\mathscr M[H]$ by~\eqref{E:MH}.
When $i=0$, we replace $\mathcal C_{i}[H] +\bar{\D}_{i-1}\mathscr M[H]$ in the above formula by $-\ve\frac{\mathscr N_0[H]}{g^{00}}$. 
\end{proposition}


\begin{proof}
We evaluate the $(\cdot,\cdot)_{\alpha+i}$-inner product of \eqref{E:DiHP} with $\tau^{\gamma-\frac53}\mathcal D_i H_\tau$, and use Definition~\ref{D:ENERGYDEF}. 
\end{proof}


\subsection{A priori bounds and the energy-norm equivalence}\label{SS:APRIORI}

Assume that $H$ is a solution to~\eqref{E:H2} on a time interval $[\kappa,T]$ for some $T\le1$. For a sufficiently small $\sigma'<1$, to be fixed later, 
we stipulate the following a priori bounds.
\begin{align}
\left\|(\r)^{\ell_1} (\tau\pa_\tau)^{\ell_2}\left(\frac{H}{r}\right)\right\|_{C^0([\kappa,T]\times[0,1])} & \le \sigma', \  \  0\le \ell_1+\ell_2 \le 2, \ \ 
\ell_1,\ell_2\in\mathbb Z_{\ge0}.
\label{E:APRIORI} 
\end{align}


\begin{lemma}
Assume that $H$ is a solution to~\eqref{E:H2} on a time interval $[\kappa,T]$ for some $T\le1$ and assume that the a priori assumptions~\eqref{E:APRIORI} hold. Then for any $( \tau, r)\in[\kappa,T]\times[0,1]$
\begin{align}
1\lesssim \left\vert \frac{\phi}{\phi_0} \right\vert  &\lesssim 1, \label{E:PHIBOUNDAPRIORI} \\
1\lesssim \left\vert \frac{\J[\phi]}{\J[\phi_0]} \right\vert  &\lesssim 1, \label{E:JPHIBOUNDAPRIORI} \\
\left\vert \pa_\tau \phi \right\vert &\lesssim 
\tau^{-\frac13}, 
\label{E:PARTIALTAUPHIBOUND} \\
\left\vert (\r) \pa_\tau \phi \right\vert &\lesssim  
\left(\ve+\sigma'\right)  \tau^{-\frac13+\delta} 
\label{E:PARTIALTAUPHIBOUND1} \\
\left\vert \phi_{\tau\tau} \right\vert &\lesssim 
\tau^{-\frac43},  
\label{E:PARTIALTAUTAUPHIBOUND}\\
\left\vert \rr^\ell \phi \right\vert &\lesssim \left(\ve+\sigma'\right)\tau^{\frac23+\delta}, \ \ \ell=1,2, \label{E:PARTIALETAPHIBOUND}  \\
\left\vert \DL\phi \right\vert &\lesssim 
\tau^{\frac23} q_1\et, 
\label{E:DPHIBOUND}\\
\left\vert \pa_\tau \DL\phi \right\vert &\lesssim  
 \tau^{-\frac13}q_1\et, 
\label{E:PARTIALTAUDPHIBOUND} \\
\left\vert \r \DL\phi \right\vert &\lesssim    
\tau^{\frac23}q_1\et, 
\label{E:PARTIALETADPHIBOUND}\\
 \left\vert (\frac{\phi}{\phi_0})_\tau \right\vert  &\lesssim \left(\ve+\sigma'\right) \tau^{\delta-1}, \label{E:PHIBOUNDAPRIORITAU} \\
 \left\vert (\frac{\J[\phi]}{\J[\phi_0]})_\tau \right\vert  &\lesssim \left(\ve+\sigma'\right) \tau^{\delta-1} \label{E:JPHIBOUNDAPRIORITAU} 
\end{align}
\end{lemma}


\begin{proof}
\noindent
{\bf Proof of~\eqref{E:PHIBOUNDAPRIORI}.}
Let $h: =\frac{\phi}{\phi_0}$. 
By~\eqref{E:THETAQDEF} and~\eqref{E:HDEF} we have 
\be\label{E:LITLLEHFORMULA}
h = \frac{\phi}{\phi_0} =1 + \sum_{j=1}^M\ve^j \frac{\phi_j}{\phi_0} +\tau^{m-\frac23}\frac{H} r.
\ee
By Proposition~\ref{P:MAINBOUNDPHI} and the a priori assumption~\eqref{E:APRIORI} for any $( \tau, r)\in[\kappa,T]\times[0,1]$ we have
\[
\left\vert 
h 
- 1 \right\vert \lesssim \sum_{j=1}^M\ve^j \tau^{j\delta} + \sigma' \tau^{m-\frac23} \le \frac1{10}
\]
for $\ve,\sigma'>0$ sufficiently small. 


\noindent
{\bf Proof of~\eqref{E:JPHIBOUNDAPRIORI}.}
Note that  
\be
\frac{\J[\phi]}{\J[\phi_0]} = \left\vert\frac\phi{\phi_0}\right\vert^2 \frac{\phi + \DL\phi}{\phi_0+\DL\phi_0}
= h^2 \frac{\phi + \DL\phi}{\phi_0+\DL\phi_0}. \label{E:Jacobianf}
\ee
 Therefore, in view of~\eqref{E:PHIBOUNDAPRIORI} it suffices to prove
\be\label{E:JPHIINTERIM}
\phi_0+\DL\phi_0 \lesssim \phi+\DL\phi\lesssim \phi_0+\DL\phi_0.
\ee
Recall that 
\[
\phi_0+\DL\phi_0 = \tau^{\frac23} + \frac23 M_g  \tau^{-\frac13} = \tau^{\frac23} \left(1 + \frac23\frac{(\tau-1)}\tau \r(\log g)\right).
\]
By~\eqref{E:GTAYLOR}--\eqref{E:GTAYLOR2}  we have
\be\label{E:PHIZERODPHIZEROFORMULA}
\tau^{\frac23} q_1(\frac{ r^n}{\tau}) \lesssim \left\vert \phi_0+\DL\phi_0\right\vert \lesssim \tau^{\frac23} q_1(\frac{ r^n}{\tau})
\ee
Moreover, 
\begin{align}\label{E:PHIDPHIFORMULA}
\phi+\DL\phi = h (\phi_0 + \DL\phi_0) + \phi_0  \DL h.
\end{align}
From~\eqref{E:LITLLEHFORMULA}, and Proposition~\ref{P:MAINBOUNDPHI} with the crude bound $\pet\lesssim1$ and the bound 
\[
\tau^{-1} = \et r^{-n} \lesssim r^{-n} q_1\et,
\]
we have 
\begin{align*}
\left\vert  \DL h \right\vert &
 \lesssim q_1\et \sum_{j=1}^M \ve^j \tau^{j\delta}  +  r^n \tau^{m-\frac23} \left\vert \frac{H_\tau}{r}\right\vert 
+ r^n \tau^{m-\frac53} \left\vert \frac{H}{r}\right\vert + \tau^{m-\frac23} \left\vert \r\left(\frac{H}{r}\right)\right\vert \notag \\
& \lesssim  \ve\tau^\delta q_1\et  + q_1\et \tau^{m-\frac23} \left(\left\vert \frac{\tau H_\tau}{r}\right\vert+\left\vert \frac{H}{r}\right\vert  \right)
+\tau^{m-\frac23} \left\vert \r\left(\frac{H}{r}\right)\right\vert \notag \\
& \lesssim q_1\et \left(\ve + \left\vert \frac{\tau H_\tau}{r}\right\vert+\left\vert \frac{H}{r}\right\vert +  \left\vert \r\left(\frac{H}{r}\right)\right\vert  \right) \notag \\
& \lesssim q_1\et \left(\ve + \sigma'\right).
\end{align*}
Now the bound~\eqref{E:JPHIBOUNDAPRIORI} follows from~\eqref{E:PHIDPHIFORMULA},~\eqref{E:PHIZERODPHIZEROFORMULA},~\eqref{E:PHIBOUNDAPRIORI}, and~\eqref{E:APRIORI}.


\noindent
{\bf Proof of~\eqref{E:PARTIALTAUPHIBOUND}.}
By~\eqref{E:THETAQDEF},~\eqref{E:HDEF}, and Proposition~\ref{P:MAINBOUNDPHI} we have
\begin{align*}
|\phi_\tau| &\lesssim \tau^{-\frac13} + \sum_{j=1}^M\ve^j \tau^{-\frac13+j\delta} \pet + \tau^{m-1}\left\vert \frac H r \right\vert
+ \tau^m \left\vert \frac{H_\tau} r \right\vert \notag \\
& \lesssim \tau^{-\frac13} + \ve \tau^{-\frac13+\delta} \pet + \sigma' \tau^{m-1} 
 \lesssim \tau^{-\frac13}, 
\end{align*}
where we have used the a priori bounds~\eqref{E:APRIORI}, the crude bound $\ve\tau^\delta\pet \lesssim1$ and the assumption $m\ge\frac{5}{2}$.


\noindent
{\bf Proof of~\eqref{E:PARTIALTAUPHIBOUND1}.} This is similar to the proof of \eqref{E:PARTIALTAUPHIBOUND}. With $\r {\phi_0} =0$,    applying $\r$ we obtain
\begin{align*}
\left\vert \r \phi_\tau \right\vert &\lesssim \ve \tau^{-\frac13+\delta} \pet + \tau^{m-1}\left\vert \r\left(\frac H r\right) \right\vert
+ \tau^m \left\vert \r\left(\frac{H_\tau} r\right) \right\vert  \\
&\lesssim (\ve+\sigma') \tau^{-\frac13+\delta},
\end{align*}
where we have used~\eqref{E:APRIORI} in the last line and the crude bound $\pet\lesssim1$. 


\noindent
{\bf Proof of~\eqref{E:PARTIALTAUTAUPHIBOUND}.}
By~\eqref{E:THETAQDEF},~\eqref{E:HDEF}, and Proposition~\ref{P:MAINBOUNDPHI} we have
\begin{align*}
|\phi_{\tau\tau}| &\lesssim \tau^{-\frac43} + \sum_{j=1}^M\ve^j \tau^{-\frac43+j\delta} \pet 
+ \tau^{m-2}\left\vert \frac H r \right\vert
+ \tau^{m-1} \left\vert \frac{H_\tau} r \right\vert + \tau^m \left\vert \frac{H_{\tau\tau}} r \right\vert \\
& \lesssim \tau^{-\frac43} + \sigma' \tau^{m-2}
 \lesssim \tau^{-\frac43}, 
\end{align*}
where we have used the a priori bounds~\eqref{E:APRIORI}, $\sigma'< 1$, $\pet\lesssim1$, and the assumption  $m\ge\frac{5}{2}$.


\noindent
{\bf Proof of~\eqref{E:PARTIALETAPHIBOUND}.}
By~\eqref{E:THETAQDEF},~\eqref{E:HDEF}, and Proposition~\ref{P:MAINBOUNDPHI}, for any $\ell=0,1,2$, we have
\begin{align*}
|\rr^\ell \phi| &\lesssim \sum_{j=1}^M\ve^j \tau^{\frac23+j\delta} + \tau^{m}\left\vert \rr^\ell\left( \frac H r\right) \right\vert  \lesssim  \ve \tau^{\frac23+\delta} + \sigma' \tau^{m} 
 \lesssim \left( \ve+\sigma'\right) \tau^{\frac23+\delta} 
\end{align*}
where we have used the a priori bounds~\eqref{E:APRIORI} and the assumption $m\ge\frac{5}{2}$. 


\noindent
{\bf Proof of~\eqref{E:DPHIBOUND}.}
By~\eqref{E:PARTIALTAUPHIBOUND} and~\eqref{E:PARTIALETAPHIBOUND} we have
\begin{align*}
\left\vert \DL\phi\right\vert &\lesssim  r^n\left( \tau^{-\frac13} + \left(\ve+\sigma'\right) \tau^{\frac23+\delta} r^{-n}q_1\et\right) + \left(\ve+\sigma'\right)\tau^{\frac23+\delta}  \lesssim \tau^{\frac23} q_1\et+  \left(\ve+\sigma'\right) \tau^{\frac23+\delta} q_1\et \\
& 
\lesssim \tau^{\frac23} q_1\et.
\end{align*}


\noindent
{\bf Proof of~\eqref{E:PARTIALTAUDPHIBOUND}.}
From the definition of $\DL$ we have
\begin{align*}
\left\vert \pa_\tau \DL\phi \right\vert &\lesssim \left\vert \r (\log g) \phi_\tau\right\vert + \left\vert \r (\log g) \phi_{\tau\tau}\right\vert 
+ \left\vert \r \phi_\tau\right\vert 
  \lesssim r^n\tau^{-\frac43} + (\ve+\sigma') \tau^{-\frac13+\delta} 
  \\
 &
 \lesssim \tau^{-\frac13}q_1\et,
\end{align*} 
where we have used the crude bound $\ve \tau^\delta \pet \lesssim1$. 

\noindent
{\bf Proof of~\eqref{E:PARTIALETADPHIBOUND}.}
From the definition of $\DL$ we have
\begin{align*}
\left\vert \r \DL\phi \right\vert &\lesssim \left\vert \rr^2 (\log g) \phi_\tau\right\vert + \left\vert \r (\log g) \r\phi_{\tau}\right\vert 
+ \left\vert \rr^2 \phi\right\vert \\
& 
\lesssim r^n \tau^{-\frac13} + \left(\ve+\sigma'\right) r^n \tau^{-\frac13+\delta} + \left(\ve+\sigma'\right) \tau^{\frac23 +\delta} 
\\
& 
\lesssim \tau^{\frac23} q_1\et, 
\end{align*} 
where we have used~\eqref{E:PARTIALTAUPHIBOUND},~\eqref{E:PARTIALTAUPHIBOUND1}, and~\eqref{E:PARTIALTAUTAUPHIBOUND}.

\noindent
{\bf Proof of~\eqref{E:PHIBOUNDAPRIORITAU} and \eqref{E:JPHIBOUNDAPRIORITAU}.}
By \eqref{E:LITLLEHFORMULA} and \eqref{E:APRIORI}, we have \eqref{E:PHIBOUNDAPRIORITAU}. 
To show~\eqref{E:JPHIBOUNDAPRIORITAU} we first observe that $| \DL h|+|\tau \pa_\tau  \DL h| \lesssim \tau^\delta$, which is a simple consequence of the bounds shown above. We recall here $h = \frac{\phi}{\phi_0}$. Now the bound follows 
from \eqref{E:Jacobianf}, \eqref{E:PHIDPHIFORMULA}, \eqref{E:APRIORI}. 
\end{proof}


\begin{lemma}
Assume that $H$ is a solution to~\eqref{E:H2} on a time interval $[\kappa,T]$ for some $T\le1$ and assume that the a priori assumptions~\eqref{E:APRIORI} hold. Then for any $( \tau, r)\in[\kappa,T]\times[0,1]$
\begin{align}
\tau^{\delta-2+\frac2n}q_{-\gamma-1}\et\lesssim c[\phi]  & \lesssim \tau^{\delta-2+\frac2n}q_{-\gamma-1}\et. \label{E:CEQUIVALENCE}\\
\left\vert \pa_\tau\J[\phi] \right\vert & \lesssim  
\tau q_1\et, 
\label{E:PARTIALTAUJBOUND}\\
\left\vert \r\J[\phi] \right\vert & \lesssim  \tau^2 q_1\et, \label{E:PARTIALETAJBOUND} \\
\left\vert \pa_\tau c[\phi] \right\vert & \lesssim  
c[\phi] \tau^{-1}, 
\label{E:PARTIALTAUCBOUND}\\
\left\vert \r c[\phi] \right\vert & \lesssim \tau^{\delta-2+\frac2n}q_{-\gamma-1}\et. \label{E:PARTIALETACBOUND} 
\end{align}
\end{lemma}


\begin{proof}
{\bf Proof of~\eqref{E:CEQUIVALENCE}.}
Recall the definition of $c[\phi]$~\eqref{E:CDEF}. By~\eqref{E:JPHIBOUNDAPRIORI} we have $\J[\phi]\approx \J[\phi_0]\approx \tau^2 q_1\et$, where we have used~\eqref{E:PHIZERODPHIZEROFORMULA} to infer the last equivalence. By~\eqref{E:JPHIBOUNDAPRIORI} $\phi^4 \approx \tau^{\frac83}$.Therefore
\[
c[\phi] \approx \tau^{\frac23-2\gamma} q_{-\gamma-1}\et = \tau^{\delta-2+\frac2n}q_{-\gamma-1}\et.
\] 


{\bf Proof of~\eqref{E:PARTIALTAUJBOUND}.}
Since $\pa_\tau \J[\phi] = 2\phi\phi_\tau(\phi + \DL\phi) + \phi^2(\phi_\tau + \pa_\tau \DL\phi)$, bounds~\eqref{E:PARTIALTAUPHIBOUND},~\eqref{E:DPHIBOUND}, and~\eqref{E:PARTIALTAUDPHIBOUND} imply
\begin{align*}
\left\vert \pa_\tau\J[\phi] \right\vert & \lesssim \tau^{\frac23}  \tau^{-\frac13} \left(\tau^{\frac23}+\tau^{\frac23} q_1\et\right) + \tau^{\frac43} \left(\tau^{-\frac13} \tau^{-\frac13} q_1\et\right)  \lesssim \tau q_1\et.
\end{align*}

\noindent
{\bf Proof of~\eqref{E:PARTIALETAJBOUND}.}
Since $\r \J[\phi] = 2\phi \r\phi (\phi + \DL\phi) + \phi^2(\r\phi + \r \DL\phi)$, bounds~\eqref{E:PARTIALETAPHIBOUND},~\eqref{E:DPHIBOUND}, and~\eqref{E:PARTIALETADPHIBOUND} imply
\begin{align*}
\left\vert \r\J[\phi] \right\vert & \lesssim \tau^{\frac23} \left(\ve+\sigma'\right) \tau^{\frac23+\delta} \left(\tau^{\frac23}+\tau^{\frac23} q_1\et \right)
+ \tau^{\frac43} \left(\left(\ve+\sigma'\right) \tau^{\frac23+\delta} + \tau^{\frac23}q_1\et\right) 
 \lesssim \tau^2 q_1\et.
\end{align*}


\noindent
{\bf Proof of~\eqref{E:PARTIALTAUCBOUND}.}
From the definition of $c[\phi]$ it is easy to check the identity
$\pa_\tau c[\phi]  = c[\phi] \left(4 \frac{\phi_\tau}{\phi} - (\gamma+1)\frac{\pa_\tau\J[\phi]}{\J[\phi]}\right)$.
Therefore
\begin{align*}
\left\vert \pa_\tau c[\phi] \right\vert & \lesssim c[\phi] \left(\tau^{-1}+ \frac{\tau q_1\et }{\tau^2q_1\et}\right)
\lesssim c[\phi] \tau^{-1},
\end{align*}
where we have used~\eqref{E:PARTIALTAUPHIBOUND},~\eqref{E:PARTIALTAUJBOUND}, 
and~\eqref{E:JPHIBOUNDAPRIORI}.


\noindent
{\bf Proof of~\eqref{E:PARTIALETACBOUND}.}
Like in the proof of~\eqref{E:PARTIALTAUCBOUND} we have
\begin{align*}
\left\vert \r c[\phi]\right\vert & \lesssim \left\vert c[\phi] \right\vert \left(\left\vert\frac{ r\pa_r \phi}{\phi}  \right\vert 
+ \left\vert\frac{\r\J[\phi]}{\J[\phi]} \right\vert \right) \\
& \lesssim \tau^{\delta-2+\frac2n}q_{-\gamma-1}\et \left(\left(\ve+\sigma'\right) \tau^\delta + 1 \right) 
\lesssim \tau^{\delta-2+\frac2n}q_{-\gamma-1}\et,
\end{align*}
where we have used bounds~\eqref{E:PARTIALETAPHIBOUND},~\eqref{E:CEQUIVALENCE}, and~\eqref{E:PARTIALETAJBOUND}.
\end{proof}


\begin{lemma}\label{L:IMPORTANT}
Assume that $H$ is a solution to~\eqref{E:H2} on a time interval $[\kappa,T]$ for some $T\le1$ and assume that the a priori assumptions~\eqref{E:APRIORI} hold. Then for any $( \tau, r)\in[\kappa,T]\times[0,1]$ the following bounds hold:
\begin{align}
1\lesssim g^{00} & \lesssim 1\label{E:GZEROZEROBOUND} \\
\left\vert \pa_r g^{00} \right\vert & \lesssim \ve \tau^{\delta-\frac1n} q_{-\gamma-1}\et \et^{2-\frac3n},\label{E:GZEROZEROBOUND1} \\
\left\vert \pa_\tau g^{00} \right\vert & \lesssim \ve \tau^{\delta-1}\label{E:GZEROZEROBOUND2} \\
\left\vert  r^{-1} g^{01} \right\vert + \left\vert \pa_r g^{01} \right\vert & \lesssim \ve \tau^{\delta-1} q_{-\gamma-1}\et \et^{1-\frac2n}, \label{E:GZEROONEBOUND} \\
\left\vert \frac{\pa_r\left(\frac{g^{01}}{g^{00}}w^\alpha r^2\right)}{w^\alpha r^2}\right\vert  & \lesssim \ve \tau^{\delta-1} \label{E:GZEROONEBOUND1}\\
\tau^{\gamma-\frac{14}3}\lesssim - \left(\frac{d(\tau,r)^2}{g^{00}}\tau^{\gamma-\frac{11}3}\right)_\tau & \lesssim \tau^{\gamma-\frac{14}3} \label{E:DBOUND}\\
w^\alpha (\tau+M_g )^{-\gamma-2}\lesssim -\left(\tau^{\gamma-\frac53} c[\phi_0]\right)_\tau & \lesssim w^\alpha (\tau+M_g )^{-\gamma-2} \label{E:CTAUBOUND}
\end{align}
\end{lemma}


\begin{proof}
{\bf Proof of~\eqref{E:GZEROZEROBOUND}.}
By definition~\eqref{E:GZEROZERODEF} of $g^{00}$ it suffices to check that $\left\| c[\phi] (\pa_r g)^2\right\|_{C^0([\kappa,T]\times[0,1])} \lesssim 1$. 
By~\eqref{E:CEQUIVALENCE} and the bound $\left\vert \pa_r g\right\vert \lesssim  r^{n-1}$ for all $ r\in[0,1]$ (by~\eqref{E:GTAYLOR2})
we have
\begin{align*}
\left\vert c[\phi] (\pa_r g)^2\right\vert & \lesssim \tau^{\delta-2+\frac2n} q_{-\gamma-1}\et r^{2n-2}   
 \lesssim \tau^\delta \frac{\et^{2-\frac2n}}{q_{\gamma+1}\et} \lesssim \tau^\delta
\end{align*}
where we recall $\delta = \frac83-2\gamma-\frac2n>0$ and $x\mapsto \frac{x^{2-\frac2n}}{(1+x)^{\gamma+1}}$ is clearly bounded for all $x\ge0$ and any $\gamma>1$. This proves~\eqref{E:GZEROZEROBOUND}. 


\noindent
{\bf Proof of~\eqref{E:GZEROZEROBOUND1}.}
From~\eqref{E:GZEROZERODEF} we have 
\begin{align*}
\left\vert \pa_r g^{00}\right\vert & \lesssim \ve |\pa_r w| |c[\phi]|  r^{2n-2} + \ve \left\vert \pa_r c[\phi]\right\vert  r^{2n-2}
+ \ve \left\vert c[\phi]\right\vert  r^{2n-3} \\
& \lesssim \ve \tau^{\delta-2+\frac2n}q_{-\gamma-1}\et  r^{2n-3} 
 = \ve \tau^{\delta-\frac1n} q_{-\gamma-1}\et \et^{2-\frac3n},
\end{align*}
where we have used~\eqref{E:PARTIALETACBOUND},~\eqref{E:CEQUIVALENCE}.


\noindent
{\bf Proof of~\eqref{E:GZEROZEROBOUND2}.}
Like above, we need to show $\left\vert \pa_\tau c[\phi]\right\vert  r^{2n-2} \lesssim \tau^{\delta-1}$.  Applying~\eqref{E:PARTIALTAUCBOUND}, it then follows
\begin{align*}
\left\vert \pa_\tau c[\phi]\right\vert  r^{2n-2} \lesssim \tau^{\delta-1}\et^{2-\frac2n}q_{-\gamma-1}\et
\lesssim \tau^{\delta-1}.
\end{align*} 

\noindent
{\bf Proof of~\eqref{E:GZEROONEBOUND}.}
From~\eqref{E:GZEROONEDEF} we have 
\begin{align*}
\left\vert \pa_r g^{01}\right\vert & \lesssim \ve |\pa_r w| |c[\phi]|  r^{n-1} + \ve \left\vert \pa_r c[\phi]\right\vert  r^{n-1}
+ \ve \left\vert c[\phi]\right\vert  r^{n-2} \\
& \lesssim \ve \tau^{\delta-2+\frac2n}q_{-\gamma-1}\et  r^{n-2} 
 = \ve \tau^{\delta-1} q_{-\gamma-1}\et \et^{1-\frac2n},
\end{align*}
where we have used~\eqref{E:PARTIALETACBOUND},~\eqref{E:CEQUIVALENCE}. The bound for $\left\vert \frac{g^{01}}{r}\right\vert$ follows analogously.


\noindent
{\bf Proof of~\eqref{E:GZEROONEBOUND1}.}
It is clear that
\begin{align*}
\left\vert \frac{\pa_r\left(\frac{g^{01}}{g^{00}}w^\alpha r^2\right)}{w^\alpha r^2}\right\vert
& \lesssim  r^{-1}\left\vert \frac{g^{01}}{g^{00}w} \right\vert 
+ \left\vert \frac{\pa_r g^{01}}{g^{00}} \right\vert + \left\vert \frac{g^{01}\pa_r g^{00}}{(g^{00})^2} \right\vert \\
& \lesssim \left\vert  c[\phi]r^{n-2} \right\vert + \left\vert \pa_r g^{01} \right\vert 
+ \left\vert g^{01} \right\vert \left\vert \pa_r g^{00} \right\vert  \\
& \lesssim \ve \tau^{\delta-1} q_{-\gamma-1}\et \et^{1-\frac2n} 
+ \ve^2  r \tau^{\delta-1} q_{-\gamma-1}\et \et^{1-\frac2n}  \\
& \quad +  \tau^{\delta-\frac1n} q_{-\gamma-1}\et \et^{2-\frac3n} \\
& \lesssim \ve \tau^{\delta-1}
\end{align*}
where we have used~\eqref{E:GZEROZEROBOUND},~\eqref{E:GZEROONEBOUND},~\eqref{E:GZEROZEROBOUND1} and $g^{01}w^{-1} = -\ve \gamma c[\phi] \frac{M_g }{r}$, $M_g $ defined in~\eqref{E:BETADEF}. Note that a negative power of $w$ is fortunately cancelled away as one positive power of $w$ is contained in the definition of $g^{01}$.


\noindent
{\bf Proof of~\eqref{E:DBOUND}.}
It clearly suffices to show $\pa_\tau\left(\frac{d(\tau,r)^2}{g^{00}}\right)\lesssim \ve\tau^{\delta-1}$.
Observe that $\pa_\tau\left(d^2\right) = \frac43 \left(\frac{\phia}{\phi_0}\right)^{-4} \pa_\tau\left(\frac{\phia}{\phi_0}\right)$. 
Since $\pa_\tau\left(\frac{\phia}{\phi_0}\right) = \sum_{j=1}^M\ve^j\pa_\tau \left(\frac{\phi_j}{\phi_0}\right)$, it follows that
$\left\vert \pa_\tau\left(\frac{\phia}{\phi_0}\right)\right\vert \lesssim \ve \tau^{\delta-1}$. Therefore
$\left\vert\pa_\tau\left(d(\tau,r)^2\right)\right\vert\lesssim \ve\tau^{\delta-1}$. Together with~\eqref{E:GZEROZEROBOUND2} the claim follows.


\noindent
{\bf Proof of~\eqref{E:CTAUBOUND}.}
Observe the identity $\tau^{\gamma-\frac53}c[\phi_0] = g^{-2}\left(\tau+\frac23M_g \right)^{-\gamma-1}$. 
Taking a $\tau$-derivative we obtain 
\[
-(\gamma+1)g^{-2}\left(\tau+\frac23M_g \right)^{-\gamma-2} (1+\frac23 \r \log r)
=-(\gamma+1)g^{-2}\left(\tau+\frac23M_g \right)^{-\gamma-2} \frac{8\pi w^\alpha}{3G},
\]
where we have used~\eqref{E:IMPORTANTAPRIORI0}. Since $1\lesssim g, G \lesssim 1$, the claim follows.
\end{proof}


A corollary of Lemma~\ref{L:IMPORTANT} is the proof of equivalence between the norms and energies given respectively by Definitions~\ref{D:NORMDEF} and~\ref{D:ENERGYDEF}.


\begin{proposition}\label{P:EQUIVALENCE}
Let $H$ be a solution to~\eqref{E:H2} on a time interval $[\kappa,T]$ for some $T\le1$. We assume that the a priori bound~\eqref{E:APRIORI} are valid on $[\kappa,T]$ for some sufficiently small $\sigma'$. 
Then there exists a $\kappa$-independent constant $C>0$ such that
\begin{align}
\frac1C S_\kappa^N(\tau) \le \sup_{\kappa\le\tau'\le\tau} \mathscr E^N(\tau') + \int_\kappa^\tau \mathscr D^N(\tau')\,d\tau' 
\le C S_\kappa^N(\tau), \ \ \tau\in [\kappa,T].
\end{align}
\end{proposition}


\subsubsection{Vector field classes $\mathcal P$ and $\bar{\mathcal P}$}


We now introduce a set of auxiliary, admissible vector fields associated with differential operators $\D_i$ and $\bar\D_i$ that allow us to circumvent coordinate singularities near the origin and to obtain high order estimates effectively. They are obtained by allowing $\frac{1}{r}$ in addition to $D_r$ whenever $D_r$ appears in the chains of $\D_i$ and $\bar\D_i$. In other words, 
\be
\mathcal P_{2j+2} :=\left\{  \prod_{k=1}^{j+1} \pa_r V_k : V_k\in \left\{D_r, \ \frac{1}{r}\right\} \right\}, \ \ \
\mathcal P_{2j+1} :=\left\{ V_{j+1} \prod_{k=1}^{j} \pa_r V_k : V_k\in \left\{D_r, \ \frac{1}{r}\right\} \right\}
\ee
for $j\geq0$ and set $\mathcal P_0=\{1\} $. Likewise, we define 
\be
\begin{split}
\bar{\mathcal P}_{2j+2} :=\left\{ W \pa_r  : W \in \mathcal P_{2j+1}\right\}, \quad 
\bar{\mathcal P}_{2j+1} :=\left\{ W \pa_r : W \in \mathcal P_{2j} \right\}
\end{split}
\ee
for $j\geq0$ and set $\bar{\mathcal P}_0=\{1\} $. The properties of $\mathcal P$ and $\bar{\mathcal P}$ are presented in detail in Appendix \ref{A:proofs}. 

In what follows, we derive the bounds of  $\bar{\mathcal P}$ of various quantities involving  $\phia$, $\phi$, $\phi+\DL\phi$ and so on that will be useful for the high-order energy estimates.



\subsection{Pointwise bounds on $\phia$}

Recall $\phia$ in \eqref{E:EXPANSIONINTRO}. 

\begin{lemma}\label{L:BARDIQBOUNDS}
The following bounds hold true: 
\begin{align}
\sum_{V\in \bar{\mathcal P}_i}\left\vert V \phia \right\vert & \lesssim   \ve r^{-i} \tau^{\frac{2}{3}+\delta} \pet, 
 \ \ i=1,\dots,N, 
\label{E:BARDIQ} \\
\sum_{V\in \bar{\mathcal P}_i}\left\vert V \DL\phia \right\vert & \lesssim  r^{n-i}\tau^{-\frac13} + \ve  r^{-i} \tau^{\frac{2}{3}+\delta} \pet \notag \\
&
\lesssim \tau^{\frac23}r^{-i} q_1\et \left(p_{1,0}\et + p_{\l,-\frac2n}\et\right),
 \ \ i=0,1,\dots,N. 
  \label{E:BARDIDQ}.
\end{align}
\end{lemma}


\begin{proof} Let $V\in \bar{\mathcal P}_i$ be given. 
By Lemma~\ref{L:SIMPLERELATIONS} we have 
\begin{align}
\left\vert V \phia  \right\vert & \lesssim  r^{-i}\sum_{\ell=1}^i \left\vert\rr^\ell \phia \right\vert 
 \lesssim  r^{-i}\sum_{j=1}^M \sum_{\ell=1}^i \ve^j\left\vert \rr^\ell \phi_j\right\vert \notag \\
 & \lesssim  r^{-i} \sum_{j=1}^M \ve^j \tau^{\frac23+ j\delta} \pet 
\lesssim \ve r^{-i} \tau^{\frac{2}{3}+\delta} \pet
\end{align}
where we have used Proposition~\ref{P:MAINBOUNDPHI} in the second line.

Recall that $\DL\phia = M_g \pa_\tau \phia  + \r \phia$. By Lemma~\ref{L:SIMPLERELATIONS}, definition~\eqref{E:BETADEF} of $M_g $, and the property~\eqref{E:GTAYLOR2} we obtain
\begin{align}
\left\vert V \DL\phia \right\vert &\lesssim  r^{n-i}\sum_{j=0}^M\sum_{\ell=0}^i\left\vert \pa_\tau\rr^\ell \phi_j\right\vert +  r^{-i}\sum_{j=1}^M\sum_{\ell=2}^{i+1} \ve^j\left\vert \rr^\ell \phi_j\right\vert \notag \\
&\lesssim  r^{n-i}\tau^{-\frac13} +  \ve r^{-i} \tau^{\frac{2}{3}+\delta} \pet,
\end{align}
where we have used the same argument as in the proof of~\eqref{E:BARDIQ} to obtain the second summand in the last bound above.
\end{proof}


A simple consequence of Lemma~\ref{L:BARDIQBOUNDS} is the following corollary.

\begin{corollary}\label{C:ROUGHBOUNDS}
The following bounds hold true: 
\begin{align}
\sum_{V\in \bar{\mathcal P}_i}\left\vert V \phia \right\vert & \lesssim \ve \tau^{\frac23+\delta^\ast} p_{\l,-\frac{N+2}{n}}\et \lesssim  \ve \tau^{\frac23+\delta^\ast}, 
\ \ i=1,\dots,N. 
 \label{E:BARDIQROUGH} \\
\sum_{V\in \bar{\mathcal P}_i}\left\vert V \DL\phia \right\vert & \lesssim \tau^{\frac23}  r^{-i} q_1\et,
\ \ i=0,1,\dots,N, 
\label{E:BARDIDQROUGH}
\end{align}
where we recall that $\delta^\ast$ is given by~\eqref{E:DELTASTARDEF}.
\end{corollary}

\begin{lemma}\label{L:DIQBOUNDS} For any $1\le i\le N$ we have 
\begin{align}
\sum_{V\in \bar{\mathcal P}_i}\left\vert V  \DL^2 \phia \right\vert &\lesssim \tau^{\frac23}   r^{-i} 
q_2\et \left(p_{1,0}\et + p_{\l,-\frac2n}\et\right) 
\label{E:D2QBOUND}\\
\sum_{V\in \bar{\mathcal P}_i}\left\vert V \DL\left(3\phia^2+2\phia D \phia  \right)\right\vert & \lesssim \tau^{\frac43}  r^{-i} 
q_2\et \left(p_{1,0}\et + p_{\l,-\frac2n}\et\right) 
\label{E:QBOUND2}\\
\sum_{V\in \bar{\mathcal P}_i}\left\vert V \DL\J[\phia]\right\vert & \lesssim \tau^2 r^{-i} 
q_2\et \left(p_{1,0}\et + p_{\l,-\frac2n}\et\right) 
\label{E:QBOUND3}\\
\sum_{V\in \bar{\mathcal P}_i}\left\vert V \left(\J[\phia]^a\right)\right\vert & \lesssim \tau^{2a} r^{-i} q_a\et \label{E:JQPOWERABOUND}\\
\sum_{V\in \bar{\mathcal P}_i}\left\vert V \left(\frac{ \DL\J[\phia]}{\J[\phia]^{\gamma+2}}\right)\right\vert & \lesssim \tau^{-2\gamma-2} r^{-i} 
q_{-\gamma}\et \left(p_{1,0}\et + p_{\l,-\frac2n}\et\right) 
\label{E:DJQBOUND}\\
\sum_{V\in \bar{\mathcal P}_i}\left\vert V \left(\frac{\phia \DL\phia }{\J[\phia]^{\gamma+1}}\right)\right\vert &\lesssim \tau^{-2\gamma-\frac23}   r^{-i} 
q_{-\gamma}\et \left(p_{1,0}\et + p_{\l,-\frac2n}\et\right)
\\
\sum_{V\in \bar{\mathcal P}_i}\left\vert V \left(\frac{\DL\phia }{\J[\phia]^{\gamma+1}}\right)\right\vert &\lesssim \tau^{-2\gamma-\frac43}   r^{-i} 
q_{-\gamma}\et \left(p_{1,0}\et + p_{\l,-\frac2n}\et\right).  
\label{E:DQJBOUND} \\
\sum_{V\in \bar{\mathcal P}_i}\left\vert V \left(\frac{P[\phia]}{\phia}\right)\right\vert &\lesssim 
\tau^{\frac23-2\gamma-\frac{i+2}n} q_{-\gamma+1}\et \left(\tau p_{1,-\frac{i+2}{n}}\et +  p_{\l,-\frac{i+4}{n}}\et \right) \notag 
 \\ 
& 
\lesssim \tau^{\frac23-2\gamma-\frac{i+2}n} 
\label{E:PQBOUND} 
\end{align}
\end{lemma}

\begin{proof}
{\bf Proof of~\eqref{E:D2QBOUND}.}
By a simple calculation $ \DL^2=M_g ^2\pa_{\tau\tau} + 2 r M_g \pa_{ r\tau}+ M_g \pa_\tau M_g  \pa_\tau+\r M_g  \pa_\tau + \rr^2$.
By the product rule $V(M_g ^2\pa_{\tau\tau}\phia )$ can be written as a linear combination of expression of the form
\begin{align*}
A (M_g ^2) \ B \pa_{\tau\tau}\phia , \ \ A\in\bar{\mathcal P}_k, \ \ B\in\bar{\mathcal P}_{i-k}, \ \ 0\le k\le i.
\end{align*}
Any such expression is bounded by $ r^{2n-i} \tau^{\frac23-2} = \tau^{\frac23} \et^2  r^{-i}$. A similar argument shows that
$\left\vert V ( r M_g \pa_{ r\tau}\phia )\right\vert \lesssim \ve \tau^{\frac23+\delta}\frac{ r^n}{\tau}\pet  r^{-i}$, $\left\vert V (M_g \pa_\tau M_g  \pa_\tau \phia )\right\vert \lesssim \tau^{\frac23}\frac{ r^{2n}}{\tau}  r^{-i}$,
$\left\vert V (\r M_g  \pa_\tau \phia )\right\vert \lesssim \tau^{\frac23}\frac{ r^n}{\tau}  r^{-i}$, 
$\left\vert V (\rr^2\phia )\right\vert \lesssim \ve \tau^{\frac23+\delta} \pet  r^{-i}$.
Summing the above bounds we obtain~\eqref{E:D2QBOUND}.

\noindent
{\bf Proof of~\eqref{E:QBOUND2}.}
Note that $\DL\left(3\phia^2+2\phia D \phia  \right)= 6\phia \DL\phia +2(\DL\phia)^2+2\phia  \DL^2\phia$. Using the product rule, bounds~\eqref{E:BARDIQ},~\eqref{E:BARDIDQ}, and~\eqref{E:D2QBOUND} we obtain~\eqref{E:QBOUND2}.

\noindent
{\bf Proof of~\eqref{E:QBOUND3}.}
The proof is similar to~\eqref{E:QBOUND2}. 
From~\eqref{E:JFORMULA} we have 
$ \DL\J[\phia]=3\phia^2\DL\phia +2\phia(\DL\phia)^2+\phia^2  \DL^2\phia$. Now the statement follows from the product rule and bounds~\eqref{E:BARDIQ},~\eqref{E:BARDIDQ}, and~\eqref{E:D2QBOUND}.

\noindent
{\bf Proof of~\eqref{E:JQPOWERABOUND}.}
We must use the chain rule. We note that $V (\J[\phia]^a)$ can be expressed as a linear combination of expressions of the form
\[
\J[\phia]^a \left(\prod_{j=1}^{j_m}\frac{W_j\J[\phia]}{\J[\phia]}\right)_{W_j\in\bar{\mathcal P}_{i_j},\ i_1+\dots+i_{j_m}=i}.
\]
We may use~\eqref{E:BARDIQ} 
and ~\eqref{E:BARDIDQROUGH} 
to conclude that $\left\vert W\J[\phia]\right\vert\lesssim \tau^2 q_1\et  r^{-j}$ for any $W\in\bar{\mathcal P}_j$.
Since $\tau^2 q_1\et \lesssim \J[\phia] \lesssim \tau^2q_1\et$, we can bound the above expression by $\tau^{2a}q_a\et  r^{-i}$.

\noindent
{\bf Proof of~\eqref{E:DJQBOUND}-\eqref{E:DQJBOUND}.}  The proof follows by the product rule~\eqref{E:product} and \eqref{E:QBOUND3}, \eqref{E:JQPOWERABOUND}, \eqref{E:BARDIQ}, \eqref{E:BARDIDQ}.

\noindent
{\bf Proof of~\eqref{E:PQBOUND}.}
Recalling~\eqref{E:PRESSURETERM} it is easy to check that
\begin{align}\label{E:PQFORMULA}
\frac{P[\phia]}{\phia} = (1+\alpha) \frac{w'}{g^2 r} \phia\J[\phia]^{-\gamma} - \gamma\frac{w}{g^2 r^2} \phia \J[\phia]^{-\gamma-1} \DL\J[\phia].
\end{align}
We now apply the product rule~\eqref{E:product} and bounds~\eqref{E:JQPOWERABOUND},~\eqref{E:QBOUND3},~\eqref{E:BARDIQ} 
and the estimate $|w'|\lesssim r^{n-1}$ to conclude 
\begin{align*}
\left\vert V \left(\frac{P[\phia]}{\phia}\right)\right\vert &\lesssim 
\tau^{\frac53 -2\gamma}\frac{r^n}{\tau} r^{-(i+2)}q_{-\gamma}\et
+ \tau^{\frac23-2\gamma}r^{-(i+2)} q_{-\gamma+1}\et \left(p_{1,0}\et + p_{\l,-\frac2n}\et\right) \\
& \lesssim \tau^{\frac23-2\gamma}r^{-(i+2)}  q_{-\gamma+1}\et \left( \tau p_{1,0}\et + p_{1,0}\et + p_{\l,-\frac2n}\et\right) \\
& \lesssim \tau^{\frac23-2\gamma}r^{-(i+2)}  q_{-\gamma+1}\et \left(p_{1,0}\et + p_{\l,-\frac2n}\et\right) 
\end{align*}
since $\tau\le1$. Replacing $r^{-(i+2)}$ by $\tau^{-\frac{i+2}{n}} \et^{-\frac{i+2}{n}}$ above, we obtain the claim, where in particular we use $\gamma>1$.
\end{proof}

\subsection{Preparatory bounds}


Recall $\phi= \phia + \tau^m \frac{H}{r}$.

\begin{lemma}\label{L:Diphi} 
For any $1\leq i\leq N$, we have 
\begin{align}
| \bar\D_i \phi | &\lesssim | \bar\D_i \phia | + \tau^m\big|  \bar\D_i \left( \frac{H}{r}\right) \big| \label{E:BARDIPHIBOUND1} \\
|\bar\D_i (\phi+ \DL\phi )| &\lesssim
 | \bar\D_i (\phia + \DL\phia)| \notag \\
& \ \ \ \ + \tau^m \left(  \left| \frac{M_g }{r} \D_{i} \pa_\tau H \right|+ \sum_{1\leq k\leq i\atop B\in {\bar{\mathcal P}}_{i-k}}  \left| \pa_r^k(M_g ) B (\frac{\pa_\tau H}{r}) \right| + |\D_{i+1} H| + \left| \bar\D_i (\frac{H}{r})\right| \right) \notag \\
& \ \ \ \   + 
\tau^{m-1}\left(    \left| \frac{M_g }{r} \D_{i}  H \right|+  \sum_{1\leq k\leq i\atop B\in {\bar{\mathcal P}}_{i-k}}   \left| \pa_r^k(M_g ) B (\frac{H}{r}) \right| \right)
\label{E:BARDIPHIBOUND2}
\end{align}
\end{lemma}


\begin{proof}
Bound~\eqref{E:BARDIPHIBOUND1} follows directly follows from 
\begin{align*}
\bar\D_i \phi = \bar\D_i (\phia + \tau^m \frac{H}{r})= \bar\D_i \phia  + \tau^m \bar\D_i \left( \frac{H}{r}\right).
\end{align*}

Further
\begin{align}
\bar\D_i (\phi+ \DL\phi )&= \bar\D_i (\phia + \DL\phia) + \bar\D_i (1+M_g \pa_\tau + \r) ( \tau^m \frac{H}{r}) \notag \\
&=\bar\D_i (\phia + \DL\phia) +\tau^m  \underbrace{\bar\D_i (M_g  \frac{\pa_\tau H}{r}) }_{(\ast)} + 
m 
\tau^{m-1} \underbrace{\bar\D_i ( M_g  \frac{H}{r})}_{(\ast\ast) } + \tau^m\left( \D_{i+1} H - 2 \bar\D_i ( \frac{H}{r}) 
\right),
\end{align}
where we have used the identities $r\pa_r\left(\frac Hr\right) = \pa_r H - \frac{H}{r}$ and $\bar\D_i \pa_r H
= \bar\D_i \left(D_r H-\frac2r H\right) = \D_{i+1}H - 2\bar D_i \left(\frac Hr\right)$.
For $(\ast)$, we first note that 
\be
(\ast) =M_g  \bar\D_i (\frac{\pa_\tau H}{r}) + \sum_{1\leq k\leq i\atop A\in {\bar{\mathcal P}}_k, B\in {\bar{\mathcal P}}_{i-k}} c_k^{iAB} A(M_g ) B( \frac{\pa_\tau H}{r}) 
\ee
For the first term, we use \eqref{XtoXover} to rewrite 
\be
M_g  \bar\D_i (\frac{\pa_\tau H}{r}) =
\begin{cases}
 \frac{M_g }{r} \left( \D_i \pa_\tau H - (i-1) \bar \D_{i-1}(\frac{\pa_\tau H}{r}) \right) &\text{ if } i \text{ is even}\\
  \frac{M_g }{r} \left( \D_i \pa_\tau H - (i+1) \bar \D_{i-1}(\frac{\pa_\tau H}{r}) \right) &\text{ if } i \text{ is odd}
\end{cases}
\ee
Therefore we deduce that 
\be
|(\ast)| \lesssim  \left| \frac{M_g }{r} \D_{i} \pa_\tau H \right|+ \sum_{1\leq k\leq i\atop B\in {\bar{\mathcal P}}_{i-k}}  \left| \pa_r^k(M_g ) B (\frac{\pa_\tau H}{r}) \right|
\ee
It is easy to see that 
\be
|(\ast\ast)| \lesssim  
  \left| \frac{M_g }{r} \D_{i}  H \right|+  \sum_{1\leq k\leq i\atop B\in {\bar{\mathcal P}}_{i-k}}   \left| \pa_r^k(M_g ) B (\frac{H}{r}) \right| 
\ee
Putting together the above bounds we obtain~\eqref{E:BARDIPHIBOUND2}.
\end{proof}

The same conclusions hold in Lemma \ref{L:Diphi} when we replace $\bar\D_i$ by any $V\in\bar{ \mathcal P}_i$.



\begin{lemma}[High-order $\phi$-bounds]
The following $L^\infty$-bounds hold:
\begin{align}
\sum_{V\in\bar{\mathcal P}_j}\left\| \frac{V \phi}{\phi} \right\|_\infty & \lesssim \ve \tau^{\delta^\ast} + \tau^{m-\frac23+\frac12(\frac{11}{3}-\gamma)} (E^N)^{\frac12} \  \text{ for } \ 1\leq  j \leq 2 \label{LinftyV1}\\
\sum_{V\in\bar{\mathcal P}_j}\left\| w^{j -2} \frac{V \phi}{\phi} \right\|_\infty & \lesssim \ve \tau^{\delta^\ast} + \tau^{m-\frac23+\frac12(\frac{11}{3}-\gamma)} (E^N)^{\frac12} \  \text{ for } \ 2\leq  j \leq N-3 \label{LinftyV2}\\
\sum_{V\in\bar{\mathcal P}_j}\left\|  r w^{j-2} \frac{V \phi}{\phi} \right\|_\infty & \lesssim \ve \tau^{\delta^\ast} + \tau^{m-\frac23+\frac12(\frac{11}{3}-\gamma)} (E^N)^{\frac12} \  \text{ for } \  j = N-2\label{LinftyV3}
\end{align}
The following $L^2$-bounds hold:
\begin{align}
\sum_{V\in\bar{\mathcal P}_j}\left\| \frac{V \phi}{\phi}  \right\|_{\alpha+2j +2-N} &\lesssim \ve \tau^{\delta^\ast} + \tau^{m-\frac23+\frac{1}{2}(\frac{11}{3}-\gamma)} (E^N)^{\frac12}  \  \text{ for } \ \frac{N-\alpha-2}{2}\leq  j \leq N-1 \label{L2V1}\\
\ve^\frac12 \sum_{V\in\bar{\mathcal P}_N}\left\|  \frac{V \phi}{\phi}  \right\|_{\alpha+N+1} &  \lesssim  \ve^\frac32 \tau^{\delta^\ast} + \tau^{m-\frac23} (E^N)^{\frac12}  \label{L2V2}
\end{align}
\end{lemma}


\begin{proof}
Note that from \eqref{E:BARDIPHIBOUND1} and \eqref{E:BARDIQROUGH} 
\be
\left| \frac{V_j \phi}{\phi}\right| \lesssim \ve \tau^{\delta^\ast} + \tau^{m-\frac23} |V_j (\frac{H}{r}) | 
\ee
Therefore from \eqref{Linfty11}, \eqref{Linfty21} and \eqref{Linfty31} we deduce~\eqref{LinftyV1}--\eqref{LinftyV3}.
Bounds~\eqref{L2V1}--\eqref{L2V2} follow from \eqref{L21} and \eqref{L22}, where 
we use the bound 
\begin{align*}
\ve \int  w^{\alpha+2k+1-N}   |\D_{k+1} H |^2  r^2 d r  \lesssim \ve \int \frac{ w^{\alpha+2k+1-N}}{(\tau+\frac23M_g )^{1+\gamma}}  
|\D_{k+1} H |^2  r^2 d r  \lesssim E^N 
\end{align*}
\end{proof}


\begin{lemma}[High-order $\phi+\DL\phi$-bounds]
The following $L^\infty$-bounds hold:
\begin{align}
\sum_{W\in\bar{\mathcal P}_j}\left\|  \frac{ W (\phi+\DL\phi ) }{\phi+ \DL\phi}  \right\|_\infty & \lesssim  \tau^{-\frac{j}{n}} \left(1+ \tau^{m-\frac23+\frac12(\frac{5}{3}-\gamma)+1} (E^N)^{\frac12}\right) \  \text{ for } \  j =1  
\label{LinftyW1}\\
\sum_{W\in\bar{\mathcal P}_j} \left\| w^{j -1} \frac{ W (\phi+\DL\phi ) }{\phi+ \DL\phi}  \right\|_\infty & \lesssim \tau^{-\frac{j}{n}} \left(1+ \tau^{m-\frac23+\frac12(\frac{5}{3}-\gamma)+1} (E^N)^{\frac12}\right) \  \text{ for } \ 2\leq  j \leq N-3\label{LinftyW2} \\
\sum_{W\in\bar{\mathcal P}_j}\left\|  r w^{j-1} \frac{ W (\phi+\DL\phi ) }{\phi+ \DL\phi}  \right\|_\infty & \lesssim \tau^{-\frac{j}{n}} \left(1+ \tau^{m-\frac23+\frac12(\frac{5}{3}-\gamma)+1} (E^N)^{\frac12}\right) \  \text{ for } \  j= N-2\label{LinftyW3}
\end{align}
The following $L^2$-bounds hold:
\begin{align}
\sum_{W\in\bar{\mathcal P}_j}\left\| \frac{ W (\phi+\DL\phi ) }{\phi+ \DL\phi}  \right\|_{\alpha+2j +2-N} &\lesssim  \tau^{-\frac{j}{n}} \left(1+ \tau^{m-\frac23+\frac12(\frac{5}{3}-\gamma)+1} (E^N)^{\frac12}\right)  \  \text{ for } \ \frac{N-\alpha-2}{2}\leq  j \leq N-1 \label{L2W1}\\
\ve^\frac12\sum_{W\in\bar{\mathcal P}_N} \left\|  \frac{ W (\phi+\DL\phi ) }{\phi+ \DL\phi}  \right\|_{\alpha+N+1} &  \lesssim  \ve^\frac12  \tau^{-\frac{N}{n}} + \tau^{m-\frac23} (E^N)^{\frac12} \label{L2W2}
\end{align}
\end{lemma}


\begin{proof}
From \eqref{E:BARDIPHIBOUND2} and 
\eqref{E:BARDIDQ}, 
we note that 
\be
\left| \frac{ W (\phi+\DL\phi ) }{\phi+ \DL\phi}  \right| \lesssim  r^{-j}
\left( p_{1,0}(\frac{ r^n}{\tau}) +\pet \right)
+ \tau^{m-\frac23} q_{-1}(\frac{ r^n}{\tau}) 
\left| W \left( \frac{H}{r} + \DL(\frac{H}{r}) \right)\right|
\ee 
Therefore, bounds~\eqref{LinftyW1}--\eqref{LinftyW3} follow from~\eqref{Linfty11}--\eqref{Linfty31}. Bounds~\eqref{L2W1}--\eqref{L2W2} follow 
from~\eqref{L21} and~\eqref{L22} respectively.
\end{proof}


Finally, the key collection of a priori bounds is provided by the following lemma, and will be used repeatedly 
in our energy estimates in Section~\ref{S:ENERGYESTIMATES}.

\begin{lemma}\label{L:PHILEMMA}
Let $a,b,c\in\mathbb R$, $b<0$, $c\le-b$, be given.
For any $i\in\{0,1\}$ we have
\begin{align}\label{E:PHILEMMA1}
\sum_{V\in\bar{\mathcal P}_i}|V\left(\phi^a\J[\phi]^b\right)| \lesssim \tau^{\frac23a+2b-\frac in}q_b\et.
\end{align}
If $2\le i\le N-1$ then
\begin{align}\label{E:PHILEMMA2}
\sum_{V\in\bar{\mathcal P}_i}\left\| q_c\et  V\left(\phi^a\J[\phi]^b\right)\right\|_{\alpha-N+2i+2} \lesssim \tau^{\frac23a+2b-\frac in}(1+(E^N)^{\frac12}).
\end{align}
If $2\le i\le N-3$ then
\begin{align}\label{E:PHILEMMA3}
\sum_{V\in\bar{\mathcal P}_i}\left\|w^{i} q_c\et V\left(\phi^a\J[\phi]^b\right)\right\|_{\infty} \lesssim \tau^{\frac23a+2b-\frac in}(1+(E^N)^{\frac12}).
\end{align}
Finally, if $i=N$ we have
\begin{align}\label{E:PHILEMMA4}
\sqrt\ve \sum_{V\in\bar{\mathcal P}_N}\left\| q_c\et V\left(\phi^a\J[\phi]^b\right)\right\|_{\alpha+N+2} \lesssim \tau^{\frac23a+2b-\frac Nn}(1+(E^N)^{\frac12}).
\end{align}
\end{lemma}


\begin{proof}
By definition of $\J[\phi]$ we have
\be\label{E:PHIAB}
\phi^a\J[\phi]^{b} = \phi^{a+2b} (\phi + \DL\phi)^{b}.
\ee 
Applying the product and the chain rule, for any $V\in \bar{\mathcal P}_i $, $i\geq1$, 
$V\left(\phi^a\J[\phi]^{b} \right) $ can be written as a linear combination of 
\be\label{E:LC}
\phi^{a+2b} (\phi + \DL\phi)^{b}  \left( \prod_{j=1}^{j_m} \frac{ V_j \phi }{\phi }   \right)_{V_j\in {\bar{\mathcal P}}_{i_j}, i_1+...+i_{j_m}=i-p}
\left( \prod_{\ell =1}^{\ell_m} \frac{ W_\ell (\phi+\DL\phi ) }{\phi+ \DL\phi} \right)_{W_\ell\in {\bar{\mathcal P}}_{a_\ell}, a_1+...+a_{\ell_m}=p}
\ee
where 
$0\leq p\leq i$. 
In order to estimate $V_j\phi$ and $W_j (\phi+\DL\phi )$, it suffices to estimate $\bar\D_i \phi$ and $\bar\D_i (\phi + \DL\phi)$. 

Let $k_* = \max\{i_j, a_\ell\}$ in~\eqref{E:LC}.
Without loss of generality, we may assume that indices appearing in \eqref{E:LC} are non-decreasing: $i_1\leq...\leq i_{j_m}$ and $a_1\leq...\leq a_{\ell_m}$. Then $k_\ast=\max\{i_{j_m}, a_{\ell_m}\}$.

\noindent
{\em Proof of~\eqref{E:PHILEMMA1}.}
Bound~\eqref{E:PHILEMMA1} is obvious from~\eqref{E:PHIAB} if $i=0$.
If $i=1$ then the claim follows from~\eqref{LinftyV1} and~\eqref{LinftyW1}.

\noindent
{\em Proof of~\eqref{E:PHILEMMA2}.}
If $k_\ast=1$, by using \eqref{LinftyV1} and \eqref{LinftyW1}, the expression in \eqref{E:LC} is bounded by 
\[
\tau^{\frac23(a+3b)}q_{b}\et\left( \ve \tau^{\delta^\ast} + \tau^{m-\frac23+\frac12(\frac{11}{3}-\gamma)} (E^N)^{\frac12}  \right)^{k-p}  \left( \tau^{-\frac{1}{n}} + \tau^{m-\frac23+\frac12(\frac{5}{3}-\gamma)+1-\frac{1}{n}} (E^N)^{\frac12} \right)^p
\]
and therefore, the worst bound occurs at $p=k$ and the last line is bounded by 
\be
 \tau^{\frac23a+b -\frac{k}{n} + m} q_b\et(1+(E^N)^{\frac12}),
\ee
where we note that that $\|w^{\alpha-N+2i+2}\|_{L^\infty} \lesssim 1$ since $i\ge2$ and $N= \lfloor \alpha \rfloor +6$. 
Suppose that $2\leq k_\ast \leq N-1$.  

We first consider  $k_\ast=a_{\ell_m}\geq i_{j_m}$. Let $j_{m_0}+1$ be the first index for which $i_{j_{m_0}+1}\geq 2 $ so that   $i_j=1$ for $j\leq j_{m_0}$. In this case, we rearrange the $w$-weight in~\eqref{E:LC} as follows:
\begin{align}
&\big\vert
w^{\frac{\alpha-N+2i+2}{2}}\phi^{a+2b} (\phi + \DL\phi)^{b}  \left( \prod_{j=1}^{j_m} \frac{ V_j \phi }{\phi } \right)_{V_j\in {\bar{\mathcal P}}_{i_j}, i_1+...+i_{j_m}=i-p}
\left( \prod_{\ell =1}^{\ell_m} \frac{ W_\ell (\phi+\DL\phi ) }{\phi+ \DL\phi} \right)_{W_\ell\in {\bar{\mathcal P}}_{a_\ell}, a_1+...+a_{\ell_m}=p}
\big\vert 
\notag \\
&=\big\vert
\phi^{a+2b} (\phi + \DL\phi)^{b}
\prod_{j=1}^{j_{m_0}} \left( \frac{ V_j \phi }{\phi }  \right) \prod_{j=j_{m_0} +1}^{j_m} \left( w^{i_j -2}\frac{ V_j \phi }{\phi }  \right)
  \prod_{\ell =1}^{\ell_{m-1}} \left( w^{a_j-1} \frac{ W_\ell (\phi+\DL\phi ) }{\phi+ \DL\phi} \right) \notag \\
&  w^{i-a_{\ell_m} -\sum_{j=j_{m_0}+1}^{j_m}(i_j-2)-\sum_{\ell=1}^{\ell_{m-1}}(a_\ell-1) } w^{\frac{\alpha + 2a_{\ell_m} + 2 -N}{2}} \frac{ W_{\ell_m} (\phi+\DL\phi ) }{\phi+ \DL\phi} \label{E:PROOF1}
\big\vert
\end{align}
The goal is to estimate the last term $w^{\frac{\alpha + 2a_{\ell_m} + 2 -N}{2}} \frac{ W_{\ell_m} (\phi+\DL\phi ) }{\phi+ \DL\phi}$ in $L^2$-norm and all the remaining ones in $L^\infty$. 
Note that  
\begin{align*}
&  \sum_{j=j_{m_{0}}+1}^{j_{m}}(i_{j}-2)+\sum_{l=1}^{l_{m-1}}(a_{l}-1)\\
&  =i-p-2\{j_{m}-j_{m_{0}}\}+p-a_{l_{m}}-\{l_{m-1}\}\\
&  =i-a_{l_{m}}-2\{j_{m}-j_{m_{0}}\}-l_{m-1}\\
&  \leq i-a_{l_{m}}.
\end{align*}
Therefore, the exponent of the first $w$ in the second line is non-negative and therefore, that factor is bounded. 
 Now all $i_j$'s and $a_\ell$'s except $a_{\ell_m}$ cannot be bigger than $N-3$, otherwise, it would contradict the definition of $k_\ast$. Thus we can apply \eqref{LinftyV2} and \eqref{LinftyW2} to the first line above. Moreover, since $2\leq a_{\ell_m}\leq N-1$, we can apply the weighted $L^2$-embedding~\eqref{L2W1} to the $W_{\ell_m}$ term in the second line of the right-hand side of~\eqref{E:PROOF1}. 
By~\eqref{E:PHILEMMA1} 
$\|q_c\et \phi^{a+2b} (\phi + \DL\phi)^{b} \|_{L^\infty} \le \tau^{\frac23 a+b}$ since $b+c\le0$ by our assumptions. This gives the bound
\begin{align}
\|w^{\frac{\alpha-N+2i+2}{2}} q_c\et V\left(\phi^a\J[\phi]^b\right)\|_{L^2} \lesssim \tau^{\frac23a+2b-\frac in}(1+(E^N)^{\frac12}), \ \ V\in \bar{\mathcal P}_i.
\end{align}

The case $k_\ast=i_{j_m} > a_\ell$ can be treated in the same fashion where we use~\eqref{L2V1} instead of~\eqref{L2W1}. 

\noindent
{\em Proof of~\eqref{E:PHILEMMA3}.}
In this case, since $k_\ast\le N-3$ and as above we first consider the case $k_\ast = a_{\ell_m}$. We then have
\begin{align}
&\big\vert
w^{i}\phi^{a+2b} (\phi + \DL\phi)^{b}  \left( \prod_{j=1}^{j_m} \frac{ V_j \phi }{\phi } \right)_{V_j\in {\bar{\mathcal P}}_{i_j}, i_1+...+i_{j_m}=i-p}
\left( \prod_{\ell =1}^{\ell_m} \frac{ W_\ell (\phi+\DL\phi ) }{\phi+ \DL\phi} \right)_{W_\ell\in {\bar{\mathcal P}}_{a_\ell}, a_1+...+a_{\ell_m}=p}
\big\vert 
\notag \\
&=\big\vert
  w^{i-\sum_{j=j_{m_0}+1}^{j_m}(i_j-2)-\sum_{\ell=1}^{\ell_{m}}(a_\ell-1) } 
\phi^{a+2b} (\phi + \DL\phi)^{b}
\prod_{j=1}^{j_{m_0}} \left( \frac{ V_j \phi }{\phi }  \right) \prod_{j=j_{m_0} +1}^{j_m} \left( w^{i_j -2}\frac{ V_j \phi }{\phi }  \right) \notag \\
 & \ \ \ \  \prod_{\ell =1}^{\ell_{m}} \left( w^{a_\ell-1} \frac{ W_\ell (\phi+\DL\phi ) }{\phi+ \DL\phi} \right) 
\big\vert
\end{align}
Note that  $\sum_{j={j_{m_0}} +1}^{j_m}(i_j-2)+\sum_{\ell=1}^{\ell_{m}}(a_\ell-1) \leq i$ and therefore, the exponent of the first $w$ in the next-to-last line above is non-negative 
as before. 
Using~\eqref{LinftyV1}--\eqref{LinftyV2}, ~\eqref{LinftyW1}--\eqref{LinftyW2}, and the bound 
$\|q_c\et \phi^{a+2b} (\phi + \DL\phi)^{b} \|_{L^\infty} \le \tau^{\frac23 a+b},$ we can bound all the remaining factors to finally obtain~\eqref{E:PHILEMMA3}.

\noindent
{\em Proof of~\eqref{E:PHILEMMA4}.}
When $i=N$ and $k_\ast\le N-1$ we may use the already proven~\eqref{E:PHILEMMA2} to infer that the $\|\cdot\|_{\alpha+N+2}$-norm of~\eqref{E:LC} is bounded by the right-hand side of~\eqref{E:PHILEMMA4}. It now remains to discuss the case $k_\ast=N$ in which case either 
1) $j_m=1$ and $i_{j_m}=N$ ($a_{\ell_m}=0$) or 2) $\ell_m=1$ and $a_{\ell_m}=N$ ($i_{j_m}=0$).  
When $a_{\ell_m}=N$, the expression~\eqref{E:LC} reads
\begin{align}
\phi^{a+2b} (\phi + \DL\phi)^{b}  \frac{ W_{\ell_m} (\phi+\DL\phi ) }{\phi+ \DL\phi} 
\end{align}
and therefore by~\eqref{L2W2} and~\eqref{E:PHILEMMA1}, we deduce
\begin{align*}
\ve \|\phi^{a+2b} (\phi + \DL\phi)^{b} q_c\et  \frac{ W_{\ell_m} (\phi+\DL\phi ) }{\phi+ \DL\phi}\|_{\alpha+N+2} 
&\lesssim \ve^\frac12 \tau^{\frac23a + 2b}  (\ve^\frac12 \tau^{-\frac{N}{n}} + \tau^{m-\frac23} (E^N)^\frac12) \\
& \lesssim \tau^{\frac23a+2b-\frac Nn}(1+(E^N)^{\frac12}),
\end{align*}
as claimed.
When $i_{j_m}=N$, the corresponding estimate reads  
\be
\ve \|\phi^{a+2b} (\phi + \DL\phi)^{b} q_c\et \frac{ V_{j_m} \phi }{\phi } \|_{\alpha+N+2} 
\lesssim \ve^\frac12 \tau^{\frac23a + 2b}  (\ve^\frac32 \tau^{\delta^\ast} + \tau^{m-\frac23} (E^N)^\frac12)
\lesssim \tau^{\frac23a+2b-\frac Nn}(1+(E^N)^{\frac12}), \notag
\ee
where we have used~\eqref{L2V2}.
\end{proof}

We conclude the section with several a priori estimates that will be important for the energy estimates in Section~\ref{S:ENERGYESTIMATES}.

\begin{lemma}\label{L:GZEROZEROLEMMA} Recall $g^{00}$ defined in \eqref{E:GZEROZERODEF}. The following bounds hold: 
\begin{align}
\sum_{V\in\bar{\mathcal P}_1}\left(\lv Vg^{00} \rv +\lv V(\frac1{g^{00}}) \rv\right) & \lesssim \ve\tau^{\delta - \frac 1n} 
\label{E:GZEROZERO1}\\
\sum_{V\in\bar{\mathcal P}_i}\left(\left\| Vg^{00}\right\|_{\alpha-N+2i+2}+\left\| V(\frac1{g^{00}})\right\|_{\alpha-N+2i+2}\right)& \lesssim 
\ve \tau^{\delta - \frac in} (1+(E^N)^{\frac12}), \ \ 2\le i \le N-1, \label{E:GZEROZERO2}\\
\sum_{V\in\bar{\mathcal P}_i}\left(\left\|w^{i}Vg^{00}\right\|_{\infty}+\left\|w^{i}V(\frac1{g^{00}})\right\|_{\infty}\right)&  
\lesssim \ve \tau^{\delta - \frac in} (1+(E^N)^{\frac12}),  \ \ 2\le i \le N-3, \label{E:GZEROZERO3}\\
\sqrt\ve \sum_{V\in\bar{\mathcal P}_N}\left(\left\| Vg^{00}\right\|_{\alpha+N}+\left\|V(\frac1{g^{00}})\right\|_{\alpha+N} \right)&  
\lesssim \ve\tau^{\delta - \frac in} (1+(E^N)^{\frac12})\label{E:GZEROZERO4}
\end{align}
\end{lemma}


\begin{proof}
It suffices to prove the bounds for $Vg^{00}$ as the corresponding bound for $V(\frac1{g^{00}})$ is a simple consequence of the chain rule~\eqref{E:CR} and
the bound~\eqref{E:GZEROZEROBOUND}.
From~\eqref{E:GZEROZERODEF} and~\eqref{E:CDEF} it follows that for any $V\in\bar{\mathcal P}_i$ with $i\ge 1$ we have
\begin{align}
Vg^{00} = - \ve\gamma \sum_{A_{1,2}\in \bar{\mathcal P}_{\ell_1,\ell_2} \atop \ell_1+\ell_2=i}
c_i^{A_1A_2} A_1(\frac{wM_g ^2}{g^2 r^2}) A_2(\phi^4\J[\phi]^{-\gamma-1}). 
\end{align}
In particular, if $\ell_2\le i-1$ we may estimate
\[
\lv A_1(\frac{wM_g ^2}{g^2 r^2}) A_2(\phi^4\J[\phi]^{-\gamma-1}) \rv \lesssim  r^{2n-2-i} \lv A_2(\phi^4\J[\phi]^{-\gamma-1}) \rv.
\]
Using Lemma~\ref{L:PHILEMMA} now 
with $c=2-\frac{2}{n} (<\gamma+1)$, 
 estimates~\eqref{E:GZEROZERO1}--\eqref{E:GZEROZERO3} follow easily. If $\ell_2=i$ and additionally $i\le N-1$ we may still run the same argument.
If however $\ell_2=N$, we lose $\sqrt\ve$ in~\eqref{E:GZEROZERO4} due to~\eqref{E:PHILEMMA4}.
\end{proof}

Since by~\eqref{E:GZEROONEDEF} for any $V\in\bar{\mathcal P}_i$ with $i\ge 1$ we have
\[
V(g^{00}) = V(\frac{M_g }{r} g^{01})
\]
by an analogous argument we have the following lemma: 


\begin{lemma}\label{L:GZEROONELEMMA} Recall $g^{01}$ defined in \eqref{E:GZEROONEDEF}. The following bounds hold: 
\begin{align}
\sum_{V\in\bar{\mathcal P}_1}\left(\lv V(\frac{g^{01}}{r}) \rv+\lv V(\frac{g^{01}}{ rg^{00}}) \rv\right) & \lesssim \ve \tau^{\delta -1-\frac 1n} 
\label{E:GZEROONE1}\\
\sum_{V\in\bar{\mathcal P}_i}\left(\left\| V(\frac{g^{01}}{r})\right\|_{\alpha-N+2i+2}+\left\|V(\frac{g^{01}}{rg^{00}})\right\|_{\alpha-N+2i+2}\right)&\lesssim \ve \tau^{\delta -1 - \frac in} (1+(E^N)^{\frac12}), \ \ 2\le i \le N-1, \label{E:GZEROONE2}\\
\sum_{V\in\bar{\mathcal P}_i}\left(\left\|w^{i}V(\frac{g^{01}}{r})\right\|_{\infty}+\left\|w^{i}V(\frac{g^{01}}{rg^{00}})\right\|_{\infty}\right)&\lesssim \ve \tau^{\delta - 1-\frac in} (1+(E^N)^{\frac12}),  \ \ 2\le i \le N-3, \label{E:GZEROONE3}\\
\sqrt\ve \sum_{V\in\bar{\mathcal P}_N}\left(\left\|V(\frac{g^{01}}{r})\right\|_{\alpha+N+2}+\left\|V(\frac{g^{01}}{rg^{00}})\right\|_{\alpha+N+2}\right) &\lesssim \ve \tau^{\delta -1- \frac in} (1+(E^N)^{\frac12})\label{E:GZEROONE4}
\end{align}
\end{lemma}



\section{Energy estimates} \label{S:ENERGYESTIMATES}

To facilitate our proof and carry out the energy estimates, for the remainder of this section we assume that 
$H$ be a solution to~\eqref{E:H2} on a time interval $[\kappa,T]$ for some $T\le1$, the a priori assumptions~\eqref{E:APRIORI} hold, and the following (rough) bootstrap condition is true:
\begin{align}\label{E:BOOTSTRAP}
S_\kappa^N(\tau) \le 1, \ \ \tau\in[\kappa,T].
\end{align}

\subsection{Estimates for $\mathscr L_{\text{low}}$-terms}\label{SS:LLOWESTIMATES}

The goal of this section is the following proposition.

\begin{proposition}\label{P:LLOW}
Let $H$ be a solution to~\eqref{E:H2} on a time interval $[\kappa,T]$ for some $T\le1$ and assume that the a priori assumptions~\eqref{E:APRIORI} and the bootstrap assumption~\eqref{E:BOOTSTRAP} hold. Then for any $( \tau, r)\in[\kappa,T]\times[0,1]$ 
the following bound holds:
\begin{align}
\ve  \tau^{\frac12(\gamma-\frac23)} \| \D_i\left(\frac1{g^{00}}\mathscr L_{\text{low}} H\right)\|_{\alpha+i} 
\lesssim \sqrt \ve \tau^{\delta^\ast} (D^N)^{\frac12} + \sqrt \ve \tau^{\min\{\delta^\ast,\frac\delta2\}-\frac12} (E^N)^{\frac12}  , \ \ i=0,1,\dots, N.
\end{align}
\end{proposition}

\subsubsection{Decomposition of $ \mathscr L_{\text{low}} H $}

We rewrite the linear operator $\mathscr L_{\text{low}}$ in the form
\be
\mathscr L_{\text{low}} H = \mathscr L_{\text{low}}^1 H + \mathscr L_{\text{low}}^2 H 
\ee
where 
\begin{align}
\mathscr L_{\text{low}}^1 H : =
  &  2 \frac{P[\phia]}{\phia} H  -\gamma w\frac{\phi^2}{g^2\J[\phi]^{\gamma+1}  r^2} \DL(3\phia^2 + 2\phia \DL\phia ) H  \notag \\
&+ \gamma(\gamma+1) w\frac{\phi^4M_g  \DL\J[\phia] }{g^2\J[\phia]^{\gamma+2}  r^2}   \left[  \pa_\tau H  + \frac{m}{\tau } H \right]   \notag \\
&+ \gamma(\gamma+1) w\frac{\phi^2 \DL\J[\phia] }{g^2\J[\phia]^{\gamma+2}  r^2 } (3\phia^2  + 2\phia \DL\phia ) H \notag \\
&-\gamma (1+\alpha) r w' \frac{\phi^4 M_g }{g^2\J[\phia]^{\gamma+1}  r^2}  \left[  \pa_\tau H  + \frac{m}{\tau } H \right] \notag 
\\& - 2 \gamma (1+\alpha) r w' \frac{\phi^2 \phia \DL\phia }{g^2\J[\phia]^{\gamma+1}  r^2} H \notag \\
&  - \gamma w c[\phi] \frac{M_g ^2}{ r^2} \left[ \frac{2m}{\tau}  \pa_\tau H  + \frac{m(m-1)}{\tau^2} H \right] \notag \\
  &   -\gamma w\frac{\phi^2}{g^2\J[\phi]^{\gamma+1}  r^2} \big[ (\DL(\phi^2M_g ) + \phi^2M_g  + 2\phi \DL\phia M_g ) \left[  \pa_\tau H  + \frac{m}{\tau } H \right]
   \label{E:LLOWONE}
\end{align}

\begin{align}\label{E:LLOWTWO}
\mathscr L_{\text{low}}^2 H : = 
  - 4 \gamma w\frac{\phi^3 \DL\phia }{g^2\J[\phi]^{\gamma+1}  r}   r \pa_r \left(
 \frac{H}{r}\right)- 2m \gamma w c[\phi] \frac{M_g }{ r\tau} \pa_r H   + \gamma(\gamma+1) w\frac{\phi^4 \DL\J[\phia] }{g^2\J[\phia]^{\gamma+2}  r }   r \pa_r \left(
 \frac{H}{r}\right) 
\end{align}


\begin{lemma}[Estimates for $\mathscr L_{\text{low}}^1$]\label{L:LLOW1LEMMA} 
Let $H$ be a solution to~\eqref{E:H2} on a time interval $[\kappa,T]$ for some $T\le1$ and assume that the a priori assumptions~\eqref{E:APRIORI} and the bootstrap assumption~\eqref{E:BOOTSTRAP} hold. Then 
\begin{align}
\ve \tau^{\frac12(\gamma-\frac23)} \| \mathcal D_i\left(\frac1{g^{00}} \mathscr L_{\text{low}}^1 H\right) \|_{\alpha+i} \lesssim \sqrt \ve \tau^{\delta^\ast} (D^N)^{\frac12}.
\end{align}
\end{lemma}


\begin{proof}
By the product rule~\eqref{E:product}
\begin{align}
\D_i\left(\frac{P[\phia]}{g^{00}\phia } H \right) = \sum_{A_1\in \bar{\mathcal P}_{\ell_1}, A_2\in \mathcal P_{\ell_2} \atop \ell_1+\ell_2 = i} c^{A_1A_2}_i A_1(\frac{P[\phia]}{g^{00}\phia } ) A_2 H.
\end{align}
We now use~\eqref{E:PQBOUND} and the $L^2$-embeddings~\eqref{L21} if $\ell_2\ge 3$ and otherwise~\eqref{Linfty1} to conclude
\begin{align}
\ve \tau^{\frac12(\gamma-\frac23)} \|A_1(\frac{P[\phia]}{\phia} ) A_2 H\|_{\alpha+i} \lesssim 
&\ve \tau^{\frac12(\gamma-\frac23)+\frac23-2\gamma +\frac12(\frac{14}3-\gamma) - \frac{(i+2)}{n}} (D^N)^{\frac12}  \notag \\
& \lesssim \ve \tau^{\delta^\ast} (D^N)^{\frac12},
\end{align}
where we have used the bound $w^{\alpha+i} \lesssim w^{\alpha+2i-N}\lesssim w^{\alpha+2\ell_2-N}$. 

We now focus on the second term in the first line of~\eqref{E:LLOWONE}.
\begin{align}
-\gamma \frac{w}{g^2 r^2g^{00}}\frac{\phi^2}{\J[\phi]^{\gamma+1} } \DL(3\phia^2 + 2\phia \DL\phia ) H = -\gamma \frac{w}{g^{00}g^2 r^2} \phi^{-2\gamma}(\phi+\DL\phi)^{-(\gamma+1)} \DL(3\phia^2 + 2\phia \DL\phia ) H.
\end{align}
By the product rule~\eqref{E:product}  
\begin{align}
&\ve \tau^{\frac12(\gamma-\frac23)} \lv \mathcal D_i\left(\frac{w}{g^2 r^2g^{00}}\frac{\phi^2}{\J[\phi]^{\gamma+1} } \DL(3\phia^2 + 2\phia \DL\phia ) H\right) \rv \notag \\
& \lesssim \sum_{A_{1,2,3}\in\bar{\mathcal P}_{\ell_1,\ell_2,\ell_3}, A_4\in \mathcal P_{\ell_4} \atop \ell_1+\dots+\ell_4=i} \notag \\
& \ve \tau^{\frac12(\gamma-\frac23)} \lv A_1\left(\frac{w}{g^2 r^2g^{00}}\right) \rv \lv A_2\left(\DL(3\phia^2 + 2\phia \DL\phia )\right)\rv \lv A_3\left(\frac{\phi^2}{\J[\phi]^{\gamma+1}}\right)\rv  \lv A_4 H\rv, \label{E:LLOWTEST}
\end{align}
since $\lv A_1\left(\frac{w}{g^2 r^2g^{00}}\right) \rv \lesssim r^{-2-\ell_1}$. 
Consider first  
{\em Case I. $\ell_3\le i-1$.}
By~\eqref{E:QBOUND2} the third line of~\eqref{E:LLOWTEST} is bounded by 
\begin{align}
\ve \tau^{\frac43+\frac12(\gamma-\frac23)-\frac{\ell_1+\ell_2+2}{n}}
\left(p_{1,-\frac{\ell_1+\ell_2+2}{n}}\et +p_{\l,-\frac{\ell_1+\ell_2+4}{n}}\et \right) 
q_2\et \lv A_3\left(\frac{\phi^2}{\J[\phi]^{\gamma+1}}\right)\rv  \lv A_4 H\rv.
\end{align}
We now distinguish two cases.

\noindent
{\em Case I-1. $\ell_3\ge\ell_4$.}
If $\ell_3\le 1$ by Lemma~\ref{L:PHILEMMA} and~\eqref{Linfty1} we then have 
\begin{align}
q_2\et \lv A_3\left(\frac{\phi^2}{\J[\phi]^{\gamma+1}}\right)\rv  \lv A_4 H\rv
 \lesssim \tau^{-\frac23-2\gamma-\frac{\ell_3}{n}+\frac12(\frac{14}3-\gamma)} (D^N)^{\frac12}. \notag
 \end{align}
 Therefore, 
 since $w^{\alpha+i}\lesssim1$, by~\eqref{E:DELTADEF}--\eqref{E:DELTASTARDEF}, 
\begin{align}
&\| \ve \tau^{\frac43+\frac12(\gamma-\frac23)-\frac{\ell_1+\ell_2+2}{n}}
\left(p_{1,-\frac{\ell_1+\ell_2+2}{n}}\et +p_{\l,-\frac{\ell_1+\ell_2+4}{n}}\et \right) 
 q_2\et   A_3\left(\frac{\phi^2}{\J[\phi]^{\gamma+1}}\right) A_4 H \|_{\alpha+i} 
\notag \\
&\lesssim \ve \tau^{\delta-\frac{\ell_1+\ell_2+\ell_3+2}{n}}(D^N)^{\frac12}
 \lesssim \ve \tau^{\delta^*}(D^N)^{\frac12}. \label{E:CASEI0}
\end{align}
If $2\le \ell_3\le i-1$ then in case $\ell_4\ge3$ 
\begin{align}
&\ve \tau^{\frac43+\frac12(\gamma-\frac23)-\frac{\ell_1+\ell_2+2}{n}}
\left\|
 \left(p_{1,-\frac{\ell_1+\ell_2+2}{n}}\et +p_{\l,-\frac{\ell_1+\ell_2+4}{n}}\et \right)
 q_2\et A_3\left(\frac{\phi^2}{\J[\phi]^{\gamma+1}}\right)   A_4 H \right\|_{\alpha+i} \notag \\
& = \ve \tau^{\frac43+\frac12(\gamma-\frac23)-\frac{\ell_1+\ell_2+2}{n}}
\Big\|w^{\frac{N+i-2\ell_3-2\ell_4+2}{2}}
\left(p_{1,-\frac{\ell_1+\ell_2+2}{n}}\et +p_{\l,-\frac{\ell_1+\ell_2+4}{n}}\et \right)
 q_2\et  \notag \\
 & \ \ \ \ w^{\frac{\alpha-N+2\ell_3+2}{2}} A_3\left(\frac{\phi^2}{\J[\phi]^{\gamma+1}}\right)  w^{\ell_4-2} A_4 H \Big\|_{L^2} \notag \\
& \lesssim  \ve \tau^{\frac43+\frac12(\gamma-\frac23)-\frac{\ell_1+\ell_2+2}{n}}\|w^{\frac{N+i-2\ell_3-2\ell_4+2}{2}}\|_{L^\infty} \notag \\
& \ \ \ \ 
\|q_2\et w^{\frac{\alpha-N+2\ell_3+2}{2}} A_3\left(\frac{\phi^2}{\J[\phi]^{\gamma+1}}\right)\|_{L^2} \|w^{\ell_4-2} A_4 H \|_{L^\infty} 
\label{E:CASEI1}
\end{align}
Recall that the total derivative number $N$ is defined in~\eqref{E:NDEF}. 
Since  $N+i-2\ell_3-2\ell_4+2 \ge N+i-2i+2= N-i+2\ge0$ the $L^\infty$-norm of $w^{\frac{N+i-2\ell_3-2\ell_4+2}{2}}$ is bounded. Moreover, by Lemma~\ref{L:PHILEMMA}
$\|q_2\et w^{\frac{\alpha-N+2\ell_3+2}{2}} A_3\left(\frac{\phi^2}{\J[\phi]^{\gamma+1}}\right)\|_{L^2 } \lesssim \tau^{-\frac23-2\gamma-\frac{\ell_3}{n}}$ and by~\eqref{Linfty2}
$\|w^{\ell_4-2} A_4 H \|_{L^\infty }\lesssim \tau^{\frac12(\frac{14}3-\gamma)}(D^N)^{\frac12}$. Plugging this into~\eqref{E:CASEI1} we obtain the upper bound $ \ve \tau^{\delta^*}(D^N)^{\frac12}$
just like in~\eqref{E:CASEI0}.
If on the other hand $\ell_4\le 2$, we replace the $L^\infty$-bound of $w^{\ell_4-2} A_4 H$ by an $L^\infty$-bound on $A_4H$ provided by~\eqref{Linfty1}. 
This allows us to estimate the first line of~\eqref{E:CASEI0} by
\begin{align*}
& \ve \tau^{\frac43+\frac12(\gamma-\frac23)-\frac{\ell_1+\ell_2+2}{n}}
\|w^{\frac{N+i-2\ell_3-2}{2}}\|_{\infty}\|q_2\et w^{\frac{\alpha-N+2\ell_3+2}{2}} A_3\left(\frac{\phi^2}{\J[\phi]^{\gamma+1}}\right)\|_{L^2 }  \| A_4 H \|_{\infty}  \\
& \lesssim \ve \tau^{\delta^*}(D^N)^{\frac12},
\end{align*}
where we have used $N+i-2\ell_3-2\ge N+i-2(i-1)-2=0$, Lemma~\ref{L:PHILEMMA}, and~\eqref{Linfty1}.

\noindent
{\em Case I-2. $\ell_3<\ell_4$.}
If $\ell_4\le2$ then we are in the regime that has already been discussed above. Assume $\ell_4\ge3$. If $\ell_3\ge2$ we use~\eqref{E:PHILEMMA3} and~\eqref{L21} to obtain
\begin{align}
& \ve \tau^{\frac43+\frac12(\gamma-\frac23)-\frac{\ell_1+\ell_2+2}{n}}  \| w^{\frac{\alpha+i}{2}} p_{1,-\frac{\ell_1+\ell_2+2}{n}}\et q_2\et  A_3\left(\frac{\phi^2}{\J[\phi]^{\gamma+1}}\right)  \lv A_4 H\rv\|_{L^2 } \notag \\
& \lesssim \ve \tau^{\frac43+\frac12(\gamma-\frac23)-\frac{\ell_1+\ell_2+2}{n}}   \| w^{\frac{i-2(\ell_3+\ell_4)+N}{2}}\|_{\infty} \|q_2\et  w^{\ell_3} A_3\left(\frac{\phi^2}{\J[\phi]^{\gamma+1}}\right)\|_{\infty}
\|w^{\frac{\alpha+2\ell_4-N}{2}}A_4H\|_{L^2 } \notag \\
& \lesssim \ve \tau^{\frac43+\frac12(\gamma-\frac23)-\frac{\ell_1+\ell_2+2}{n}-\frac23-2\gamma-\frac{\ell_3}{n}+\frac12(\frac{14}3-\gamma)} (D^N)^{\frac12} \notag \\
& \lesssim \ve \tau^{\delta^*} (D^N)^{\frac12}.
\end{align}
We have used the inequality $i-2(\ell_3+\ell_4)+N\ge N-i\ge0$. The case $\ell_3\le2$ is handled similarly, with~\eqref{E:PHILEMMA1} instead of~\eqref{E:PHILEMMA3}.

\noindent
{\em Case II. $\ell_3=i$.}
In this case we need to bound 
\be\label{E:VCASE2}
\ve \tau^{\frac12(\gamma-\frac23)}  \|\frac{1}{g^2 r^2g^{00}} V\left(\frac{\phi^2}{\J[\phi]^{\gamma+1}}\right) \DL(3\phia^2 + 2\phia \DL\phia ) H \|_{\alpha+i+2}
\ee
with $V\in \bar{\mathcal P}_i$. If $i\in\{0,1\}$ we can use~\eqref{E:PHILEMMA1} and if $2\le i\le N-1$ we may use~\eqref{E:PHILEMMA2} to bound $\|V\left(\frac{\phi^2}{\J[\phi]^{\gamma+1}}\right)\|_{L^2 }$
and $\|w^{\frac{\alpha-N+2i+2}{2}}V\left(\frac{\phi^2}{\J[\phi]^{\gamma+1}}\right)\|_{L^2 }$ respectively. The remaining terms are estimated in $L^\infty$ and we conclude that the expression in~\eqref{E:VCASE2} 
is bounded by $ \ve \tau^{\delta^*} (D^N)^{\frac12}$ just like above. If however $i = N$ we must use~\eqref{E:PHILEMMA4}. It then follows that the expression in~\eqref{E:VCASE2} is bounded by
\begin{align}
&\sqrt\ve \tau^{\frac12(\gamma-\frac23)}  \sqrt\ve \|q_2\et w^{\frac{\alpha+N+2}2}V\left(\frac{\phi^2}{\J[\phi]^{\gamma+1}}\right)\|_{L^2 }\|q_{-2}\et \DL(3\phia^2 + 2\phia \DL\phia )\|_{L^\infty} \|H\|_{\infty} \notag \\
& \lesssim \sqrt\ve  \tau^{\frac12(\gamma-\frac23)-\frac23-2\gamma-\frac{N+2}{n}+\frac43+ \frac12(\frac{14}3-\gamma)}  (D^N)^{\frac12} \notag \\
& \lesssim \sqrt\ve \tau^{\delta^*} (D^N)^{\frac12}. \label{E:SAMEBOUND}
\end{align} 

The 3rd-7th 
 term 
 in~\eqref{E:LLOWONE} are estimated analogously. 
Note that the terms $\pa_\tau H$ and $\frac{H}{\tau}$ and similarly $\frac{\pa_\tau H}{\tau}$ and
$\frac{H}{\tau^2}$ are on equal footing from the energy stand point or more precisely
\begin{align*}
\tau^{\gamma-\frac53} \left(\|\pa_\tau H\|_{\alpha+j}^2 + \|\frac H\tau\|_{\alpha+j}^2\right) & \lesssim E^N \\
\tau^{\gamma-\frac83} \left(\|\pa_\tau H\|_{\alpha+j}^2 + \|\frac H\tau\|_{\alpha+j}^2\right) & \lesssim D^N,
\end{align*}
where we recall the definitions~\eqref{E:ENDEF}--\eqref{E:DNDEF} of $E^N$ and $D^N$.
In particular, the estimates for the 3rd, 5th, and the 7th term in~\eqref{E:LLOWONE} are very similar and we sketch the details for the 7th (next-to-last) term.
By the product rule~\eqref{E:product} we have
\begin{align}
\D_i \left(w c[\phi] \frac{M_g ^2}{ g^2r^2} \frac{H_\tau}{\tau} \right)
= \sum_{A_{1,2}\in\bar{\mathcal P}_{\ell_1,\ell_2}, A_3 \in\mathcal P_{\ell_3} \atop \ell_1+\ell_2+\ell_3=i}
\tau^{-1}A_1\left(\frac {M_g ^2 w}{g^2r^2}\right) A_2 \left(\frac{\phi^4}{\J[\phi]^{\gamma+1}}\right) A_3 H_\tau
\end{align}
A case-by-case analysis analogous to the one above, Lemma~\ref{L:PHILEMMA}, and Lemmas~\ref{L:L2WEIGHTED}--\ref{L:LINFTYWEIGHTED} yield
\begin{align}
\tau^{\frac12(\gamma-\frac23)} \|\D_i\left(\left(w c[\phi] \frac{M_g ^2}{ g^2r^2} \frac{H_\tau}{\tau}\right)\right)\|_{\alpha+i}
&\lesssim \| r^{2n-2}q_{-\gamma-1}\et \|_{L^\infty} \tau^{\frac12(\gamma-\frac23)} \tau^{\frac23 -2\gamma +\frac 12(\frac 23-\gamma)} (D^N)^{\frac12} \notag \\
& \lesssim \tau^{\frac 83-2\gamma-\frac 2n} \left\| \frac{\et^{2-\frac 2n}}{(1+\et)^{\gamma+1}}\right\|_\infty (D^N)^{\frac12} \notag \\
& \lesssim \tau^\delta \left\|p_{2,-\frac 2n} \et\right\|_{\infty} (D^N)^{\frac12} \lesssim \tau^\delta (D^N)^{\frac12}
\end{align}
The same bound, with $\frac{H_\tau}{\tau}$ replaced by $\frac{H}{\tau^2}$ follows analogously.

The 4-th and the 6-th term on the right-hand side of~\eqref{E:LLOWONE} are easier to bound. In the 6-th term the factor $w'$ gives a regularising power of $r$ near the center $r=0$
due to the bound $|\pa_r^a w'|\le r^{n-a-1}$ (which in turn follows from~\eqref{E:FM}). Similarly, the presence of $ \DL\J[\phia]$ in the 4-th term, by virtue of~\eqref{E:QBOUND3} affords a power of $\et$ in our estimates, which again counteracts any potential singularities coming from the negative powers of $r$ near $r=0$. Routine application of Lemmas~\ref{L:L2WEIGHTED}--\ref{L:LINFTYWEIGHTED} and Lemma~\ref{L:PHILEMMA} yields the desired bound.

To estimate the last line in~\eqref{E:LLOWONE} we first observe that
\[
 (\DL(\phi^2M_g ) + \phi^2M_g  + 2\phi \DL\phia M_g ) = \phi^2\left( \DL M_g +M_g  \right) +2\phi (2\DL\phia M_g  + \DL\phi M_g ).
\]
Therefore
\begin{align*}
& -\gamma w\frac{\phi^2}{g^2\J[\phi]^{\gamma+1}  r^2} \big[ (\DL(\phi^2M_g ) + \phi^2M_g  + 2\phi \DL\phia M_g ) \\
& = -\gamma w\frac{\phi^4}{g^2\J[\phi]^{\gamma+1}  r^2}   \left( \DL M_g +M_g  \right)  -  2\gamma w\frac{\phi^3}{g^2\J[\phi]^{\gamma+1}  r^2} (2\DL\phia M_g  + \DL\phi M_g ).
\end{align*}
We can therefore break up the last line of~\eqref{E:LLOWONE} into a sum of terms that are of similar structure as the ones showing up above, and thus the estimate follows analogously and thus obtain the same bound as in~\eqref{E:SAMEBOUND}.
\end{proof}


\begin{lemma}[Estimates for $\mathscr L_{\text{low}}^2$]\label{L:LLOW2LEMMA} 
Let $H$ be a solution to~\eqref{E:H2} on a time interval $[\kappa,T]$ for some $T\le1$ and assume that the a priori assumptions~\eqref{E:APRIORI} and the bootstrap assumption~\eqref{E:BOOTSTRAP} hold. Then 
\begin{align}
\ve \tau^{\frac12(\gamma-\frac23)} \| \mathcal D_i \left(\frac1{g^{00}}\mathscr L_{\text{low}}^2 H\right) \|_{\alpha+i} \lesssim \sqrt \ve \tau^{\min\{\delta^\ast,\frac\delta2\}-\frac12} (E^N)^{\frac12}.
\end{align}
\end{lemma}


\begin{proof}
We focus on the first and the most complicated term on the right-hand side of~\eqref{E:LLOWTWO}.
Recall that $\r\left( \frac{H}{r}\right) = D_r H- 3\frac H r$.
By analogy to~\eqref{E:LLOWTEST} we have
\begin{align}
&\ve \tau^{\frac12(\gamma-\frac23)} \lv \mathcal D_i\left(\frac{w}{g^2 r g^{00}}\frac{\phi^3}{\J[\phi]^{\gamma+1} } \DL\phia (D_r H- 3\frac H r)\right) \rv \notag \\
& \lesssim \sum_{A_{2,3,4}\in\bar{\mathcal P}_{\ell_2,\ell_3,\ell_4}, A_1\in \mathcal P_{\ell_1} \atop \ell_1+\dots+\ell_4=i}  
 \ve \tau^{\frac12(\gamma-\frac23)} \lv A_1\left(\frac{w}{g^2 r g^{00}}\right) \rv \lv A_2 \DL\phia \rv \lv A_3\left(\frac{\phi^3}{\J[\phi]^{\gamma+1}}\right)\rv  \lv A_4 (D_r H- 3\frac H r)\rv \label{E:LLOWTEST2}
\end{align}

{\em Case I-1. $\ell_3=i$.}
In this case $\ell_1=\ell_2=\ell_4=0$ and we note that 
\be\label{E:NOTETHAT}
\lv \frac{w}{g^2 r g^{00}} \rv \lesssim w \et^{-\frac1n} \tau^{-\frac1n}, \ \ 
\lv \DL\phia\rv \lesssim 
\tau^{\frac23} q_1\et \left(p_{1,0}\et + p_{\l,-\frac2n}\et\right), 
\ee
where we we have used~\eqref{E:BARDIDQ}.  
Therefore we bound the $\|\cdot\|_{\alpha+i}$ norm of the last line of~\eqref{E:LLOWTEST2} by
\begin{align}
\ve \tau^{\frac12(\gamma-\frac23)-\frac1n+\frac23+\frac12(\frac{11}{3}-\gamma)} 
\left(p_{1,0}\et + p_{\l,-\frac2n}\et\right) 
 \|q_1\et w^{\frac{\alpha+i+2}{2}}A_3\left(\frac{\phi^3}{\J[\phi]^{\gamma+1}}\right)\|_{L^2 } (E^N)^{\frac12}
\label{E:INBETWEEN1}
\end{align}
If $i=N$ by~\eqref{E:PHILEMMA4} we have
\begin{align*}
\sqrt \ve \|q_1\et w^{\frac{\alpha+N+2}{2}}A_3\left(\frac{\phi^3}{\J[\phi]^{\gamma+1}}\right)\|_{L^2 } \lesssim \tau^{-2\gamma-\frac{N}{n}}.
\end{align*}
Since $\frac12(\gamma-\frac23)-\frac1n+\frac23+\frac12(\frac{11}{3}-\gamma)-2\gamma = \delta^*+\frac1n-\frac12$ 
(recall~\eqref{E:DELTASTARDEF}), and $\left(p_{1,0}\et + p_{\l,-\frac2n}\et\right)\lesssim 1$, 
it follows that 
~\eqref{E:INBETWEEN1} is bounded by
\begin{align}\label{E:THEBOUND}
\sqrt\ve \tau^{\delta^\ast-\frac12}(E^N)^{\frac12}
\end{align}
as needed.
If $2\le i \le N-1$ we use~\eqref{E:PHILEMMA2} instead and if $i=1$ we use~\eqref{E:PHILEMMA1} instead, to bound~\eqref{E:INBETWEEN1} by $\ve \tau^{\delta^\ast-\frac12}(E^N)^{\frac12}$.

\noindent
{\em Case I-2. $\ell_4=i$.}
In this case $\ell_1=\ell_2=\ell_3=0$ and $A_4=\bar{\D}_i$. Using~\eqref{E:NOTETHAT} and~\eqref{E:PHILEMMA1} we can bound the last line of~\eqref{E:LLOWTEST2} by
\begin{align}
&\ve \tau^{\frac12(\gamma-\frac23)-\frac1n+\frac23} p_{1,0}\et \| w^{\frac{\alpha+i+2}{2}} q_{-\frac{\gamma+1}{2}}\et A_4(D_r H- 3\frac H r)\|_{L^2}
\| q_{\frac{3+\gamma}{2}}\et\left(\frac{\phi^3}{\J[\phi]^{\gamma+1}}\right)\|_{\infty} \notag \\
&\lesssim \ve \tau^{\frac12(\gamma-\frac23)-\frac1n+\frac23-2\gamma}  \| q_{-\frac{\gamma+1}{2}}\et w^{\frac{\alpha+i+2}{2}}A_4(D_r H- 3\frac H r)\|_{L^2}.
\label{E:INBETWEEN2}
\end{align}
If $i=N$ we have $\bar{\D}_ND_r = \D_{N+1}$. Therefore, by~\eqref{L22}, 
$w^{\alpha+N+2}\lesssim w^{\alpha+N+1}$, and $q_{-\frac{\gamma+1}{2}}\et\lesssim 1$, 
\[
\sqrt \ve\| q_{-\frac{\gamma+1}{2}} w^{\frac{\alpha+N+2}{2}}\bar{\D}_ND_r H\|_{L^2} 
\lesssim  \tau^{\frac{\gamma+1}{2}} (E^N)^{\frac12}.
\]
On the other hand, using~\eqref{E:TOPORDERBOUND} we additionally have
\begin{align}
\sqrt \ve\| q_{-\frac{\gamma+1}{2}}\et w^{\frac{\alpha+N+2}{2}}{\bar \D}_i\left(\frac H r\right)\|_{L^2}
\lesssim \sqrt \ve \| q_{-\frac{\gamma+1}{2}}\et w^{\frac{\alpha+N+2}{2}}\D_{N+1} X\|_{L^2} 
\lesssim \tau^{\frac12(\gamma+1)}(E_N)^{\frac12}. \notag 
\end{align}
Plugging the last bounds into the last line of~\eqref{E:INBETWEEN2} 
and recalling~\eqref{E:DELTADEF} 
we bound it by 
\begin{align}
\sqrt \ve\tau^{\frac12(\gamma-\frac23)-\frac1n+\frac23-2\gamma+\frac{\gamma+1}{2}} (E^N)^{\frac12} =\sqrt\ve \tau^{\frac{\delta-1}{2}}(E^N)^{\frac12}.
\end{align}
If $2\le i \le N-1$ we use~\eqref{L21} instead of~\eqref{L22} above and obtain the upper bound $\ve \tau^{\delta^\ast-\frac12}(E^N)^{\frac12}$. Similarly, if $i\le2$ we may use~\eqref{Linfty1} instead.

\noindent
{\em Case II. $\ell_3,\ell_4\le i-1$.}
Recalling~\eqref{E:BARDIDQ} and the bound $|A_{1}(\frac{w}{g^{2}rg^{00}})|\lesssim r^{-1-\ell_{1}}$, 
by~\eqref{E:NOTETHAT} we have 
\begin{align}
& \ve \tau^{\frac12(\gamma-\frac23)} \lv A_1\left(\frac{w}{g^2 rg^{00}}\right) \rv \lv A_2\DL\phia\rv \lv A_3\left(\frac{\phi^3}{\J[\phi]^{\gamma+1}}\right)\rv  \lv A_4 (D_r H- 3\frac H r)\rv  \notag \\
& \lesssim \ve  \tau^{\frac12(\gamma-\frac23)+\frac23 - \frac{\ell_1+\ell_2+1}{n} }  
 \left(p_{1,-\frac{\ell_2}{n}}\et + p_{\l,-\frac{\ell_2+2}n}\et\right)
 \lv q_1\et A_3\left(\frac{\phi^3}{\J[\phi]^{\gamma+1}}\right)\rv  \lv A_4 (D_r H- 3\frac H r)\rv
\label{E:LLOWTEST3}
\end{align}

\noindent
{\em Case II-1. $\ell_3\le\ell_4\le i-1$.}
If $\ell_4\le1$ 
and therefore $\ell_3\le1$,
we can estimate the $\|\cdot\|_{\alpha+i}$-norm of the last line of~\eqref{E:LLOWTEST3}  using~\eqref{E:PHILEMMA1} and~\eqref{Linfty11} by
\begin{align}
& \ve  \tau^{\frac12(\gamma-\frac23)+\frac23 - \frac{\ell_1+\ell_2+1}{n} }  \|q_1\et A_3\left(\frac{\phi^3}{\J[\phi]^{\gamma+1}}\right)\|_{\infty}  
\| A_4 (D_r H- 3\frac H r)\|_{\infty} \notag \\
&\lesssim \ve \tau^{ \frac12(\gamma-\frac23)-\frac{\ell_1+\ell_2+\ell_3+1}{n}+\frac23+\frac12(\frac{11}{3}-\gamma)-2\gamma }(E^N)^{\frac12} \notag \\
&\lesssim \ve \tau^{\delta^\ast-\frac12}(E^N)^{\frac12}
\end{align}

If $2\le \ell_4\le i-1$, assume first that $\ell_3\ge2$. We rely on~\eqref{L21} and~\eqref{E:PHILEMMA3} to bound the $\|\cdot\|_{\alpha+i}$-norm of the last line of~\eqref{E:LLOWTEST3} by
\begin{align}
&\ve \tau^{\frac12(\gamma-\frac23)+\frac23-\frac{\ell_1+\ell_2+\ell_3+1}{n}} \|w^{\frac{i+N-2(\ell_3+\ell_4+1)}{2}}\|_{\infty}
\|w^{\ell_3}q_1\et A_3\left(\frac{\phi^3}{\J[\phi]^{\gamma+1}}\right)\|_{\infty} \notag \\
& \|w^{\frac{\alpha+2(\ell_4+1)-N}{2}}A_4(D_r H- 3\frac H r)\|_{L^2} \notag \\
&\lesssim \ve \tau^{\frac12(\gamma-\frac23)+\frac23-2\gamma+\frac12(\frac{11}3-\gamma)-\frac{\ell_1+\ell_2+\ell_3+1}{n}} (E^N)^{\frac12}
 \lesssim \ve \tau^{\delta^\ast-\frac12}(E^N)^{\frac12}, \label{E:OBTAINEDBOUND}
\end{align}
where we have used the bound $i+N-2(\ell_3+\ell_4+1)\ge0$, which is true if $\ell_3+\ell_4\le i-1$,  to bound $ \|w^{\frac{i+N-2(\ell_3+\ell_4+1)}{2}}\|_{\infty}$ by a constant. If on the other hand $\ell_3+\ell_4=i$, then $\ell_1=0$ and therefore we have an additional power of $w$ in our estimate which by the same idea as above allows us to obtain the bound~\eqref{E:OBTAINEDBOUND}.
 
If $\ell_3\le1$ we then use~\eqref{E:PHILEMMA1} instead of~\eqref{E:PHILEMMA3} and deduce the same bound analogously.

\noindent
{\em Case II-2. $\ell_4\le\ell_3\le i-1$.}
This case is handled analogously to the case II-1 above and relies on a similar case distinction ($\ell_4\ge2$ and $\ell_4\le1$) as well as Lemma~\ref{L:PHILEMMA} and
estimates~\eqref{Linfty1},~\eqref{L21}.

This completes the bound of the first term on the right-hand side of~\eqref{E:LLOWTWO}. The estimates for the remaining 2 terms proceed analogously. Note that we use~\eqref{E:DJQBOUND}
crucially to estimate the third term on the right-hand side of~\eqref{E:LLOWTWO}.
\end{proof}



\subsection{High order commutator estimates}

The goal of this section is the following proposition.


\begin{proposition}\label{P:COMMESTIMATES}
Let $H$ be a solution to~\eqref{E:H2} on a time interval $[\kappa,T]$ for some $T\le1$ and assume that the a priori assumptions~\eqref{E:APRIORI} holds. Then 
\begin{align}
  \tau^{\frac12(\gamma-\frac23)} \| \mathcal C_i[H] \|_{\alpha+i} 
\lesssim \sqrt \ve \tau^{\delta^\ast} (D^N)^{\frac12} + \sqrt \ve \tau^{\frac\delta2-\frac12} (E^N)^{\frac12}  , \ \ i=1,\dots, N.
\end{align}
\end{proposition}


\begin{lemma}[The commutator estimates] \label{L:COMMUTATORLEMMA}
Let $H$ be a solution to~\eqref{E:H2} on a time interval $[\kappa,T]$ for some $T\le1$ and assume that the a priori assumptions~\eqref{E:APRIORI} holds.  
\begin{align}
\tau^{\frac12(\gamma-\frac23)} \left(\|\left[\D_i, \frac{1}{g^{00}} \right] \frac{\pa_\tau H }{\tau}\|_{\alpha+i} +
\| \left[\D_i,  \frac{g^{01}}{g^{00}} \pa_r\right]\pa_\tau H \|_{\alpha+i}
+ \|\left[\D_i, \frac{d^2}{g^{00}}\right] \frac{H}{\tau^2}\|_{\alpha+i} \right)
& \lesssim  \sqrt\ve \tau^{\delta^\ast} (D^N)^{\frac12} \label{E:COMMESTIMATES} \\
\ve \tau^{\frac12(\gamma-\frac23)}  \|\left[\bar\D_{i-1},   \frac{c[\phi]}{g^{00}}  \right] D_r L_\alpha H \|_{\alpha+i} \lesssim  
\sqrt\ve \tau^{\frac\delta2 -\frac12} (E^N)^{\frac12} + \ve \tau^{\delta^\ast} (D^N)^{\frac12}
\label{E:COMMESTIMATES2}\\
\ve \tau^{\frac12(\gamma-\frac23)} \| \frac{c[\phi]}{g^{00}}  \sum_{j=0}^{i-1} \zeta_{ij} \mathcal D_{i-j} H\|_{\alpha+i}  \lesssim 
\ve \tau^{\delta^\ast} (D^N)^{\frac12},
\label{E:COMMESTIMATES3}
\end{align}
where we remind the reader that the coefficients $\zeta_{ij}$, $i=1,\dots, N$, $j=0,\dots i-1$ are defined in Lemma~\ref{L:COMM1}. 
\end{lemma}


\begin{proof}
{\em Proof of~\eqref{E:COMMESTIMATES}.}
By~\eqref{E:DiXP} we have the formula
\begin{align}
\left[\D_i,  \frac{g^{01}}{g^{00}} \pa_r\right]\pa_\tau H
= i  \pa_r \left(\frac{g^{01}}{g^{00}}\right)\D_i H_\tau 
+ \sum_{A_1\in \bar{\mathcal P}_{\ell_1}, A_2\in {\mathcal P}_{\ell_2} \atop \ell_1+\ell_2 = i, \ \ell_1\ge1} c_i^{A_1A_2} A_1(\frac{g^{01}}{r g^{00}}) A_2 H_\tau \notag \\
 +\sum_{A_1 \in \bar{\mathcal P}_{\ell_1},  A_2\in \bar{\mathcal P}_{\ell_2} \atop \ell_1+\ell_2=i, \ \ell_1\ge2} {\bar c}_i^{A_1 A_2}  r A_1(\frac{g^{01}}{r g^{00}})  A_2 D_r H_\tau \label{E:COMMEXPLICIT}
\end{align}
Since $\lv\pa_r\left(\frac{g^{01}}{g^{00}} \right) \rv\lesssim \ve \tau^{\delta^\ast -1}$ by Lemma~\ref{L:GZEROONELEMMA}, 
the bound $w^{\alpha+i}\lesssim w^{\alpha+2i-N}$, and definitions~\eqref{E:DELTADEF}--\eqref{E:DELTASTARDEF} of $\delta$ and $\delta^\ast$, 
we have
\begin{align}
\tau^{\frac12(\gamma-\frac23)} \|w^{\frac{\alpha+i}{2}}\pa_r \left(\frac{g^{01}}{g^{00}}\right)\D_i H_\tau \|_{L^2 } 
& \lesssim \ve \tau^{\delta^\ast + \frac12(\gamma-\frac83)} \|w^{\frac{\alpha+i}{2}}\D_i H_\tau\|_{L^2 } \notag \\
& \lesssim \ve \tau^{\delta^\ast} (D^N)^{\frac12}.
\end{align}

In order to bound the second term on the right-hand side of~\eqref{E:COMMEXPLICIT}
we distinguish several cases by analogy to Lemma~\ref{L:LLOW1LEMMA}.

\noindent
{\em Case I: $\ell_1\le\ell_2$.}
If $\ell_2\le2$ 
and therefore $\ell_1\le2$, 
we can use~\eqref{Linfty1} and~\eqref{E:GZEROONE1} to obtain
\begin{align}
\tau^{\frac12(\gamma-\frac23)}\|w^{\frac{\alpha+i}{2}}A_1(\frac{g^{01}}{rg^{00}})  A_2 H_\tau\|_{L^2} & \lesssim \ve 
\tau^{\frac12(\gamma-\frac23)+\delta -1- \frac{\ell_1}{n}+\frac12(\frac83 -\gamma)}(D^N)^{\frac12} \notag \\
& \lesssim \ve \tau^{\delta^\ast} (D^N)^{\frac12}. 
\end{align}
If $3\le \ell_2\le N$ we again distinguish 2 cases. If $\ell_1\ge2$ we can use~\eqref{L21},~\eqref{E:GZEROONE3}, 
and~\eqref{E:DELTADEF}--\eqref{E:DELTASTARDEF} 
to obtain
\begin{align}
& \tau^{\frac12(\gamma-\frac23)}\|w^{\frac{\alpha+i}{2}}A_1(\frac{g^{01}}{g^{00} r})  A_2 H_\tau\|_{L^2}  \notag \\
&\lesssim \tau^{\frac12(\gamma-\frac23)}\|w^{\frac{i+N-2(\ell_1+\ell_2)}{2}}\|_{\infty}\|w^{\ell_1}A_1(\frac{g^{01}}{g^{00} r})\|_\infty\|w^{\frac{\alpha-N+2\ell_2}{2}}  A_2 H_\tau\|_{L^2} \notag \\
&\lesssim \ve \tau^{\delta^\ast} (D^N)^{\frac12}.
\end{align}
If $\ell_1=1$ we then use~\eqref{E:GZEROONE1} instead of~\eqref{E:GZEROONE3} and obtain the same conclusion.

\noindent
{\em Case II: $\ell_1\ge\ell_2$.}
In this case we proceed analogously and rely crucially on Lemmas~\ref{L:GZEROONELEMMA} and estimates~\eqref{Linfty1}--\eqref{Linfty2}.
The only nonstandard situation occurs when $\ell_1=N$. In that case $\ell_2=0$ and we must use the bound~\eqref{E:GZEROONE4} together with~\eqref{Linfty1}. We then obtain
\begin{align}
\tau^{\frac12(\gamma-\frac23)}\|w^{\frac{\alpha+N}{2}}A_1(\frac{g^{01}}{g^{00} r}) H_\tau\|_{L^2} 
&\lesssim \tau^{\frac12(\gamma-\frac23)}\|w^{\frac{\alpha+N}{2}}A_1(\frac{g^{01}}{g^{00} r})\|_{L^2} \|H_\tau\|_{L^\infty} \notag \\
& \lesssim  \tau^{\frac12(\gamma-\frac23)+\frac12(\frac83-\gamma)+\delta-1 - \frac{N}{n}} (D^N)^{\frac12} \notag \\
& \lesssim 
\sqrt\ve 
\tau^{\delta^\ast} (D^N)^{\frac12}. \label{E:DELTASTARVALUE?}
\end{align}

To estimate the last term on the right-hand side of~\eqref{E:COMMEXPLICIT} we note that for any $A_2\in\bar{\mathcal P}_{\ell_2}$, we have 
$A_2D_r \in \mathcal P_{\ell_2+1}$ and since $\ell_2\le i-2$ we are in the regime treated above. This concludes the proof of the bound for
$\| \left[\D_i,  \frac{g^{01}}{g^{00}} \pa_r\right]\pa_\tau H \|_{\alpha+i}$. The remaining 2 terms on the left-hand side of~\eqref{E:COMMESTIMATES}
are 
estimated analogously and their proofs rely 
crucially on Lemmas~\ref{L:GZEROZEROLEMMA} and~\ref{L:GZEROONELEMMA}.
The second term is less singular with respect to $\tau$ and the presence of the $g^{01}$ does not change the structure of the estimates due to Lemma~\ref{L:GZEROONELEMMA}. The third term contains the factor $\frac{H}{\tau^2}$ which, from the point of view of the energy, scales just like  $\frac{H_\tau}{\tau}$ and thus the 
structure of the estimates is similar to the above. 

\noindent
{\em Proof of~\eqref{E:COMMESTIMATES2}.}
From~\eqref{E:LBETADEF} we have
\be\label{E:LALPHAEXPANDED}
L_\alpha H = - w \D_2 H - (1+\alpha)w' D_r H.
\ee
By the commutator formula~\eqref{E:DYP} we have
\begin{align}
\left[\bar\D_{i-1},   \frac{c[\phi]}{g^{00}}  \right] D_r L_\alpha H
=&(i-1)\pa_r \left(\frac{c[\phi]}{g^{00}} \right) {\bar\D}_{i-2} D_r L_\alpha H  \notag \\
& +  \sum_{A_{1,2}\in \bar{\mathcal P}_{\ell_1,\ell_2}, \, \ell_1+\ell_2 =i -1\atop \ell_1\ge2} \bar c_i^{A_1A_2} A_1\left(\frac{c[\phi]}{g^{00}}\right) A_2D_r L_\alpha H.
\label{E:COMMAUX1}
\end{align}
The second sum on the right-hand side of~\eqref{E:COMMAUX1} can be estimated analogously to the estimates for~\eqref{E:COMMESTIMATES} above, using~\eqref{E:LALPHAEXPANDED}. Thereby we observe that the total number of derivatives in the operator $A_2D_r L_\alpha$ is at most $i$, since $\ell_2\le i-3$.
We next focus on the first term on the right-hand side of~\eqref{E:COMMAUX1}. Since  ${\bar\D}_{i-2} D_r = \D_{i-1}$, using~\eqref{E:LALPHAEXPANDED} we can write it as
\begin{align}\label{E:COMMAUX2}
-(i-1)\pa_r \left(\frac{c[\phi]}{g^{00}} \right) \D_{i-1}\left(w \D_2 H + (1+\alpha)w' D_r H\right).
\end{align}
By the product rule~\eqref{E:product} we can isolate the top-order term 
\begin{align*}
 \D_{i-1}\left(w \D_2 H + (1+\alpha)w' D_r H\right) = w \D_{i+1}H 
 +\sum_{A_1\in\bar{\mathcal P}_{\ell_1}, A_2\in\mathcal P_{\ell_2} \atop \ell_1+\ell_2=i-1,\, \ell_1\ge1} A_1 w A_2\D_2H + (1+\alpha)\D_{i-1}\left(w' D_r H\right)
 \end{align*}
We now use~\eqref{E:PHILEMMA1}, 
\eqref{L22} 
and conclude, in the case $i=N$
\begin{align}
&\ve \tau^{\frac12(\gamma-\frac23)} \|w^{\frac{\alpha+N}{2}}\pa_r \left(\frac{c[\phi]}{g^{00}}\right) w \D_{N+1}H \|_{L^2}   \notag \\
& \lesssim \ve \tau^{\frac12(\gamma-\frac23)}\|q_{\frac{\gamma+1}{2}} \pa_r \left(\frac{c[\phi]}{g^{00}}\right)\|_{\infty} 
\|q_{-\frac{\gamma+1}{2}}w^{\frac{\alpha+N+2}{2}} \D_{N+1}H \|_{L^2} \notag \\
&\lesssim  \sqrt\ve \tau^{\frac12(\gamma-\frac23)+\frac23-2\gamma -\frac1n+\frac12(\gamma+1)}(E^N)^{\frac12} \notag \\
& = \sqrt\ve \tau^{\frac\delta2 -\frac12} (E^N)^{\frac12},
\end{align}
where the estimate~\eqref{E:PHILEMMA1} has been used in the third line.
When using~\eqref{E:PHILEMMA1}, we first recall~\eqref{E:CDEF} and use the product rule to write 
\[
\pa_r \left(\frac{c[\phi]}{g^{00}}\right) = \pa_r\left(\frac{\phi^4}{\J[\phi]^{\gamma+1}}\right) \frac{1}{g^2 g^{00}} + \frac{\phi^4}{\J[\phi]^{\gamma+1}}\pa_r\left( \frac{1}{g^2 g^{00}}\right).
\]
We note that by Lemma~\ref{L:GZEROZEROLEMMA} and~\eqref{E:FM} we have $\lv \frac{1}{g^2 g^{00}}  \rv + \lv\pa_r\left( \frac{1}{g^2 g^{00}}\right) \rv\lesssim 1$ and therefore~\eqref{E:PHILEMMA1} yields the third line above.
The remaining below-top-order terms can be estimated analogously to~\eqref{E:COMMESTIMATES} to finally obtain~\eqref{E:COMMESTIMATES2}.

\noindent
{\em Proof of~\eqref{E:COMMESTIMATES3}.}
Since $\left\vert \pa_r^kw\right\vert\lesssim  r^{n-k}$ for any $k\in\{1,\dots,n\}$ it follows from~\eqref{pij} $|\zeta_{ij}|\lesssim  r^{n-j-2}$.
Therefore by~\eqref{E:PHILEMMA1} for any $j\le i$ we have 
\begin{align}
\lv \frac{c[\phi]}{g^{00}}\zeta_{ij} \rv & \lesssim \tau^{\frac23-2\gamma} q_{-\gamma-1}\et  r^{n-j-2} \notag \\
& \lesssim  \tau^{\frac53-2\gamma - \frac{j+2}{n}} q_{-\gamma-1}\et \et^{1-\frac{j+2}{n}} \notag \\
& \lesssim \tau^{\delta^\ast -1}. \notag
\end{align}

Therefore for any $j\ge3$ we have
\begin{align}
\ve \tau^{\frac12(\gamma-\frac23)} \| \frac{c[\phi]}{g^{00}} \zeta_{ij} \mathcal D_{i-j} H\|_{\alpha+i} 
&\lesssim \ve \tau^{\frac12(\gamma-\frac23)+\delta^\ast-1} \|w_{\frac{N-2(i-j)+i}{2}}\|_\infty \| w^{\frac{\alpha+2(i-j)-N}{2}} \mathcal D_{i-j} H\|_{L^2 }  \notag \\
& \lesssim \ve \tau^{\frac12(\gamma-\frac23)+\delta^\ast-1+\frac12(\frac83-\gamma)} (D^N)^{\frac12} & \notag \\
& \lesssim \ve \tau^{\delta^\ast} (D^N)^{\frac12},
\end{align}
where we have used~\eqref{L21} in the second line. If $j\le2$ we use~\eqref{Linfty1} instead of~\eqref{L21} and obtain the same bound.
\end{proof}

\subsection{High-order estimates for $\mathscr M[H]$}\label{SS:MH}

We first recall $K_a[\theta]$, $a\in \mathbb R$ in \eqref{E:KM}: 
\[
K_a[\theta]= \J[\phi]^a - \J[\phia]^a 
\]
and also 
\be\label{E:QUADRATIC}
K_{a}[\theta]  = a \J[\phia]^{a-1} K_1[\theta] + a(a-1)\J[\phia]^{a-2} \left(\int_0^1(1-s)(1+s \frac{K_1[\theta]}{\J[\phia]}  )^{a-2} ds\right)  (K_1[\theta])^2
\ee
\begin{lemma} \label{L:KABOUNDS}
We have the following bound: 
\begin{align}
&|K_a[\theta]|  \lesssim \tau^{m-\frac23+\frac12(\frac53-\gamma)} \tau^{2a} q_a(\frac{ r^n}{\tau})  (E^N)^\frac12 \label{Kabound} \\
&|w^{j-1} \bar \D_j K_a[\theta]| \lesssim \tau^{m-\frac23+\frac12(\frac53-\gamma)} r^{-j} \tau^{2a} q_{a}(\frac{ r^n}{\tau})  (E^N)^\frac12 , \quad 1\leq j\leq N-3 \label{Kainfty} \\
&\|  r^j q_{-1} (\frac{ r^n}{\tau}) \bar\D_j K_1[\theta]  \|_{\alpha+ 2j +2 -N} \lesssim \tau^{m+\frac43+\frac12(\frac53-\gamma)} (E^N)^\frac12 , \quad 2\leq j\leq N-1 \label{KaL2}
\end{align}
\end{lemma}

\begin{remark}\label{R:KABOUNDS}
$ \tau^{\frac12(\frac53-\gamma)} (E^N)^\frac12 $ can be replaced by $ \tau^{\frac12(\frac83-\gamma)} (D^N)^\frac12 $ in the above bounds. 
\end{remark}

\begin{proof}
First we recall that~\eqref{E:KM} implies 
\be\label{K1}
K_1[\theta]=(2\phia\theta +\theta^2) (\phia+\DL\phia ) +\phi^2 (\theta + \DL\theta) 
\ee
which 
together with~\eqref{Linfty11} and $\theta=\tau^m\frac Hr$ 
yields 
\be
|K_1[\theta]| \lesssim \tau^{m-\frac23+\frac12(\frac53-\gamma)} \tau^2 q_1(\frac{ r^n}{\tau})  (E^N)^\frac12 
\ee
or equivalently, 
\be
\left| \frac{K_1[\theta]}{\J[\phia]}\right| \lesssim \tau^{m-\frac23+\frac12(\frac53-\gamma)}  (E^N)^\frac12 
\ee
The representation \eqref{E:QUADRATIC} then gives \eqref{Kabound} or equivalently 
\be
\left| \frac{K_a[\theta]}{\J[\phia]^a} \right|\lesssim \tau^{m-\frac23+\frac12(\frac53-\gamma)}  (E^N)^\frac12 
\ee

Next we evaluate $\bar\D_j K_a [\theta]$. We start with $a=1$. By applying the product rule to \eqref{K1} and using \eqref{Linfty11},  \eqref{Linfty2} and \eqref{Linfty21},  \eqref{E:BARDIQ}, \eqref{E:BARDIDQ}, we deduce that
\be
|w^{j-1} \bar \D_j K_1[\theta]| \lesssim \tau^{m-\frac23+\frac12(\frac53-\gamma)} r^{-j} \tau^2 q_1(\frac{ r^n}{\tau})  (E^N)^\frac12 , \quad 1\leq j\leq N-3
\ee

For general $a\in \mathbb R$, let us write down the expression for $\bar\D_j K_a [\theta]$. For $j=1$, 
using~\eqref{exp2}, 
we have 
\begin{align}
&\bar\D_1 K_a[\theta] = a \J[\phi]^{a-1} \bar\D_1 K_1[\theta] - a K_{a-1}[\theta] \bar\D_1 \J[\phia]
\end{align}
For $j\geq 2$, by applying the product rule and chain rule, we deduce that
\begin{align}
&\bar\D _j K_a[\theta] = \sum_{ 1\leq \ell\leq j \atop C_1\in\bar{\mathcal P}_{j-\ell}, C_2\in \bar{\mathcal P}_{\ell}  }c_j^{C_1C_2}C_1 \left( \J[\phi]^{a-1} \right)  C_2 \left( K_1[\theta] \right)\label{jKa}  \\
&+ \sum_{1\leq \ell\leq j, 1\leq k_1\leq k_2\leq j-1 \atop B_1\in \bar{\mathcal P}_{j-k_2-\ell}, B_2\in \bar{\mathcal P}_{\ell}}  c^{jB_1B_2}_{k_1k_2\ell}B_1 \left( \J[\phi]^{a-1-k_1}\left( \prod_{k'=1}^{k_1} V_{k'}\J[\phia] \right)_{ j_1+...+j_{k_1}=k_2, j_{k'}\geq1 \atop V_{k'}\in {\bar{\mathcal P}}_{j_{k'}}} \right)  B_2 \left( K_1[\theta] \right) \notag \\
&+ \sum_{1\leq k_1\leq k_2\leq  j}c^j_{k_1}K_{a-k_1} [\theta] \left( \prod_{k'=1}^{k_1} V_{k'}\J[\phia] \right)_{ j_1+...+j_{k}=k_2, j_{k'}\geq1\atop V_{k'}\in {\bar{\mathcal P}}_{j_{k'}}}\notag
\end{align}
which can be proved based on the induction argument on $j$. Therefore, we deduce \eqref{Kainfty}. 
Also, by~\eqref{L21} we have \eqref{KaL2}. 
\end{proof}


Before we proceed with the estimates, we examine the structure of $\mathscr M[H]$. 
Recall \eqref{E:MH} and the formula~\eqref{E:CDEF} $c[\phi] = \frac{\phi^4}{g^2\J[\phi]^{\gamma+1}}$. Then  
 \begin{align}
\mathscr M[H]&=\ve\gamma  \pa_r \left(  \frac{c[\phi]}{g^{00}} \right)  L_\alpha H  +\ve  D_r \left( \frac{\mathscr N_0[H]}{g^{00}}\right) \notag \\
&=\frac{\ve}{g^{00}} \left[ -\gamma (1+\alpha) w'  \frac{\phi^4 }{g^2 } \pa_r (\J[\phi]^{-\gamma-1}) D_r H + \pa_r (\mathscr N_0[H] )  \right] \notag \\
&  - \ve\gamma w \frac{\phi^4 }{g^2g^{00} } \pa_r (\J[\phi]^{-\gamma-1}) \pa_r D_r H + \ve\gamma \pa_r \left( \frac{\phi^4}{g^2 g^{00}} \right) \J[\phi]^{-\gamma-1} L_\alpha H  \notag \\
&  + \ve \left( \pa_r\left( \frac{1}{g^{00}}\right) + \frac{2}{r} \right) \mathscr N_0[H], \label{E:MHEXPANSION}
\end{align}
where we have used~\eqref{E:LBETADEF} and written $L_\alpha H =  - (1+\alpha) w' D_r H - w\pa_r D_r H$. Our source of concern is rectangular bracket above, as it contains top-order terms with two derivatives falling on $H$ (either through $\pa_r(\J[\phi]^{-\gamma-1})$ or $\pa_r\mathscr N_0[H]$) and seemingly insufficient number of $w$-powers to allow us to bound them through our $w$-weighted norms. This situation is a typical manifestation of the vacuum singularity at the outer boundary. Our key insight is that, due to {\em special algebraic} structure of the equation, the terms involving two spatial derivatives of $H$ without the corresponding multiple of $w$ will be cancelled. 
In the following lemma, we present the rearrangement of $\mathscr M[H]$ that elucidates such an important cancelation. 

\begin{lemma}[Cancellation lemma]  \ \label{L:CANCEL}

\begin{itemize}
\item[(i)] 
$\mathscr M[H]$ can be rewritten into the following form 
\begin{align} 
\mathscr M[H] = &  \ve \gamma(\gamma+1) (1+\alpha) w'  \frac{\phi^4 }{g^2 g^{00} } \J[\phi]^{-\gamma-2}\pa_r \J[\phia]  D_r H   \notag \\
&   - \ve\gamma(1+\alpha)w' \pa_r \left( \frac{\phi^4}{g^2 g^{00}} \right) \J[\phi]^{-\gamma-1} D_r H - \ve\gamma w \pa_r\left( \frac{c[\phi]}{g^{00} }\right) \pa_r D_r H \notag \\
& + \ve \frac{1}{\tau^m g^{00}} \mathfrak K_4[\frac{\tau^m H}{r}] + \ve \left( \pa_r\left( \frac{1}{g^{00}}\right) + \frac{1}{g^{00}} \frac{2}{r} \right) \mathscr N_0[H]  \label{MH}
\end{align}
where 
\begin{align}
 \mathfrak K_4[\theta] &=-\gamma(1+\alpha) w' \frac{\phi^2}{g^2} K_{-\gamma-1}[\theta]\big(\pa_r K_1[\theta] -\phi^2[\r^2 \theta + 4\pa_r\theta]\big)   \notag\\
&+(1+\alpha) w' \frac{\phi^2}{g^2} \big\{ -\gamma K_{-\gamma-1}[\theta]\pa_r \J[\phia] + \gamma\pa_r( \J[\phia]^{-\gamma-1})  K_1[\theta]  \notag\\
&  -\gamma (\gamma+1)  K_{-\gamma-2}[\theta] \pa_r \J[\phia]  \phi^2[\r\theta + 3\theta]  
+2\gamma  K_{-\gamma-1}[\theta]\phi \pa_r \phi [\r\theta + 3\theta]  \big\} \notag \\
&+  (1+\alpha) \pa_r \left( w' \frac{\phi^2}{g^2} \right)\left( K_{-\gamma}[\theta] + \gamma \J[\phia]^{-\gamma-1} K_1[\theta] +\gamma  K_{-\gamma-1}[\theta]\phi^2[\r\theta + 3\theta]  \right) \label{k40}
\end{align}
and $\mathscr N_0[H]=\frac{r}{\tau^m}  \mathfrak K_3[\frac{\tau^m H}{r}]$ where $\mathfrak K_3$ is defined by~\eqref{E:MATHFRAKK3}. 

\item[(ii)]  Each expression in the right-hand side of \eqref{MH} contains at most two spatial derivatives. If two spatial derivatives of $H$ appear in the expression, they always contain a factor of  $w$.  In particular, the last bracket of the first line of $ \mathfrak K_4[\theta]$ in \eqref{k40} can be rewritten without any second spatial derivatives of $H$: 
\begin{align}
&\pa_r K_1[\theta] -\phi^2[\r^2 \theta + 4\pa_r\theta] \notag \\
&  = \phi^2 M_g  \pa_r \pa_\tau \theta + \pa_r(\phi^2M_g ) \pa_\tau\theta 
+ 2\phi ( \DL\phia +  r\pa_r \phi )\pa_r\theta  \notag \\
& \ \ \ \ +\left[ \pa_r (3\phia^2 + 2\phia \DL\phia ) +\pa_r(3\phia + \DL\phia  )\theta\right]\theta\big\}. \label{A3K}
\end{align}
\end{itemize}
\end{lemma}

\begin{proof}
To verity \eqref{MH}, we will first rewrite the rectangular bracket in \eqref{E:MHEXPANSION}. 
By~\eqref{E:NZERODEF} $\mathscr N_0[H]=\frac{r}{\tau^m}  \mathfrak K_3[\frac{\tau^m H}{r}]$ where $\mathfrak K_3$ is defined by~\eqref{E:MATHFRAKK3}. Then 
\begin{align*}
\partial_{r} (r\mathfrak K_{3}[\theta] ) &  =\partial_{r}\left[(1+\alpha) w' \frac{\phi^2}{g^2 } \left( K_{-\gamma}[\theta] + \gamma \J[\phia]^{-\gamma-1} K_1[\theta] +\gamma  K_{-\gamma-1}[\theta]\phi^2[\r\theta + 3\theta]  \right)\right]\\
&  =(1+\alpha)\pa_r\left( w' \frac{\phi^2}{g^2 } \right)\left( K_{-\gamma}[\theta] + \gamma \J[\phia]^{-\gamma-1} K_1[\theta] +\gamma  K_{-\gamma-1}[\theta]\phi^2[\r\theta + 3\theta]  \right)\\
&\quad  +(1+\alpha) w' \frac{\phi^2}{g^2 } \underbrace{\pa_r\left[ K_{-\gamma}[\theta] + \gamma \J[\phia]^{-\gamma-1} K_1[\theta] +\gamma  K_{-\gamma-1}[\theta]\phi^2[\r\theta + 3\theta]  \right]}_{(\ast)}
\end{align*}
By using $
\partial_{r}K_{-\gamma}[\theta]=-\gamma \J[\phi]^{-\gamma-1}\partial_{r}
K_{1}[\theta]-\gamma K_{-\gamma-1}[\theta]\partial_{r}\J[\phia]$, we have
\begin{align*}
(\ast) &  =-\gamma \J[\phi]^{-\gamma-1}\partial_{r}K_{1}[\theta]-\gamma
K_{-\gamma-1}[\theta]\partial_{r}\J[\phia]\\
&\quad  +\gamma \J[\phia]^{-\gamma-1}\partial_{r}K_{1}[\theta]-\gamma(
\gamma+1)\J[\phia]^{-\gamma-2}\partial_{r}\J[\phia]K_{1}[\theta]\\
& \quad -\{\gamma(\gamma+1)\J[\phi]^{-\gamma-2}\partial_{r}K_{1}[\theta
]+\gamma(\gamma+1)K_{-\gamma-2}[\theta]\partial_{r}\J[\phia] \}\phi^{2}%
[r\partial_{r}\theta+3\theta]\\
& \quad +\gamma K_{-\gamma-1}[\theta] \pa_r(\phi^{2}[r\partial_{r}\theta+3\theta
])\\
&  =-\gamma(\gamma+1)\J[\phi]^{-\gamma-2}\partial_{r}K_{1}[\theta]\phi
^{2}[r\partial_{r}\theta+3\theta]\\
& \quad -\gamma K_{-\gamma-1}[\theta]\{ \partial_{r}\J[\phia]-\pa_r(\phi^{2}[r\partial
_{r}\theta+3\theta]) \}\\
& \quad +\{\gamma \J[\phia]^{-\gamma-1}-\gamma \J[\phi]^{-\gamma-1}\}\partial_{r}%
K_{1}[\theta]\\
& \quad -\gamma(\gamma+1)\J[\phia]^{-\gamma-2}\partial_{r}\J[\phia]K_{1}[\theta]\\
& \quad -\gamma(\gamma+1)K_{-\gamma-2}[\theta]\partial_{r}\J[\phia] \phi^{2}%
[r\partial_{r}\theta+3\theta]\\
&  =-\gamma(\gamma+1)\J[\phi]^{-\gamma-2}\partial_{r}K_{1}[\theta]\phi
^{2}[r\partial_{r}\theta+3\theta]\\
& \quad -\gamma K_{-\gamma-1}[\theta] \{\partial_{r}K_{1}[\theta]-\phi^2
[r\partial_{r}^2\theta+4\partial_{r}\theta]\}\\
& \quad -\gamma K_{-\gamma-1}[\theta]\partial_{r}\J[\phia]-\gamma(\gamma
+1)\J[\phia]^{-\gamma-2}\partial_{r}\J[\phia]K_{1}[\theta]   \\
&\quad  -\gamma(\gamma+1)K_{-\gamma-2}[\theta]\partial_{r}\J[\phia]\phi^{2}%
[r\partial_{r}\theta+3\theta] + 2\gamma K_{-\gamma-1}[\theta] \phi\pa_r\phi [r\pa_r\theta + 3\theta]
\end{align*}
which in turn implies 
\be\label{N0Hr}
\pa_r \mathscr N_0[H] = -  (1+\alpha)\gamma (\gamma+1)  w' \frac{\phi^2}{g^2}   \J[\phi]^{-\gamma-2} \pa_r K_1[\theta]  \phi^2 D_r H + \tau^{-m}\mathfrak K_4[\theta] 
\ee
where we have used  $ \r\left(\frac Hr\right) + 3\frac H r = D_r H$ and $\theta = \frac{\tau^m H}{r}$. 

For the first term in the rectangular bracket in \eqref{E:MHEXPANSION}, we note 
\[
\pa_r  (\J[\phi]^{-\gamma-1}) = - (\gamma+1)\J[\phi]^{-\gamma-2} \pa_r K_1[\theta] -(\gamma+1)\J[\phi]^{-\gamma-2} \pa_r \J[\phia]. 
\]
Together with \eqref{N0Hr}, the rectangular bracket in \eqref{E:MHEXPANSION} gives rise to the first line and the first term of the third line of \eqref{MH}. The following line of \eqref{E:MHEXPANSION} corresponds to the second line of \eqref{MH} where we have used $L_\alpha H =  - (1+\alpha) w' D_r H - w\pa_r D_r H$. 

Finally we will count the number of  spatial derivatives and the weight $w$. First of all, it is clear that all the terms appearing in \eqref{MH} contain at most two spatial derivatives of $H$. For instance, the first term in \eqref{MH} does not contain the second derivatives of $H$. In the second line, both terms contain the second derivatives of $H$ and they have a factor of $w$. Note that $\pa_r g^{00}$ has a term involving two derivatives (see \eqref{E:GZEROZERODEF}) but that comes with $w$. The same counting applies to the rest. The only expression that is not obvious at first sight is the first line of \eqref{k40} because we do see the two spatial derivatives of $H$ without the weight $w$. It turns out that those second derivatives disappear after cancelation. A direct computation using~\eqref{K1} yields the identity \eqref{A3K}. 
\end{proof}


\begin{lemma}\label{L:MH} 
Let $H$ be a solution to~\eqref{E:H2} on a time interval $[\kappa,T]$ for some $T\le1$ and assume that the a priori assumptions~\eqref{E:APRIORI} holds. Then
\begin{align}
\tau^{\frac12(\gamma-\frac23)}\|\bar{\D}_{i-1}\mathscr M[H]\|_{\alpha+i} \lesssim \sqrt{\ve}\tau^{\min\{\delta^\ast, \frac{\delta}{2}\}-\frac12} (E^N)^\frac12 +  
\ve  \tau^{m-\frac12+\frac52(\frac43-\gamma)} (E^N)^\frac12 (D^N)^\frac12 
\end{align}
\end{lemma}


\begin{proof} 
We note that the terms in the first two lines of \eqref{MH} have the similar structure as the terms resulting from $D_r (\frac{1}{g^{00}} \mathscr L_{\text{low}}^2 H)$ in terms of the highest order derivative count and the weight $w$ count. For instance, the first line of  \eqref{MH} is comparable to the case when the derivative falls into $w$ of the last term of \eqref{E:LLOWTWO}. The difference is whether the coefficients are set by $\phia$, $\J[\phia]$ or $\phi$, $\J[\phi]$, but the coefficients enjoy similar bounds due to Lemmas \ref{L:BARDIQBOUNDS} \ref{L:DIQBOUNDS} for $\phia$, Lemma \ref{L:PHILEMMA} for $\phi$ and our a priori assumption \eqref{E:BOOTSTRAP}. 
 We therefore have
\begin{align}
\tau^{\frac12(\gamma-\frac23)}\|\bar\D_{i-1}(\mathscr M_{12})\|_{\alpha+i}  \lesssim \sqrt \ve \tau^{\min\{\delta^\ast,\frac\delta2\}-\frac12} (E^N)^{\frac12}
\end{align}
where $\mathscr M_{12}$ denotes the first two lines of  \eqref{MH}. 

We focus on the last line of \eqref{MH} and present the detail for the bound on 
$\ve\tau^{-m}\bar\|\bar\D_{i-1}(\frac{1}{g^{00}} \mathfrak K_4[\frac{\tau^m H}{r}])\|_{\alpha+i}$. We restrict our attention first to the following term coming from the first line of \eqref{k40}
\[
(\star):=\ve\frac{1}{\tau^m}w' \frac{\phi^2}{g^2 g^{00}} K_{-\gamma-1}[\theta] \left(\pa_r K_1[\theta] - \phi^2[\r^2 \theta + 4\pa_r\theta]\right) \text{ where }  \theta=\frac{\tau^m H}{r}.
\]
As shown in the previous lemma, the identity \eqref{A3K} assures that $\pa_r K_1[\theta] - \phi^2[\r^2 \theta + 4\pa_r\theta]$ contains at most one spatial derivative of $H$ and therefore no issues associated with the $w$-weights near the boundary will occur. 

We proceed with $\bar\D_{i-1} (\star)$ for $1\leq i\leq N$. By the product rule, $\bar\D_{i-1} (\star)$ can be written as a linear combination of the following form 
\be
\ve\frac{1}{\tau^m}A_1\left( \frac{w'}{r} \frac{\phi^2}{g^2 g^{00}}\right) A_2 (K_{-\gamma-1}[\theta] )A_3\left(r \{ \pa_r K_1[\theta] - \phi^2[\r^2 \theta + 4\pa_r\theta] \}\right)
\ee
where $A_1\in \bar{\mathcal P}_{\ell_1}$,  $A_2\in \bar{\mathcal P}_{\ell_2}$,  $A_3\in \bar{\mathcal P}_{\ell_3}$, $\ell_1+\ell_2+\ell_3=i-1$. As before, we divide into several cases. If $\ell_k\leq 2$ for all $k=1,2,3$, all the indices are low and we just use $L^\infty$ bounds \eqref{LinftyV1}, \eqref{E:GZEROZERO1}, \eqref{E:GZEROZERO3}, \eqref{Kainfty}. 
In the following, we assume that at least one index is greater than 2. 

\underline{Case I: $\ell_3\geq \max\{\ell_1,\ell_2\}$.} 
In this case,  $3\leq \ell_3\leq i-1$.  
Since $\ell_1,\ell_2 \leq \frac{N-1}{3}\leq N-4$, and we apply $L^\infty$ bounds for $A_1$ and $A_2$ factors and $L^2$ bounds for $A_3$ factor. In particular, by assuming $\ell_1\geq 1$ (the case of $\ell_1=0$ follows similarly), we arrange the $w$ weights as follows: 
\[
\ve\frac{1}{\tau^m}  w^{\ell_1} A_1\left(\frac{ w' }{r}\frac{\phi^2}{g^2 g^{00}}\right) w^{\ell_2}A_2 (K_{-\gamma-1}[\theta] ) w^{\frac{\alpha+ i - 2(\ell_1+\ell_2) }{2}}A_3\left(r\{\pa_r K_1[\theta] - \phi^2[\r^2 \theta + 4\pa_r\theta]\}\right)
\]
By the product rule and by using 
\eqref{LinftyV1}, \eqref{LinftyV2}, \eqref{E:GZEROZERO1}, \eqref{E:GZEROZERO3},  we deduce that 
\be\label{A1infty}
\left| w^{\ell_1} A_1\left( \frac{w'}{r} \frac{\phi^2}{g^2 g^{00}}\right) \right| \lesssim  r^{n-2-\ell_1}\tau^{\frac43} (1+ \ve \tau^{\delta^\ast})
\ee
and by further using \eqref{Kainfty} 
\be
\left|w^{\ell_2}A_2 (K_{-\gamma-1}[\theta] )  \right|\lesssim \tau^{m-\frac23+\frac12(\frac83-\gamma)} r^{-\ell_2} \tau^{-2\gamma-2} q_{-\gamma-1}(\frac{ r^n}{\tau})  (D^N)^\frac12
\ee
We have derived so far 
\[
\begin{split}
&\| \bar\D_{i-1} (\star) \|_{\alpha+i} \lesssim  \ve \tau^{-\frac{5\gamma}{2}} (D^N)^{\frac12}   \| r^{n-\ell_1-\ell_2-2} q_{-\gamma-1}(\frac{ r^n}{\tau}) A_3\left(r\{ \pa_r K_1[\theta] - \phi^2[\r^2 \theta + 4\pa_r\theta]\}\right) \|_{\alpha+ i - 2(\ell_1+\ell_2)}
\end{split}
\]
We claim that 
\be\label{claim2}
 \| r^{n-\ell_1-\ell_2-2} q_{-\gamma-1}(\frac{ r^n}{\tau}) A_3\left(r\{\pa_r K_1[\theta] - \phi^2[\r^2 \theta + 4\pa_r\theta]\}\right) \|_{\alpha+ i - 2(\ell_1+\ell_2)} \lesssim \tau^{m+\frac{19}{6} -\frac{\gamma}{2}}(E^N)^{\frac12}
\ee
Note that from  \eqref{A3K} we may rewrite $r\{ \pa_r K_1[\theta] - \phi^2[\r^2 \theta + 4\pa_r\theta] \}$ as
\be\label{A3K2}
\begin{split}
&r\{ \pa_r K_1[\theta] - \phi^2[\r^2 \theta + 4\pa_r\theta] \}\\
&=\tau^m \Big\{ \phi^2M_g (D_r \pa_\tau H - 3 \frac{\pa_\tau H}{r} + \frac{m}{\tau} (D_r H - 3 \frac{ H}{r})  ) + r\pa_r(\phi^2M_g )(\frac{\pa_\tau H}{r} +  \frac{m}{\tau} \frac{H}{r}) \\
&+2\phi(\DL\phia +\r\phi) ( D_r H - 3 \frac{ H}{r} ) + \left[ r\pa_r (3\phia^2 + 2\phia \DL\phia ) +\tau^mr\pa_r(3\phia + \DL\phia  )\frac{H}{r}\right] \frac{H}{r}\Big\}
\end{split}
\ee
Apply $A_3$ to the above. We focus on the first term which can be written as
\[
A_{31} (\phi^2M_g ) A_{32} ( D_r \pa_\tau H - 3 \frac{\pa_\tau H}{r} + \frac{m}{\tau} \frac{H}{r} )
\]
for $A_{31}\in {\bar {\mathcal P}}_{\ell_{31}}$ and $A_{32}\in {\bar{ \mathcal P}}_{\ell_{32}}$ where $\ell_{31}+\ell_{32} = \ell_3 \leq i-1$. As previously done, depending on the size of $\ell_{31}, \ell_{32}$, we may use $L^\infty$ and $L^2$ bounds. We verify the claim \eqref{claim2} when $\ell_{31}=0$ and $\ell_{32}=\ell_3$. Note that 
\[
|q_{-\gamma-1}(\frac{ r^n}{\tau})\phi^2M_g  A_3(D_r \pa_\tau H - 3 \frac{\pa_\tau H}{r} + \frac{m}{\tau} \frac{H}{r} ) | \lesssim \tau^{\frac73}(|B \pa_\tau H | + |A_3(\frac{\pa_\tau H}{r})| + \frac{1}{\tau}|A_3(\frac{H}{r})|)
\]
for $B\in \mathcal P_{\ell_3+1}$. Now we have $w^{\alpha+i-2(\ell_1+\ell_2)}= w^{\alpha+ 2 (\ell_3+1) -N} w^{N+i - 2(\ell_1+\ell_2+\ell_3+1)}\leq w^{\alpha+ 2 (\ell_3+1) -N}$ since $\ell_1+\ell_2+\ell_3 +1=i$ and $i\leq N$ and hence by the definition of $E^N$ and $L^2$ embedding, we obtain 
\[
 \| r^{n-\ell_1-\ell_2-2} q_{-\gamma-1}(\frac{ r^n}{\tau})\phi^2M_g  A_3( D_r \pa_\tau H - 3 \frac{\pa_\tau H}{r} + \frac{m}{\tau} \frac{H}{r}) \|_{\alpha+i -2(\ell_1+\ell_2)} \lesssim \tau^{\frac{19}{6} -\frac{\gamma}{2}}(E^N)^{\frac12}
\]
which gives \eqref{claim2}. Other terms can be estimated similarly. 

Therefore we deduce that 
\be
\| \bar\D_{i-1} (\star) \|_{\alpha+i}  \lesssim  \ve \tau^{m-\frac56+3(\frac43-\gamma)} (E^N)^\frac12 (D^N)^\frac12 
\ee

\underline{Case II: $\ell_2\geq \max\{\ell_1,\ell_3\}$.} In this case,  $3\leq \ell_2\leq i-1$ and $\ell_1,\ell_3 \leq\frac{i-1}{3}$. We apply $L^\infty$ bounds for $A_1$ and $A_3$ factors and $L^2$ bounds for $A_2$ factor. We arrange the $w$ weights as follows: 
\[
\ve\frac{1}{\tau^m}  w^{\ell_1} A_1\left( \frac{w'}{r} \frac{\phi^2}{g^2 g^{00}}\right) w^{\frac{\alpha+ i -2(\ell_1+\ell_3)}{2}}A_2 (K_{-\gamma-1}[\theta] ) w^{\ell_3}A_3\left(r\{\pa_r K_1[\theta] - \phi^2[\r^2 \theta + 4\pa_r\theta]\}\right)
\]
We have the same bound for $A_1$ factor as in \eqref{A1infty}. For $A_3$ factor, from \eqref{A3K2}, we deduce that 
\[
|w^{\ell_3}A_3(r\{ \pa_r  K_1[\theta] - \phi^2[\r^2 \theta + 4\pa_r\theta]\}) | \lesssim \tau^{m+\frac73+\frac{1}{2}(\frac83-\gamma)}  r^{-\ell_3}q_1(\frac{ r^n}{\tau}) (D^N)^\frac12 
\]
It suffices to estimate 
\[
 \| r^{n-\ell_1-\ell_3-2} q_{1}(\frac{ r^n}{\tau}) \bar\D_{\ell_2}\left(K_{-\gamma-1}[\theta]\right) \|_{\alpha+i -2(\ell_1+\ell_3)} 
\]
Using \eqref{jKa} for $a=-\gamma-1$, we have the following expression 
\begin{align}
&\bar\D _{\ell_2}K_{-\gamma-1}[\theta] = \sum_{ 1\leq \ell\leq \ell_2 \atop C_1\in\bar{\mathcal P}_{\ell_2-\ell}, C_2\in \bar{\mathcal P}_{\ell}  }c_{\ell_2}^{C_1C_2}C_1 \left( \J[\phi]^{-\gamma-2} \right)  C_2 \left( K_1[\theta] \right) \label{jka1} \\
&+ \sum_{1\leq \ell\leq \ell_2, 1\leq k_1\leq k_2\leq \ell_2-1 \atop B_1\in \bar{\mathcal P}_{\ell_2-k_2-\ell}, B_2\in \bar{\mathcal P}_{\ell}}  c^{\ell_2B_1B_2}_{k_1k_2\ell}B_1 \left( \J[\phi]^{-\gamma-2-k_1}\left( \prod_{k'=1}^{k_1} V_{k'}\J[\phia] \right)_{ j_1+...+j_{k_1}=k_2, j_{k'}\geq1 \atop V_{k'}\in {\bar{\mathcal P}}_{j_{k'}}} \right)  B_2 \left( K_1[\theta] \right) \label{jka2} \\
&+ \sum_{1\leq k_1\leq k_2\leq  \ell_2}c^j_{k_1}K_{-\gamma-1-k_1} [\theta] \left( \prod_{k'=1}^{k_1} V_{k'}\J[\phia] \right)_{ j_1+...+j_{k}=k_2, j_{k'}\geq1\atop V_{k'}\in {\bar{\mathcal P}}_{j_{k'}}}\label{jka3}
\end{align}
Following the case-by-case analysis as before and using \eqref{Kabound}, \eqref{Kainfty}, \eqref{KaL2}, Lemma \ref{L:PHILEMMA} and \eqref{E:JQPOWERABOUND}, we deduce that 
\[
 \| r^{n-\ell_1-\ell_3-2} q_{1}(\frac{ r^n}{\tau}) \bar\D_{\ell_2}\left(K_{-\gamma-1}[\theta]\right) \|_{\alpha+i -2(\ell_1+\ell_3)}  \lesssim  \tau^{m-\frac23+\frac12(\frac53-\gamma)} \tau^{-2\gamma-2}  (E^N)^\frac12 
\]

Therefore we obtain the same bound as Case I 
\be
\| \bar\D_{i-1} (\star) \|_{\alpha+i}  \lesssim   \ve \tau^{m-\frac56+3(\frac43-\gamma)} (E^N)^\frac12 (D^N)^\frac12, 
\ee
where we have used~\eqref{E:GZEROZERO2}. 

\underline{Case III: $\ell_1\geq \max\{\ell_2,\ell_3\}$.} In this case,  $3\leq \ell_1\leq i-1$ and $\ell_2,\ell_3 \leq\frac{i-1}{3}$. We apply $L^\infty$ bounds for $A_2$ and $A_3$ factors and $L^2$ bounds for $A_1$ factor. For $L^2$ bounds for $A_1(\frac{1}{g^{00}})$, we use Lemma \ref{L:GZEROZEROLEMMA}. The proof follows in the same fashion and we get the same bound as in the previous cases. 

All the other terms in~\eqref{k40} are estimated analogously and we have the following bound: 
\be
\ve\tau^{-m}\|\bar\D_{i-1}(\frac{1}{g^{00}} \mathfrak K_4[\frac{\tau^m H}{r}])\|_{\alpha+i} \lesssim \ve \tau^{m-\frac56+3(\frac43-\gamma)} (E^N)^\frac12 (D^N)^\frac12 
\ee

The last term in ~\eqref{MH} can be estimated similarly by using Lemma \ref{L:GZEROZEROLEMMA}, \eqref{N0Hr}, and the previous estimates on $\mathfrak K_4[\theta]$: 
\be
\ve\| \bar\D_{i-1} \left( ( \pa_r\left( \frac{1}{g^{00}}\right) + \frac{1}{g^{00}} \frac{2}{r} ) \mathscr N_0[H]  \right) \|_{\alpha+i}  \lesssim  \ve \tau^{m-\frac56+3(\frac43-\gamma)} (E^N)^\frac12 (D^N)^\frac12 
\ee
This finishes the proof Lemma. 
\end{proof}


\subsection{Nonlinear estimates}\label{SS:NONLINEARESTIMATES}


Before we formulate the main estimate in Proposition~\ref{P:NONLINEARESTIMATE}, we collect 
several identities that can be regarded as a special form of the product rule that connects the algebraic structure
of the nonlinearity to the algebraic properties of the vector field class $\mathcal P$. 

\begin{lemma}\label{L:HH}
For any $i\in\{0,\dots,N\}$ there hold the identities
\begin{align}
\D_i\left(\frac{1}{r} \left( r\pa_r \left(
 \frac{H}{r}\right) \right)^2 \right) &= \sum_{1\leq k\leq i} \sum_{ B \in \mathcal P_{k+1} \atop C\in {\mathcal P}_{i-k+2}} c_k^{iBC} (BH ) (CH)  \label{E:DG}\\
\D_i\left(\frac{H^2}{r}\right) &= \sum_{A_{1,2}\in \mathcal P_{\ell_1,\ell_2} \atop \ell_1+\ell_2=i+1, \ \ell_1,\ell_2\le i} a_i^{A_1A_2}A_1 H A_2 H \label{E:HH} \\
\D_i\left(\frac{H^3}{ r^2}\right) & = \sum_{A_{1,2,3}\in \mathcal P_{\ell_1,\ell_2,\ell_3} \atop \ell_1+\ell_2+\ell_3=i+1, \ \ell_1,\ell_2,\ell_3\le i} \label{E:HHH}
b_i^{A_1A_2A_3}A_1 H A_2 H A_3 H,
\end{align}
where $a_i^{A_1A_2},b_i^{A_1A_2A_3}, c_k^{iBC}$ are some universal real constants.
Note that the operators $A_j$, $j=1,2,3$ are at most of order $i$. 
\end{lemma}

\begin{proof}
{\em Proof of~\eqref{E:DG}.}
The proof is based on the induction on $i$. 
Note that 
\[
 \r \left(\frac{H}{r} \right)=  D_r H - 3 \frac{H}{r} 
\]
Let $i=1$. Then 
\begin{align*}
&D_r \left[  \frac{1}{r} (  D_r H - 3 \frac{H}{r} )^2 \right] \\
&=\frac{2}{r} (  D_r H - 3 \frac{H}{r} ) \left( \pa_r D_r H - 3\pa_r( \frac{H}{r} )\right) +  \frac{1}{ r^2} (  D_r H - 3 \frac{H}{r} )^2 \\
&= 2 \pa_r \left(\frac{H}{r} \right) \left( \pa_r D_r H - 3\pa_r( \frac{H}{r} )\right) + \left( \pa_r \left(\frac{H}{r} \right)\right)^2 \\
&= 2 \pa_r \left(\frac{H}{r} \right) \pa_r D_r H  -5 \left( \pa_r \left(\frac{H}{r} \right)\right)^2
\end{align*}
Since both $\pa_r \left(\frac{\cdot}{r} \right)$ and $\pa_r D_r $ belong to $ \mathcal P_{2}$, the claim is true for $i=1$. Now suppose the claim is true for all $i\leq \ell$ and let
\[
\mathscr G:=  \frac{1}{r} \left( r\pa_r \left(\frac{H}{r}\right) \right)^2.
\] 
If $\ell$ is even, 
\[
\begin{split}
\D_{\ell+1} \mathscr G &= D_r \sum_{1\leq k\leq \ell \atop k: \text{ even}} \sum_{ B \in \mathcal P_{k+1} \atop C\in {\mathcal P}_{\ell-k+2}} c_k^{\ell BC} (BH ) (CH) 
+  D_r \sum_{1\leq k\leq \ell\atop k: \text{ odd}} \sum_{ B \in \mathcal P_{k+1} \atop C\in {\mathcal P}_{\ell-k+2}} c_k^{\ell BC} (BH ) (CH) \\
&=  \sum_{1\leq k\leq \ell \atop k: \text{ even}} \sum_{ B \in \mathcal P_{k+1} \atop C\in {\mathcal P}_{\ell-k+2}}c_k^{\ell BC}\left[ (BH ) D_r (CH) + \pa_r (BH) (CH) \right] 
\\
&\quad + \sum_{1\leq k\leq \ell\atop k: \text{ odd}} \sum_{ B \in \mathcal P_{k+1} \atop C\in {\mathcal P}_{\ell-k+2}} c_k^{\ell BC} \left[ D_r (BH ) (CH)
+(BH) \pa_r (CH) \right]
\end{split}
\]
Note that each term in the summation belongs to $\mathcal P_j$ for some $j$. Therefore, the claim is true for $i=\ell+1$. If $\ell$ is odd, we can rearrange terms as follows: 
\[
\begin{split}
\D_{\ell+1} \mathscr G &= \pa_r \sum_{1\leq k\leq \ell \atop k: \text{ even}} \sum_{ B \in \mathcal P_{k+1} \atop C\in {\mathcal P}_{\ell-k+2}} c_k^{\ell BC} (BH ) (CH) 
+  \pa_r \sum_{1\leq k\leq \ell\atop k: \text{ odd}} \sum_{ B \in \mathcal P_{k+1} \atop C\in {\mathcal P}_{\ell-k+2}} c_k^{\ell BC}(BH ) (CH) \\
&=  \sum_{1\leq k\leq \ell \atop k: \text{ even}} \sum_{ B \in \mathcal P_{k+1} \atop C\in {\mathcal P}_{\ell-k+2}}c_k^{\ell BC}\left[ \pa_r(BH ) (CH) +  (BH) \pa_r (CH) \right] 
\\
&\quad + \sum_{1\leq k\leq \ell\atop k: \text{ odd}} \sum_{ B \in \mathcal P_{k+1} \atop C\in {\mathcal P}_{\ell-k+2}} c_k^{\ell BC} \left[ (D_r -\frac{2}{r}) (BH ) (CH)
+(BH)  (D_r -\frac{2}{r}) (CH) \right]
\end{split}
\]
which shows that the claim is true for $i=\ell+1$. 
The proofs of~\eqref{E:HH}--\eqref{E:HHH} are similar.
\end{proof}


\begin{proposition}[Estimates for the nonlinear term]\label{P:NONLINEARESTIMATE}
Let $H$ be a solution of~\eqref{E:H1}. Then for any $i\in\{0,1,\dots,N\}$ the following bound holds
\begin{align}\label{E:NONLINEARESTIMATE}
\tau^{\frac12(\gamma-\frac23)} \|\D_i\mathscr N[H]\|_{\alpha+i} \lesssim 
\sqrt\ve  \tau^{m+\frac34 \delta^\ast}E^N +\tau^{m+\delta^\ast -\frac32-\frac32(\frac43-\gamma)}(E^N)^{\frac12} (D^N)^{\frac12}.
\end{align}
\end{proposition}


Since $\phi=\phia +\theta$, we have by simple algebra
\begin{align}
 \frac{1}{\phi^2}-\frac{1}{\phia^2} +  \frac{2\theta}{\phia^3}  = \frac{3\phia \theta^2+2\theta^3}{\phi^2\phia^3}
\end{align}
From~\eqref{E:NONLINEARTERM} we may write $\mathscr N[H]$ in the form
\begin{align}\label{E:NH1}
\mathscr N [H]= - \ve  r\tau^{-m} \mathfrak K_2[\tau^m \frac{H}{r}] -\frac{2}{3\phi^2\phia^2 }\tau^m \frac{H^2}{r}-\frac{4}{9\phi^2\phia^3}\tau^{2m} \frac{H^3}{ r^2} 
- \ve \tau^m \frac{P[\phia] H^2}{\phia^2  r}
\end{align}

Using~\eqref{E:MATHFRAKK2}, the first term on the right-hand side of~\eqref{E:NH1} takes the form
\begin{align}
- \ve  r\tau^{-m} \mathfrak K_2[\tau^m \frac{H}{r}] &= 2\ve \gamma\tau^m w\frac{\phi^3 }{g^2\J[\phi]^{\gamma+1} } \frac{1}{r} \left( r\pa_r \left(
 \frac{H}{r}\right) \right)^2 \notag \\
&+ 2 \ve \gamma \tau^m w\frac{\phi^3 }{g^2\J[\phi]^{\gamma+1}} \frac{M_g }{ r^2}  \left(  \pa_\tau H  + \frac{m}{\tau } H \right)   r\pa_r \left( \frac{H}{r}\right)\notag\\
&+  \ve \gamma \tau^{m}  w\frac{\phi^2}{g^2\J[\phi]^{\gamma+1}  r} \DL(3\phia + \DL\phia  )H^2 \notag \\
&- \ve\gamma(\gamma+1) \tau^m w\frac{\phi^2}{g^2\J[\phia]^{\gamma+2}  r} \big[  (3\phia + \DL\phia  )\frac{H^2}{ r^2} +\tau^m\frac{H^3}{ r^3} \big] \DL\J[\phia]\notag\\
&+\ve \gamma (1+\alpha)\tau^m  w' \frac{\phi^2}{g^2\J[\phia]^{\gamma+1} } \big[ (\DL\phia -3\phia  )\frac{H^2}{ r^2} -2\tau^m\frac{H^3}{ r^3} \big] \notag \\
&+\ve \gamma \tau^{-m}w\frac{\phi^2}{g^2 r} (K_{-\gamma-1}[\tau^m\frac{H}{r}] + (\gamma+1) \J[\phia]^{-\gamma-2} K_1[\tau^m\frac{H}{r}] )  
 \DL\J[\phia] \label{E:MATHFRAKK2NEW}
\end{align}


 We denote the first and second terms of the right-hand side of \eqref{E:MATHFRAKK2NEW} by $\mathscr N_1[H]$ and $\mathscr N_2[H]$. We first present the estimation of $\mathscr N_1[H]$ and $\mathscr N_2[H]$. 

\subsubsection*{Estimates for $\mathscr N_1[H]=2\gamma\tau^m w\frac{\phi^3 }{g^2\J[\phi]^{\gamma+1} } \frac{1}{r} \left( r\pa_r \left(
 \frac{H}{r}\right) \right)^2$}


\begin{lemma} \label{L:N1}
For each $1\leq i\leq N-1$, we have the following:
\be
\ve \tau^{\frac12(\gamma-\frac23)} \|\D_i ( \mathscr N_1[H]  )\|_{\alpha+i} \lesssim 
\ve \tau^{m+\frac54\delta^\ast}E^N 
\ee
For $i=N$, 
\be
\ve\tau^{\frac12(\gamma-\frac23)} \|\D_N ( \mathscr N_1[H]  )\|_{\alpha+N} \lesssim 
\ve^\frac12  \tau^{m+\frac34 \delta^\ast}E^N 
\ee
\end{lemma}


\begin{proof} 
Using the product rule~\eqref{E:product}
we  have
\begin{align*}
\ve \tau^{\frac12(\gamma-\frac23)} \D_i ( \mathscr N_1[H]  )
= -2\gamma\ve \tau^{\frac12(\gamma-\frac23)+m}
\sum_{j+k+\ell =i} \sum_{A\in \bar{\mathcal P}_j , B\in \bar{\mathcal P}_k \atop C \in {\mathcal P}_\ell} c^{iABC}_{kj} 
\underbrace{A\left( \frac{w}{g^2}\right) B\left( \frac{\phi^3 }{\J[\phi]^{\gamma+1}}\right)\left(C \mathscr G\right)}_{I^{iABC}}, 
\end{align*}
where we recall the notation $\mathscr G =\frac{1}{r} \left( r\pa_r \left(
 \frac{H}{r}\right) \right)^2$ from the proof of Lemma~\ref{L:HH}. 

\underline{Case I: $\ell=0$.} First we have $C\mathscr G=\mathscr G$ and $\mathscr G=\pa_r (\frac{H}{r}) (D_r H -3\frac{H}{r})$. By using $L^\infty$ bound \eqref{Linfty1}, we have 
\be
|\mathscr G |\lesssim \tau^{\frac12(\frac{11}{3} - \gamma)} (E^N)^\frac12 |\D_1 H| \lesssim \tau^{\frac{11}{3} - \gamma} E^N
\ee

\underline{Case I-1: $\ell=0$ and $0\le k\le1$.} By~\eqref{E:PHILEMMA1} $|A\left( \frac{w}{g^2}\right) B\left( \frac{\phi^3 }{\J[\phi]^{\gamma+1}}\right) |\lesssim \tau^{-2\gamma}q_{-\gamma-1}(\frac{ r^n}{\tau})$. 
Therefore, 
recalling~\eqref{E:DELTASTARDEF} and the above definition of $I^{iABC}$ we obtain 
\be\label{Bound1}
\ve  \tau^{\frac12(\gamma-\frac23)} \lv I^{iABC} \rv \lesssim \ve \tau^{\frac12(\gamma -\frac23)+m -2\gamma-\frac in+\frac{11}{3}-\gamma}E^N  = 
\ve \tau^{\frac56(4-3\gamma)- \frac in + m}E^N \lesssim \ve \tau^{m+\frac54\delta^\ast}E^N.
\ee

\underline{Case I-2: $\ell=0$, $k\geq2$.} 
If $j=0$ and $k=i\le N-1$ we use~\eqref{E:PHILEMMA2} to conclude
\begin{align}
\ve  \tau^{\frac12(\gamma-\frac23)+m} \| I^{iABC} \|_{\alpha+i} 
& \lesssim \ve  \tau^{\frac12(\gamma-\frac23)+m} \|A\left( \frac{w}{g^2}\right)\|_{L^\infty} \|B\left( \frac{\phi^3 }{\J[\phi]^{\gamma+1}}\right)\|_{\alpha -N+2i+2}
\|\left(C \mathscr G\right)\|_{L^\infty} \|w^{N-i-1}\|_{L^\infty} \notag \\
& \lesssim  \ve  \tau^{\frac12(\gamma-\frac23)-2\gamma -  \frac{k}{n}  + \frac{11}{3} - \gamma+m} E^N =\ve \tau^{\frac56(4-3\gamma)- \frac in + m}E^N \notag \\
& \lesssim \ve \tau^{m+\frac54\delta^\ast}E^N \label{E:CASE12}
\end{align}

If $j=0$ and $k=i=N$ we use~\eqref{E:PHILEMMA4} instead of~\eqref{E:PHILEMMA2} and 
thanks to an additional power of $w$, 
this leads to
\begin{align*}
\ve  \tau^{\frac12(\gamma-\frac23)+m} \| I^{NABC} \|_{\alpha+ N} 
& \lesssim \ve  \tau^{\frac12(\gamma-\frac23)+m} \| \frac{w}{g^2}B\left( \frac{\phi^3 }{\J[\phi]^{\gamma+1}}\right)\|_{\alpha+ N}
\|\left(C \mathscr G\right)\|_{L^\infty} \\
& \lesssim \ve  \tau^{\frac12(\gamma-\frac23)+m} \|B\left( \frac{\phi^3 }{\J[\phi]^{\gamma+1}}\right)\|_{\alpha+ N+1}
\|\left(C \mathscr G\right)\|_{L^\infty}  \\
& \lesssim  \sqrt\ve  \tau^{\frac12(\gamma-\frac23)-2\gamma - \frac{N}{n} + \frac{11}{3} - \gamma+m} E^N 
\lesssim \sqrt\ve \tau^{m+\frac54\delta^\ast}E^N.
\end{align*}

If $j\ge1$ we have $k\le N-1$ and $\lv A\left( \frac{w}{g^2}\right) \rv \lesssim  r^{n-j} \lesssim \tau^{1-\frac jn}q_{1-\frac{j}{n}}\et.$
Using the bound analogous to~\eqref{E:CASE12} we conclude
\begin{align}
\ve  \tau^{\frac12(\gamma-\frac23)+m} \| I^{iABC} \|_{\alpha+i}  
& \lesssim \ve  \tau^{\frac12(\gamma-\frac23)+m} \|A\left( \frac{w}{g^2}\right) B\left( \frac{\phi^3 }{\J[\phi]^{\gamma+1}}\right)\|_{\alpha -N+2k+2}
\|\left(C \mathscr G\right)\|_{L^\infty} \|w^{N-k-1}\|_{L^\infty} \notag \\
& \lesssim \tau^{\frac12(\gamma-\frac23)+m+1-\frac jn} \|q_{1-\frac{j}{n}}\et B\left( \frac{\phi^3 }{\J[\phi]^{\gamma+1}}\right)\|_{\alpha -N+2k+2} \|\left(C \mathscr G\right)\|_{L^\infty} \notag \\
& \lesssim \ve  \tau^{\frac12(\gamma-\frac23)+m-2\gamma+1 - \frac{k+j}{n} + \frac{11}{3} - \gamma} E^N \notag \\
& \lesssim \ve \tau^{1+\frac56(4-3\gamma)- \frac in + m}E^N \lesssim 
\ve \tau^{m+\frac54\delta^\ast+1}E^N
\end{align}

\underline{Case II: $\ell\geq1$.}  In this case, we will make use of the representation obtained in \eqref{E:DG}: for $C\in \mathcal P_\ell$, we write it as 
\be
C\mathscr G = \sum_{1\leq q\leq \ell } \sum_{ C_1 \in \mathcal P_{q+1} \atop C_2\in {\mathcal P}_{\ell-q+2}} c_q^{\ell C_1C_2} (C_1H ) (C_2H)
\ee
Let $\ell_\ast=\max\{ q+1, \ell-q+2 \}$. Without loss of generality, we may assume that $\ell_\ast = \ell-q+2$ so that $C_2\in \mathcal P_{\ell_\ast}$ and $C_1\in \mathcal P_{\ell-\ell_\ast+3}$.  Note that $\frac{\ell+3}{2}\leq \ell_\ast\leq \ell+1$ and $1\leq q=\ell-\ell_\ast+2\leq \frac{\ell+1}{2}$.

\underline{Case II-1: $\ell\geq1$, $j=0$ and $k=0$.}  We first consider $\frac{N-\alpha}{2}\leq \ell_\ast \leq N$.  In this case, 
by~\eqref{E:PHILEMMA1} and thanks to an additional power of $w$, 
\begin{align}
&\ve  \tau^{\frac12(\gamma-\frac23)+m} \| I^{iABC} \|_{\alpha+i} \notag\\
& = \ve \tau^{\frac{10}{3}-\frac{\gamma}{2}+m}  \| \frac{\phi^3 }{g^2\J[\phi]^{\gamma+1}}  \tau^{\frac{1}{2} (\gamma-\frac{11}{3})}  w^{\frac{ N+i -2\ell_\ast +2 }{2}-q+1} w^{q-1}C_1H \tau^{\frac{1}{2} (\gamma-\frac{11}{3})}  w^{\frac{\alpha+2\ell_\ast -N}{2}} C_2H  \|_{L^2} \notag \\
& \lesssim \ve \tau^{\frac{10}{3}-\frac{\gamma}{2}-2\gamma+m}  \tau^{\frac{1}{2} (\gamma-\frac{11}{3})}  \|w^{q-1}C_1H\|_{L^\infty} 
\tau^{\frac{1}{2} (\gamma-\frac{11}{3})}  \|w^{\frac{\alpha+2\ell_\ast -N}{2}} C_2H  \|_{L^2} \notag \\
& \lesssim  \ve \tau^{\frac{10}{3}-\frac{\gamma}{2}-2\gamma+m} E^N
= \ve \tau^{\frac52(\frac43-\gamma)+m}E^N \lesssim \ve \tau^{m+\frac54\delta^\ast}E^N, \label{Bound11}
\end{align}
where we note that since $p+\ell_\ast =\ell+2$, $\frac{ N+i -2\ell_\ast +2 }{2}-q+1 = \frac{N+i - 2\ell}{2} \geq 0$, and 
thus $\|w^{\frac{ N+i -2\ell_\ast +2 }{2}-q+1} \|_{L^\infty}\lesssim 1$.  

Furthermore, since $\ell-\ell_\ast+1=q-1 \leq \frac{\ell-1}{2} \leq N-4$ due to $N=\lfloor\alpha\rfloor+6\geq 9$, $\tau^{\frac{1}{2} (\gamma-\frac{11}{3})}  w^{\frac{ N+i -2\ell_\ast +2 }{2}-q+1} w^{q-1}C_1H$ is bounded by $(E^N)^{\frac12}$ via \eqref{Linfty2}. Finally, we used the 
$L^2$ embedding \eqref{L21} to bound $\|w^{\frac{\alpha+2\ell_\ast -N}{2}} C_2H  \|_{L^2}$.

Now suppose $\ell_\ast =N+1$ ($\ell=N$ and $q=1$). We then have 
\begin{align}
\ve  \tau^{\frac12(\gamma-\frac23)+m} w^{\frac{\alpha+N}{2}}I^{NABC}=& \ve^\frac12 \tau^{\frac{3}{2} }   \tau^{\frac{\gamma+1}{2}}q_{\frac{\gamma+1}{2}}(\frac{ r^n}{\tau})  \frac{\phi^3 }{g^2\J[\phi]^{\gamma+1}} \tau^{\frac{1}{2} (\gamma-\frac{11}{3})}   C_1H \notag \\
&\ve^\frac12  \tau^{-\frac{\gamma+1}{2}}q_{-\frac{\gamma+1}{2}}(\frac{ r^n}{\tau}) w^{\frac{\alpha+N+2}{2}} C_2H 
\end{align}
We estimate the $L^2$-norm of the above expression by estimating the first line in $L^\infty$ norm and the second line in the $L^2$-norm.
Recalling~\eqref{E:NDEF}, by~\eqref{L22},~\eqref{E:PHILEMMA1},~\eqref{Linfty1}
we obtain  
\be
\ve  \tau^{\frac12(\gamma-\frac23)+m} \| I^{NABC}\|_{\alpha+i}\lesssim \sqrt\ve  \tau^{\frac12(4-3\gamma) + m}E^N \lesssim \sqrt\ve \tau^{m+\frac34 \delta^\ast }E^N .
\ee
The only remaining case is when $\ell_\ast <\frac{N-\alpha}{2}$, namely $\ell_\ast=2$ and $\ell=1$. In this case, we can just use the $L^\infty$ bound \eqref{Linfty2} to derive the same bound as in \eqref{Bound11}. 

\underline{Case II-2: $\ell\geq1$, $j=0$ and $k\geq 1$.} In this case, $2\leq \ell_\ast \leq i$ and $k\leq N-1$ since $\ell+k=i\le N$. 
If $\ell_\ast \ge k$ we have 
\begin{align}
& \ve\tau^{\frac12(\gamma-\frac23)+m}   \| I^{iABC} \|_{\alpha+i}
\lesssim  \ve \tau^{\frac{10}{3}-\frac{\gamma}{2}+m}\|w^k B\left(\frac{\phi^3 }{\J[\phi]^{\gamma+1}}\right)\|_{L^\infty} 
\|w^{\frac{ N+i -2k-2q+4 -2 \ell_\ast }{2}}\|_{L^\infty} \notag \\
& \ \ \ \  \tau^{\frac{1}{2} (\gamma-\frac{11}{3})}    \| w^{q-1}C_1H\|_{L^\infty}  
 \tau^{\frac{1}{2} (\gamma-\frac{11}{3})}  \|w^{\frac{\alpha+2\ell_\ast -N}{2}} C_2H\|_{L^2} \\
 & \lesssim \ve \tau^{\frac56(4-3\gamma)+ m}E^N 
 \lesssim \ve \tau^{m+\frac54\delta^\ast}E^N, \label{E:CASEII2}
\end{align}
where we have used~\eqref{E:PHILEMMA3} and the embeddings~\eqref{L21} and~\eqref{Linfty21}. Moreover, 
$\|w^{\frac{ N+i -2k-2q+2 -2 \ell_\ast }{2}}\|_{L^\infty}\lesssim 1$ since $N+i -2k-2q+4 -2 \ell_\ast = N-i\ge0$.

If $\ell_\ast < k$, as in Case I-2, we estimate $w^{\frac{\alpha+N-2k}{2}}B\left(\frac{\phi^3}{\J[\phi]^{\gamma+1}}\right)$ in the $L^2$-norm and the 
appropriately weighted terms $C_1H$ and $C_2H$ in the $L^\infty$-norm and obtain the same bound as in~\eqref{E:CASEII2}.

\underline{Case II-3: $\ell\geq1$ and $j\geq 1$.} In this case, we have $2\leq \ell_\ast\leq i+1-j$, $k\leq i-j-1$ and $\ell+k= i-j$. If $k=0$ we proceed as in Case II-1:
\begin{align}
 \ve \tau^{\frac12(\gamma-\frac23)+m}   \| I^{iABC} \|_{\alpha+i}
&\lesssim  \tau^{\frac{10}{3}-\frac{\gamma}{2}+1-\frac{j}{n}+m }  \|q_{1-\frac{j}{n}}\et \frac{\phi^3}{\J[\phi]^{\gamma+1}} \|_{L^\infty} 
\|w^{\frac{ N+i -2\ell -2 }{2}}\|_{L^\infty}\notag \\
& \ \ \ \ \quad \tau^{\frac{1}{2} (\gamma-\frac{11}{3})} \|w^{q-1}C_1H \|_{L^\infty} 
 \tau^{\frac{1}{2} (\gamma-\frac{11}{3})}  \|w^{\frac{\alpha+2\ell_\ast -N}{2}} C_2H\|_{L^2} \notag \\
 & \lesssim  \ve  \tau^{\frac56(4-3\gamma) + m+1-\frac{N}{n}} E^N \lesssim 
\ve \tau^{m+\frac54\delta^\ast+1}E^N
\end{align}
where we have used $N+i -2\ell_\ast - 2q+2= N+i -2\ell -2 \geq 0$ because $\ell = i-j \leq i-1$. 
If $k \geq 1$. We proceed as in Case II-2. We distinguish the two cases $\ell_\ast\ge k$ and $\ell_\ast<k$. Proceeding analogously to Case II-2, relying on the embeddings~\eqref{L21} and~\eqref{Linfty21}, and Lemma~\ref{L:PHILEMMA} we conclude
\be
\ve \tau^{\frac12(\gamma-\frac23)+m}   \| I^{iABC} \|_{\alpha+i} \lesssim \ve  \tau^{\frac56(4-3\gamma) + m+1-\frac{N}{n}} E^N \lesssim 
\ve \tau^{m+\frac54\delta^\ast+1}E^N. 
\ee
\end{proof}


\subsubsection*{Estimates for $\mathscr N_2[H]= 2\gamma \tau^m w\frac{\phi^3 }{g^2\J[\phi]^{\gamma+1}} \frac{M_g }{ r^2}  \left(  \pa_\tau H  + \frac{m}{\tau } H \right)   r\pa_r \left( \frac{H}{r}\right)$}


\begin{lemma}\label{L:N2} For any $i\in\{0,1,\dots,N\}$ the following bound holds
\begin{align}
\ve \tau^{\frac12(\gamma-\frac23)} \| \mathcal D_i \mathscr N_2[H] \|_{\alpha+i} \lesssim \sqrt \ve \tau^{m+\frac54\delta^\ast-\frac12} (E^N)^{\frac12}(D^N)^{\frac12}.
\end{align}
\end{lemma}


\begin{proof}
By the product rule~\eqref{E:product} and the identity $\r\left(\frac H r\right)=D_r H - 3\frac H r$ we have 
\begin{align}
&\ve \tau^{\frac12(\gamma-\frac23)}\D_i\left(    2\gamma \tau^m w\frac{\phi^3 }{g^2\J[\phi]^{\gamma+1}} \frac{M_g }{ r^2}  
\left(  \pa_\tau H  + \frac{m}{\tau } H \right)   r\pa_r \left( \frac{H}{r}\right) 
\right) \notag \\
& = 
2\ve \gamma \tau^{\frac12(\gamma-\frac23)+m}  \sum_{A_{1,2,3}\in\bar{\mathcal P}_{\ell_1,\ell_2,\ell_3}, A_4\in \mathcal P_{\ell_4} \atop \ell_1+\dots+\ell_4=i}
A_1\left(\frac{wM_g }{g^2 r^2}\right) A_2\left(\frac{\phi^3 }{\J[\phi]^{\gamma+1}}\right) A_3\left(D_r H - 2\frac H r\right) \notag \\
& \ \ \ \ \left(A_4\pa_\tau H + \frac m\tau A_4H\right) \label{E:N11}
\end{align}
 
{\em Case I. $\ell_3\le i-1$.}
Each 
factor 
in the last line of~\eqref{E:N11} can be estimated by
\begin{align}
\ve \tau^{\frac12(\gamma-\frac23)+m}  r^{n-\ell_1-2} \lv A_2\left(\frac{\phi^3 }{\J[\phi]^{\gamma+1}}\right) A_3\left(D_r H - 2\frac H r\right)
\left(A_4\pa_\tau H + \frac m\tau A_4H\right) \rv \notag \\
\ve \tau^{\frac12(\gamma-\frac23)+m+1-\frac{\ell_1+2}{n}}p_{1,-\frac{\ell_1+2}{n}} \lv q_1\et A_2\left(\frac{\phi^3 }{\J[\phi]^{\gamma+1}}\right) A_3\left(D_r H - 3\frac H r\right)
\left(A_4\pa_\tau H + \frac m\tau A_4H\right) \rv \label{E:N12}
\end{align}
We now distinguish several cases.

\noindent
{\em Case I-1. $\ell_3=\max\{\ell_2,\ell_3,\ell_4\}$.}
Assume first $\ell_2\le\ell_4\le\ell_3$ and $\ell_2\ge2$.
In this case the $\|\cdot\|_{\alpha+i}$ norm of~\eqref{E:N12} is bounded by
\begin{align}
&\ve \tau^{\frac12(\gamma-\frac23)+m+1-\frac{\ell_1+2}{n}} 
\|w^{\frac{N+i -2(\ell_2+\ell_3+\ell_4)}{2}}\|_{\infty}
\|w^{\ell_2} q_1\et A_2\left(\frac{\phi^3 }{\J[\phi]^{\gamma+1}}\right) \|_{\infty}  \notag \\
&\ \ \ \  \|w^{\ell_4-2}(A_4H_\tau-\frac m\tau A_4H)\|_{\infty} 
\| w^{\frac{\alpha+2\ell_3-N}{2}}A_3\left(D_r H - 3\frac H r\right)\|_{L^2 } \notag \\
& \lesssim \ve \tau^{\frac12(\gamma-\frac23)+m+1-\frac{\ell_1+2}{n} - 2\gamma - \frac{\ell_2}{n}+\frac{1}{2}(\frac53-\gamma)+\frac12(\frac83-\gamma)} E_N^{\frac12} D_N^{\frac12} \notag \\
& \lesssim \ve \tau^{m+\frac{17}6-\frac52 \gamma-\frac{i+2}{n}}  (E^N)^{\frac12}(D^N)^{\frac12} 
\lesssim \ve \tau^{m+\frac54\delta^\ast-\frac12}  (E^N)^{\frac12}(D^N)^{\frac12}, \label{E:B1}
\end{align}
where we have used~\eqref{E:PHILEMMA3} to bound $\|w^{\ell_2} q_1\et A_2\left(\frac{\phi^3 }{\J[\phi]^{\gamma+1}}\right) \|_{\infty}$,~\eqref{Linfty2} to 
bound $ \|w^{\ell_4-2}(A_4H_\tau-\frac m\tau A_4H)\|_{\infty} $ and~\eqref{L21} to bound $\| w^{\frac{\alpha+2\ell_3-N}{2}}A_3\left(D_r H - 3\frac H r\right)\|_{L^2}$. 
Note that we have used the bounds
$\tau^{\frac12(\gamma-\frac53)}\|\D_j H_\tau\|_{\alpha+j}+\tau^{\frac12(\gamma-\frac{11}3)}\|\D_j H\|_{\alpha+j}
\lesssim (E^N)^{\frac12}$,  $j\in\{0,1,\dots, N\}.$ The case $\ell_4\le\ell_2\le\ell_3$ is handled analogously. 

If $\ell_2\le1 $ and $\ell_3\ge3$ we then use~\eqref{E:PHILEMMA1} instead of~\eqref{E:PHILEMMA3} above and obtain the same bound. 
If $\ell_2\le1$ and $\ell_3\le2$ we then use~\eqref{E:PHILEMMA1} and~\eqref{Linfty1} instead of~\eqref{Linfty2} in the aboive argument and obtain the same 
upper bound. 

\noindent
{\em Case I-2. $\ell_4= \max\{\ell_2,\ell_3,\ell_4\}$ or $\ell_2 =\max\{\ell_2,\ell_3,\ell_4\}$.}
These cases can be treated similarly, with a similar case distinction, and with help of Lemma~\ref{L:PHILEMMA},~\eqref{Linfty1}--\eqref{Linfty2} and~\eqref{L21}.

{\em Case II. $\ell_3=i$.}
If $i\le N-1$ we proceed as in Case I. Assume now $\ell_3=N$.
Since in this case $A_3 = \bar \D_N$ the last line of~\eqref{E:N11} takes the form
\begin{align}
2\ve \gamma \tau^{\frac12(\gamma-\frac23)+m}  
\frac{wM_g }{g^2 r^2}\frac{\phi^3 }{\J[\phi]^{\gamma+1}}\left(\D_{N+1} H - 2\bar D_N\left(\frac H r\right)\right)
\left(\pa_\tau H + \frac m\tau H\right).
\end{align}
We take special notice of the additional power of $w$ available in this case.
Since $\lv \frac{M_g }{g^2 r^2}\rv \lesssim \tau^{1-\frac2n}p_{1,-\frac2n}\et q_1\et$ we can estimate the $\|\cdot\|_{\alpha+N}$-norm of the above expression
by 
\begin{align}
&\ve \tau^{\frac12(\gamma-\frac23)+m+1-\frac2n} 
\|q_{\frac{\gamma+3}{2}}\et \frac{\phi^3 }{\J[\phi]^{\gamma+1}}\left(\pa_\tau H + \frac m\tau H\right)\|_{\infty}
\|w^{\frac{\alpha+N+2}{2}}q_{-\frac{\gamma+1}{2}}\et \left(\D_{N+1} H - 2\bar D_i\left(\frac H r\right)\right)\|_{L^2 } \notag \\
& \lesssim \sqrt\ve \tau^{\frac12(\gamma-\frac23)+m-2\gamma+1-\frac{2}{n}+\frac12(\frac83-\gamma)+\frac12(\gamma+1)}   
(E^N)^{\frac12}(D^N)^{\frac12} \notag \\
& =\sqrt\ve  \tau^{\frac12(4-3\gamma)+\frac12-\frac2n + m}  (E^N)^{\frac12}(D^N)^{\frac12} 
\lesssim \sqrt\ve \tau^{m+\frac34\delta^\ast+\frac12}(E^N)^{\frac12}(D^N)^{\frac12} . \label{E:B2}
\end{align}
where we have used~\eqref{E:PHILEMMA1},
~\eqref{L22}, the bound
$\tau^{\frac12(\gamma-\frac53)}\|\D_j H_\tau\|_{\alpha+j}+\tau^{\frac12(\gamma-\frac{11}3)}\|\D_j H\|_{\alpha+j}
\lesssim (E^N)^{\frac12}$,  $j\in\{0,1,\dots, N\},$ 
and~\eqref{E:TOPORDERBOUND}.
The proof follows from~\eqref{E:B1} and~\eqref{E:B2}.
\end{proof}

The third, fourth, and fifth lines on the right-hand side of~\eqref{E:MATHFRAKK2NEW} are easily bounded by the same ideas as above, where we systematically use the product rule~\eqref{E:product}, Lemmas~\ref{L:PHILEMMA} and~\ref{L:DIQBOUNDS}. 

We only highlight the potential difficulties and how they can be overcome. 
In the 3rd term on the right-hand side of~\eqref{E:MATHFRAKK2NEW} there is nothing dangerous; we may write it as 
$\ve \gamma \tau^{m}  w\frac{\phi^2}{\J[\phi]^{\gamma+1}}\frac1{g^1} \DL(3\phia + \DL\phia  )H \frac Hr $ and then estimate its $\D_i$ derivative using the case-by-case analysis analogous to the above, the product rule~\eqref{E:product}, Lemma~\ref{L:PHILEMMA}, and the bounds~\eqref{E:BARDIDQ},~\eqref{E:D2QBOUND}. Note that the last two estimates
afford the presence of a power of $\frac{r^n}{\tau}$ in our bounds, which in turn has the regularising effect of diminishing any potential singularities due to negative powers of $r$ at $r=0$.

The 4-th term on the right-hand side of~\eqref{E:MATHFRAKK2NEW} looks potentially dangerous due to the presence of the negative powers of $r$ in $\frac{H^2}{r^2}$ and $\frac{H^3}{r^3}$. However, by~\eqref{E:QBOUND3} the bounds on $V \DL\J[\phia]$, $V\in\bar{\mathcal P}_i$, will afford a presence of a power of of $\frac{r^n}{\tau}$, thus averting all difficulties with potential singularities at $r=0$. 

Finally, the 5th term on the right-hand side of~\eqref{E:MATHFRAKK2NEW} contains $w'$ explicitly, and since $|\pa_r^\ell w'|\lesssim r^{n-\ell-1}$ in the vicinity of $r=0$ we have the above mentioned regularising effect. The estimates are then routinely performed using the the product rule~\eqref{E:product} and Lemma~\ref{L:PHILEMMA}.
The outcome is 
\begin{align}
\ve \tau^{\frac12(\gamma-\frac23)} \|\mathcal D_i \left(\text{ $j$-th line of~\eqref{E:MATHFRAKK2NEW}}\right) \|_{\alpha+i}
\lesssim \ve \tau^{m+\frac 54 \delta^\ast} (E^N)^{\frac12}(D^N)^{\frac12}, \ \ j=3,4,5. \label{E:B3}
\end{align}
To estimate the last line of~\eqref{E:MATHFRAKK2NEW} the  crucial insight is that
\begin{align*}
K_{-\gamma-1}[\tau^m\frac{H}{r}]  &+ (\gamma+1) \J[\phia]^{-\gamma-2} K_1[\tau^m\frac{H}{r}]  \\
& = (\gamma+1)(\gamma+2)\J[\phia]^{-\gamma-3} \left(\int_0^1(1-s)(1+s \frac{K_1[\theta]}{\J[\phia]}  )^{-\gamma-3} ds\right)  (K_1[\theta])^2,
\end{align*}
which follows from~\eqref{E:QUADRATIC} with $a=-\gamma-1$. Therefore the left-hand side above 
is in fact quadratic in $K_1[\theta]$. We now estimate the high-order derivatives of the above left-hand side using 
the product rule~\eqref{E:product}, Lemma~\ref{L:KABOUNDS}, and Remark~\ref{R:KABOUNDS}. By analogy to the proof of
Lemma~\ref{L:MH}, we obtain
\begin{align}
& \ve \tau^{\frac12(\gamma-\frac23)} \|\mathcal D_i \left(\tau^{-m}w\frac{\phi^2}{g^2 r} (K_{-\gamma-1}[\tau^m\frac{H}{r}] + (\gamma+1) \J[\phia]^{-\gamma-2} K_1[\tau^m\frac{H}{r}] )\right)   \|_{\alpha+i} \notag \\
&\lesssim \ve \tau^{m+\frac 54 \delta^\ast} (E^N)^{\frac12}(D^N)^{\frac12}, \ \ i\le N-1. \label{E:B4.0}
\end{align}
On the other hand, when $i=N$, $\D_NK_1[\theta]$ contains a top-order term $\D_{N+1}H$ in which case we have to use \eqref{L22} with loss of $\sqrt{\ve}$: 
\begin{align}
& \ve \tau^{\frac12(\gamma-\frac23)} \|\mathcal D_i \left(\tau^{-m}w\frac{\phi^2}{g^2 r} (K_{-\gamma-1}[\tau^m\frac{H}{r}] + (\gamma+1) \J[\phia]^{-\gamma-2} K_1[\tau^m\frac{H}{r}] )  
 \DL\J[\phia]\right) \|_{\alpha+i} \notag \\
& \lesssim \sqrt\ve \tau^{m+ \frac34\delta^\ast-\frac12} (E^N)^\frac12(D^N)^{\frac12}. \label{E:B4}
\end{align}

We next discuss the rest of terms in \eqref{E:NH1}. 
Using the product rule~\eqref{E:product} and~\eqref{E:HH} we have
\begin{align}
& \D_i\left(\frac{1}{\phi^2\phia^2 }\tau^m \frac{H^2}{r}\right) \notag \\
&= \tau^m \sum_{B_{1}\in \bar{\mathcal P}_{\ell_1}, B_2\in\mathcal P_{\ell_2} \atop \ell_1+\ell_2 = i, \  \ell_2\le i-1} c_i^{B_1B_2}
B_1\left(\frac{1}{\phi^2\phia^2 }\right)
B_2\left(\frac{H^2}{r}\right) + \tau^m\frac{1}{\phi^2\phia^2 } \D_i\left(\frac{H^2}{r}\right) \notag \\
& = \tau^m \sum_{B_{1}\in \bar{\mathcal P}_{\ell_1}, B_2\in\mathcal P_{\ell_2} \atop \ell_1+\ell_2 = i, \  \ell_2\le i-1} c_i^{B_1B_2}
B_1\left(\frac{1}{\phi^2\phia^2 }\right)
\sum_{B_{2,1}\in \bar{\mathcal P}_{\ell_{2,1}}, B_{2,2}\in\mathcal P_{\ell_{2,2}} \atop \ell_{2,1}+\ell_{2,2}=\ell_2}
c_{\ell_2}^{B_{2,1}B_{2,2}} 
B_{2,1}\left(\frac H r\right) B_{2,2} H \notag \\
& \ \ \ \ +  \tau^m\frac{1}{\phi^2\phia^2 } \sum_{A_{1,2}\in \mathcal P_{\ell_1,\ell_2} \atop \ell_1+\ell_2=i+1, \ \ell_1,\ell_2\le i} a_i^{A_1A_2}A_1 H A_2 H
\end{align} 
With this decomposition, Lemma~\ref{L:PHILEMMA} (applied with $b=0$), Lemma~\ref{L:BARDIQBOUNDS}, and the case-by-case analysis analogous to
the proof of Lemma~\ref{L:N2} we obtain the bound
\begin{align}
\tau^{\frac12(\gamma-\frac23)+m} \|\D_i \left(\frac{1}{\phi^2\phia^2 } \frac{H^2}{r}\right) \|_{\alpha+i}
&\lesssim 
\tau^{\frac12(\gamma-\frac23)+m-\frac83- \frac{i+2}{n} +\frac12(\frac53-\gamma)+\frac12(\frac83-\gamma)} (E^N)^{\frac12} (D^N)^{\frac12} \notag \\
& \lesssim \tau^{-\frac12 \gamma - \frac56-\frac{i+2}{n} + m}(E^N)^{\frac12} (D^N)^{\frac12} \notag \\
& \lesssim \tau^{m+\frac14\delta^\ast- \frac32-\frac34\frac{N+2}{n}}(E^N)^{\frac12} (D^N)^{\frac12} \notag \\
&= \tau^{m+\delta^\ast -\frac32-\frac32(\frac43-\gamma)}(E^N)^{\frac12} (D^N)^{\frac12} 
\label{E:B5}
\end{align}
where we have used $\frac{N+2}{n}= 2(\frac43-\gamma)-\delta^\ast$ at the last step. 
Similarly
\begin{align}
\tau^{\frac12(\gamma-\frac23)+2m} \|\D_i \left(\frac{1}{\phi^2\phia^3} \frac{H^3}{ r^2}\right) \|_{\alpha+i}
&
\lesssim  \tau^{\frac12(\gamma-\frac23)+2m-\frac{10}3- \frac{i+2}{n} +2\times \frac12(\frac53-\gamma)+\frac12(\frac83-\gamma)} E^N (D^N)^{\frac12}\notag 
\\
&  \lesssim \tau^{- \gamma -\frac23 -\frac{i+2}{n} + 2m}E^N (D^N)^{\frac12} \notag \\
 &\lesssim \tau^{2m+\frac12\delta^\ast -2 - \frac12\frac{N+2}n}E^N (D^N)^{\frac12}. \label{E:B6}
\end{align}

In order to estimate $ \ve \tau^m \frac{P[\phia] H^2}{\phia^2  r}$ (the last term on the right-hand side of~\eqref{E:NH1}) 
Using the product rule~\eqref{E:product}
\begin{align}
\D_i\left(\frac{P[\phia] H^2}{\phia^2  r}\right) =
\sum_{A_{1,2}\in \bar{\mathcal P}_{\ell_1,\ell_2}, A_3\in\mathcal P_{\ell_3} \atop \ell_1+\dots+\ell_3=i} c_i^{A_1A_2A_3}
A_1(\frac{1}{\phia}) A_2(\frac{P[\phia]}{\phia})A_3 \left(\frac{H^2}{r}\right).
\end{align}
For $A_3 \left(\frac{H^2}{r}\right)$, we may use \eqref{E:HH} to further decompose it into a linear combination of $A_{31}H A_{32} H$ where $A_{31}, A_{32} \in \mathcal P_{i_1},\mathcal P_{i_2} $, $i_1+i_2=\ell_3+1$, $i_1,i_2\le \ell_3$. 
We next recall~\eqref{E:PQFORMULA}. 
Applying~\eqref{E:PQBOUND} and the case-by-case analysis analogous to
the proof of Lemma~\ref{L:N2} we obtain the bound
\begin{align}
\ve \tau^{\frac12(\gamma-\frac23)+m} \|\D_i \left(\frac{P[\phia] H^2}{\phia^2  r}\right) \|_{\alpha+i}
& \lesssim \ve \tau^{\frac12(\gamma-\frac23)+m -2\gamma - \frac{i+2}{n}+ \frac12(\frac53-\gamma)+\frac12(\frac83-\gamma)} 
(E^N)^{\frac12}(D^N)^{\frac12} \notag \\
& \lesssim \ve\tau^{m+\frac54 \delta^\ast -\frac{3}{2}}(E^N)^{\frac12}(D^N)^{\frac12}. \label{E:B7}
\end{align}

{\bf Proof of Proposition~\ref{P:NONLINEARESTIMATE}.}
Bound~\eqref{E:NONLINEARESTIMATE} follows from~\eqref{E:MATHFRAKK2NEW}, Lemmas~\ref{L:N1}--\ref{L:N2}, bounds~\eqref{E:B3}--\eqref{E:B4},~\eqref{E:B5}--\eqref{E:B6}, and~\eqref{E:B7}.

\subsection{Source term estimates}

Recall the definition~\eqref{E:SOURCETERM} and formula~\eqref{E:SOURCETERM1}.

\begin{lemma}[Source term estimates]\label{L:SOURCETERMBOUNDS}
For any $i\in\{0,1,\dots N\}$ the following bounds hold
\begin{align}
\| \mathcal D_i\left( r\phi_0^{-2}R^\ve_{M,2}[\frac{\phi_1}{\phi_0}, \dots, \frac{\phi_M}{\phi_0}]\right)\|_{\alpha+i} &\lesssim \tau^{-\frac43+(M+1)\delta-\frac{i}{n}}, 
\label{E:BOUNDSOURCE1} \\
\| \mathcal D_i \left( rR_P^\ve\right) \|_{\alpha+i} &\lesssim  \tau^{-\frac43+(M+1)\delta-\frac{i-1}{n}},\label{E:BOUNDSOURCE2} \\
\| \mathcal D_i \mathscr S(\phia) \|_{\alpha+i} & \lesssim\ve^{M+1} \tau^{-m} \tau^{-\frac43+(M+1)\delta
-\frac{i}{n}}.
\label{E:BOUNDSOURCE3}
\end{align}
\end{lemma}

\begin{proof}
{\em Proof of~\eqref{E:BOUNDSOURCE1}.}
Recall that $R^\ve_{M,2}$ is defined through~\eqref{E:EPSILONEXPANSION}. 
A detailed look at the Taylor expansion of the function $R^\ve_{M,\nu}$ reveals
that for any $D\in\mathbb N$ there exist constants $c_{\alpha_1,\dots\alpha_M}^j$, $j\in\{1,\dots, D\}$ and a smooth function $r_{M,\nu}^{D,\ve}$ such that 
\begin{align}
R^\ve_{M,\nu}(x_1,\dots x_m) = \sum_{j=1}^D\ve^{j-1}\sum_{(\alpha_1,\dots,\alpha_M)\in\mathbb Z_{\ge0}^M \atop \sum_{i=1}^Mi\alpha_i = M+j} 
c^j_{\alpha_1,\dots,\alpha_M} x_1^{\alpha_1}\dots x_M^{\alpha_M}
+ r^{D,\ve}_{M,\nu}(x_1,\dots x_m)  \label{E:TAYLORR}
\end{align} 
where the remainder term $r^{D,\ve}_{M,\nu}(x_1,\dots x_m)$ has the property that all mixed derivatives 
$\pa_{x_1}^{\alpha_1}\dots\pa_{x_M}^{\alpha_m}r^{D,\ve}_{M,\nu}$ vanish at ${\bf 0}$ if $\sum_{i=1}^Mi \alpha_i\le M+D$.
Using the chain rule,~\eqref{E:TAYLORR}, and the bound
\[
 r^\ell \pa_r^\ell \left(\frac{\phi_i}{\phi_0}\right) \lesssim \tau^{i\delta},
\]
bound~\eqref{E:BOUNDSOURCE1} follows immediately.

\noindent
{\em Proof of~\eqref{E:BOUNDSOURCE2}.}
Recall that $R_P^\ve$ is defined through~\eqref{E:RPEPSILONDEF}.
To estimate the first term on the right-hand side of~\eqref{E:RPEPSILONDEF} we use the following crude bound
\begin{align}
&\sum_{k=0}^i\lv \D_i \left(
\frac{\phi_k\phi_{i-k}}{w^\alpha  r} \DL\left(w^{1+\alpha}\mathscr J[\phi_0]^{-\gamma}\frac{h_m}{m!}\right) \right)\rv \notag \\
& \lesssim  r^{-i} \sum_{\ell=0}^i \lv \rr^\ell
\left(\frac{\phi_k\phi_{i-k}}{w^\alpha  r} \DL\left(w^{1+\alpha}\mathscr J[\phi_0]^{-\gamma}\frac{h_m}{m!}\right) \right)\rv \notag \\
&\lesssim  
r^{-i} \sum_{\ell_1+\ell_2+\ell_3\le i} 
\lv \rr^{\ell_1}\left(\frac{1}{ r}\right) \rv
\lv\rr^{\ell_2}\left(\phi_k\phi_{i-k}\right)\rv 
\lv\rr^{\ell_3}\left(w^{-\alpha}\DL\left(w^{1+\alpha}\mathscr J[\phi_0]^{-\gamma}\frac{h_m}{m!} \right)\right)\rv \notag. 
\end{align}
By~\eqref{E:JSMALLHBOUND}, Proposition~\ref{P:MAINBOUNDPHI} we can bound the last line above by 
\begin{align}
 r^{-i-1}\tau^{\frac23+k\delta+\frac23 + (i-k)\delta-2\gamma + m\delta} p_{1,0}\et 
= \tau^{\frac43 + (i+m)\delta - 2\gamma - \frac{i+1}{n}} p_{1,-\frac{i+1}{n}}\et \lesssim \tau^{\frac43-2\gamma -\frac{i+1}{n} + M\delta}. \label{E:SB1}
\end{align}

To estimate the second term on the right-hand side of~\eqref{E:RPEPSILONDEF} we
first note that
\begin{align}
 \frac{R^\ve}{g^2(r) w^\alpha  r} \DL\left(w^{1+\alpha}\J[\phia]^{-\gamma} \right)
 =  -\gamma \frac{R^\ve w}{g^2(r)  r} \J[\phia]^{-\gamma-1} \DL\J[\phia]
 +  (1+\alpha) \frac{R^\ve}{g^2(r) } w' \J[\phia]^{-\gamma} . \label{E:SOURCEPRELIM1}
\end{align}
When we apply $\D_i$ to the first term on the right-hand side of~\eqref{E:SOURCEPRELIM1} we use
the product rule~\eqref{E:product} to break down the resulting expression into a linear combination of terms of the form
\begin{align}\label{E:SOURCEPRELIM2}
A_1\left(\frac{R^\ve}{wg^2(r)  r} \right) A_2\left(\J[\phia]^{-\gamma-1}\right) A_3 \DL\J[\phia]
\end{align}
with $A_1\in \mathcal P_{\ell_1}$ and $A_{2,3}\in\bar{\mathcal P}_{\ell_2,\ell_3}$ with $\ell_1+\ell_2+\ell_3 = i$. 
By Proposition~\ref{P:MAINBOUNDPHI} we have
\begin{align*}
\lv A_1\left(\frac{R^\ve w}{g^2(r)  r^2} \right) \rv \lesssim \tau^{\frac43 + M\delta - \frac{\ell_1+1}{n}} p_{\l, - \frac{\ell_1+3}{n}}\et.
\end{align*}
Combining the previous line with~\eqref{E:QBOUND3} and~\eqref{E:JQPOWERABOUND} we can bound the absolute value of~\eqref{E:SOURCEPRELIM2} by
\begin{align}
&\tau^{\frac43 + M\delta - \frac{\ell_1+1}{n}} p_{\l, - \frac{\ell_1+3}{n}}\et \tau^{-2\gamma-2-\frac{\ell_2}{n}} q_{-\gamma-1}\et 
\et^{-\frac{\ell_2}{n}}
\tau^{2-\frac{\ell_3}{n}} \et^{-\frac{\ell_3}{n}} q_2\et \left(p_{1,0}\et + p_{\l,-\frac2n}\et\right)
\notag \\
& \lesssim \tau^{\frac43-2\gamma + M\delta - {\frac{i+1}{n}}} 
p_{\l, - \frac{\ell_1+4}{n}}\et q_{-\gamma+1}\et \left(p_{1,-\frac{\ell_2+\ell_3}{n}}\et + p_{\l,-\frac{\ell_2+\ell_3+2}{n}}\et\right) 
\notag \\
& \lesssim \tau^{\frac43-2\gamma + M\delta - {\frac{i+1}{n}}}. \label{E:SB2}
\end{align}

To bound the last term on the right-hand side of~\eqref{E:RPEPSILONDEF} we can use the refined expansion~\eqref{E:TAYLORR}
to obtain the bound
\begin{align*}
\lv \rr^\ell R_{M,\gamma}^\ve  \rv
& \lesssim \tau^{(M+1)\delta} 
\end{align*}
where we have used~\eqref{E:PARTIALBARLAMBDAJBOUNDK},~\eqref{E:PARTIALTAUJPOWERMINUSK},~\eqref{E:REPSILONMNUPROPERTIES}, 
and~\eqref{E:RJDEF}, and the bound $|R_\J|\lesssim \tau^{M\delta}$.
By an analogous argument we have the bound $\lv\rr^\ell h_M\rv \lesssim \tau^{(M+1)\delta}$.
Using the last two bounds, the product rule, Proposition~\ref{P:MAINBOUNDPHI}, by analogy to the above we obtain
\begin{align}
\lv\D_i\left( \frac{\sum_{j=0}^{M-1} \ve^j \sum_{k=0}^j \phi_k\phi_{i-k}}{g^2(r) w^\alpha  r^2} \DL\left(w^{1+\alpha}\mathscr J[\phi_0]^{-\gamma}\left(\frac{h_{M}}{M!} + \ve R^\ve_{M,\gamma} \right)\right)\right)\rv 
\lesssim \tau^{\frac43 - 2\gamma +(M+1)\delta -\frac{i+1}{n}}. \label{E:SB3}
\end{align}
Since $\tau^{\frac43 - 2\gamma +(M+1)\delta -\frac{i+1}{n}}=\tau^{-\frac43+(M+1)\delta-\frac{i-1}{n}}$, the claim follows from~\eqref{E:SB1},~\eqref{E:SB2}, and~\eqref{E:SB3}.

\noindent
{\em Proof of~\eqref{E:BOUNDSOURCE2}.}
Since $\mathscr S(\phia) =  r\tau^{-m} S(\phia)$, from~\eqref{E:SOURCETERMEXACT} and~\eqref{E:BOUNDSOURCE1}--\eqref{E:BOUNDSOURCE2} we obtain
\[
\| \mathcal D_i \mathscr S(\phia) \|_{\alpha+i}  \lesssim \ve^{M+1} \tau^{-m} \tau^{-\frac43+(M+1)\delta
-\frac{i}{n}},
\]
where we have used the bound $\tau^{-\frac{i-1}{n}}\le \tau^{-\frac in}$, for $\tau\in(0,1]$.
\end{proof}

As a corollary, we obtain the following bound for the source terms.


\begin{proposition}[Source term estimates] \label{P:SOURCETERMESTIMATE}
Let $H$ be a solution of~\eqref{E:H1}. Then for any $i\in\{0,\dots, N\}$ the following bound holds
\begin{align}
\tau^{\gamma-\frac53} \left\vert\left( \mathcal D_i\mathscr S (\phia)  \ , \ \mathcal D_i H_\tau\right)_{\alpha+i}  \right\vert
 \lesssim \ve D^N +  \ve^{2M+1} \tau^{2(1-\frac2n)+ (2M-2)\delta + 2\delta^\ast-3\gamma-2m}.
 \label{E:SOURCETERMMAINBOUND}
\end{align}
\end{proposition}


\begin{proof}
By the previous lemma and the Cauchy-Schwarz inequality we obtain
\begin{align}
\tau^{\gamma-\frac53} \left\vert\left( \mathcal D_i\mathscr S (\phia)  \ , \ \mathcal D_i H_\tau\right)_{\alpha+i}  \right\vert
& \lesssim  \tau^{\frac12(\gamma-\frac23)} \| \mathcal D_i \mathscr S(\phia) \|_{\alpha+i}
 \tau^{\frac12(\gamma-\frac83)}\|\mathcal D_i H_\tau\|_{\alpha+i}  \notag \\
& \lesssim \ve^{M+1} \tau^{\frac12(\gamma-\frac23)-m-\frac43+(M+1)\delta-\frac{i}{n}}(D^N)^{\frac12} \notag \\
& = \ve^{M+1} \tau^{1+M\delta - \frac{i+2}{n}-\frac32\gamma-m} (D^N)^{\frac12}.
\end{align}
Since $i\le N$ and $\delta^\ast = \delta - \frac{N}{n}$, we can estimate the above expression by a multiple of
\[
\ve D^N +  \ve^{2M+1} \tau^{2(1-\frac2n)+ (2M-2)\delta + 2\delta^\ast-3\gamma-2m}.
\]
\end{proof}


\begin{remark}\label{R:M}
In order for the $\tau$-power to be integrable on $[0,T]$ we need to impose
$2(1-\frac2n)+ (2M-2)\delta + 2\delta^\ast-3\gamma-2m>-1$ which is equivalent to
$(M-1)\delta+\delta^\ast >\frac32(\gamma-1)+m+\frac2n$. Since $\delta^\ast>\frac2n$ by~\eqref{E:DELTALOWBOUND} and $0<\gamma-1<\frac13$, a sufficient condition for the previous estimate is for $M$ to be sufficiently large so that
\begin{align}
(M-1)\delta>\frac12 +m.
\end{align}
\end{remark}


\subsection{Proof of Theorem~\ref{T:MAINEXISTENCETHEOREM}}
\label{SS:PROOFOFKAPPATHEOREM}

We are ready to estimate $ \int_\kappa^\tau \mathcal R_i  d\tau' $ where  $\mathcal R_i$ is given in \eqref{E:Ri}. The only missing estimate  is the last term of \eqref{E:Ri}. By \eqref{E:PHIBOUNDAPRIORITAU}, \eqref{E:JPHIBOUNDAPRIORITAU}, and \eqref{E:GZEROZEROBOUND2}, we have 
$|\left(\frac{c[\phi]}{c[\phi_0]g^{00}}\right)_\tau|\lesssim \ve \tau^{\delta-1} $ and hence we obtain the following estimate of the last term of $\mathcal R_i$ in \eqref{E:Ri}: 
\be\label{E:FINALTERM}
\left\vert\frac12 \ve\gamma \tau^{\gamma-\frac53}  
\int_0^1  c[\phi_0]\left(\frac{c[\phi]}{c[\phi_0]g^{00}}\right)_\tau w^{1+\alpha} \left\vert \mathcal D_{j+1}H\right\vert^2 \,w^j r^{2}\,d r \right\vert\lesssim \ve \tau^{\delta-1} E^N
\ee

Combining Propositions~\ref{P:LLOW},~\ref{P:COMMESTIMATES}, Lemma~\ref{L:MH}, Propositions~\ref{P:NONLINEARESTIMATE},~\ref{P:SOURCETERMESTIMATE}, and~\eqref{E:FINALTERM}, we obtain
the bound
\begin{align}
\sum_{i=0}^N\lv \mathcal R_i \rv \lesssim &\left(\ve + \sqrt\ve \tau^{\delta^\ast} +\ve \tau^{m-\frac12+\frac52\left(\frac43-\gamma\right)}+\tau^{m+\delta^\ast-\frac32-\frac32(\frac43-\gamma)}\sqrt{E^N}\right) D^N \notag \\
& + \left(\sqrt\ve \tau^{\min\{\delta^\ast, \frac\delta2\}-\frac12}+\sqrt\ve\tau^{m+\frac34\delta^\ast}\sqrt{E^N}+\sqrt{\ve}\tau^{\frac\delta2 -\frac12}+\sqrt\ve \tau^{\min\{\delta^\ast,\frac\delta2\}-\frac12}\right)
\sqrt{E^N}\sqrt{D^N} \notag \\
& + \ve\tau^{\delta-1}E^N \notag \\
& +\ve^{2M+1} \tau^{2(1-\frac2n)+ (2M-2)\delta + 2\delta^\ast-3\gamma-2m}. \label{E:RIBOUND1}
\end{align}
We note that the last line of~\eqref{E:RIBOUND1}
Let $\bar\delta: = \min\{\delta^\ast, \frac\delta2\}>0$. With the choices 
\be\label{E:MFORMULA}
m=\frac52, \ \ M = \lfloor 1+\frac{2m+1}{2\delta} \rfloor +1=  \lfloor 1+\frac{3}{\delta} \rfloor +1,
\ee
we have $m-\frac12+\frac52\left(\frac43-\gamma\right)\ge0$, $m+\delta^\ast-\frac32-\frac32(\frac43-\gamma)\ge0$,
$2(1-\frac2n)+ (2M-2)\delta + 2\delta^\ast-3\gamma-2m\ge0$ (for the last bound we use~\eqref{E:DELTALOWBOUND} which implies $\delta^\ast>\frac2n$). Consequently, bound~\eqref{E:RIBOUND1} together with the a priori assumption $E^N\le 1$ implies
\begin{align}
\sum_{i=0}^N\lv \mathcal R_i \rv \lesssim & \sqrt\ve D^N 
+\sqrt{E^N} D^N + \sqrt\ve \tau^{\bar\delta-\frac12}\sqrt{E^N}\sqrt{D^N} + \sqrt\ve \tau^{\frac52}E^N\sqrt{D^N}   \notag \\
& + \ve\tau^{\delta-1}E^N + \ve^{2M+1} \notag \\
\lesssim & \sqrt\ve D^N  +\sqrt{E^N} D^N 
+ \sqrt\ve \tau^{2\bar\delta-1}E^N + \sqrt\ve \tau^5 (E^N) \notag \\
& + \ve\tau^{\delta-1}E^N + \ve^{2M+1},
\label{E:RIBOUND2}
\end{align}
where we have used the bound $2|ab|\le a^2+b^2$ to go from the first to the second estimate and the a priori bound $E^N\lesssim1$.

We now integrate the energy identity~\eqref{E:ENERGYIDENTITY} over the time interval $[\kappa,\tau]$, $\tau\le1$, and obtain
by virtue of Proposition~\ref{P:EQUIVALENCE}
conclude that there exists a universal constant $C_0>2$ such that
\begin{align}
S_\kappa^N(\tau) \le& \frac{C_0}{2} S_\kappa^N(\tau)\Big|_{\tau=\kappa} +\ve^{2M+1}(\tau-\kappa) \notag \\
& + C\left(\sqrt\ve +\sup_{\kappa\le\tau'\le\tau} \sqrt{E^N(\tau')} \right) \int_\kappa^\tau D^N(\tau')\,d\tau' \notag   \\
& + C \sqrt\ve \sup_{\kappa\le\tau'\le\tau}E^N(\tau') 
\int_\kappa^\tau \left( (\tau')^{2\bar\delta-1} +(\tau')^{\delta-1} + (\tau')^5\right) \,d\tau'. \label{E:ENERGYBOUND}
\end{align}

The positivity of $\delta$ and $\bar\delta=\min\{\delta^\ast, \frac\delta2\}$ guarantees that the last time integral on the right-most side of~\eqref{E:ENERGYBOUND} is finite and bounded independently of the constant $\kappa$. As a consequence of~\eqref{E:ENERGYBOUND} and Proposition~\ref{P:EQUIVALENCE} we conclude 
\begin{align}\label{E:ENERGYBOUND2}
S_\kappa^N(\tau) \le& \frac{C_0}2 S_\kappa^N(\tau)\Big|_{\tau=\kappa} + \ve^{2M+1}
+ C\sqrt \ve S_\kappa^N(\tau) + \left(S_\kappa^N(\tau)\right)^{\frac32}, \ \ \tau\in[\kappa,1].
\end{align}

Since by the local well-posedness theorem Proposition~\ref{T:LOCAL}, the map $\tau\mapsto S_\kappa^N(\tau)$ is continuous, a standard continuity argument applied to~\eqref{E:ENERGYBOUND2} implies that there exist $0<\sigma_\ast,\ve_\ast<1$  such that for any $0<\sigma<\sigma_\ast$ the following is true: for any choice of initial data $(H,H_\tau)\big|_{\tau=\kappa}$ satisfying 
\[
S_\kappa^N(H_0^\kappa,H_1^\kappa)(\tau=\kappa) \le \sigma^2,
\]
and any $0<\ve<\ve_\ast$
the solution exists on the interval $[\kappa,1]$ and satisfies the uniform-in-$\kappa$ bound
\be\label{E:FINALBOUND}
S_\kappa^N(\tau) \le C_0\left(\sigma^2+\ve^{2M+1}\right), \ \ \tau\in[\kappa,1].
\ee

\noindent
{\bf Justification of the a priori assumptions~\eqref{E:APRIORI} and~\eqref{E:BOOTSTRAP}.}
The size restrictions $0<\ve<\ve_\ast$, $0<\sigma<\sigma_\ast$ for $\ve_\ast,\sigma_\ast$ sufficiently small are necessary to ensure that the a priori assumptions~\eqref{E:APRIORI} and~\eqref{E:BOOTSTRAP} can be consistently recovered from the standard continuity argument. 
First let $(\ell_1,\ell_2)\neq (0,2)$. The embedding inequality \eqref{Linfty11} immediately gives
\be\label{E:AB}
\left\| (r\pa_r)^{\ell_1} (\tau\pa_\tau)^{\ell_2} \left( \frac{H}{r}\right) \right\|_{\infty} \lesssim \tau^{\frac{1}{2}(\frac{11}{3}-\gamma)} (E^N)^\frac12  \lesssim \ve + \sigma
\ee
for $ 0\leq \ell_1+\ell_2\leq 2,  \  (\ell_1,\ell_2)\neq (0,2)$. Now for $(\ell_1,\ell_2)= (0,2)$, it suffices to derive the bound for $\| \tau^2\pa_\tau^2(\frac{H}{r})\|_\infty$. Since $(H,\pa_\tau H)$ is a classical solution to ~\eqref{E:H2}, we may use the equation directly:  
\begin{align*}
\tau^2\pa_\tau^2(\frac{ H }{r})&=  - 2\frac{\tau g^{01}}{rg^{00}} {\tau\pa_r\pa_\tau H}  - \frac{2m}{g^{00}}  \tau \pa_\tau(\frac{ H}{r}) - \frac{d^2}{g^{00}}\frac{H}{r} 
+ \ve\gamma  \frac{ \tau^2 c[\phi]}{g^{00}}   \frac{1}{ w^{\alpha} r}\pa_r \left(w^{1+\alpha} \frac{1}{ r^2}\pa_r[ r^2 H] \right)  
\notag \\
& \quad -  \ve \frac{\tau^2 \mathscr N_0[H]}{r g^{00}}  + \frac{\tau^2}{rg^{00}}\left(\mathscr S(\phia) -\ve \mathscr L_{\text{low}} H + \mathscr N [H]\right). 
\end{align*}
Using \eqref{E:AB}, \eqref{E:GZEROZEROBOUND}, \eqref{E:GZEROONEBOUND} it is easy to see that the first three terms of the right-hand side are bounded by $\tau^{\frac{1}{2}(\frac{11}{3}-\gamma)} (E^N)^\frac12$. For the fourth term, by  \eqref{E:CEQUIVALENCE}, $ |\tau^2 c[\phi] |\lesssim \tau^{\delta +\frac{2}{n}}q_{-\gamma-1} (\frac{r^n}{\tau})$ and moreover by \eqref{infty}, we have $\| w \frac{ \D_2 H }{r} \|_\infty \lesssim \tau^{\frac{1}{2}(\frac{11}{3}-\gamma)} (E^N)^\frac12 $. Hence, the fourth term is also bounded by $\tau^{\frac{1}{2}(\frac{11}{3}-\gamma)} (E^N)^\frac12$. Similarly, by using the source estimates and estimating $L^\infty$ norm of $\frac{\mathscr N_0[H]}{r}$, $\frac{ \mathscr L_{\text{low}} H}{r}$ and $\frac{\mathscr N[H]}{r}$, we deduce that the second line is also bounded by $\tau^{\frac{1}{2}(\frac{11}{3}-\gamma)} (E^N)^\frac12$ and $\ve^{M+1}$. Therefore, we obtain 
\[
\left\| (\tau\pa_\tau)^{2} \left( \frac{H}{r}\right) \right\|_{\infty} \lesssim  \ve+\sigma. 
\]
It is now clear that there exists a universal $C$ so that if we choose $\sigma'=C(\ve_\ast+\sigma_\ast)$ the bound~\eqref{E:APRIORI} is consistent and can be justified by a classical continuity argument. The same comment applies to~\eqref{E:BOOTSTRAP}.


\section{Compactness as $\kappa\to0$ and proof of the Theorem~\ref{T:MAINMAIN}}
\label{S:COMPACTNESS}

Let $B^k$ be the Hilbert space generated by the norm: 
\begin{align*}
\|f\|_{B^k}:=\sum_{j=0}^{k}\|\D_j  f  \|_{\alpha+j} 
\end{align*}
namely $B^k=\overline{C_c(0,1)}^{\| \cdot \|_{B^k}}$. 
From the theory of weighted Sobolev spaces \cite{GuOp1,GuOp2}, we deduce that $B^k$ is compactly embedded into $B^{k-1}$ for $k\ge1$.

Let a family of given initial data $(H_0^\kappa,H_1^\kappa)$ satisfy the uniform bound
\be\label{E:initialc}
S_\kappa^N(H_0^\kappa,H_1^\kappa)(\tau=\kappa) <\sigma^2 \ \  \text{for each} \ \ \kappa\in(0, \frac12].
\ee
 In particular,  this gives the uniform bound of $\|(\kappa^{\frac12(\gamma-\frac{11}{3})} H_0^\kappa , \kappa^{\frac12(\gamma-\frac53)} H_1^\kappa)\|_{B^N\times B^N} <\sqrt{2}\sigma$. By compact embedding of $B^N$ into $B^{N-1}$, there exists a sequence of $\{\kappa_j \}_{j=1}^\infty$ such that $\kappa_j \rightarrow 0$ and $({\kappa_j}^{\frac12(\gamma-\frac{11}{3})}  H_0^{\kappa_j} ,  {\kappa_j}^{\frac12(\gamma-\frac53)} H_1^{\kappa_j})$ converge in $B^{N-1}\times B^{N-1}$. Fix such a sequence $\kappa_j$ and initial data $H_0^{\kappa_j}$ and $H_1^{\kappa_j}$.

Now let $(H^{\kappa_j}, \pa_\tau H^{\kappa_j})$ be the solution to the initial value problem~\eqref{E:H1} with initial data  $H_0^{\kappa_j}$ and $H_1^{\kappa_j}$  given by
Theorem~\ref{T:MAINEXISTENCETHEOREM}.  
Consider its well-defined trace at time $\tau=1$ (i.e. $t=0$ in the original coordinates). Since 
$S_{\kappa_j}^N(\tau=1)<C_0\left(\sigma^2+\ve^{2M+1}\right)$ for all $\kappa_j$, 
in particular we have the uniform bound
\begin{align*}
\sum_{j=0}^{N}\|\D_j \pa_\tau H^{\kappa_j}\big|_{\tau=1}\|_{\alpha+j}
+\sum_{j=0}^{N}\|\D_j H^{\kappa_j}\big|_{\tau=1}\|_{\alpha+j} < \sqrt{2C_0}\left(\sigma+\ve\right),
\end{align*}
where we have used the crude bound $\ve^{2M+1}\le\ve^2$.
Therefore, there exists a subsequence of $\kappa_j$, denoted by $\kappa_j$ again and $(H_0,H_1)\in B^{N}\times B^{N}$  so that  
\[
\lim_{j\to\infty} \|(H^{\kappa_j} \big|_{\tau=1}, \pa_\tau H^{\kappa_j}\big|_{\tau=1})-(H_0,H_1)\|_{B^{N-1}\times B^{N-1}} = 0.
\]

We now consider the solution of~\eqref{E:H1} with the final value $(H_0,H_1)$ at time $\tau=1$. By the local well-posedness theory obtained similarly as in Proposition~\ref{T:LOCAL}, 
there exists a unique solution $(H, \pa_\tau H)$ to~\eqref{E:H1} on a maximal interval of existence $(T,1]$ for some $T<1$. 

We claim that $T=0$. To see this, assume the opposite, i.e. $0<T<1$. Then for each $\kappa_j\in(0,\frac T2]$, consider the sequence of solutions $(H^{\kappa_j}, \pa_\tau H^{\kappa_j}  ) $ to~\eqref{E:H1} with given initial condition $H_0^{\kappa_j}$ and $H_1^{\kappa_j}$.   
Then, on the interval $[\frac T2,1]$ the sequence $(H^{\kappa_j} , \pa_\tau H^{\kappa_j})$ satisfies the 
uniform-in-$j$ bound~\eqref{E:FINALBOUND}. In particular, as $j\to\infty$, possibly along a subsequence, 
$(H^{\kappa_j}, \pa_\tau H^{\kappa_j})$ converges to some $(\bar H,\pa_\tau \bar H)$ in 
$C^0\left([\frac T2,1], B^{N-1}\times B^{N-1}\right)$ 
 and the resulting limit $(\bar H,\pa_\tau \bar H)$ is a classical solution of~\eqref{E:H1} on $[\frac T2,1]$. 
Since the final condition at $\tau=1$ has to coincide for $(\bar H,\pa_\tau\bar H)$ and
$(H,\pa_\tau H)$, by the uniqueness part of the local well-posedness theorem, $\bar H$ and $H$ coincide on $[\frac T2,1]$ which contradicts the assumption that $(T,1]$ is the maximal interval of existence for $H$. 

Therefore we have established the existence of a classical solution
\[
\phi(\tau,r) = \phia(\tau,r) + \tau^m\frac{H(\tau,r)}{r} = \tau^{\frac23} + \sum_{j=1}^M \ve^j \phi_j(\tau,r) + \tau^m\frac{H(\tau,r)}{r} 
\]
to~\eqref{E:BASICPDEPHI} on the space-time domain $(\tau,r)\in(0,1]\times[0,1]$. 
In particular, the leading order behavior of $\phi$ is driven by the dust solution $\phi_0=\tau^{\frac23}$ and
we have
\begin{align}
 1\lesssim \lv \frac{\phi}{\phi_0} \rv \lesssim 1, \ \ 1\lesssim\lv \frac{\J[\phi]}{\J[\phi_{0}]}\rv &\lesssim  1,  \ \ (\tau,r)\in(0,1]\times[0,1] ; \notag\\
 \lim_{\tau\to0^+}\frac{\phi}{\phi_0}=\lim_{\tau\to0^+}\frac{\J[\phi]}{\J[\phi_{0}]} &= 1. \notag
\end{align}
Claims~\eqref{E:CLAIM1}--\eqref{E:CLAIM2} follow easily by going back to the $(s,r)$-coordinate system, which in turn give~\eqref{E:CP1}--\eqref{E:CP3}.
This completes the proof of Theorem~\ref{T:MAINMAIN}.

\medskip 
{\bf Data at $s=0$.}
Note that the initial conditions~\eqref{E:INITIALCONDITIONS} that correspond to the obtained collapsing solution are now given by 
\begin{align*}
\chi_0(r)&=\phi(1,r)=1 + \sum_{j=1}^M\ve^j \phi_j(1,r) + \frac{H(1,r)}{r}, \\ 
\chi_1(r)&=-\frac1{g(r)}\phi_\tau(1,r) = -\frac2{3g(r)} + \sum_{j=1}^M \ve^j\pa_\tau \phi_j(1,r) + \frac{mH(1,r)+\pa_\tau H(1,r)}{r}.
\end{align*}

In particular, by the smallness of weighted norms of $H$ and Hardy-Sobolev embeddings, we conclude
\begin{align}\label{E:INITIALEPS}
\|\chi_0-1\|_{C^2([0,1])} + \|\chi_1+\frac2{3g(r)}\|_{C^2([0,1])} = O(\sigma+\ve).
\end{align}
We may now express the 
initial density $\tilde\rho_0$ and the initial velocity vector field $\tilde{\bf u}_0$ (at time $s=0$) in Eulerian variables. 
 Let $Y=\chi_0(r)y=\phi(1,r)y$, $\bar\chi_0(R)=\chi_0(r)$, $R=|Y|=r\chi_0(r)$.
By~\eqref{E:ETAINITIAL}
\begin{align*}
\tilde {\bf u}_0(Y) & = \frac{\chi_1(r)}{\chi_0(r)} Y  = -\frac1{g(r)} \frac{\pa_\tau\phi(1,r)}{\phi(1,r)} Y 
= -\frac2{3g(r)}Y -\frac{\pa_\tau(\phi-\phi_0)(1,r)}{g(r)\phi(1,r)}  Y \\
& =  \left( -\frac{2}{3g(r)}+O(\sigma+\ve) \right)Y \\
& = \left( -\frac{2}{3g(\frac R{\bar\chi_0(R)})}+O(\sigma+\ve) \right)Y, \ \ Y\in B_{\chi_0(1)}(0).
\end{align*}
By~\eqref{E:DENSITYFORMULA} we have
\be\label{E:TILDERHOZERO}
\tilde\rho_0(Y) = w^\alpha(\frac{R}{\bar\chi_0(R)}) \left(\J[\chi_0]\circ\frac{R}{\bar\chi_0(R)}\right)^{-1}.
\ee
From~\eqref{E:SCALING}
we then conclude 
\begin{align*}
{\bf u}_0(x) & = \ve^{-\frac3{2(4-3\gamma)}}\frac1{g_\ve(|x|)}\left(-\frac{2}{3}+O(\ve+\sigma)\right) x, \ \ 
g_\ve(R)=g(\frac{R}{\epsilon^{\frac1{4-3\gamma}}}), \\
\rho_0(x)  & = \ve^{-\frac3{4-3\gamma}}\tilde\rho_0(\frac{x}{\ve^{\frac1{4-3\gamma}}}), \\
\Omega(t)\Big|_{t=0} & = B_{\ve^{\frac1{4-3\gamma}}\chi_0(1)}(0).
\end{align*}

Since by~\eqref{E:RHOTAYLORINTRO} $w^\alpha(r) = 1-  c r^n + o_{r\to0}(r^n)$ for some $c>0$ in the vicinity of $r=0$, we conclude that we have the expansion
\begin{align}\label{E:RHOZEROFLAT}
\tilde\rho_0(Y) = \left(1- \tilde c \frac{R^n}{\bar\chi_0(R)^n} + o_{R\to0}(R^n)\right) \left(\J[\chi_0]\circ\frac{R}{\bar\chi_0(R)}\right)^{-1}.
\end{align}
This formula in view of~\eqref{E:INITIALEPS} gives a quantified sense in which the initial density is {\em flat} about the origin.

\medskip
{\bf The Eulerian description of collapsing solutions.} 
Let $0<\tau\leq 1$ be fixed. Note that $\J[\phi]>0$ and the Eulerian density is given by 
\[
\tilde \varrho (\tau, \phi(\tau,r) r) = \frac{w^\alpha(r)  }{\J[\phi](\tau,r)} 
\]
where we have written $\tilde\varrho(\tau, \phi(\tau,r) r) = \tilde \rho (\frac{1-\tau }{g(r)}, \phi(\tau,r) r) = \tilde \rho (s, \chi(s,r) r) $.  
Let $\tilde R:= \phi(\tau,r) r$. Then since $\J[\phi]>0$, there exists the inverse mapping $ r =\tilde r(\tau, \tilde R)$ such that $ \tilde r (\tau, \phi(\tau,r) r) = r $ for all $r\in [0,1]$. We may rewrite the Eulerian density 
\[
\tilde \varrho (\tau, \tilde R) = \frac{w^\alpha( \tilde r (\tau, \tilde R))  }{\J[\phi](\tau,\tilde r (\tau, \tilde R) )}  = 
\frac{w^\alpha( \frac{\tilde R}{\phi(\tau, \tilde R) })  }{\J[\phi](\tau, \frac{\tilde R}{\phi(\tau, \tilde R) })}    \ \text{ for } \ 0\leq  \tilde R \leq \phi(\tau,1)
\]
where we have written $\phi(\tau, \tilde R)=\phi(\tau, r)$ for $\tilde R = \phi(\tau,r)r$. 
By our construction, $\lim_{\tau \to 0^+}\frac{\phi }{\phi_0}=1 $ and $\lim_{\tau \to 0^+}\frac{\J[\phi] }{\J[\phi_0]}=1$ for all $\tilde R \in [0,\phi(\tau,1)]$. Therefore, we deduce that 
\[
\tilde \varrho (\tau, \tilde R) =  \tilde \rho (\frac{1-\tau }{g(r)}, \tilde R)  \approx_{\tau\to0^+}
  \frac{w^\alpha(\frac{\tilde R }{\tau^\frac23})}{\tau(\tau + \frac23 |M_g(\frac{\tilde R }{\tau^\frac23})|)}
\]
The right-hand side is nothing but the density driven by the dust profile (1.34) written in $\tau$ coordinate.
Switching back to the $(s,r)$-coordinate system, this is precisely in agreement with~\eqref{E:CP1} and highlights the role of the dust profile in our collapse.

\section{Acknowledgments}
Y. Guo is indebted to D. Christodoulou for bringing his attention to the study of the Newtonian gravitational 
collapse in 1996. He acknowledges the support of the NSF grant DMS-1810868, 
Chinese NSF grant 10828103 and BICMR. M. Had\v zi\'c acknowledges the support of the EPSRC grant 
EP/N016777/1. J. Jang acknowledges the support of the NSF grant DMS-1608494.

\appendix
\renewcommand{\theequation}{\Alph{section}.\arabic{equation}}
\setcounter{theorem}{0}\renewcommand{\theorem}{\Alph{section}.\??arabic{prop}}

\section{Vector field classes $\mathcal P$ and $\bar{\mathcal P}$}\label{A:proofs}

In what follows, we present key lemmas that highlight the importance of $\mathcal P_i$ and $\bar{\mathcal P}_i$.


The first key is that the original $\D_i$'s control those admissible vector fields in $L^2$ sense so that the members of $\mathcal P_i$ and $\bar{\mathcal P}_i$ can be freely used in the energy estimates. A clue is in the divergence structure of $D_r$ that grants an extra control of $\frac{1}{r}$. More precisely, we have   

\begin{lemma}\label{L:XtoXover} Let $i\in \mathbb N$ be given. Then we have the following identity. 
\be\label{XtoXover}
\D_i X =  r \pa_r \bar{\D}_{i-1} \left(\frac{X}{r}\right) + (i+2) \bar{\D}_{i-1} \left(\frac{X}{r}\right)
\ee
Moreover, we have the following estimate: 
\be\label{XtoXoverB}
\begin{split}
 \int_0^\frac{3}{4}   \left( \left|  r \pa_r \bar{\D}_{i-1} \left(\frac{X}{r}\right) \right|^2 + \left| \bar{\D}_{i-1} \left(\frac{X}{r}\right)\right|^2  \right)  r^2 \chi^2 d r \lesssim \int_0^\frac{3}{4} |\D_i X|^2  r^2 \chi^2 d r 
\end{split}
\ee 
where $\chi \geq 0$ is a smooth cutoff function satisfying $\chi=1$ on $[0,\frac12]$, $\chi=0$ on $[\frac34,1]$, and $\chi'\leq 0$. 
\end{lemma}

\begin{proof}

We first establish \eqref{XtoXover}. The proof is based on the induction on $i$. 
First observe 
\[
\D_1 X= D_r X  =\pa_r X + 2 \frac{X}{r} = \r \left(\frac{X}{r}\right) + 3 \frac{X}{r}
\]
and hence \eqref{XtoXover} holds for $i=1$. Suppose  \eqref{XtoXover} is valid for all $i\leq \ell$. If $\ell$ is even, $\D_{\ell+1} = D_r \D_\ell $, $\bar{\D}_\ell = D_r \bar{\D}_{\ell-1}$, and 
$  r\pa_r \bar{\D}_{\ell-1} \left(\frac{X}{r}\right) + (\ell+2) \bar{\D}_{\ell-1} \left(\frac{X}{r}\right)  =  r \bar{\D}_\ell  \left(\frac{X}{r}\right)  + \ell \bar{\D}_{\ell-1} \left(\frac{X}{r}\right) $. Then by the induction hypothesis, we deduce 
\[
\begin{split}
\D_{\ell+1}X & = (\pa_r + \frac{2}{r}) \left( r\pa_r \bar{\D}_{\ell-1} \left(\frac{X}{r}\right) + (\ell+2) \bar{\D}_{\ell-1} \left(\frac{X}{r}\right) \right) \\
&= (\pa_r + \frac{2}{r})  \left(  r \bar{\D}_\ell  \left(\frac{X}{r}\right)  + \ell \bar{\D}_{\ell-1} \left(\frac{X}{r}\right) \right)\\
&=\r \bar{\D}_\ell  \left(\frac{X}{r}\right)  + (\ell + 3)\bar{\D}_\ell \left( \frac{X}{r} \right)
\end{split}
\]
which verifies \eqref{XtoXover} for $i=\ell+1$. If $\ell$ is odd, $\D_{\ell+1} = \pa_r \D_\ell $, $\bar{\D}_\ell = \pa_r \bar{\D}_{\ell-1}$, and $  r\pa_r \bar{\D}_{\ell-1} \left(\frac{X}{r}\right) + (\ell+2) \bar{\D}_{\ell-1} \left(\frac{X}{r}\right)  =  r \bar{\D}_\ell  \left(\frac{X}{r}\right)  + (\ell +2 )\bar{\D}_{\ell-1} \left(\frac{X}{r}\right) $. Then 
\[
\begin{split}
\D_{\ell+1}X & =\pa_r \left( r\pa_r \bar{\D}_{\ell-1} \left(\frac{X}{r}\right) + (\ell+2) \bar{\D}_{\ell-1} \left(\frac{X}{r}\right) \right) \\
&= \pa_r \left(  r \bar{\D}_\ell  \left(\frac{X}{r}\right)  + (\ell +2 )\bar{\D}_{\ell-1} \left(\frac{X}{r}\right) \right)\\
&=\r \bar{\D}_\ell  \left(\frac{X}{r}\right)  + (\ell + 3)\bar{\D}_\ell \left( \frac{X}{r} \right)
\end{split}
\]
which also verifies \eqref{XtoXover} for $i=\ell+1$. This finishes the proof of \eqref{XtoXover}. 

We will prove \eqref{XtoXoverB} only since the other can be shown similarly. To this end, we compute the square integral of $\D_i X$ by using \eqref{XtoXover}.  Then 
\[
\begin{split}
 \int_0^\frac{3}{4} |\D_i X|^2  r^2 \chi^2 d r  & =  \int_0^\frac{3}{4} \left( r\pa_r \bar{\D}_{i-1} \left(\frac{X}{r}\right) + (i+2) \bar{\D}_{i-1} \left(\frac{X}{r}\right) \right)^2  r^2 \chi^2 d r  \\
 & =   \int_0^\frac{3}{4} \left[ \left( r\pa_r \bar{\D}_{i-1} \left(\frac{X}{r}\right) \right)^2 + (i+2)^2 \left( \bar{\D}_{i-1} \left(\frac{X}{r}\right) \right)^2 \right]  r^2 \chi^2 d r \\
 &\quad+ 2(i+2)   \int_0^\frac{3}{4}  \pa_r \bar{\D}_{i-1} \left(\frac{X}{r}\right)  \bar{\D}_{i-1} \left(\frac{X}{r}\right)    r^3 \chi^2 d r \\
 &=   \int_0^\frac{3}{4} \left[ \left( r\pa_r \bar{\D}_{i-1} \left(\frac{X}{r}\right) \right)^2 + (i+2)(i-1) \left( \bar{\D}_{i-1} \left(\frac{X}{r}\right) \right)^2 \right]  r^2 \chi^2 d r \\
 &\quad- 2(i+2)   \int_0^\frac{3}{4}   \left(\bar{\D}_{i-1} \left(\frac{X}{r}\right) \right)^2    r^3 \chi\chi' d r 
\end{split}
\]
Since $\chi'\leq 0$, the result follows for $i> 1$. For $i=1$, we need to show that $\frac{X}{r}$ is bounded by $\D_1 X$.  Observe that 
\[
\begin{split}
 \int_0^\frac{3}{4} |\D_1 X|^2  r^2 \chi^2 d r  & =  \int_0^\frac{3}{4} \left( \pa_r X +2\frac{X}{r} \right)^2  r^2 \chi^2 d r  \\
 & =   \int_0^\frac{3}{4}\left[ (\pa_r X)^2+ 4\frac{X^2}{ r^2} \right]  r^2 \chi^2 d r + 4  \int_0^\frac{3}{4} \pa_r X X    r \chi^2 d r \\
 &=   \int_0^\frac{3}{4}\left[ (\pa_r X)^2+ 2\frac{X^2}{ r^2} \right]  r^2 \chi^2 d r - 4  \int_0^\frac{3}{4} X^2   r \chi \chi' d r
 \end{split}
\]
from which   \eqref{XtoXoverB} follows also for $i=1$. 
\end{proof}

\begin{remark}
We note that the cut-off $\chi$ in the above proof can be replaced by any nonnegative $C^1$-function supported on $[0,1]$ such that  $\chi'\le0$.
In particular, choosing $\chi(r)  = q_{-\gamma-1}\et w^{\alpha+i+1}$ as a corollary we obtain 
\begin{align}\label{E:TOPORDERBOUND}
\int_0^1 w^{\alpha+i+1} q_{-\gamma-1}\et \lv {\bar \D}_i\left(\frac X r\right)\rv^2  r^2\,d r
\lesssim \int_0^1 w^{\alpha+i+1} q_{-\gamma-1}\et \lv \D_{i+1} X\rv^2  r^2\,d r.
\end{align}
\end{remark}

An important consequence of Lemma \ref{L:XtoXover} is the following: 

\begin{lemma}\label{L:control} Suppose $\D_i X$ is bounded in $L^2([0,\frac34], r^2d r)$. Then we have the following estimate: 
\be
\sum_{ \mathfrak D_i   \in \mathcal P_i}\int_0^\frac{3}{4} | \mathfrak D_i X|^2  r^2\chi^2 d r \lesssim \int_0^\frac{3}{4} |\D_i X|^2  r^2\chi^2 d r 
\ee
where $\chi$ is the same cutoff function in Lemma \ref{L:XtoXover}. 
\end{lemma}

\begin{proof}
If there is no $\frac1 r$ in a given $\mathfrak D_i$,  we are done since it is the same as $\D_i$. Suppose $\frac1 r$ appears in $\mathfrak D_i$. Then we may write $\mathfrak D_i X = \bar\D_{i-j-1} \frac1 rY_j X$ for some $Y_j\in \mathcal P_j$. Apply Lemma \ref{L:XtoXover} to get the $L^2$ bound for $\mathfrak D_i X$ by $\D_{i-j} Y_j X$. If $Y_j$ does not have $\frac1 r$, we are done. If it does, then we repeat the same argument by writing $\D_{i-j} Y_j = \bar\D_{i-j-k-1} \frac1 rZ$ for some $Z \in \mathcal P_k$ and applying Lemma \ref{L:XtoXover} etc. The repetition ends in at most $\lfloor\frac{i+1}{2}\rfloor$ steps. 
\end{proof}

Another appealing feature of $\mathcal P_i$ and $\bar{\mathcal P}_i$ is that they give rise to an algebraic structure via the following Leibniz rule: 

\begin{lemma}[Product rule]\label{L:product}  
Let $i\in\mathbb N$ be given. 
\begin{enumerate}
\item[(a)]
For any $A\in \mathcal P_i$ the following identity holds:
\be\label{E:product}
A \left( f g \right)= \sum_{k=0}^i \sum_{ B \in \mathcal P_k \atop C\in \bar{\mathcal P}_{i-k}} c^{ABC}_{k} \ Bf\, Cg,
\ee
for some real-valued constants $c^{ABC}_{k}$.
\item[(b)]
For any $A\in \bar{\mathcal P}_i$ the following identity holds:
\be\label{E:PRODUCTRULEDBAR}
A \left( f g \right)= \sum_{k=0}^i \sum_{ B \in \bar{\mathcal P}_k \atop C\in \bar{\mathcal P}_{i-k}} {\bar c}^{ABC}_{k} \ Bf\, Cg,
\ee
for some real-valued constants ${\bar c}^{ABC}_{k}$.
\end{enumerate}
\end{lemma}

\begin{proof}
The proof is based on the induction on $i$. We start with \eqref{E:product}. 
Let $i=1$. Then $\mathcal P_1 =\{D_r, \frac1 r\}$. Since we have 
\[
\D_1 (fg) = D_r (fg) = (D_r f) g + g \pa_r g = (\D_1 f) g + f \bar{\D}_1 g  
\]
and 
\[
\frac{1}{r}(fg) =( \frac{1}{r}f ) g
\]
\eqref{E:product} holds for $i=1$. Suppose  \eqref{E:product}  is true for all $i\leq \ell$. We will show that it is true for $i=\ell +1$. 

First let  $\ell$ be even. Then $A\in \mathcal P_{\ell+1}$ is either $A=D_r A'$ or $A=\frac1 rA'$ for some $A' \in \mathcal P_{\ell}$. By induction hypothesis, it suffices to show that  $D_r (B' f C' g )$ and $\frac1 r(B' f C' g )$ for $B'\in \mathcal P_{k'}$ and $C'\in \bar{\mathcal P}_{\ell-k'}$ can be written as a linear combination of $(Bf )(Cg)$ for some $B\in \mathcal P_k$ and $C\in \bar{\mathcal P}_{\ell+1-k}$. If $k'$ is even, we write $D_r (B' f C' g ) = (D_r B' f)  (C' g) + (B' f) (\pa_r C' g )$  and  $\frac1 r(B' f C' g ) = (\frac1 rB' f ) C' g$. Since $D_r B' \in \mathcal P_{k'+1}$, $\frac1 rB'\in \mathcal P_{k'+1}$ and $\pa_r C' \in \bar{\mathcal P}_{\ell+1-(k'+1)}$, both expressions are in the desirable form. If $k'$ is odd, we may write $D_r (B' f C' g ) = (\pa_r B' f)  (C' g) + (B' f) (D_r C' g )$  and  $\frac1 r(B' f C' g ) =  (B' f ) (\frac1 rC' g)$ so that they are in the desirable form. 

Now let $\ell$ be odd. Then $A=\pa_r A'$ where $A'\in \mathcal P_\ell$. We will show that $\pa_r(B' f C' g) $ for  $B'\in \mathcal P_{k'}$ and $C'\in \bar{\mathcal P}_{\ell-k'}$ can be written into a desirable form. As before, we consider $k'$ even, odd separately. If $k'$ is even, we write 
\[
\pa_r(B' f C' g) =( D_r B' f ) (C'g) + (B'f) (D_r C'g) - 4 (B'f )(\frac1 rC' g)
\]
Note that $D_r B' \in \mathcal P_{k'+1}$ and $D_r C', \  \frac{1}{r} C' \in \bar{\mathcal P}_{\ell+1-k'}$, and hence it has the desirable form.  If $k'$ is odd, $
\pa_r( B'f C' g) = (\pa_r B' f)  C' g  +B'f  (\pa_r C' g)$. The result follows since $\pa_r  B' \in \mathcal P_{k'+1}$ and $\pa_r C'  \in \bar{\mathcal P}_{\ell+1-k'}$. 

\eqref{E:PRODUCTRULEDBAR} is an easy consequence of  \eqref{E:product}, because any  $A\in \bar{\mathcal P}_i$ can be written as $A = A ' \pa_r$ for some $A' \in \mathcal P_{i-1}$. Write $A(fg) = A' ((\pa_r f) g) + A'((\pa_r g) f)$ and apply \eqref{E:product} with $A'$ to obtain the desired result. 
\end{proof}


\begin{lemma}[Chain rule]\label{L:CHAINRULE}
Let $a\in\mathbb R, i\in\mathbb N$ be given and fix a vectorfield $W\in\bar{\mathcal P}_i$.  Then for any sufficiently smooth $f$ the following identity holds
\begin{align}\label{E:CR}
W(f^a) = \sum_{k=1}^i f^{a-k} \sum_{i_1+\dots i_k=i \atop W_j\in \bar{\mathcal P}_{i_j}} c_{k,i_1,\dots,i_k} \prod_{j=1}^k W_j f
\end{align}
for some real constants $c_{k,i_1,\dots,i_k}$.
\end{lemma}


\begin{proof}
The proof relies on an induction argument. Let $i=1$. Then $W=\pa_r$ and 
$\pa_r (f^a) = a f^{a-1} \pa_r f $ which verifies \eqref{E:CR}. Suppose \eqref{E:CR} is true for all $i\leq \ell$. Then we will show that it is true for $i=\ell+1$. First let $\ell$ be even. Then $W\in \bar{\mathcal P}_{\ell+1}$ can be written as $W=\pa_r W'$ for some $W'\in \bar{\mathcal P}_{\ell}$. By induction hypothesis, 
\[
\begin{split}
W(f^a)=\pa_r W' (f^a) = \sum_{k=1}^{\ell}(a-k)  f^{a-k-1} \pa_r f \sum_{i_1+\dots i_k=\ell \atop W'_j\in \bar{\mathcal P}_{i_j}} c_{k,i_1,\dots,i_k}  \prod_{j=1}^k W'_j f \\
+  \sum_{k=1}^{\ell}  f^{a-k}  \sum_{i_1+\dots i_k=\ell \atop W'_j\in \bar{\mathcal P}_{i_j}} c_{k,i_1,\dots,i_k} \pa_r\Big( \prod_{j=1}^k W'_j f \Big)
\end{split}
\]
The first sum of the right-hand side is a linear combination of $f^{a-k'} \prod_{j=1}^{k'} W'_j f$ where $2\leq k'\leq \ell+1$, $W_j =W'_j$ for $j\leq k'-1$, $W_{k'}= \pa_r $ and $W_j\in  \bar{\mathcal P}_{i_j}$ with $i_1+\dots i_{k'}=\ell+1 $. Hence it is in the desirable form. For the second sum, no increase in the number of the product occurs, but the number of derivatives increases by one. Note that $\pa_r\Big( \prod_{j=1}^k W'_j f \Big) =\sum_{m=1}^k \prod_{j=1}^k W^m_j f $ where $W^m_j = W'_j$ for $j\neq m$ and $W^m_j= \pa_r W'_j$ for $j=m$. Now if the corresponding $i_m$ is even, $\pa_r W'_m \in \bar{\mathcal P}_{i_m+1}$. If $i_m$ is odd, we write 
$\pa_r W'_m = D_r W'_m - \frac{2}{r} W'_m$ such that each operation belongs to $\bar{\mathcal P}_{i_m+1}$. In both cases, we have $i_1+\dots i_k =\ell+1$ and hence \eqref{E:CR} is valid for $i=\ell+1$. Next let $\ell$ be odd. Then $W\in \bar{\mathcal P}_{\ell+1}$ can be written as either $W=D_r W'$ or $W=\frac1 r w'$ for some $W'\in \bar{\mathcal P}_{\ell}$. We consider $W=\frac1 r w'$ only since the other case follows similarly by combining the previous cases. Write 
\[
W(f^a)=\frac1 r w' (f^a) = \sum_{k=1}^{\ell}f^{a-k} \sum_{i_1+\dots i_k=\ell \atop W'_j\in \bar{\mathcal P}_{i_j}} c_{k,i_1,\dots,i_k}   \frac1 r\prod_{j=1}^k W'_j f 
\]
Now we claim that $\frac1 r\prod_{j=1}^k W'_j f  =  \prod_{j=1}^k W_j f  $ for some $W_j \in \bar{\mathcal P}_{i_j}$ with $i_1+\dots i_k =\ell+1$. To this end, we first observe because $\ell$ is odd, there exists at least one index $m$, $1\leq m\leq k$ whose corresponding $i_m$ is odd. Now let $W_j = W'_j$ for $j\neq m$ and $W_m=\frac1 r w'_m$. Then it is easy to see that all properties are satisfied so that the expression has its desirable form required by \eqref{E:CR} for $i=\ell+1$. 
\end{proof}

The next lemma implies that $\pa_r^i X$ can be expressed as a linear combination of admissible operations belonging to $\mathcal P_i$. Note that the other way around is not true in general: for instance, $D_r $ can't be expressed in terms of $\pa_r $ only. Hence, our energy built upon $\D_i$'s controls $\pa_r^i $'s as well as $\mathcal P_i$'s. 

\begin{lemma}\label{L:etaiX} Let $i\in\mathbb N$ be given. 
\be\label{E:dd}
\pa_r^i X = \sum_{A  \in \mathcal P_i}  c_{i}^{A} A X 
\ee
where $c_{i}^{A}$'s are constants. 
\end{lemma} 


\begin{proof}
 The proof is based on the induction on $i$. Let $i=1$. Then we have 
\[
\pa_r X = D_r X -2 \frac{1}{r} X
\]
Since $D_r$ and $\frac{1}{r}$ belong to $\mathcal P_1$, \eqref{E:dd} holds. Suppose \eqref{E:dd} is true for all $i\leq \ell$. We will show that \eqref{E:dd} holds for $i=\ell+1$. By induction hypothesis, 
\[
\begin{split}
\pa_r^{\ell+1} X = \sum_{ A \in \mathcal P_\ell}  c_{\ell}^ {A}  \pa_r A X 
\end{split}
\]
Now if $\ell$ is odd, $ \pa_r A$ is an admissible vector field belonging to $\mathcal P_{\ell+1}$, and hence we are done. If $\ell$ is even, then we write 
\[
\pa_r A = D_r A - 2 \frac{1}{r} A
\]
Then $D_r A$ and $\frac{1}{r} A$ are both admissible vector fields belonging to $\mathcal P_{\ell+1}$ and thus \eqref{E:dd} holds for $i=\ell+1$. 
\end{proof}

The same conclusion holds in Lemma \ref{L:etaiX} when we replace $A\in \mathcal P_i$ by $A\in\bar{ \mathcal P}_i$. 


We also write a few useful identities relating high-order $\D$, $\r$, and $\pa_r$ derivatives.

\begin{lemma}\label{L:SIMPLERELATIONS}
\begin{enumerate}
\item[(i)]
For any $j\in\mathbb Z_{\ge0}$ there exist constants $c_{k}$, $k\in\{0,\dots,j\}$ such that 
\begin{align}\label{E:AUX3}
\rr^j = \sum_{k=0}^j c_{k}  r^{k}\pa_r^k.
\end{align} 
\item[(ii)]
For any $j\in\mathbb Z_{\ge0}$ there exist constants $\bar c_{k}$, $k\in\{0,\dots,j\}$ such that 
\begin{align}
\pa_r^j =  r^{-j}\sum_{k=0}^j \bar c_{k} \rr^k.
\end{align} 
\item[(iii)]
For any $j\in\mathbb Z_{\ge0}$ there exist constants $\tilde c_{k}$, $k\in\{0,\dots,j\}$ such that 
\begin{align}
\D_j = r^{-j}\sum_{k=0}^j  \tilde c_{k} \rr^k.
\end{align} 
The same conclusion holds when we replace $\D_j$ by $A\in \mathcal P_j$.  
\item[(iv)] For any $j\in\mathbb Z_{>0}$ there exist constants $\hat c_{k}$, $k\in\{0,\dots,j\}$ such that 
\begin{align}
\bar\D_j = r^{-j}\sum_{k=1}^j  \hat c_{k} \rr^k.
\end{align} 
The same conclusion holds when we replace $\bar\D_j$ by $A\in \bar{\mathcal P}_j$. 
\end{enumerate}
\end{lemma}

\begin{proof}
The above statements follow easily by induction.
\end{proof}

\section{High-order commutators}\label{A:COMM}

From definitions~\eqref{E:LBETADEF}--\eqref{E:LBETASTARDEF} 
and the product rule
\[
D_r(fg) = D_r f g + f \pa_r g,
\]
it is easy to check that 
the following commutation rules hold:
\begin{align}
D_r L_k f = L_{1+k}^* D_r f - (1+k) (w''+\frac2 r w')D_r f \label{E:com1}\\
\pa_r  L_k^* h =  L_{1+k} \pa_r h - (1+k) (w''-\frac2 r w')\pa_r h \label{E:com2}
\end{align}
More generally, we have the following commutation rules for $\D_i L_\alpha$:

\begin{lemma}\label{L:COMM1} 
For any $i\in\mathbb Z_{>0}$ there exist constants $c_{ijk}$, $j\in\mathbb Z_{\ge0}$, $k\in\mathbb Z_{>0}$ such that  
\begin{align}
\mathcal D_iL_\alpha X = \mathcal L_{i+\alpha} \mathcal D_i X  + \sum_{j=0}^{i-1} \zeta_{ij} \mathcal D_{i-j} X\label{E:COMM_L}
\end{align}
where 
\be
\zeta_{ij} = \sum_{k=1}^{2+j} c_{ijk} \frac{ \pa_r^k w}{ r^{2+j-k}} \label{pij}.
\ee
\end{lemma}

\begin{proof} 
The proof is based on the induction on $i$. First let $i=1$. From \eqref{E:com1}, we have 
\[
\begin{split}
\mathcal D_1 L_\alpha X&= D_r L_\alpha X = L^\ast_{1+\alpha} D_r X - (1+\alpha) (w''+\frac2 r w')D_r X \\
&= \mathcal L_{1+\alpha}\mathcal D_1 X 
- (1+\alpha) (w''+\frac2 r w')\mathcal D_1 X
\end{split}
\]
and hence \eqref{E:COMM_L} holds with $p_{10}= -(1+\alpha) \pa_ r^2 w-2(1+\alpha)\frac{\pa_r w}{r}$. Suppose \eqref{E:COMM_L} is valid for all $i\leq \ell$. It suffices to show that \eqref{E:COMM_L} holds for $i=\ell+1$. If $\ell$ is even, by \eqref{E:COMM_L} and \eqref{E:com1}, we have 
\[
\begin{split}
&\mathcal D_{\ell+1} L_\alpha X = D_r \mathcal D_{\ell} L_\alpha X = D_r \left(  \mathcal L_{\ell+\alpha} \mathcal D_i X  + \sum_{j=0}^{\ell-1} \zeta_{\ell j} \mathcal D_{\ell-j} X  \right)  \\
&= L^\ast_{1+\ell+\alpha} D_r \mathcal D_\ell X  - (1+\ell+\alpha) (w''+\frac{2}{r} w') D_r \mathcal D_\ell X +
\sum_{j=0}^{\ell-1} D_r(  \zeta_{\ell j} \mathcal D_{\ell-j} X )\\
&= \mathcal L_{1+\ell+\alpha} \mathcal D_{1+\ell} X -  (1+\ell+\alpha) (w''+\frac{2}{r} w')  \mathcal D_{1+\ell} X 
+\sum_{\substack{0\leq j<\ell-1\\ j: \text{ even}}} ( \zeta_{\ell j} \mathcal D_{1+\ell-j} X+ \pa_r(\zeta_{\ell j}) \mathcal D_{\ell-j} X  ) \\
&\quad+ \sum_{\substack{0< j\leq \ell-1\\ j: \text{ odd}}} ( \zeta_{\ell j} \mathcal D_{1+\ell-j} X  + (\pa_r(\zeta_{\ell j}) +\frac{2}{r}\zeta_{\ell j} )\mathcal D_{\ell-j} X   )
\end{split} 
\]
The last three terms can be rearranged as $\sum_{j=0}^\ell \zeta_{1+\ell j} \mathcal D_{1+\ell -j} X$ where 
\[
\begin{split}
\zeta_{1+\ell 0} &= - (1+\ell+\alpha) (w''+\frac{2}{r} w')  + \zeta_{\ell 0},  \ \ \   j=0\\
\zeta_{1+\ell j} &=\zeta_{\ell j} + \pa_r(\zeta_{\ell j-1}), \ \ \   j\geq1 \text{ and odd} \\
\zeta_{1+\ell j} &=\zeta_{\ell j} + \pa_r(\zeta_{\ell j-1}) + \frac{2}{r} \zeta_{\ell j-1}, \ \ \   j\geq 2 \text{ and even} 
\end{split}
\]
Note that $\zeta_{1+\ell j}$ takes the form given in \eqref{pij} because of the induction assumption. Therefore,  \eqref{E:COMM_L} holds for $i=\ell+1$. If $\ell$ is odd, we use \eqref{E:com2} in place of \eqref{E:com1} to derive the same conclusion. This finishes the proof. 
\end{proof}

Next we present the commutator identities useful for derivation of high order equations as well as for high order estimates. 

\begin{lemma}\label{L:COMM2}
Let $i\in\mathbb Z_{>0}$ be given and let $e, X$ be sufficiently smooth functions. For given differential operators A and B, let 
\begin{align}\label{E:COMMDEF}
\left[ A\,,\, e B\right] X: = A(e BX) - e ABX
\end{align}
denote the usual commutator. 
Then for any $1\le k \le i$ and any $A\in \bar{\mathcal P}_k$, $B\in\mathcal P_{i-k}$ there exists a constant$c_k^{iAB}\in\mathbb R$, and similarly for any $A\in \bar{\mathcal P}_k$, $B\in\bar{\mathcal P_{i-k}}$ there exists a constant $\bar c_k^{iAB}$ such that following identities hold 
\begin{align}
[\D_i, e\pa_r]X
& =i \pa_r e \mathcal D_i X + 
\sum_{1 \leq k\leq i \atop A\in \bar{\mathcal P}_{k}, B\in {\mathcal P}_{i-k}} c_k^{iAB} A(\frac{e}{r}) (B X)
 + 
\sum_{1 \leq k\leq i-1\atop  A \in \bar{\mathcal P}_{k+1},  B\in \bar{\mathcal P}_{i-k-1}} \bar c_k^{i A B}  r A(\frac{e}{r}) ( B D_r X) 
\label{E:DiXP} \\
[\D_i, e ]X & =  \sum_{1 \leq k\leq i \atop A\in \bar{\mathcal P}_{k}, B\in {\mathcal P}_{i-k}} c_k^{iAB} (Ae) (B X) \label{E:DXP} \\
 [\bar{\D}_i, e]X & = 
 i \pa_r e {\bar\D}_{i-1} X
+ \sum_{2 \leq k\leq i \atop A\in \bar{\mathcal P}_{k}, B\in \bar{\mathcal P}_{i-k}} \bar c_k^{iAB} (A e) (B X) \label{E:DYP}
\end{align}
\end{lemma}

\begin{proof}
{\em Proof of~\eqref{E:DiXP}.}
The proof is based on the induction on $i$. Let $i=1$. Then by using the identity $D_r \pa_r X= \pa_r D_r X + \frac{2}{ r^2}X$, 
\begin{align*}
\mathcal D_1 \left( e \pa_r X \right) = D_r  \left( e \pa_r X \right) = e D_r \pa_r X + \pa_r e \pa_r X = e \pa_r D_r X
+ \pa_r e D_r X  -2 \pa_r( \frac{e}{r}  )X 
\end{align*}
which yields \eqref{E:DiXP} for $i=1$. Suppose \eqref{E:DiXP} is valid for all $i\leq \ell$. It suffices to show \eqref{E:DiXP} for $i=\ell +1$. We will verify it when $\ell $ is odd. The other case ($\ell$ is even) will follow similarly. By using the induction assumption, 
\begin{align*}
&\mathcal D_{\ell+1} \left( e \pa_r X \right) =\pa_r \Big( e \pa_r (\mathcal D_\ell X) + \ell \pa_r e \mathcal D_\ell X + 
\sum_{1 \leq k\leq \ell \atop A\in \bar{\mathcal P}_{k}, B\in {\mathcal P}_{\ell-k}} c_k^{\ell AB} A(\frac{e}{r}) (B X)\\
 &\qquad\qquad\qquad\qquad+ 
\sum_{1 \leq k\leq \ell-1\atop  A \in \bar{\mathcal P}_{k+1},  B\in \bar{\mathcal P}_{\ell-k-1}} \bar c_k^{\ell A B}  r A(\frac{e}{r}) ( B D_r X)\Big)
 \\
&= e \pa_r \mathcal D_{\ell+1} X +(\ell +1) \pa_r e \mathcal D_{\ell +1}X \\
& + \sum_{1 \leq k\leq \ell, k:\text{odd}\atop A\in \bar{\mathcal P}_{k}, B\in {\mathcal P}_{\ell-k}} c_k^{\ell AB}\left\{ (D_r -\frac{2}{r}) A(\frac{e}{r}) (B X)+A(\frac{e}{r}) ((D_r -\frac{2}{r}) B X) \right\} \\
&+ \sum_{1 \leq k\leq \ell , k:\text{even}\atop A\in \bar{\mathcal P}_{k}, B\in {\mathcal P}_{\ell-k}} c_k^{\ell AB} \left\{ \pa_r A(\frac{e}{r}) (B X) + A(\frac{e}{r}) (\pa_r B X) \right\} \\
& + \sum_{1 \leq k\leq \ell-1, k:\text{odd}\atop  A \in \bar{\mathcal P}_{k+1},  B\in \bar{\mathcal P}_{\ell-k-1}} 
\bar c_k^{\ell A B}
\left\{ (A(\frac{e}{r})  +  r\pa_r A(\frac{e}{r})) ( B D_r X) +   rA(\frac{e}{r}) ( (D_r-\frac{2}{r}) B D_r X)  \right\} \\
&+ \sum_{1 \leq k\leq \ell-1,k:\text{even}\atop  A \in \bar{\mathcal P}_{k+1},  B\in \bar{\mathcal P}_{\ell-k-1}} 
\bar c_k^{\ell A B}
\left\{   r(D_r-\frac{1}{r}) A(\frac{e}{r}) ( B D_r X)+  r A'(\frac{e}{r}) ( \pa_r B D_r X) \right\}
\end{align*}
We note that each expression in the above summations belongs to either summation in \eqref{E:DiXP} for $i=\ell+1$. 
Proof of~\eqref{E:DXP} and~\eqref{E:DYP} follows analogously. 
\end{proof}

\section{Hardy-Sobolev embedding}\label{A:HSembedding}

Let $\chi, \; \psi \geq 0$ be smooth cutoff functions satisfying $\chi=1$ on $[0,\frac12]$, $\chi=0$ on $[\frac34,1]$ and $\psi=1$ on $[\frac12,1]$, $\psi=0$ on $[0,\frac14]$, 
satisfying in addition
\[
\chi'(r)\le0, \ \ \psi'(r)\ge0, \ \ r\in[0,1].
\]

\begin{lemma}[Localized Hardy inequalities] \label{L:LOCHARDY}
Let $\chi,\psi$ be the above defined cut-off functions and let $u: B_1(0)\to \mathbb R$ be a given smooth radially symmetric function, where $B_1(0)=\{x, \big| |x|\le 1\}$ is the unit ball in $\mathbb R^3$. Then 
\begin{enumerate}
\item 
\be \label{hardy0}
\int_0^\frac34 
 |u|^2 \chi^2 d r \lesssim \int_\frac12^\frac34 |u|^2  r^2 d r + \int_0^\frac34 |A u|^2 r^2 \chi^2 d r
\ee
where $A=\mathcal D_1 = \pa_r+\frac2r$ or $A = \bar{\mathcal D}_1 = \pa_r$. 
\item Let $a>1$ be given. 
\be \label{hardy1}
\int_\frac14^1 w^{a-2} |u|^2\psi^2 d r \lesssim \int_\frac14^\frac12 w^a |u|^2 d r + \int_\frac14^1 w^a |A u|^2\psi^2 d r
\ee
where $A=\mathcal D_1 = \pa_r$ or $A = \bar{\mathcal D}_1 = \pa_r+\frac2r$.  
\end{enumerate}
\end{lemma}

\begin{proof}
The proof is based on the standard Hardy inequality and the cutoff function argument. 
For $A=\bar{\mathcal D}_1$, see~\cite{Jang2014}. The case $A=\mathcal D_1=\pa_r$ follows from 
\be\label{identity01}
\begin{split}
\int_0^\frac{3}{4} |\pa_r u|^2  r^2 \chi^2 d r &= \int_0^\frac{3}{4} |D_r u -\frac{2}{r} u  |^2  r^2 \chi^2 d r \\
&= 
 \int_0^\frac{3}{4} |D_r u |^2  r^2 \chi^2 d r + 4 \int_0^\frac{3}{4} u ^2  \chi^2 d r - 4
 \int_0^\frac{3}{4} D_r u  u  r\chi^2 d r \\
 &=  \int_0^\frac{3}{4} |D_r u |^2  r^2 \chi^2 d r - 2  \int_0^\frac{3}{4} u ^2  \chi^2 d r  + 4 \int_0^\frac{3}{4} u ^2  r \chi \chi' d r 
\end{split}
\ee
The very same bound implies 
\[
2 \int_0^\frac{3}{4} u ^2  \chi^2 d r \l \le  4 \int_0^\frac{3}{4} u ^2  r \chi \chi' \, d r  + \int_0^\frac{3}{4} |D_r u |^2   \chi^2 \, r^2 d r, 
\]
which immediately yields~\eqref{hardy0} with $A=D_r$.
The localized Hardy inequality near the boundary follows similarly. 
We note that $w\sim 1-r$ in the vicinity of the boundary $r=1$. 
\end{proof}

As a consequence of the above lemma, we have that for 
any smooth $u:B_1(0)\to \mathbb R$ and any $m\in\mathbb Z_{>0}$
\be\label{L1bound}
\| u \|_{L^1}^2 \lesssim 
 \int_0^\frac34 |A u|^2 r^2 \chi^2 d r  + \sum_{i=0}^{m} \int_{\frac14}^1 w^{\alpha-\lfloor \alpha \rfloor + 2m} |A_i u|^2 d r,
\ee
where either $A=D_r$ and $A_i = \D_i$ for all $i=0,1,\dots,m$, or $A=\pa_r$ and $A_i = \bar{\D}_i$ 
for all $i=0,1,\dots,m$.
and the same estimate holds with $D_r$ replaced by $\pa_r$ and $\D_i$ by $\bar{\D}_i$. 
See Lemma 3.3 of \cite{Jang2015} for the proof. Note  the term $\int_\frac12^\frac34 |u|^2  r^2 d r$ in \eqref{hardy0} has been absorbed into the second summation in \eqref{L1bound}.  We remark that away from the origin, both $\sum_{i=0}^{m} \int_{\frac14}^1 w^{\alpha-\lfloor \alpha \rfloor + 2m} |{\mathcal D}_i u|^2 d r$ 
and $\sum_{i=0}^{m} \int_{\frac14}^1 w^{\alpha-\lfloor \alpha \rfloor + 2m} |\bar{\mathcal D}^i u|^2 d r$ 
are equivalent to 
$\sum_{i=0}^{m} \int_{\frac14}^1 w^{\alpha-\lfloor \alpha \rfloor + 2m} |\pa_r^i u|^2 d r$ for any 
$m\in\mathbb Z_{>0}$. 

The next result concerns the $L^\infty$ bound: 

\begin{lemma}\label{L:EMBEDDINGS}
Under the same assumptions as in Lemma~\ref{L:LOCHARDY} and any
$m\in\mathbb Z_{>0}$, we have
\be\label{infty0}
\left\| u\right\|_\infty^2 \lesssim \sum_{i=1}^{2} \int_0^\frac{3}{4} |B_i u |^2  r^2 d r  + \sum_{i=0}^{m+1} \int_{\frac14}^1 w^{\alpha-\lfloor \alpha \rfloor + 2m} |B_i u|^2 d r
\ee
where either $B_i=\mathcal D_i$ or $B_i = \bar{\mathcal D}_i$, $i=0,1,\dots,m+1$. 
Moreover, 
\be\label{infty}
\left\| \frac{u}{r} \right\|_\infty^2 \lesssim \sum_{i=2}^{3} \int_0^\frac{3}{4} |{\mathcal D}_i u|^2  r^2 d r  + \sum_{i=0}^{m+1} \int_{\frac14}^1 w^{\alpha-\lfloor \alpha \rfloor + 2m} |{\mathcal D}_i u|^2 d r
\ee
\end{lemma}

\begin{proof} The proof follows from the Sobolev embedding inequality 
\[
\|u\|_\infty \lesssim \|u\|_{L^1} + \| \pa_r u  \|_{L^1}.
\]   
While the summed norms generated by $\mathcal D_i$ or $\bar{\mathcal D}_i$ or $\pa_r^i$ are all equivalent away from the origin, the ordered derivatives near the origin require some attention.  We start with \eqref{infty0}. 
First, we will verify~\eqref{infty0} in the case  $B_i = \bar{\mathcal D}_i$, $i=0,1,\dots,m+1$. 
By the above Sobolev embedding and~\eqref{L1bound}, 
\begin{align*}
\|u\|_\infty &\lesssim  \int_0^\frac{3}{4} |\pa_r u |^2  r^2\chi^2 d r +
\sum_{i=0}^m \int_{\frac14}^1 w^{\alpha-\lfloor \alpha \rfloor + 2m} |\bar{\D}_i u|^2 \,dr  \\
& \ \ \ \ +  \int_0^\frac{3}{4}  |D_r \pa_r u |^2 r^2 \chi^2 d r   +  \sum_{i=0}^m \int_{\frac14}^1 w^{\alpha-\lfloor \alpha \rfloor + 2m} |\D_i\pa_r u|^2 \,dr \\
& \lesssim \int_0^\frac{3}{4} \left(|\pa_r u |^2 +|D_r \pa_r u |^2\right) r^2\chi^2 d r +
\sum_{i=0}^{m+1} \int_{\frac14}^1 w^{\alpha-\lfloor \alpha \rfloor + 2m} |\bar\D_i u|^2 d r 
\end{align*}
where we have applied~\eqref{L1bound} to $\|u\|_{L^1}$ with  the choice $A= \pa_r$ and 
$A_i =\bar{\mathcal D}_i$ for $i=0,\dots,m$ and then 
 to $\|\pa_r u\|_{L^1}$ with the choice $A=D_r$ and 
$A_i =\mathcal D_i$ for $i=0,\dots,m$. For the last line,
we simply remark that away from the origin, both $\sum_{i=0}^{m} \int_{\frac14}^1 w^{\alpha-\lfloor \alpha \rfloor + 2m} |{\mathcal D}_i u|^2 d r$ 
and $\sum_{i=0}^{m} \int_{\frac14}^1 w^{\alpha-\lfloor \alpha \rfloor + 2m} |\bar{\mathcal D}^i u|^2 d r$ 
are equivalent to 
$\sum_{i=0}^{m} \int_{\frac14}^1 w^{\alpha-\lfloor \alpha \rfloor + 2m} |\pa_r^i u|^2 d r$ for any 
$m\in\mathbb Z_{>0}$. 
This proves~\eqref{infty0} when $B_i = \bar{\mathcal D}_i$ for $i=0,\dots,m+1$. 

In the case $B_i = \D_i$ for $i=0,\dots,m+1$, we apply~\eqref{L1bound} to $\|u\|_{L^1}$ with  $A= D_r$ and $A_i = \D_i$ for $i=0,\dots,m$ and to $\|\pa_r u\|_{L^1}$ with $A=\pa_r$ and $A_i = \bar\D_i$ for $i=0,\dots,m$, 
to obtain 
\[
\|u\|_\infty \lesssim  \int_0^\frac{3}{4} |D_r u |^2  r^2\chi^2 d r +  \int_0^\frac{3}{4}  | \pa_{r}^2 u |^2 r^2 \chi^2d r   +  \sum_{i=0}^{m+1} \int_{\frac14}^1 w^{\alpha-\lfloor \alpha \rfloor + 2m} |{\mathcal D}_i u|^2 d r 
\]
Now it suffices to show that $ \int_0^\frac{3}{4}  | \pa_{r}^2 u |^2 r^2\chi^2 d r$ is bounded by $  \sum_{i=1}^{2} \int_0^\frac{3}{4} |{\mathcal D}_i u |^2  r^2\chi^2 d r  + \sum_{i=0}^{m+1} \int_{\frac14}^1 w^{\alpha-\lfloor \alpha \rfloor + 2m} |{\mathcal D}_i u|^2 d r $. To this end, we first note that 
\be\label{identity00}
\pa_ r^2 u = \pa_r D_r u  - 2 \pa_r(\frac{u}{r}), \quad \pa_r D_r u =  r\pa_ r^2 (\frac{u}{r}) + 4\pa_r(\frac{u}{r})
\ee
Therefore, 
\begin{align*}
\int_0^\frac{3}{4}  |\pa_r^2 u |^2 r^2\chi^2 d r &= \int_0^\frac{3}{4}  | \pa_r D_r u  - 2 \pa_r(\frac{u}{r}) |^2 r^2\chi^2 d r\\
&=  \int_0^\frac{3}{4}  | \pa_r D_r u |^2 r^2\chi^2 d r - 12  \int_0^\frac{3}{4}  | \pa_r(\frac{u}{r}) |^2 r^2\chi^2 d r 
- 4  \int_0^\frac{3}{4}   \pa_ r^2 (\frac{u}{r})   \pa_r(\frac{u}{r}) r^3 \chi^2 d r \\
&= \int_0^\frac{3}{4}  |\D_2u |^2 r^2\chi^2 d r - 6  \int_0^\frac{3}{4}  | \pa_r(\frac{u}{r}) |^2 r^2\chi^2 d r + 
4 \int_0^\frac{3}{4}  | \pa_r(\frac{u}{r}) |^2 r^3\chi \chi' d r\\
&\le  \int_0^\frac{3}{4}  | \mathcal D_2 u |^2 r^2\chi^2 d r, 
\end{align*}
where we have used $\chi'\le0$ in the last estimate.
This yields~\eqref{infty0} for $B = D_r$ and $B_i=\mathcal D_i$ for $i=0,1,\dots,m+1$. 

Next we will prove \eqref{infty}. First we have 
\[
\left\| \frac{u}{r} \right\|_\infty^2 \lesssim \int_0^\frac34 |\frac{u}{r}|^2  \chi^2 d r + \int_0^\frac34 |\pa_r(\frac{u}{r})|^2  \chi^2 d r + \sum_{i=0}^{m+1} \int_{\frac14}^1 w^{\alpha-\lfloor \alpha \rfloor + 2m} |{\mathcal D}_i u|^2 d r,
\]
where we have used the bound $\sum_{i=0}^{m+1} \int_{\frac14}^1 w^{\alpha-\lfloor \alpha \rfloor + 2m} |{\mathcal D}_i (\frac{ u}{r} )|^2 d r \lesssim \sum_{i=0}^{m+1} \int_{\frac14}^1 w^{\alpha-\lfloor \alpha \rfloor + 2m} |{\mathcal D}_i u|^2 d r$ since $\frac14\le r\le 1$. 
By applying \eqref{hardy0},  we see that 
\[
\int_0^\frac34 |\frac{u}{r}|^2  \chi^2 d r  + \int_0^\frac34 |\pa_r(\frac{u}{r})|^2  \chi^2 d r \lesssim 
\int_\frac{1}{2}^\frac{3}{4} (|u|^2 +|\pa_r u|^2) d r 
+ \underbrace{\int_0^\frac34 |\pa_r (\frac{u}{r})|^2  r^2 \chi^2  d r}_{(a)} + \underbrace{ \int_0^\frac34 |\pa_ r^2 (\frac{u}{r})|^2  r^2 \chi^2  d r}_{(b)}
\]
For $(a)$, we apply  \eqref{hardy0} to obtain 
\[
(a)\lesssim \int_\frac{1}{2}^\frac{3}{4} (|u|^2 +|\pa_r u|^2) d r  + \int_0^\frac34 |\pa_r( \r (\frac{u}{r}))|^2  r^2 \chi^2  d r
\]
Note that by using \eqref{identity00}
\[
\begin{split}
&\int_0^\frac34 |\pa_r( \r (\frac{u}{r}))|^2  r^2 \chi^2  d r = \int_0^\frac34 |\pa_r D_r u  - 3\pa_r (\frac{u}{r})|^2  r^2 \chi^2  d r \\
&=  \int_0^\frac{3}{4}  | \pa_r D_r u |^2 r^2\chi^2 d r - 15  \int_0^\frac{3}{4}  | \pa_r(\frac{u}{r}) |^2 r^2\chi^2 d r 
- 6  \int_0^\frac{3}{4}   \pa_ r^2 (\frac{u}{r})   \pa_r(\frac{u}{r}) r^3 \chi^2 d r \\
&= \int_0^\frac{3}{4}  | \pa_r D_r u |^2 r^2\chi^2 d r - 6  \int_0^\frac{3}{4}  | \pa_r(\frac{u}{r}) |^2 r^2\chi^2 d r + 
6 \int_0^\frac{3}{4}  | \pa_r(\frac{u}{r}) |^2 r^3\chi \chi' d r
\end{split}
\]
which yields 
\[
(a)\lesssim  \int_\frac{1}{2}^\frac{3}{4} (|u|^2 +|\pa_r u|^2) d r + \int_0^\frac{3}{4}  | \mathcal D_2 u |^2 r^2\chi^2 d r
\]
For $(b)$, we apply  \eqref{hardy0} again to obtain 
\[
(b)\lesssim \int_\frac{1}{2}^\frac{3}{4} (|u|^2 +|\pa_r u|^2 +| \pa_ r^2 u|^2 ) d r  + \int_0^\frac34 |\pa_r( \r^2 (\frac{u}{r}))|^2  r^2 \chi^2  d r
\]
By \eqref{identity00} and also using $\pa_ r^2 D_r u =  r \pa_r^3 (\frac{u}{r})  + 5  \pa_ r^2 (\frac{u}{r}) $, 
\[
\begin{split}
&\int_0^\frac34 |\pa_r( \r^2 (\frac{u}{r}))|^2  r^2 \chi^2  d r = \int_0^\frac34 |\pa_ r^2 D_r u  - 4\pa_ r^2 (\frac{u}{r})|^2  r^2 \chi^2  d r \\
&=  \int_0^\frac{3}{4}  | \pa_ r^2 D_r u |^2 r^2\chi^2 d r +16 \int_0^\frac{3}{4}  | \pa_ r^2(\frac{u}{r}) |^2 r^2\chi^2 d r 
- 8  \int_0^\frac{3}{4}   \pa_ r^2 D_r u  \pa_ r^2 (\frac{u}{r})  r^2 \chi^2 d r \\
&=  \int_0^\frac{3}{4}  | \pa_ r^2 D_r u |^2 r^2\chi^2 d r - 24 \int_0^\frac{3}{4}  | \pa_ r^2(\frac{u}{r}) |^2 r^2\chi^2 d r 
- 8  \int_0^\frac{3}{4}  \pa_r^3 (\frac{u}{r})   \pa_ r^2 (\frac{u}{r})  r^3 \chi^2 d r  \\
&= \int_0^\frac{3}{4}  | \pa_ r^2 D_r u |^2 r^2\chi^2 d r - 12   \int_0^\frac{3}{4}  | \pa_ r^2(\frac{u}{r}) |^2 r^2\chi^2 d r + 
8 \int_0^\frac{3}{4}  | \pa_ r^2 (\frac{u}{r}) |^2 r^3\chi \chi' d r \\
&=  \int_0^\frac{3}{4} |D_r \pa_r D_r u |^2  r^2 \chi^2 d r - 2  \int_0^\frac{3}{4} |\pa_r D_r u| ^2  \chi^2 d r - 12   \int_0^\frac{3}{4}  | \pa_ r^2(\frac{u}{r}) |^2 r^2\chi^2 d r \\
& \quad +8 \int_0^\frac{3}{4}  | \pa_ r^2 (\frac{u}{r}) |^2 r^3\chi \chi' d r+ 4 \int_0^\frac{3}{4} |\pa_r D_r u|^2  r \chi \chi' d r 
\end{split}
\]
where we have used \eqref{identity01} at the last equality. It in turn yields 
\[
(b)\lesssim  \int_\frac{1}{2}^\frac{3}{4} (|u|^2 +|\pa_r u|^2+| \pa_ r^2 u|^2) d r +  \int_0^\frac{3}{4} |\mathcal D_3 u |^2  r^2 \chi^2 d r 
\]
This finishes the proof of \eqref{infty}. 
\end{proof}

The same argument gives the following bound for $\|ru\|_\infty$: 
\be\label{inftyru}
\left\| ru\right\|_\infty^2 \lesssim \sum_{i=0}^{1} \int_0^\frac{3}{4} |B_i u |^2  r^2 d r  + \sum_{i=0}^{m+1} \int_{\frac14}^1 w^{\alpha-\lfloor \alpha \rfloor + 2m} |B_i u|^2 d r
\ee
where we apply \eqref{hardy0} just once near the origin, since $\| ru\|_\infty \lesssim \| ru \|_{L^1} + \| \pa_r (r u)\|_{L^1}$, to derive the first sum.

We now recall the Hilbert space $B^N$ with the norm: 
\begin{align*}
\|f\|_{B^N}:=\sum_{j=0}^{N}\|\D_j  f  \|_{\alpha+j}. 
\end{align*}
In what follows, we will derive the weighted $L^2$ and $L^\infty$ embedding inequalities for functions in $B^N$ based on Lemma \ref{L:LOCHARDY}, Lemma \ref{L:EMBEDDINGS}. 

\begin{lemma}[$L^2$ weighted embeddings]\label{L:L2WEIGHTED}  Let $(H,\pa_\tau H)\in B^N\times B^N$ be given. Then we have 
\begin{enumerate}
\item[(i)] 
For any $ \frac{N-\alpha }{2} \leq k\leq N$ 
\be\label{L21}
\begin{split}
\tau^{\gamma-\frac53} \int  w^{\alpha+2k-N} |\D_{k} \pa_\tau H |^2  r^2 d r+\tau^{\gamma-\frac{11}{3}} \int  w^{\alpha+2k-N} |\D_{k} H |^2  r^2 d r  \lesssim E^N \\
\tau^{\gamma-\frac83} \int  w^{\alpha+2k-N} |\D_{k} \pa_\tau H |^2  r^2 d r+\tau^{\gamma-\frac{14}{3}} \int  w^{\alpha+2k-N} |\D_{k} H |^2  r^2 d r  \lesssim D^N
\end{split}
\ee

\item[(ii)] We further assume that $\sum_{j=0}^N\int \frac{ w^{\alpha+j+1}}{(\tau+\frac23M_g )^{1+\gamma}}  |\D_{j+1} H |^2  r^2 d r <\infty$. 
Then for any  $ \frac{N-\alpha -1}{2} \leq k\leq N$ 
\be\label{L22}
\ve \int \frac{ w^{\alpha+2k+1-N}}{(\tau+\frac23M_g )^{1+\gamma}}  |\D_{k+1} H |^2  r^2 d r  \lesssim E^N 
\ee
\end{enumerate}
\end{lemma}

\begin{proof} We start with \eqref{L21}. Divide each integral appearing in the left-hand sides into two $\int = \int_0^\frac12 + \int_{\frac12}^1$. Then $w$ is strictly positive for $r\in[0,\frac12]$ and hence $w^{\alpha+2k-N}\lesssim w^{\alpha+k}$ for $r\in[0,\frac12]$, the first pieces are trivially bounded by the right-hand sides. Now for the second pieces, we apply \eqref{hardy1} repeatedly $(N-k)$ times starting with $a-2= \alpha+ 2k -N \geq 0$ to deduce the result. 

Likewise, for \eqref{L22}, we divide the integral into two pieces. Then the integral restricted to $[0,\frac12]$ is bounded by $E^N$ because $w^{\alpha+2k+1-N}\lesssim w^{\alpha+k+1}$ for $r\in[0,\frac12]$. For the integral from $\frac12$ to $1$, we first observe that $(\tau+\frac23M_g ) $ for $r\in[\frac12,1]$ is bounded from below and above by positive constants  and hence by applying  \eqref{hardy1} repeatedly $(N-k)$ times, we deduce the desired bound. 
\end{proof}

\begin{lemma}[$L^\infty$ embedding] Let $(H,\pa_\tau H)\in B^N\times B^N$ be given. Then we have
\begin{enumerate}
\item[(i)] For any $k\in \mathbb Z_{\geq 0}$ such that $k\leq \frac{N-\lfloor\alpha\rfloor -2}{2}$, 
\be\label{Linfty1}
\begin{split}
\tau^{\frac12(\gamma-\frac53)}\| \D_k \pa_\tau H \|_\infty  + \tau^{\frac12(\gamma-\frac{11}{3})}\| \D_k  H \|_\infty   \lesssim (E^N)^\frac12 \\
\tau^{\frac12(\gamma-\frac83)}\| \D_k \pa_\tau H \|_\infty  + \tau^{\frac12(\gamma-\frac{14}{3})}\| \D_k  H \|_\infty   \lesssim (D^N)^\frac12 
\end{split}
\ee
\item[(ii)] Similarily, for any $k\in \mathbb Z_{\geq 0}$ such that $k\leq \frac{N-\lfloor\alpha\rfloor -2}{2}$, 
\be\label{Linfty11}
\begin{split}
\tau^{\frac12(\gamma-\frac53)}\| \bar\D_k \pa_\tau  (\frac{H}{r}) \|_\infty  + \tau^{\frac12(\gamma-\frac{11}{3})}\| \bar\D_k  (\frac{H}{r}) \|_\infty   \lesssim (E^N)^\frac12   \\
\tau^{\frac12(\gamma-\frac83)}\| \bar\D_k \pa_\tau  (\frac{H}{r}) \|_\infty  + \tau^{\frac12(\gamma-\frac{14}{3})}\| \bar\D_k  (\frac{H}{r}) \|_\infty   \lesssim (D^N)^\frac12 
\end{split}
\ee
\end{enumerate}
\end{lemma}

\begin{proof} We start with \eqref{Linfty1}. We present the detail for $\|\D_k \pa H\|_\infty$ and other cases follow in the same way. By using \eqref{infty0} with $u=\D_k H$ and $m=\lfloor \alpha\rfloor + k +1$, we see that 
\[
\|\D_k  H\|_\infty^2 \lesssim \sum_{i=1}^{2} \int_0^\frac{3}{4} |B_i \D_k  H |^2  r^2 d r  + \sum_{i=0}^{\lfloor \alpha\rfloor + k +2} \int_{\frac14}^1 w^{\alpha-\lfloor \alpha \rfloor + 2 (  \lfloor \alpha\rfloor + k +1) } |B_i \D_k  H |^2 d r
\]
where we take $B_i =\D_i$ for $k$ even and $B_i=\bar \D_i$ for $k$ odd. The first sum is trivially bounded by $ \tau^{-(\gamma-\frac{11}{3})} (E^N)$ or $ \tau^{-(\gamma-\frac{14}{3})} (D^N) $ since $1\lesssim w^{\alpha+i+k}$ for $r\in [0,\frac34]$. For the second sum, since 
$w^{\alpha-\lfloor \alpha \rfloor + 2 (  \lfloor \alpha\rfloor + k +1) } \lesssim w^{\alpha+ i +k}$ for $0\leq i \leq \lfloor \alpha\rfloor + k +2 $ and also the total number of derivatives appearing in the sum $ \lfloor \alpha\rfloor + k +2 + k\leq N$, it is bounded by $ \tau^{-(\gamma-\frac{11}{3})} (E^N) $ or $ \tau^{-(\gamma-\frac{14}{3})} (D^N) $. 

For \eqref{Linfty11}, we apply  \eqref{infty0} with $m=\lfloor \alpha\rfloor + k +1$ with $u=\bar\D_k (\frac{H}{r})$ and $m=\lfloor \alpha\rfloor + k +1$, 
\[
\|\bar\D_k (\frac{H}{r})\|_\infty^2 \lesssim \sum_{i=1}^{2} \int_0^\frac{3}{4} |B_i (\bar\D_k (\frac{H}{r}) ) |^2  r^2 d r  + \sum_{i=0}^{\lfloor \alpha\rfloor + k +2} \int_{\frac14}^1 w^{\alpha-\lfloor \alpha \rfloor + 2 (  \lfloor \alpha\rfloor + k +1) } |B_i (\bar\D_k (\frac{H}{r}) ) |^2 d r
\]
where we take $B_i =\bar\D_i$ for $k$ even and $B_i=\D_i$ for $k$ odd. Now for the first sum, note that 
$B_i \bar \D_k \frac1r \in {\mathcal P}_{i+k+1}$. Therefore, by Lemma \ref{L:control}, it is bounded by $\sum_{j=0}^{k+3} \int_0^\frac{3}{4} | \D_j H |^2  r^2 d r $ and thus by $ \tau^{-(\gamma-\frac{11}{3})} (E^N)$ or $ \tau^{-(\gamma-\frac{14}{3})} (D^N) $ since $k+3\leq N$. Now for the second sum, note that when $r\in [\frac14,1]$, in contrast to the first sum, $\frac{1}{r}$ does not act as a derivative, in other words, $|B_i (\bar\D_k (\frac{H}{r}) ) |^2$ is bounded by $\sum_{j=0}^{i+k} |\D_j H|^2$. Therefore, by the same reasoning as in the previous case, we obtain the result. 
\end{proof}

\begin{lemma}[$L^\infty$ $w-$weighted embedding]\label{L:LINFTYWEIGHTED} Let $(H,\pa_\tau H)\in B^N\times B^N$ be given. Then we have
\begin{itemize}
\item[(i)] 
For any $k\in \mathbb Z_{\geq 0}$ such that $k+4 \leq N$, 
\be\label{Linfty2}
\begin{split}
\tau^{\frac12(\gamma-\frac53)}\| w^k \D_{k+2} \pa_\tau H \|_\infty  + \tau^{\frac12(\gamma-\frac{11}{3})}\| w^k \D_{k+2}  H \|_\infty   \lesssim (E^N)^\frac12 \\
\tau^{\frac12(\gamma-\frac83)}\| w^k \D_{p+2} \pa_\tau H \|_\infty  + \tau^{\frac12(\gamma-\frac{14}{3})}\| w^k \D_{k+2}  H \|_\infty   \lesssim (D^N)^\frac12
\end{split}
\ee
for $k=N-3$, 
\be\label{Linfty3}
\begin{split}
\tau^{\frac12(\gamma-\frac53)}\|  r \,w^{N-3} \D_{N-1} \pa_\tau H \|_\infty  + \tau^{\frac12(\gamma-\frac{11}{3})}\| r\, w^{N-3} \D_{N-1}  H \|_\infty   \lesssim (E^N)^\frac12 \\
\tau^{\frac12(\gamma-\frac83)}\|  r \,w^{N-3} \D_{N-1} \pa_\tau H \|_\infty  + \tau^{\frac12(\gamma-\frac{14}{3})}\| r\, w^{N-3} \D_{N-1}  H \|_\infty   \lesssim (D^N)^\frac12
\end{split}
\ee
\item[(ii)]
Similarly, for any $k\in \mathbb Z_{\geq 0}$ such that $k+5 \leq N$, 
\be\label{Linfty21}
\begin{split}
\tau^{\frac12(\gamma-\frac53)}\| w^k \bar\D_{k+2} \pa_\tau (\frac{H}{r}) \|_\infty  + \tau^{\frac12(\gamma-\frac{11}{3})}\| w^k \bar\D_{k+2} (\frac{ H}{r}) \|_\infty   \lesssim (E^N)^\frac12 \\
\tau^{\frac12(\gamma-\frac83)}\| w^k \bar\D_{p+2} \pa_\tau (\frac{H}{r}) \|_\infty  + \tau^{\frac12(\gamma-\frac{14}{3})}\| w^k \bar\D_{k+2} (\frac{ H}{r}) \|_\infty   \lesssim (D^N)^\frac12 
\end{split}
\ee
for $k=N-4$, 
\be\label{Linfty31}
\begin{split}
\tau^{\frac12(\gamma-\frac53)}\|  r \,w^{N-4} \bar\D_{N-2} \pa_\tau (\frac{H}{r}) \|_\infty  + \tau^{\frac12(\gamma-\frac{11}{3})}\| r\, w^{N-4} \bar\D_{N-2} (\frac{ H }{r})\|_\infty   \lesssim (E^N)^\frac12 \\
\tau^{\frac12(\gamma-\frac83)}\|  r \,w^{N-4} \bar\D_{N-2} \pa_\tau (\frac{H}{r}) \|_\infty  + \tau^{\frac12(\gamma-\frac{14}{3})}\| r\, w^{N-4} \bar\D_{N-2} (\frac{ H }{r})\|_\infty   \lesssim (D^N)^\frac12
\end{split}
\ee
\end{itemize}
\end{lemma}

\begin{proof} The proof follows in the same spirit as in the previous lemma. For $w^k \D_{k+2} H$ in \eqref{Linfty2}, we first apply \eqref{infty0} with $u = w^k \D_{k+2} H$ and $m= 3+\lfloor \alpha \rfloor -k$: 
\[
\|w^k\D_{k+2}  H\|_\infty^2 \lesssim \sum_{i=1}^{2} \int_0^\frac{3}{4} |B_i ( w^k\D_{k+2}  H )  |^2  r^2 d r  + \sum_{i=0}^{4+\lfloor \alpha \rfloor -k} \int_{\frac14}^1 w^{\alpha-\lfloor \alpha \rfloor + 2 (  3+\lfloor \alpha \rfloor -k) } |B_i (w^k\D_{k+2}  H) |^2 d r
\]
where we take $B_i =\D_i$ for $k$ even and $B_i=\bar\D_i$ for $k$ odd. Then it is easy to see that the first sum is bounded by our energy since $k+4\leq N$. For the second sum, by using the product rule and smoothness of $w$, we first note that 
$|B_i \D_k (w^k\D_{k+2}  H) |^2  \lesssim \sum_{j=0}^{i} w^{2k-2j} | W_j H  |^2 $ where $W_j\in \mathcal P_{k+2+i-j}$.  Therefore, by further using Lemma \ref{L:control}, the second sum is bounded by 
\[
\sum_{i=0}^{4+\lfloor \alpha \rfloor -k} \sum_{j=0}^i  \int_{\frac14}^1 w^{\alpha + \lfloor \alpha \rfloor + 6 -2j } | \mathcal D_{k+2+i-j} H |^2 d r
\]
But note that $w^{\alpha + \lfloor \alpha \rfloor + 6 -2j }= w^{\alpha +k + 2+i -j + (4+ \lfloor \alpha \rfloor -k -j) } \lesssim w^{\alpha +k + 2+i -j }$ because $j\leq i \leq 4+ \lfloor \alpha \rfloor -k$. Furthermore, the total number of derivatives appearing is $k+2+ i \leq 6+\lfloor \alpha \rfloor =N $. Hence we obtain the desired bound. Other cases in \eqref{Linfty2} and \eqref{Linfty21} follow in the same way. \eqref{Linfty3} and \eqref{Linfty31} can be obtained similarly by applying \eqref{inftyru} instead of \eqref{infty0}. 
\end{proof}

\section{Local-in-time well-posedness}\label{A:LOCAL}

Let $\kappa>0$ be a sufficient small fixed number. In this section, we discuss the existence of $H$ solving \eqref{E:H1} in $[\kappa, T]$ with $S_\kappa^N(\tau)<\infty$ for all $\tau\in [\kappa, T]$ for some time $\kappa<T\leq 1$.

\begin{proposition}\label{T:LOCAL}

Let $1<\gamma<\frac43$, assume that the physical vacuum condition~\eqref{E:PHYSICALVACUUM} is satisfied,  and let $N=\lfloor \alpha\rfloor +6$. If $(H_0^\kappa,H_1^\kappa)$ satisfy $S_\kappa^N(H_0^\kappa ,H_1^\kappa )\le\sigma^2$, there exists a time $T=T(\sigma)>\kappa$ and a unique solution $\tau\to(H(\tau,\cdot))$ of the initial value problem~\eqref{E:H1} such that the map $[\kappa,T]\ni \tau\mapsto S_\kappa^N(\tau) $ is continuous and the solution satisfies the bound
\[
S_\kappa^N(\tau) \le  \tilde C
\]
where the constant $\tilde C$ depends only on $\ve$ and $\sigma$. 
\end{proposition}

\noindent
{\em Sketch of the proof of Proposition~\ref{T:LOCAL}.}
The proof of Proposition~\ref{T:LOCAL} follows by 
the well-posedness proof for the compressible Euler system of Jang \& Masmoudi~\cite{JaMa2009,JaMa2015}. The argument of \cite{JaMa2009,JaMa2015} will render the existence theory based on an appropriate approximate scheme and  a priori bounds. To apply the result of \cite{JaMa2009,JaMa2015}, we will design the approximate scheme for $\mathfrak H:= D_r H$ and $H$ from $\mathfrak H$. We construct $j^{\text{th}}$ approximations $(\mathfrak H_j,\pa_\tau \mathfrak H_j)$ and $(H_{j}, \pa_\tau H_{j})$ as follows. The first approximation for $j=1$, we use the initial data: let $(\mathfrak H_{1}, \pa_\tau \mathfrak H_1 )=(D_rH_0^\kappa, D_r H_1^\kappa)$ and $(H_1, \pa_\tau H_1) = (H_0^\kappa, H_1^\kappa)$ where $S_\kappa^N(H_0^\kappa ,H_1^\kappa )\le\sigma^2$. Inductively we obtain the approximate solutions $(j+1)^{\text{th}}$ approximations as follows: For $j\geq 1$, let $(\mathfrak H_{j+1}, \pa_\tau \mathfrak H_{j+1})$ be the solution to the initial value problem for the following linear PDE: 
\begin{align}
&\pa_\tau^2\mathfrak H_{j+1} + 2\frac{g^{01}_j}{g^{00}_j}\pa_r  \pa_\tau  \mathfrak H_{j+1} +\frac{2m}{g^{00}_j}  \frac{ \partial_\tau \mathfrak H_{j+1}}{\tau} +\frac{d^2}{g^{00}_j}\frac{\mathfrak H_{j+1}}{\tau^2} 
+ \ve\gamma  \frac{c[\phi_j]}{g^{00}_j}  \mathcal L_{1+\alpha} \mathfrak H_{j+1} \notag \\
&  = \mathcal D_1 \left( \frac1{g^{00}_j}\left(\mathscr S(\phia) -\ve \mathscr L_{\text{low}} H_j + \mathscr N [H_j]\right)\right) + \mathcal C_1[H_j] +  \mathscr M[H_j], \label{E:approx}\\
& (\mathfrak H_{j+1}, \pa_\tau \mathfrak H_{j+1}) |_{\tau=\kappa}= (D_rH_0^\kappa, D_r H_1^\kappa). \notag 
\end{align}
Note that the schemes mimic the equation \eqref{E:DiHP} for $i=1$. The subscript $j$ implies that the coefficients appearing in the expression are evaluated by using $H_j$, $\pa_\tau H_j$. With the bounds $S^N_\kappa (H_{j}, \pa_\tau H_j)<\infty$ depending only on $\ve$ and $\sigma$, the existence of $(\mathfrak H_{j+1}, \pa_\tau \mathfrak H_{j+1})$ follows from the duality argument  in \cite{JaMa2009,JaMa2015}. By defining $H_{j+1}$ by
\[
H_{j+1} = \frac{1}{r^2}\int_0^r \mathfrak H_{j+1}  (r')^2 dr'
\]
and based on a priori estimates, we also deduce  $S^N_\kappa (H_{j+1}, \pa_\tau H_{j+1})<\infty$ whose bound depends only on $\ve$ and $\sigma$. As $j\rightarrow\infty$, after extracting a subsequence, we obtain the limit   $(H,\pa_\tau H)$ of $(H_{j}, \pa_\tau H_{j})$ that solves \eqref{E:H1} in $[\kappa, T]$ for some $T=T(\sigma)>\kappa$ with $S^N_\kappa (H, \pa_\tau H)<\infty$.

\end{document}